\documentclass[11pt,notitlepage]{amsart} 
\usepackage[utf8]{inputenc}
\usepackage[all,cmtip]{xy}
\usepackage{amssymb,amsmath,amsthm}
\usepackage{mathtools,mathdots}
\usepackage{fancyhdr} 
\usepackage{mathrsfs,calligra, bbm}
\usepackage[top=1.5in, bottom=1.5in, left=1in, right=1in]{geometry}
\usepackage{graphicx}
\usepackage{CJKutf8}
\usepackage{dirtytalk}
\usepackage{hyperref}
\usepackage{tikz-cd}
\usepackage{tikz}
\usetikzlibrary{decorations.pathmorphing}
\usepackage{subcaption}
\usepackage{enumitem}
\usepackage{array}
\usepackage[toc,page]{appendix}
\usepackage{stfloats}
\usepackage[T1]{fontenc}
\usepackage{mathrsfs}
\DeclareMathAlphabet{\mathpzc}{OT1}{pzc}{m}{it}
\usepackage{lipsum}

\DeclareMathAlphabet{\mathpzc}{OT1}{pzc}{m}{it}

\usepackage{amsfonts}
\usepackage{scalerel}

\newcounter{oldtocdepth}

\let\emptyset\varnothing

\theoremstyle{definition}
\newtheorem{Definition}{Definition}[subsection]
\newtheorem{Example}[Definition]{Example}

\newtheorem{Proposition-Definition}[Definition]{Proposition-Definition}

\newtheorem{Remark}[Definition]{Remark}
\newtheorem*{Remark*}{Remark}

\newtheorem{Theorem}[Definition]{Theorem}
\newtheorem{Lemma}[Definition]{Lemma}
\newtheorem{Proposition}[Definition]{Proposition}
\newtheorem{Corollary}[Definition]{Corollary}

\newtheorem{Theorem*}{Theorem}[]
\newtheorem*{hyp}{Hypothesis}

\DeclareMathOperator\A{\mathbf{A}}
\DeclareMathOperator\B{\mathbf{B}}
\DeclareMathOperator\C{\mathbf{C}}

\DeclareMathOperator\F{\mathbf{F}\!}

\DeclareMathOperator\Q{\mathbf{Q}}
\DeclareMathOperator\R{\mathbf{R}}
\DeclareMathOperator\T{\mathbf{T}}

\DeclareMathOperator\V{\mathbf{V}}

\DeclareMathOperator\Z{\mathbf{Z}}

\DeclareMathOperator\OC{\mathbf{O\!C}}

\DeclareMathOperator\bfu{\mathbf{u}}

\DeclareMathOperator\bfits{\textbf{\textit{s}}}
\DeclareMathOperator\bfitu{\textbf{\textit{u}}}

\DeclareMathOperator\bfitx{\textbf{\textit{x}}}
\DeclareMathOperator\bfitz{\textbf{\textit{z}}}
\DeclareMathOperator\bfitm{\textbf{\textit{m}}}

\DeclareMathOperator\bbA{\mathbb{A}}
\DeclareMathOperator\bbB{\mathbb{B}}
\DeclareMathOperator\bbD{\mathbb{D}}
\DeclareMathOperator\bbG{\mathbb{G}}
\DeclareMathOperator\bbH{\mathbb{H}}

\DeclareMathOperator\bbT{\mathbb{T}}

\DeclareMathOperator\calA{\mathcal{A}}
\DeclareMathOperator\calB{\mathcal{B}}
\DeclareMathOperator\calC{\mathcal{C}}

\DeclareMathOperator\calE{\mathcal{E}}

\DeclareMathOperator\calG{\mathcal{G}}
\DeclareMathOperator\calH{\mathcal{H}}

\DeclareMathOperator\calM{\mathcal{M}}

\DeclareMathOperator\calO{\mathcal{O}}

\DeclareMathOperator\calS{\mathcal{S}}
\DeclareMathOperator\calT{\mathcal{T}}
\DeclareMathOperator\calU{\mathcal{U}}
\DeclareMathOperator\calV{\mathcal{V}}
\DeclareMathOperator\calW{\mathcal{W}}
\DeclareMathOperator\calX{\mathcal{X}}
\DeclareMathOperator\calY{\mathcal{Y}}
\DeclareMathOperator\calZ{\mathcal{Z}}

\DeclareMathOperator\scrC{\mathscr{C}}
\DeclareMathOperator\scrD{\mathscr{D}}
\DeclareMathOperator\scrE{\mathscr{E}}
\DeclareMathOperator\scrF{\mathscr{F}}
\DeclareMathOperator\scrG{\mathscr{G}}
\DeclareMathOperator\scrH{\mathscr{H}}
\DeclareMathOperator\scrI{\mathscr{I}}

\DeclareMathOperator\scrK{\mathscr{K}}

\DeclareMathOperator\scrM{\mathscr{M}}

\DeclareMathOperator\scrO{\mathscr{O}}

\DeclareMathOperator\scrS{\mathscr{S}}
\DeclareMathOperator\scrT{\mathscr{T}}

\DeclareMathOperator\scrV{\mathscr{V}}
\DeclareMathOperator\scrW{\mathscr{W}}

\DeclareMathOperator\frakC{\mathfrak{C}}

\DeclareMathOperator\frakG{\mathfrak{G}}

\DeclareMathOperator\frakS{\mathfrak{S}}

\DeclareMathOperator\frakX{\mathfrak{X}}
\DeclareMathOperator\frakY{\mathfrak{Y}}
\DeclareMathOperator\frakZ{\mathfrak{Z}}

\DeclareMathOperator\fraka{\mathfrak{a}}

\DeclareMathOperator\frakm{\mathfrak{m}}

\DeclareMathOperator\fraks{\mathfrak{s}}

\DeclareMathOperator\frakz{\mathfrak{z}}

\DeclareMathOperator\GL{GL}
\DeclareMathOperator\GSp{GSp}

\DeclareMathOperator\Aut{Aut}
\DeclareMathOperator\End{End}
\DeclareMathOperator\Hom{Hom}

\DeclareMathOperator\sheafIm{\scrI\!\!\!\textit{m}}
\DeclareMathOperator\sheafKer{\scrK\!\!\!\!\textit{er}}
\DeclareMathOperator\sheafOC{\scrO\!\!\scrC}
\DeclareMathOperator\sheafOD{\scrO\!\!\scrD}
\DeclareMathOperator\sheafOV{\scrO\!\!\scrV}

\DeclareMathOperator\adicFL{\mathcal{F}\!\mathbf{\ell}}

\DeclareMathOperator\adicIW{\mathcal{IW}}

\DeclareMathOperator\AIP{AIP}
\DeclareMathOperator\AL{AL}
\DeclareMathOperator\alg{alg}

\DeclareMathOperator\an{an}

\DeclareMathOperator\can{can}

\DeclareMathOperator\cl{cl}

\DeclareMathOperator\Cov{Cov}

\DeclareMathOperator\cts{cts}
\DeclareMathOperator\cusp{cusp}

\DeclareMathOperator\diag{diag}

\DeclareMathOperator\ES{ES}
\DeclareMathOperator\EScal{\mathcal{E}\!\mathcal{S}}
\DeclareMathOperator\et{\textrm{\'{e}t}}
\DeclareMathOperator\fet{\textrm{f\'{e}t}}
\DeclareMathOperator\Fil{Fil}

\DeclareMathOperator\fket{\textrm{fk\'{e}t}}
\DeclareMathOperator\Fl{Fl}

\DeclareMathOperator\Gal{Gal}

\DeclareMathOperator\Ha{Ha}

\DeclareMathOperator\hst{hst}
\DeclareMathOperator\HT{HT}
\DeclareMathOperator\id{id}

\DeclareMathOperator\image{image}
\DeclareMathOperator\Isom{Isom}
\DeclareMathOperator\Iw{Iw}
\DeclareMathOperator\ket{\textrm{k\'{e}t}}

\DeclareMathOperator\Lie{Lie}

\DeclareMathOperator\mf{mf}

\DeclareMathOperator\oc{oc}

\DeclareMathOperator\one{\mathbbm{1}}
\DeclareMathOperator\oneanti{\breve{\one}}
\DeclareMathOperator\opp{opp}
\DeclareMathOperator\Par{par}
\DeclareMathOperator\pr{pr}
\DeclareMathOperator\proet{\textrm{pro\'{e}t}}

\DeclareMathOperator\proket{\textrm{prok\'{e}t}}

\DeclareMathOperator\rig{rig}
\DeclareMathOperator\Res{Res}

\DeclareMathOperator\Sen{Sen}

\DeclareMathOperator\Siegel{Siegel}
\DeclareMathOperator\Spa{Spa}

\DeclareMathOperator\Spec{Spec}
\DeclareMathOperator\Spf{Spf}

\DeclareMathOperator\Sym{Sym}

\DeclareMathOperator\std{std}

\DeclareMathOperator\tol{tol}

\DeclareMathOperator\Tr{Tr}

\DeclareMathOperator\trans{^{\mathtt{t}}\!}

\DeclareMathOperator\univ{univ}

\DeclareMathOperator\wt{wt}

\DeclareMathOperator\Alg{\textbf{\textsc{Alg}}}
\DeclareMathOperator\Groups{\textbf{\textsc{Group}}}

\DeclareMathOperator\Sets{\textbf{\textsc{Sets}}}
\DeclareMathOperator\FSets{\textbf{\textsc{FSets}}}

\DeclareMathOperator\Mod{\textbf{\textsc{Mod}}}

\DeclareMathOperator\bfalpha{\boldsymbol{\alpha}}
\DeclareMathOperator\bfbeta{\boldsymbol{\beta}}
\DeclareMathOperator\bfgamma{\boldsymbol{\gamma}}
\DeclareMathOperator\bfdelta{\boldsymbol{\delta}}

\DeclareMathOperator\bflambda{\boldsymbol{\lambda}}

\DeclareMathOperator\bfnu{\boldsymbol{\nu}}

\DeclareMathOperator\bfsigma{\boldsymbol{\sigma}}
\DeclareMathOperator\bftau{\boldsymbol{\tau}}
\DeclareMathOperator\bfupsilon{\boldsymbol{\upsilon}}

\DeclareMathOperator\llbrack{\![\![\!}
\DeclareMathOperator\rrbrack{\!]\!]}

\DeclareMathOperator\bla{\boldsymbol{\langle}}
\DeclareMathOperator\bra{\boldsymbol{\rangle}}

\title{Perfectoid overconvergent Siegel modular forms and the overconvergent Eichler--Shimura morphism}
\author{Hansheng Diao, Giovanni Rosso, and Ju-Feng Wu}

\date{}

\begin{document}

\begin{abstract}
The aim of this paper is twofold. We first present a construction of the overconvergent automorphic sheaves for Siegel modular forms by generalising the perfectoid method, originally introduced by Chojecki--Hansen--Johansson for automorphic forms on compact Shimura curves over $\Q$. The global sections of these automorphic sheaves are precisely the overconvergent Siegel modular forms. In particular, one can compare these automorphic sheaves with the ones constructed by Andreatta--Iovita--Pilloni. Secondly, we establish an (explicit) overconvergent Eichler--Shimura morphism for Siegel modular forms, generalising the result of Andreatta--Iovita--Stevens for the elliptic modular forms. 
\end{abstract}

\maketitle
\thispagestyle{empty}
\tableofcontents

\section{Introduction}
A classical result, due to M. Eichler and G. Shimura, states that the first cohomology of the complex modular curve with coefficients in $\mathrm{Sym}^{k-2}\Q^2$, after scalar extension to $\C$, admits a Hecke-equivariant decomposition as the direct sum of the space of weight $k$ holomorphic modular forms and the space of weight $k$ anti-holomorphic cusp forms.

G. Faltings establishes a $p$-adic analogue of the Eichler--Shimura decomposition. Let $p$ be a prime number and let $\C_p$ be the completion of a fixed algebraic closure $\overline{\Q}_p$ of $\Q_p$. Suppose $X$ is the modular curve (of some tame level $N$) over $\Q_p$ and let $\overline{X}$ be its compactification. Let $\pi: E\rightarrow \overline{X}$ be the universal generalised elliptic curve with identity section $e$ and let $\underline{\omega}:=e^*\Omega^1_{E/\overline{X}}$. In \cite{FaltingsHT}, Faltings constructs a Hecke- and $\mathrm{Gal}(\overline{\Q}_p/\Q_p)$-equivariant decomposition
\[
H_{\et}^1(X_{\overline{\Q}_p},\mathrm{Sym}^{k}\Q_p^2)\otimes_{\Q_p} \C_p(1) \simeq \big(H^0(\overline{X},\underline{\omega}^{k+2})\otimes_{\Q_p} \C_p\big) \oplus \big(H^1(\overline{X},\underline{\omega}^{-k})\otimes_{\Q_p} \C_p (k+1)\big) 
\]
where the Galois actions on the coherent cohomology groups are trivial. 

An analogue of Faltings' result for overconvergent modular forms is established by F. Andreatta, A. Iovita and G. Stevens in \cite{AIS-2015} and later extended to the case of compact Shimura curves over $\Q$ by P. Chojecki, D. Hansen and C. Johansson in \cite{CHJ-2017}. The novelty of the second work consists of a ``perfectoid construction'' of families of overconvergent automorphic forms as well as the usage of the pro-\'etale site. The same method is also adapted to the cases of elliptic and Hilbert modular forms in \cite{BHW-2019}.

There are two main goals in this paper. We first construct the automorphic sheaves for overconvergent Siegel modular forms (of genus $g$) using the perfectoid method. Secondly, we establish an overconvergent Eichler--Shimura morphism for Siegel modular forms, generalising the results of \cite{AIS-2015} and \cite{CHJ-2017}.

\subsection{Overconvergent automorphic sheaves}

Compared with the aforementioned works, one of the new ingredients in this paper is the {\it toroidally compactified} perfectoid Siegel modular variety $\overline{\calX}_{\Gamma(p^{\infty})}$ at the infinite level studied in \cite{Pilloni-Stroh-CoherentCohomologyandGaloisRepresentations}. This perfectoid space is equipped with the \emph{Hodge-Tate period map} $\pi_{\HT}: \overline{\calX}_{\Gamma(p^{\infty})}\rightarrow \adicFL$ where $\adicFL$ is the flag variety parameterising maximal lagrangian subspaces of a fixed symplectic space of rank $2g$. We also consider the (toroidally compactified) Siegel modular variety $\overline{\calX}_{\Iw^+}$ at the \emph{strict Iwahori level} (see Definition \ref{Definition: Siegel modular varieties of (strict) Iwahoris level} for details). The strict Iwahori level here is a certain deeper level compared with the usual Iwahori level.\footnote{ The use of the strict Iwahori level is due to a technicality in our construction of the overconvergent Eichler--Shimura morphisms. Our discussions in \S \ref{section:constructionsheaf} and \S \ref{section:overconvergentcohomologies} can also be carried out using the usual Iwahori level. } Moreover, in order to investigate the overconvergent Siegel modular forms, we consider certain open subspaces $\overline{\calX}_{\Gamma(p^{\infty}), w}$ and $\overline{\calX}_{\Iw^+, w}$ of $\overline{\calX}_{\Gamma(p^{\infty})}$ and $\overline{\calX}_{\Iw^+}$, respectively. They are strict neighborhoods of the usual ordinary loci and are referred to as the \emph{$w$-ordinary loci} (see Definition \ref{Definition: w-ordinary}).

We briefly describe our construction of the overconvergent automorphic sheaves. Let $(R_{\calU}, \kappa_{\calU})$ be a weight (see Definition \ref{Definition: weights}) and let $w\in \Q_{>0}$ such that $w>1+r_{\calU}$ for some integer $r_{\calU}$ defined in Definition \ref{Definition: w-analytic weight}. One can think of $(R_{\calU}, \kappa_{\calU})$ as a family of $p$-adic weights. Given these data, we construct a sheaf $\underline{\omega}_w^{\kappa_{\calU}}$ over $\overline{\calX}_{\Iw^+, w}$ whose global sections are precisely the $w$-overconvergent Siegel modular forms of strict Iwahori level and weight $\kappa_{\calU}$. The construction goes as follows: consider the natural projection $\overline{\calX}_{\Gamma(p^{\infty})}\rightarrow \overline{\calX}_{\Iw^+}$ which is a Galois cover with Galois group equal to the strict Iwahori subgroup $\Iw_{\GSp_{2g}}^+$ of $\GSp_{2g}(\Z_p)$, then the sections of $\underline{\omega}_w^{\kappa_{\calU}}$ consist of functions $f$ on $\overline{\calX}_{\Gamma(p^{\infty}),w}$ which 
\begin{enumerate}
    \item[$\bullet$] take value in a certain analytic representation $C_{\kappa_{\calU}}^{w-\an}(\Iw_{\GL_g}, \C_p\widehat{\otimes}R_{\calU})$ of the Iwahori subgroup $\Iw_{\GL_g}$ of $\GL_g(\Z_p)$, and
    \item[$\bullet$] satisfy the following condition with respect to the action of $\Iw_{\GSp_{2g}}^+$:
    \[
        \bfgamma^* f=\rho_{\kappa_{\calU}}(\bfgamma_a+\frakz\bfgamma_c)^{-1}f\qquad\text{for any} \qquad\bfgamma=\begin{pmatrix}\bfgamma_a & \bfgamma_b\\ \bfgamma_c & \bfgamma_d\end{pmatrix}\in \Iw^+_{\GSp_{2g}},
    \]
    where $\frakz$ stands for the pullback of the coordinate function on the flag variety $\adicFL$ along $\pi_{\HT}$ and  $\rho_{\kappa_{\calU}}(\bfgamma_a+\frakz\bfgamma_c)$ stands for a certain automorphism on $C_{\kappa_{\calU}}^{w-\an}(\Iw_{\GL_g}, \C_p\widehat{\otimes}R_{\calU})$ (see Definition \ref{Definition: strict Iwahori action on w-analytic representation for IwGL_g}). 
\end{enumerate}

When $p>2g$, we show that $\underline{\omega}_w^{\kappa_{\calU}}$ coincides with (the pullback to the strict Iwahori level of) the automorphic sheaf constructed in \cite{AIP-2015}. See \S \ref{section:constructionsheaf} for a complete story.

\begin{Remark}\label{Remark: analogy to the complex version}
\normalfont Our definition yields a strong analogy to the complex theory of classical Siegel modular forms!

Suppose $k=(k_1, \ldots, k_g)\in \Z_{\geq 0}^g$ is a dominant weight for $\GL_g$ and let $\rho_k:\GL_g(\C)\rightarrow \GL(\V_k)$ be the corresponding irreducible representation of $\GL_g$ of highest weight $k$. Let $\bbH_g^+$ be the (complex) Siegel upper-half space. Then a classical Siegel modular form of weight $k$ and level $\Gamma$ is a holomorphic function $f:\bbH_g^+\rightarrow \V_k$ such that
$$f(\bfgamma\cdot \bfitx)=\rho_k(\bfgamma_c\bfitx+\bfgamma_d)f(\bfitx)$$
for all $\bfitx\in \bbH_g^+$ and $\bfgamma=\begin{pmatrix}\bfgamma_a & \bfgamma_b\\ \bfgamma_c & \bfgamma_d\end{pmatrix}\in \Gamma\subset \GSp_{2g}(\Z)$.

In our case, a $w$-overconvergent Siegel modular form $f$ can be viewed as a function 
$$f: \overline{\calX}_{\Gamma(p^{\infty}), w}\rightarrow C^{w-\an}_{\kappa_{\calU}}(\Iw_{\GL_g}, \C_p\widehat{\otimes}R_{\calU})$$ 
satisfying $$f(\bfitx\cdot \bfgamma)=\rho_{\kappa_{\calU}}(\bfgamma_a + \frakz\bfgamma_c)^{-1}f(\bfitx)$$
for all $\bfitx\in \overline{\calX}_{\Gamma(p^{\infty}), w}$ and $\bfgamma=\begin{pmatrix}\bfgamma_a & \bfgamma_b\\ \bfgamma_c & \bfgamma_d\end{pmatrix}\in \Iw_{\GSp_{2g}}^+\subset \GSp_{2g}(\Z_p)$. Notice that $C^{w-\an}_{\kappa_{\calU}}(\Iw_{\GL_g}, \C_p\widehat{\otimes}R_{\calU})$ is an analytic analogue of the algebraic representation $\V_k$ and $\rho_{\kappa_{\calU}}(\bfgamma_a + \frakz\bfgamma_c)^{-1}$ is an analogue of the automorphy factor $\rho_k(\bfgamma_c\bfitx+\bfgamma_d)$.

In fact, it is possible to modify $\rho_{\kappa_{\calU}}(\bfgamma_a + \frakz\bfgamma_c)^{-1}$ into $\rho_{\kappa_{\calU}}(\bfgamma_c \frakz+\bfgamma_d)$, which yields a perfect match with the classical theory, if we choose to work with the left action of $\GSp_{2g}$ on the Siegel modular variety. However, we eventually work with the right action in order to be compatible with  popular literature such as \cite{Boxer--Pilloni--higherColeman}.
\end{Remark}

\subsection{The (explicit) overconvergent Eichler--Shimura morphisms}

The second goal of the paper is to construct the overconvergent Eichler--Shimura morphisms which relates the so-called overconvergent cohomology groups with the space of overconvergent modular forms.

To describe the overconvergent cohomology groups, we restrict our attention to a \emph{small} weight $(R_{\calU}, \kappa_{\calU})$ in the sense of Definition \ref{Definition: weights}. For such a weight, we consider a space $D^r_{\kappa}(\T_0, R_{\calU})$ of analytic distributions on a certain $p$-adic manifold $\T_0$. (See \S \ref{subsection:continuousfunctions} for the precise definitions.) This space of distributions is equipped with a natural action of the strict Iwahori subgroup and hence gives rise to a sheaf $\scrD_{\kappa_{\calU}}^r$ on the Siegel modular variety $\overline{\calX}_{\Iw^+}$. The cohomology groups of $\scrD_{\kappa_{\calU}}^r$ are referred as the \emph{overconvergent cohomology groups}, also known as the \emph{overconvergent modular symbols} studied by A. Ash and G. Stevens, among others (see for example \cite{Ash-Stevens, Hansen-PhD, Johansson-Newton}).

The goal is to construct a natural morphism from an overconvergent cohomology group to the space of overconvergent Siegel modular forms. The morphism should be compatible with Eichler--Shimura morphisms at classical weights and should interpolate in $p$-adic families. To this end, we compute both objects on the so-called \emph{pro-Kummer \'etale site}. Our situation is very different from previous works in the literature, such as \cite{CHJ-2017}, where the authors work exclusively with compact Shimura curves which have no \emph{boundary issues}. The presence of the boundaries in the case of Siegel modular varieties introduces several technical difficulties. Nonetheless, we obtain two sheaves $\sheafOD_{\kappa_{\calU}}^r$ and $\widehat{\underline{\omega}}_{w}^{\kappa_{\calU}}$ on the pro-Kummer \'etale site, coming from the overconvergent cohomology group and the automorphic sheaf, respectively.

The final and the key step is in \S \ref{subsection: OES} where we \emph{explicitly} construct a morphism of sheaves $\sheafOD_{\kappa_{\calU}}^r\rightarrow \widehat{\underline{\omega}}_{w}^{\kappa_{\calU}}$. This morphism can be viewed as an analytic analogue of the ``projection onto the highest weight vector'' and it is the main novelty of the paper compared with other works in the literature. We remark that the construction involves taking transposes of matrices while the usual Iwahori subgroup is not closed under taking transposes. This is why we have to work with a smaller subgroup---the strict Iwahori subgroup---as a compromise.

Putting everything together, we obtain the \emph{overconvergent Eichler--Shimura morphism} \[
    \ES_{\kappa_{\calU}}: H_{\proket}^{n_0}(\overline{\calX}_{\Iw^+}, \sheafOD_{\kappa_{\calU}}^r) \rightarrow H^0(\overline{\calX}_{\Iw^+, w}, \,\underline{\omega}_w^{\kappa_{\calU}+g+1})(-n_0),
\] where $n_0 = \dim_{\C_p} \overline{\calX}_{\Iw^+}$ and ``$(-n_0)$'' stands for the Tate twist. It is a Hecke- and Galois-equivariant morphism from the ($p$-adic family of) overconvergent cohomology groups to the ($p$-adic family of) overconvergent Siegel modular forms.

We list some properties of this map. First of all, it is compatible with base changes on the weights. Secondly, we are able to control its image when specialising to a dominant algebraic weight $k\in \Z_{\geq 0}^{g}$. In particular, we show in Theorem \ref{thm:imageclassicalweight} that the image of $\ES_{k}$ is contained in the space of classical algebraic Siegel modular forms. The proof uses the fact that when the weight is a dominant algebraic weight, the ``highest weight vector'' is an algebraic function. Lastly, we conclude the paper by showing that $\ES_{\kappa_{\calU}}$ can be glued to a morphism of sheaves on a suitable cuspidal eigenvariety.

A very interesting by-product of the Eichler--Shimura morphisms is an explicit construction of the Galois representations associated with overconvergent Siegel modular forms. In Theorem \ref{thm:GaloisRepresentations}, under certain hypotheses on the eigenvariety and the system of Hecke-eigenvalues, we show that the middle degree cohomology of $\sheafOD_{\kappa_{\calU}}^r$ realises the $2^g$-dimensional Galois representations associated with (a $p$-adic family of) overconvergent Siegel modular forms. These Galois representations are usually obtained using deformation of pseudo-representations. However, a direct construction like ours is more useful for arithmetic applications, for example in the construction of $p$-adic $L$-functions as in \cite{LoefflerZerbesColeman}.

There are many things we don't do: we don't control the cokernel of the map $\ES_{\kappa_{\calU}}$ in general and we do not investigate the kernel of this map at all. A natural expectation is to construct a full filtration of $H_{\proket}^{n_0}(\overline{\calX}_{\Iw^+}, \sheafOD_{\kappa_{\calU}}^r) $ using higher Coleman theory, recently developed by G. Boxer and V. Pilloni \cite{Boxer--Pilloni--higherColeman}, in terms of information on suitable strata of $\overline{\calX}_{\Iw^+}$. One result in this direction is recently announced by J. E. Rodr\'iguez Camargo (\cite{Rodriguez-BGG}) in terms of $p$-adic completed cohomology groups. Comparing with Camargo's work, our construction using overconvergent cohomology groups has the advantage that everything is explicit, which is useful for computational purpose. For example, in the case of $g=1$, an explicit overconvergent Eichler--Shimura morphism is used to deduce a complete proof of the Halo conjecture in \cite{DY-2023}.

We also remark that while preparing this paper, Andreatta and Iovita have posted the article \cite{AI-2020}, in which they upgrade their previous work \cite{AIS-2015} to an ``overconvergent de Rham Eichler--Shimura morphism'', meaning that their Eichler--Shimura map has as source the overconvergent cohomology groups tensored with $\bbB_{\mathrm{dR}}$. They have also announced upcoming work concerning this type of de Rham Eichler--Shimura morphisms for overconvergent Siegel modular forms of genus $2$. This suggests that the results in this paper can be upgraded to investigate finer $p$-adic Hodge theoretic properties of overconvergent cohomology groups; for example, the construction of de Rham (or even crystalline) periods in $p$-adic families. We shall leave this to future studies.

\subsection{Plan of the paper}

The paper is organised as follows. In \S \ref{section:PerfectoidSMV}, we define the main geometric objects of interest, including the Siegel modular varieties for various level structures, the perfectoid Siegel modular variety at infinite level, the flag variety, and the Hodge--Tate period map. The next section, \S \ref{section:constructionsheaf}, contributes to the construction of the overconvergent automorphic sheaves. In particular, when $p>2g$, we show that our construction coincides with the (pullback to the strict Iwahori level of the) automorphic sheaves of Andreatta--Iovita--Pilloni. We warn the reader that \S \ref{subsection: admissibility}-\ref{subsection:comparison sheaf aip} is the most technical part of the paper. These subsections can be skipped on a first reading. In \S \ref{section:overconvergentcohomologies}, we study the space of analytic distributions and the overconvergent cohomology groups. Finally, in \S \ref{section:EichlerShimura} and \S \ref{section:oncuspidaleigenvariety}, we construct the Eichler--Shimura morphism as well as its upgraded version on the cuspidal eigenvariety. 

In \S \ref{Section: Kummer etale and pro-Kummer etale sites of log adic spaces}, we recall the basics of logarithmic adic spaces and their pro-Kummer \'etale topology following \cite{Diao}. Also in the same section, we include some technical calculations of the derived functor $R^i \nu_*$ and a generalised projection formula, both of which are used in the main text and could be of general interest and utility. In \S \ref{section: boundary}, we recall the basics of the toroidal compactifications of Siegel modular varieties. We also review the ``modified integral structures'' studied in \cite{Pilloni-Stroh-CoherentCohomologyandGaloisRepresentations} as well as the construction of the perfectoid Siegel modular variety at the infinite level.

\section*{Acknowledgement} 
The authors would like to express their gratitude to Przemy{\l}aw Chojecki, David Hansen, and Christian Johansson for discussions on the perfectoid construction of sheaves of families of Siegel modular forms. We also thank David Hansen, Adrian Iovita, James Newton, and Ruishen Zhao for pointing out inaccuracies in the early versions of this paper. G.R. would like to thank Daniel Barrera-Salazar and Riccardo Brasca. J.-F.W. would like to thank Ulrich G\"ortz for helpful conversations about toroidal compactifications of Siegel modular varieties and $p$-divisible groups. He would also like to thank Marc Levine for a succinct introduction on $1$-motives during a casual conversation. During the preparation of the article, H.D. was partially supported by the National Natural Science Foundation of China (Grant No. 12422101) and the National Key R{\&}D Program of China (Grant Nos. 2023YFA1009703 and 2021YFA1000704); G.R. and J.-F.W. were partially supported by the Natural Sciences and Engineering Research Council of Canada (NSERC), funding reference number 2018-04392, and by the Fonds de recherche du Qu\'ebec Nature and technologies (FRQNT), grant 2019-NC-254031. While revising the paper, J.-F.W. was supported by the  ERC Consolidator grant ``Shimura varieties and the BSD conjecture'' and Taighde \'{E}ireann -- Research Ireland under Grant number IRCLA/2023/849 (HighCritical).

\section*{Conventions and notations} 
Throughout this paper, we fix the following: \begin{enumerate}
    \item[$\bullet$] $g\in \Z_{\geq 1}$.
    \item[$\bullet$] $p\in \Z_{> 0}$ is an odd prime number. Due to a certain technicality, we will have to assume $p>2g$ at some places. Such an assumption shall be clear in the context. 
    \item[$\bullet$] $N\in \Z_{\geq 3}$ is an integer coprime to $p$.
    \item[$\bullet$] We fix once and forever an algebraic closure $\overline{\Q}_p$ of $\Q_p$ and an algebraic isomorphism $\C_p\simeq \C$, where $\C_p$ is the $p$-adic completion of $\overline{\Q}_p$. We write $G_{\Q_p}$ for the absolute Galois group $\Gal(\overline{\Q}_p/\Q_p)$. We also fix the $p$-adic absolute value on $\C_p$ so that $|p|=p^{-1}$.
    \item[$\bullet$] For any $w\in \Q_{>0}$, we denote by ``$p^w$'' an element in $\C_p$ with absolute value $p^{-w}$. All constructions in the paper will not depend on such choices.
    \item[$\bullet$] We adopt the language of almost mathematics. In particular, for an $\calO_{\C_p}$-module $M$, we denote by $M^a$ for the associated almost $\calO_{\C_p}$-module.
    \item[$\bullet$] For $n\in \Z_{\geq 1}$ and any set $R$, we denote by $M_n(R)$ the set of $n$ by $n$ matrices with coefficients in $R$.
    \item[$\bullet$] The transpose of a matrix $\bfalpha$ is denoted by $\trans\bfalpha$.
    \item[$\bullet$] For any $n\in \Z_{\geq 1}$, we denote by $\one_n$ the $n\times n$ identity matrix and denote by $\oneanti_n$ the $n\times n$ anti-diagonal matrix whose non-zero entries are $1$; \emph{i.e.,} $$\one_n=\begin{pmatrix} 1& & \\ & \ddots & \\ & &1\end{pmatrix}\quad\text{ and }\quad\oneanti_n=\begin{pmatrix} & & 1\\ & \iddots & \\ 1 & &\end{pmatrix}$$
    \item[$\bullet$] We use $\cong$ to denote canonical isomorphisms and $\simeq$ to denote non-canonical ones.
    \item[$\bullet$] In principle (except for \S \ref{Section: Kummer etale and pro-Kummer etale sites of log adic spaces}), symbols in Gothic font (\emph{e.g.} $\frakX, \frakY, \frakZ$) stand for formal schemes; symbols in calligraphic font (\emph{e.g.} $\calX, \calY, \calZ$) stand for adic spaces; and symbols in script font (\emph{e.g.} $\scrO, \scrF, \scrE$) stand for sheaves (over various geometric objects). 
\end{enumerate}
\section{Siegel modular varieties and the Hodge--Tate period map}\label{section:PerfectoidSMV}
In this section, we introduce the Siegel modular varieties for various level structures, viewed as adic spaces, as well as their toroidal compactifications. We also recall the perfectoid Siegel modular variety introduced in \cite{Pilloni-Stroh-CoherentCohomologyandGaloisRepresentations} together with the Hodge--Tate period map. The notion of perfectoid Siegel modular variety has its root in \cite{Scholze-2015}. However, we point out that the author of \textit{loc. cit.} only considers minimal compactifications while it is important for us to work with the toroidal compactifications.

\subsection{\texorpdfstring{Algebraic and $p$-adic groups}{Algebraic and p-adic groups}}\label{subsection: Algebraic and p-adic groups}
We start with a list of algebraic and $p$-adic groups that will appear repeatedly throughout the paper. 

Let $V:=\Z^{2g}$ equipped with an alternating pairing \footnote{We choose to work with $\begin{pmatrix} & -\oneanti_g\\ \oneanti_g & \end{pmatrix}$ instead of $\begin{pmatrix} & -\one_g\\ \one_g & \end{pmatrix}$ because we prefer to work with Borel subgroups that are strictly upper-triangular.} $$\bla\cdot, \cdot\bra:V\times V\rightarrow \Z, \quad (\vec{v}, \vec{v}')\mapsto \vec{v}\begin{pmatrix} & -\oneanti_g\\ \oneanti_g & \end{pmatrix}\trans\vec{v}',$$ where we view elements in $V$ as row vectors. Let $e_1, ..., e_{2g}$ be the standard basis for $V$ so that $\vec{v}=(a_1, \cdots, a_{2g})\in V$ corresponds to $a_1e_1+\cdots+a_{2g}e_{2g}$. Then we have $$\bla e_i, e_j\bra=\left\{\begin{array}{ll}
    -1, & \text{if }i<j\text{ and } j=2g+1-i\\
    1, & \text{if }i>j\text{ and }j=2g+1-i\\
    0, & \text{else}
\end{array}\right. .$$ The group $\GL_{2g}$ acts on $V$ via right multiplication. We define the algebraic group $\GSp_{2g}$ to be the subgroup of $\GL_{2g}$ that preserves this pairing up to a unit. In other words, for any ring $R$, $$\GSp_{2g}(R):=\left\{\bfgamma\in \GL_{2g}(R): \bfgamma \begin{pmatrix} & -\oneanti_g\\ \oneanti_g\end{pmatrix}\trans\bfgamma = \varsigma(\bfgamma)\begin{pmatrix} & -\oneanti_g\\ \oneanti_g\end{pmatrix}\text{ for some }\varsigma(\bfgamma)\in R^{\times}\right\}.$$ Equivalently, for any $\bfgamma=\begin{pmatrix}\bfgamma_a & \bfgamma_b\\ \bfgamma_c & \bfgamma_d\end{pmatrix}\in \GL_{2g}$, $\bfgamma\in \GSp_{2g}$ if and only if \begin{equation}\label{eq: relation of blocks for GSp}
    \trans\bfgamma_a\oneanti_g\bfgamma_c=\trans\bfgamma_c\oneanti_g\bfgamma_a, \quad \trans\bfgamma_b\oneanti_g\bfgamma_d=\trans\bfgamma_d\oneanti_g\bfgamma_b, \text{ and }\trans\bfgamma_a\oneanti_g\bfgamma_d-\trans\bfgamma_c\oneanti_g\bfgamma_b=\varsigma(\bfgamma)\oneanti_g
\end{equation}for some $\varsigma(\bfgamma)\in \bbG_m$.

Taking base change to $\Z_p$, we consider $V_p:=V\otimes_{\Z}\Z_p$, equipped with the induced alternating pairing
$$\bla\cdot, \cdot\bra:V_p\times V_p\rightarrow \Z_p, \quad (\vec{v}, \vec{v}')\mapsto \vec{v}\begin{pmatrix} & -\oneanti_g\\ \oneanti_g & \end{pmatrix}\trans\vec{v}'.$$ Let $e_1, \ldots, e_{2g}$ be the standard basis for $V_p$ and let $\Fil^{\std}_{\bullet}$ denote the standard increasing filtration on $V_p$ defined by $\Fil^{\std}_0=0$ and $$\Fil^{\std}_i=\langle e_1, \ldots, e_i\rangle,$$ for $i=1, \ldots, 2g$.

The algebraic and $p$-adic subgroups of $\GL_g$ and $\GSp_{2g}$ considered in the present paper are the following:\begin{enumerate}
    \item[$\bullet$] For every $m\in \Z_{\geq 1}$, we write $$\Gamma(p^m):=\ker\left(\GSp_{2g}(\Z_p)\xrightarrow{\mod p^m}\GSp_{2g}(\Z/p^m\Z)\right).$$

    \item[$\bullet$] The Borel subgroups are \begin{align*}
        B_{\GL_g} & := \text{the Borel subgroup of upper triangular matrices in $\GL_g$},\\
        B_{\GSp_{2g}} & := \text{the Borel subgroup of upper triangular matrices in $\GSp_{2g}$.}
    \end{align*}
    
    \item[$\bullet$] The corresponding unipotent radicals are \begin{align*}
        U_{\GL_g} & := \text{ the upper triangular $g\times g$ matrices whose diagonal entries are all $1$},\\
        U_{\GSp_{2g}} & := \text{ the upper triangular $2g\times 2g$ matrices in $\GSp_{2g}$ whose diagonal entries are all $1$}.
    \end{align*} 
    
    \item[$\bullet$] The corresponding maximal tori for $\GL_g$ and $\GSp_{2g}$ are the maximal tori of diagonal matrices, denoted by $T_{\GL_g}$ and $T_{\GSp_{2g}}$, respectively. The Levi decomposition then yields $$B_{\GL_g}=T_{\GL_g}U_{\GL_g}\quad \text{ and }\quad B_{\GSp_{2g}}=T_{\GSp_{2g}}U_{\GSp_{2g}}.$$
    
    \item[$\bullet$] Let $B_{\GL_g}^{\opp}$ and $B_{\GSp_{2g}}^{\opp}$ be the opposite Borel subgroups of $B_{\GL_g}$ and $B_{\GSp_{2g}}$, respectively. They consist of lower triangular matrices of the corresponding algebraic groups. Similarly, $U_{\GL_g}^{\opp}$ and $U_{\GSp_{2g}}^{\opp}$ stand for the opposite unipotent radicals of $U_{\GL_g}$ and $U_{\GSp_{2g}}$, respectively. 
    
    \item[$\bullet$] To simplify the notations, we write $$T_{\GL_g, 0}=T_{\GL_g}(\Z_p),\quad U_{\GL_g, 0}=U_{\GL_g}(\Z_p), \quad B_{\GL_g, 0}=B_{\GL_g}(\Z_p),$$
$$T_{\GSp_{2g}, 0}=T_{\GSp_{2g}}(\Z_p),\quad U_{\GSp_{2g}, 0}=U_{\GSp_{2g}}(\Z_p),\quad B_{\GSp_{2g}, 0}=B_{\GSp_{2g}}(\Z_p).$$ 

The subgroups $B^{\opp}_{\GL_g,0}$, $B^{\opp}_{\GSp_{2g},0}$, $U^{\opp}_{\GL_g,0}$, and $U^{\opp}_{\GSp_{2g},0}$ are defined similarly.
 
For every $s\in \Z_{\geq 1}$, define \begin{align*}
        T_{\GL_g, s}:=\ker(T_{\GL_g}(\Z_p)\rightarrow T_{\GL_g}(\Z/p^s\Z)), & \qquad  T_{\GSp_{2g}, s}:=\ker(T_{\GSp_{2g}}(\Z_p)\rightarrow T_{\GSp_{2g}}(\Z/p^s\Z)),\\
        U_{\GL_g, s}:=\ker(U_{\GL_g}(\Z_p)\rightarrow U_{\GL_g}(\Z/p^s\Z)), & \qquad U_{\GSp_{2g}, s}:=\ker(U_{\GSp_{2g}}(\Z_p)\rightarrow U_{\GSp_{2g}}(\Z/p^s\Z)),\\
        B_{\GL_g, s}:=\ker(B_{\GL_g}(\Z_p)\rightarrow B_{\GL_g}(\Z/p^s\Z)), & \qquad  B_{\GSp_{2g}, s}:=\ker(B_{\GSp_{2g}}(\Z_p)\rightarrow B_{\GSp_{2g}}(\Z/p^s\Z)),
    \end{align*} where all of the maps are reduction modulo $p^s$. 
    
    The subgroups $B^{\opp}_{\GL_g,s}$, $B^{\opp}_{\GSp_{2g},s}$, $U^{\opp}_{\GL_g,s}$, and $U^{\opp}_{\GSp_{2g},s}$ are defined similarly.
    
    \item[$\bullet$] The Iwahori subgroups of $\GL_g(\Z_p)$ and $\GSp_{2g}(\Z_p)$ are \begin{align*}
        \Iw_{\GL_g} & := \text{ the preimage of $B_{\GL_g}(\F_p)$ under the reduction map $\GL_g(\Z_p)\rightarrow \GL_g(\F_p)$},\\
        \Iw_{\GSp_{2g}} & :=\text{ the preimage of $B_{\GSp_{2g}}(\F_p)$ under the reduction map $\GSp_{2g}(\Z_p)\rightarrow \GSp_{2g}(\F_p)$}.
    \end{align*} The Iwahori decomposition yields $$\Iw_{\GL_g}= U_{\GL_g, 1}^{\opp}T_{\GL_g, 0}U_{\GL_g, 0}\quad \text{ and }\quad \Iw_{\GSp_{2g}} =  U_{\GSp_{2g}, 1}^{\opp}T_{\GSp_{2g}, 0}U_{\GSp_{2g}, 0}.$$
    
    \item[$\bullet$] We consider the \textbf{\textit{strict Iwahori subgroups}} of $\GL_g(\Z_p)$ and $\GSp_{2g}(\Z_p)$ defined as \begin{align*}
        \Iw_{\GL_g}^+ & := \text{ the preimage of $T_{\GL_g}(\F_p)$ under the reduction map $\GL_g(\Z_p)\rightarrow \GL_g(\F_p)$},\\
        \Iw_{\GSp_{2g}}^+ & := \text{ the preimage of $T_{\GSp_{2g}}(\F_p)$ under the reduction map $\GSp_{2g}(\Z_p) \rightarrow \GSp_{2g}(\F_p)$.}
        %\Iw_{\GSp_{2g}}^{+} & := \left\{\bfgamma\in %\GSp_{2g}(\Z_p): \bfgamma\equiv \left(\begin{array}{ccc|ccc}
        %    * &&  & * &\cdots & * \\ & \ddots & & \vdots & & \vdots \\ && * & * & \cdots & *\\ \hline &&& * &  &  \\ &&&& \ddots &  \\ &&&&& * 
        %\end{array}\right)\mod p\right\},
    \end{align*} %where the blocks in the definition of $\Iw_{\GSp_{2g}}^+$ are all $(g\times g)$-matrices.  
    
    Clearly, $\Iw_{\GL_g}^+\subset \Iw_{\GL_g}$ and $\Iw_{\GSp_{2g}}^+\subset \Iw_{\GSp_{2g}}$. Also observe that, for any $\bfgamma= \begin{pmatrix}\bfgamma_a & \bfgamma_b\\ \bfgamma_c & \bfgamma_d\end{pmatrix}\in \Iw_{\GSp_{2g}}^+$, we have $\bfgamma_a\in \Iw_{\GL_g}^+$. Moreover, the Iwahori decomposition induces decompositions 
    \[
    \Iw_{\GL_g}^+ = U_{\GL_g, 1}^{\opp}T_{\GL_g, 0}U_{\GL_g, 1} \quad \text{ and }\quad \Iw_{\GSp_{2g}}^+ = U_{\GSp_{2g}, 1}^{\opp} T_{\GSp_{2g}, 0} U_{\GSp_{2g}, 1}.
    \] %where \( U_{\GSp_{2g}, 0}^+ := U_{\GSp_{2g}, 0}\cap \Iw_{\GSp_{2g}}^+\).

    \item[$\bullet$] Finally, we introduce the notion of ``$w$-neighbourhood'' of some aforementioned $p$-adic groups. For any $w\in \Q_{>0}$, define
\[T^{(w)}_{\GL_g, 0}:=\left\{ \bflambda=(\bflambda_{ij})_{i,j}\in T_{\GL_g}(\calO_{\C_p}):|\bflambda_{ij}-\bflambda'_{ij}|\leq p^{-w}\,\,\textrm{for some }\bflambda'=(\bflambda'_{ij})_{i,j}\in T_{\GL_g, 0}\right\},\]
\[U^{(w)}_{\GL_g, 0}:=\left\{ \bflambda=(\bflambda_{ij})_{i,j}\in U_{\GL_g}(\calO_{\C_p}):|\bflambda_{ij}-\bflambda'_{ij}|\leq p^{-w}\,\,\textrm{for some }\bflambda'=(\bflambda'_{ij})_{i,j}\in U_{\GL_g, 0}\right\},\]
\[B^{(w)}_{\GL_g, 0}:=\left\{ \bflambda=(\bflambda_{ij})_{i,j}\in B_{\GL_g}(\calO_{\C_p}):|\bflambda_{ij}-\bflambda'_{ij}|\leq p^{-w}\,\,\textrm{for some }\bflambda'=(\bflambda'_{ij})_{i,j}\in B_{\GL_g, 0}\right\}.\]
The groups $B^{\opp, (w)}_{\GL_g, 0}$ and $U^{\opp, (w)}_{\GL_g, 0}$ are defined similarly.

For every $s\in \Z_{\geq 1}$, define
\[T^{(w)}_{\GL_g, s}:=\left\{ \bflambda=(\bflambda_{ij})_{i,j}\in T_{\GL_g}(\calO_{\C_p}):|\bflambda_{ij}-\bflambda'_{ij}|\leq p^{-w}\,\,\textrm{for some }\bflambda'=(\bflambda'_{ij})_{i,j}\in T_{\GL_g, s}\right\},\]
\[U^{(w)}_{\GL_g, s}:=\left\{ \bflambda=(\bflambda_{ij})_{i,j}\in U_{\GL_g}(\calO_{\C_p}):|\bflambda_{ij}-\bflambda'_{ij}|\leq p^{-w}\,\,\textrm{for some }\bflambda'=(\bflambda'_{ij})_{i,j}\in U_{\GL_g, s}\right\},\]
\[B^{(w)}_{\GL_g, s}:=\left\{ \bflambda=(\bflambda_{ij})_{i,j}\in B_{\GL_g}(\calO_{\C_p}):|\bflambda_{ij}-\bflambda'_{ij}|\leq p^{-w}\,\,\textrm{for some }\bflambda'=(\bflambda'_{ij})_{i,j}\in B_{\GL_g, s}\right\}.\]
The groups $U^{\opp, (w)}_{\GL_g, s}$ and $B^{\opp, (w)}_{\GL_g, s}$ are defined similarly.

Define
\[\Iw^{(w)}_{\GL_g}:=\left\{ \bflambda=(\bflambda_{ij})_{i,j}\in \GL_g(\calO_{\C_p}):|\bflambda_{ij}-\bflambda'_{ij}|\leq p^{-w}\,\,\textrm{for some }\bflambda'=(\bflambda'_{ij})_{i,j}\in \Iw_{\GL_g}\right\}\]
and
\[\Iw^{+, (w)}_{\GL_g}:=\left\{ \bflambda=(\bflambda_{ij})_{i,j}\in \GL_g(\calO_{\C_p}):|\bflambda_{ij}-\bflambda'_{ij}|\leq p^{-w}\,\,\textrm{for some }\bflambda'=(\bflambda'_{ij})_{i,j}\in \Iw^+_{\GL_g}\right\}.\]
The Iwahori decomposition induces
$$\Iw_{\GL_g}^{(w)}=U_{\GL_g, 1}^{\opp,(w)}T_{\GL_g,0}^{(w)}U_{\GL_g,0}^{(w)}$$
and
$$\Iw_{\GL_g}^{+, (w)}=U_{\GL_g, 1}^{\opp,(w)}T_{\GL_g,0}^{(w)}U_{\GL_g,1}^{(w)}.$$

We also define
$$T_w=\ker(T_{\GL_g}(\calO_{\C_p})\rightarrow T_{\GL_g}(\calO_{\C_p}/p^w)),$$
$$U_w=\ker(U_{\GL_g}(\calO_{\C_p})\rightarrow U_{\GL_g}(\calO_{\C_p}/p^w)),$$
$$B_w=\ker(B_{\GL_g}(\calO_{\C_p})\rightarrow B_{\GL_g}(\calO_{\C_p}/p^w)).$$
The groups $U_w^{\opp}$ and $B_w^{\opp}$ are defined similarly.

Then we have 
$$T_{\GL_g,0}^{(w)}=T_{\GL_g,0}T_w,\quad U_{\GL_g,0}^{(w)}=U_{\GL_g,0}U_w,\quad B_{\GL_g,0}^{(w)}=B_{\GL_g,0}B_w.$$
There are similarly identities for $U_{\GL_g,0}^{\opp, (w)}$ and $B_{\GL_g,0}^{\opp, (w)}$.
\end{enumerate}

\subsection{Siegel modular varieties}\label{subsection: Siegel modular varieties} 
We consider Siegel modular varieties of genus $g$ (of principal tame level $N$) for various level structures at $p$.
\begin{Definition}\label{Definition: Siegel modular varieties of (strict) Iwahoris level}
\begin{enumerate}
\item[(i)] The \textbf{Siegel modular scheme} is the scheme $X_0$ over $\calO_{\C_p}$ that parameterises triples $(A, \lambda, \psi_N)$, where $(A, \lambda)$ is a principally polarised abelian scheme over $\calO_{\C_p}$ and $$\psi_N: V\otimes_{\Z}(\Z/N\Z)\xrightarrow[]{\sim} A[N]$$ 
is a symplectic isomorphism with respect to the pairing induced by $\bla\cdot, \cdot\bra$ on the left and the Weil pairing on the right. Let $X$ denote the base change of $X_0$ to $\C_p$.
\item[(ii)] For every $n\in \Z_{\geq 1}$, the \textbf{Siegel modular variety of principal $p^n$-level} is the algebraic variety $X_{\Gamma(p^n)}$ over $\C_p$ that parameterises quadruples $(A, \lambda, \psi_N, \psi_{p^n})$, where $(A, \lambda)$ is a principally polarised abelian variety over $\C_p$ and $$\psi_N: V\otimes_{\Z}(\Z/N\Z)\xrightarrow[]{\sim} A[N]$$ 
and
$$\psi_{p^n}:V\otimes_{\Z}(\Z/p^n\Z)\xrightarrow[]{\sim} A[p^n]$$
are symplectic isomorphisms.
 \item[(iii)] The \textbf{Siegel modular variety of Iwahori level} is the algebraic variety $X_{\Iw}$ over $\C_p$ that parameterises quadruples $(A, \lambda, \psi_N, \Fil_{\bullet}A[p])$, where $(A, \lambda, \psi_N)$ is as in (ii) and $\Fil_{\bullet}A[p]$ is a full flag of $A[p]$ that satisfies $$(\Fil_{\bullet}A[p])^{\perp}\cong \Fil_{2g-\bullet}A[p]$$ with respect to the Weil pairing.
    %\item[(iv)] The \textbf{Siegel modular variety of strict Iwahori level} is the algebraic variety $X_{\Iw^+}$ over $\C_p$ that parameterises quintuples $(A, \lambda, \psi_N, \Fil_{\bullet}A[p], \{C_i:i=1, ..., g\})$, where $(A, \lambda, \psi_N, \Fil_{\bullet}A[p])$ is as in (iii) and $\{C_i: i=1, ..., g\}$ is a collection of subgroups $C_i\subset A[p]$ of order $p$ such that $$\Fil_iA[p]=\langle C_1, \ldots, C_i\rangle$$ for all $i=1, \ldots, g$.
    \item[(iv)] The \textbf{Siegel modular variety of strict Iwahori level} is the algebraic variety $X_{\Iw^+}$ over $\C_p$ that parameterises quadruples $(A, \lambda, \psi_N, \{C_i:i=1, ..., g\})$, where $(A, \lambda, \psi_N)$ is as in (i) and $\{C_i: i=1, ..., g\}$ is a collection of subgroups $C_i \subset A[p]$ of order $p$ such that \[
        C_i \cap C_j = 0
    \] for any $i\neq j$.
\end{enumerate}
\end{Definition}

For $\Gamma \in \{\Gamma(p^n), \Iw^+ = \Iw_{\GSp_{2g}}^+, \Iw = \Iw_{\GSp_{2g}}, \emptyset\}$, it is well-known that the $\C$-points of the algebraic variety $X_{\Gamma}$ can be identified with the locally symmetric space \[
    X_{\Gamma}(\C) = \GSp_{2g}(\Q)\backslash \GSp_{2g}(\A_f)\times \bbH_g/\Gamma(N)\cdot\Gamma,
\] where \begin{enumerate}
    \item[$\bullet$] $\A_f$ is the ring of finite ad\`eles of $\Q$;
    \item[$\bullet$] $\bbH_g$ is the disjoint union of the Siegel upper- and lower-half spaces;
    \item[$\bullet$] $\Gamma(N) = \{\bfgamma \in \GSp_{2g}(\widehat{\Z}): \bfgamma \equiv 1 \mod N\}$.
\end{enumerate} Here, we use the fixed isomorphism $\C_p \simeq \C$ to view $\C$ as a $\C_p$-algebra.

\begin{Remark}\label{Remark: moduli problem for the strict Iwahori level}
    For $\Gamma \in \{\Gamma(p^n), \Iw, \emptyset\}$, the above identification is well-known. To justify the moduli problem in Definition \ref{Definition: Siegel modular varieties of (strict) Iwahoris level} (iv) corresponds to the level $\Iw_{\GSp_{2g}}^+$, observe that the automorphisms on $A[p]$ preserving the Weil pairing can be identified with $\GSp_{2g}(\F_p)$. The subgroup fixing the prescribed order-$p$ subgroups $\{C_i: i=1, ...g\}$ is then, up to conjugation, isomorphic to \[
        \overline{\Gamma} = \left\{ \bfgamma = \begin{pmatrix} \bfgamma_a & \\  & \bfgamma_d \end{pmatrix} \in \GSp_{2g}(\F_p): \bfgamma_a \in T_{\GL_g}(\F_p) \right\}.
    \]
    However, by the relations \eqref{eq: relation of blocks for GSp}, one sees that, for $\bfgamma \in \overline{\Gamma}$, $\bfgamma_d\in T_{\GL_g}(\F_p)$ and so $\overline{\Gamma} = T_{\GSp_{2g}}(\F_p)$.
\end{Remark}

Moreover, we have a chain of forgetful maps
$$X_{\Gamma(p^n)}\rightarrow X_{\Gamma(p)}\rightarrow X_{\Iw^+}\rightarrow X_{\Iw}\rightarrow X$$
with the arrows described as follows:
\begin{enumerate}
\item[$\bullet$] The first arrow sends $(A, \lambda, \psi_N, \psi_{p^n})$ to $(A, \lambda, \psi_N, p^{n-1}\psi_{p^n})$.
\item[$\bullet$] The second arrow sends $(A, \lambda, \psi_N, \psi_p)$ to $(A, \lambda, \psi_N, \{\langle \psi_p(e_i)\rangle: i=1, ..., g\})$. %$(A, \lambda, \psi_N, \Fil_{\bullet}^{\psi_p}A[p], \{\langle \psi_p(e_i)\rangle: i=1, ..., g\})$ where $\Fil_{\bullet}^{\psi_p}A[p]$ stands for the full flag $$0\subset \langle \psi_p(e_1)\rangle\subset\langle \psi_p(e_1), \psi_p(e_2)\rangle\subset \cdots\subset \langle\psi_p(e_1), \ldots, \psi_p(e_{2g})\rangle.$$
\item[$\bullet$] The third arrow sends $(A, \lambda, \psi_N, \{C_i: i=1, ..., g\})$ to $(A, \lambda, \psi_N, \Fil_{\bullet}A[p])$, where \[
    \Fil_iA[p] = \left\{ \begin{array}{ll}
        \langle C_1, ..., C_i \rangle, & 1 \leq i \leq g  \\
        (\Fil_{2g-i}A[p])^{\perp}, & g+1 \leq i \leq 2g 
    \end{array} \right. .
\]
\item[$\bullet$] The fourth arrow sends $(A, \lambda, \psi_N, \Fil_{\bullet}A[p])$ to $(A, \lambda, \psi_N)$.
\end{enumerate}

For $\Gamma \in \{\Gamma(p^n), \Iw^+, \Iw, \emptyset\}$, let $\calX_{\Gamma}$ be the adic space over $\Spa(\C_p, \calO_{\C_p})$ associated with $X_{\Gamma}$. The analytifications of the forgetful maps yield 
\begin{equation}\label{eq: chain of adic Siegel varieties}
    \calX_{\Gamma(p^n)} \rightarrow \calX_{\Gamma(p)} \rightarrow \calX_{\Iw^+} \rightarrow \calX_{\Iw} \rightarrow \calX.
\end{equation} By fixing a $\GSp_{2g}(\Z)$-admissible polyhedral cone decomposition as in \S \ref{section: boundary}, we show in \S \ref{subsection: boundary strata} that the chain (\ref{eq: chain of adic Siegel varieties}) extends to a chain of log adic spaces\footnote{For a quick review of log adic spaces and the pro-Kummer \'etale site, see \S \ref{Section: Kummer etale and pro-Kummer etale sites of log adic spaces}.} \[
    \overline{\calX}_{\Gamma(p^n)} \rightarrow \overline{\calX}_{\Gamma(p)} \rightarrow \overline{\calX}_{\Iw^+} \rightarrow \overline{\calX}_{\Iw} \rightarrow \overline{\calX},
\] where, for each $\Gamma \in \{\Gamma(p^n), \Iw^+, \Iw, \emptyset\}$, 
\begin{enumerate}
    \item[$\bullet$] $\overline{\calX}_{\Gamma}$ is the adic space over $\Spa(\C_p, \calO_{\C_p})$ associated with the toroidal compactification $\overline{X}_{\Gamma}$ of $X_{\Gamma}$, determined by the fixed polyhedral cone decomposition; 
    \item[$\bullet$] the log structure on $\overline{\calX}_{\Gamma}$ is the divisorial log structure associated with the boundary divisor $\calZ_{\Gamma} := \overline{\calX}_{\Gamma}\smallsetminus \calX_{\Gamma}$; namely, the corresponding sheaf of monoids $\scrM_{\Gamma}$ on $\overline{\calX}_{\Gamma, \et}$ consists of sections of $\scrO_{\overline{\calX}_{\Gamma,\et}}$ that are invertible on the locus away from the boundary divisor;
    \item[$\bullet$] $\overline{\calX}_{\Gamma}$ is finite Kummer \'etale over $\overline{\calX}$ and \begin{enumerate}
        \item[(i)] $\overline{\calX}_{\Gamma(p^n)}\rightarrow \overline{\calX}$ is Galois with Galois group $\GSp_{2g}(\Z/p^n\Z)$;
        \item[(ii)] $\overline{\calX}_{\Gamma(p)}\rightarrow \overline{\calX}_{\Iw}$ is Galois with Galois group $B_{\GSp_{2g}}(\Z/p\Z)$;
        \item[(iii)] $\overline{\calX}_{\Gamma(p)} \rightarrow \overline{\calX}_{\Iw^+}$ is Galois with Galois group $T_{\GSp_{2g}}(\Z/p\Z)$. %$$B^+_{\GSp_{2g}}(\Z/p\Z) := \left\{\begin{pmatrix}\bfgamma_a &\bfgamma_b\\ &\bfgamma_d\end{pmatrix}\in B_{\GSp_{2g}}(\Z/p\Z): \bfgamma_a\text{ is diagonal}\right\}.$$
    \end{enumerate}
\end{enumerate}
We call $\overline{\calX}_{\Gamma}$ the \textbf{\textit{toroidal compactification}} of $\calX_{\Gamma}$ (determined by the fixed polyhedral cone decomposition). 

Furthermore, we have the \emph{perfectoid Siegel modular variety of infinite level} constructed in \cite{Pilloni-Stroh-CoherentCohomologyandGaloisRepresentations}. See \S \ref{subsection: perfectoid Siegel modular variety} for details.

\begin{Theorem}[\text{\cite[Corollaire 4.14]{Pilloni-Stroh-CoherentCohomologyandGaloisRepresentations}}]\label{Theorem: perfectoid toroidally compactified Siegel modular variety}
There exists a perfectoid space $\overline{\calX}_{\Gamma(p^{\infty})}$ such that $$\overline{\calX}_{\Gamma(p^{\infty})}\sim \varprojlim_{n}\overline{\calX}_{\Gamma(p^n)},$$ where ``$\sim$'' is in the sense of \cite[Definition 2.4.1]{Scholze-Weinstein}.
\end{Theorem}

\begin{Remark}
\normalfont The perfectoid Siegel modular variety is constructed by introducing certain \emph{modified integral structures} at  finite levels. More precisely, consider the toroidal compactification $\overline{X}_0$ of $X_0$ and let $\overline{\frakX}$ denote the formal scheme obtained by taking completion along the special fibre of $\overline{\frakX}_0$. Let $\overline{\frakX}_{\Gamma(p^n)}$ denote the normalisation of $\overline{X}$ inside $\overline{\calX}_{\Gamma(p^n)}$. In \cite{Pilloni-Stroh-CoherentCohomologyandGaloisRepresentations}, the authors consider the modified formal schemes $\overline{\frakX}^{\text{mod}}_{\Gamma(p^n)}$ obtained by taking certain admissible formal blowups from $\overline{\frakX}_{\Gamma(p^n)}$ and then consider the projective limit
$$\overline{\frakX}^{\text{mod}}_{\Gamma(p^{\infty})}=\varprojlim_n \overline{\frakX}^{\text{mod}}_{\Gamma(p^n)}$$
in the category of $p$-adic formal schemes. Finally, $\overline{\calX}_{\Gamma(p^{\infty})}$ is defined to be the adic generic fibre of $\overline{\frakX}^{\text{mod}}_{\Gamma(p^{\infty})}$.
\end{Remark}

We summarise the discussion above in the following commutative diagram $$\begin{tikzcd}
\calX_{\Gamma(p^\infty)}\arrow[r, hook]\arrow[d] & \overline{\calX}_{\Gamma(p^\infty)}\arrow[d]\\
\calX_{\Gamma(p^n)}\arrow[r, hook]\arrow[d] & \overline{\calX}_{\Gamma(p^n)}\arrow[d]\\
\calX_{\Gamma(p)}\arrow[r, hook]\arrow[d] & \overline{\calX}_{\Gamma(p)}\arrow[d]\\
\calX_{\Iw^+}\arrow[r, hook]\arrow[d] & \overline{\calX}_{\Iw^+}\arrow[d]\\
\calX_{\Iw}\arrow[r, hook]\arrow[d] & \overline{\calX}_{\Iw}\arrow[d]\\
\calX \arrow[r, hook] & \overline{\calX}\\
\end{tikzcd}.$$ 
where $\calX_{\Gamma(p^\infty)}$ is the part of $\overline{\calX}_{\Gamma(p^\infty)}$ away from the boundary.

There is a natural $\GSp_{2g}(\Z_p)$-action on $\overline{\calX}_{\Gamma(p^{\infty})}$ permuting the $p$-power level structures. In particular, the chain of natural projections
$$\overline{\calX}_{\Gamma(p^{\infty})}\rightarrow \overline{\calX}_{\Gamma(p^n)}\rightarrow \overline{\calX}_{\Gamma(p)}\rightarrow \overline{\calX}_{\Iw^+}\rightarrow \overline{\calX}_{\Iw}\rightarrow  \overline{\calX}$$
is $\GSp_{2g}(\Z_p)$-equivariant. According to Proposition \ref{Proposition: perfectoid toroidal compactification} (ii), the projection $h_{\Gamma(p^n)}:\overline{\calX}_{\Gamma(p^{\infty})}\rightarrow \overline{\calX}_{\Gamma(p^n)}$ (resp., $h_{\Iw^+}:\overline{\calX}_{\Gamma(p^{\infty})}\rightarrow \overline{\calX}_{\Iw^+}$, $h_{\Iw}:\overline{\calX}_{\Gamma(p^{\infty})}\rightarrow \overline{\calX}_{\Iw}$) is a pro-Kummer \'{e}tale Galois cover with Galois group $\Gamma(p^n)$ (resp., $\Iw^+_{\GSp_{2g}}$, $\Iw_{\GSp_{2g}}$). \footnote{Here we have abused the notation and identify the perfectoid space $\overline{\calX}_{\Gamma(p^{\infty})}$ with the object $\varprojlim_n \overline{\calX}_{\Gamma(p^n)}$ in the pro-Kummer \'{e}tale site $\overline{\calX}_{\proket}$.}

\begin{Lemma}\label{Lemma: structure sheaves at the infinite level and the structure sheaves at the Iwahori level}
We have the following identities of sheaves \begin{align*}
    \scrO_{\overline{\calX}_{\Iw}}^{+}=\left(h_{\Iw, *}\scrO_{\overline{\calX}_{\Gamma(p^{\infty})}}^+\right)^{\Iw_{\GSp_{2g}}}, &  \quad \scrO_{\overline{\calX}_{\Iw}}=\left(h_{\Iw, *}\scrO_{\overline{\calX}_{\Gamma(p^{\infty})}}\right)^{\Iw_{\GSp_{2g}}}\\
    \scrO_{\overline{\calX}_{\Iw^+}}^{+}=\left(h_{\Iw^+, *}\scrO_{\overline{\calX}_{\Gamma(p^{\infty})}}^+\right)^{\Iw^+_{\GSp_{2g}}}, &  \quad \scrO_{\overline{\calX}_{\Iw^+}}=\left(h_{\Iw, *}\scrO_{\overline{\calX}_{\Gamma(p^{\infty})}}\right)^{\Iw^+_{\GSp_{2g}}}.
\end{align*}
\end{Lemma}
\begin{proof}
We give the proof of the first pair of identities. The second pair can be proven by the same argument.

It suffices to prove the first identity. For any affinoid open $\calV\subset \overline{\calX}_{\Iw}$ with preimages $\calV_{n}\subset \overline{\calX}_{\Gamma(p^{n})}$ for $n\in \Z_{>0}\cup \{\infty\}$ such that \[
    \scrO_{\overline{\calX}_{\Gamma(p^{\infty})}}^+(\calV_{\infty}) = \left(\varinjlim_{n}\scrO_{\overline{\calX}_{\Gamma(p^n)}}^+(\calV_n)\right)^{\wedge},
\] we have to show $$\scrO_{\overline{\calX}_{\Iw}}^+(\calV)=\left(\scrO_{\overline{\calX}_{\Gamma(p^{\infty})}}^+(\calV_{\infty})\right)^{\Iw_{\GSp_{2g}}}.$$ Here,
{``$\wedge$''} stands for the $p$-adic completion.
Consider the object $\widetilde{\calV}_{\infty}:=\varprojlim_{n}\calV_n$ in the pro-Kummer \'etale site $\overline{\calX}_{\Iw, \proket}$. By Lemma \ref{Kummer etale Galois cover}, each $\calV_n$ is finite Kummer \'{e}tale over $\calV$ with Galois group $G_n:=\Iw_{\GSp_{2g}}/\Gamma(p^n)$. Thus, $$\scrO_{\overline{\calX}_{\Gamma(p^{\infty})}}^+(\calV_{\infty})=\left(\varinjlim_{n}\scrO_{\overline{\calX}_{\Gamma(p^n)}}^+(\calV_n)\right)^{\wedge}=\left(\scrO_{\overline{\calX}_{\Iw, \proket}}^{+}(\widetilde{\calV}_{\infty})\right)^{\wedge}.$$ By \cite[Lemma 4.1.7 \& Corollary 4.4.13]{Diao}, we know $$\left(\scrO_{\overline{\calX}_{\Gamma(p^n)}}^+(\calV_n)/p^m\right)^{\Iw_{\GSp_{2g}}}=\left(\scrO_{\overline{\calX}_{\Gamma(p^n)}}^+(\calV_n)/p^m\right)^{G_n}=\scrO_{\overline{\calX}_{\Iw}}^+(\calV)/p^m$$ for every $m\in \Z_{\geq 1}$. This implies $$\left(\scrO_{\overline{\calX}_{\Iw, \proket}}^+(\widetilde{\calV}_{\infty})/p^m\right)^{\Iw_{\GSp_{2g}}}=\scrO_{\overline{\calX}_{\Iw}}^+(\calV)/p^m.$$ Consequently, we have \begin{align*}
    \left(\scrO_{\overline{\calX}_{\Gamma(p^{\infty})}}^+(\calV_{\infty})\right)^{\Iw_{\GSp_{2g}}} & = \left(\left(\scrO_{\overline{\calX}_{\Iw, \proket}}^{+}(\widetilde{\calV}_{\infty})\right)^{\wedge}\right)^{\Iw_{\GSp_{2g}}}
    = \left(\varprojlim_{m}\left(\scrO_{\overline{\calX}_{\Iw, \proket}}^{+}(\widetilde{\calV}_{\infty})/p^m\right)\right)^{\Iw_{\GSp_{2g}}}\\
    & = \varprojlim_{m}\left(\left(\scrO_{\overline{\calX}_{\Iw, \proket}}^{+}(\widetilde{\calV}_{\infty})/p^m\right)^{\Iw_{\GSp_{2g}}}\right)
    = \varprojlim_{m} \scrO_{\overline{\calX}_{\Iw}}^+(\calV)/p^m
    = \scrO_{\overline{\calX}_{\Iw}}^+(\calV).
\end{align*}
\end{proof}

\begin{Remark}\label{Remark: why strict Iwahori level}
\normalfont We point out that the main geometric object studied in \cite{AIP-2015} is the (toroidally compactified) Siegel modular variety of Iwahori level while ours is of strict Iwahori level. We introduce this deeper level to deal with a certain technical issue involved in the construction of the overconvergent Eichler--Shimura morphism in \S \ref{subsection: OES}. 
\end{Remark}

\subsection{The flag variety}\label{subsection: flag varieties}
The Hodge--Tate period map is a $\GSp_{2g}(\Z_p)$-equivariant morphism from the perfectoid Siegel modular variety to certain flag variety. In this section, let us first describe the target flag variety (and its variants) carefully. 

Recall that $V_p=V\otimes_{\Z}\Z_p$ is the standard symplectic space of rank $2g$ over $\Z_p$. Let $P_{\Siegel}$ be the Siegel parabolic subgroup of $\GSp_{2g}$ defined by \[P_{\Siegel}:=\begin{pmatrix}\GL_g &  \\ M_g & \GL_g\end{pmatrix}\cap \GSp_{2g}.\] Let $\Fl:=P_{\Siegel}\backslash\GSp_{2g}$ be the flag variety over $\Z_p$, parameterising the maximal lagrangians $W\subset V_p$; any representative $\bfgamma \in \GSp_{2g}$ corresponds to the maximal lagrangian spanned by the first $g$ rows of $\bfgamma$. There is a natural action of $\GSp_{2g}$ on $\Fl$ by right multiplication. Let $\adicFL$ be the associated adic space of $\Fl$ over $\Spa(\Q_p, \Z_p)$, equipped with the induced right action of $\GSp_{2g}(\Q_p)$. Hence, for any $p$-adically complete sheafy $(\Q_p, \Z_p)$-algebra $(R, R^+)$, $\adicFL(R, R^+)$ parameterises maximal lagrangians $W\subset V_p\otimes_{\Z_p}R$. Consider the open subset $\adicFL^{\times}\subset \adicFL$ whose $(R, R^+)$-points are $$\adicFL^{\times}(R, R^+)=\left\{(W\subset V_p\otimes_{\Z_p}R)\in \adicFL(R, R^+):\begin{array}{l}
    \text{there exists a basis $\{w_i\}$ of $W$ such that}  \\
    \text{the matrix $(\bla w_i, e_{2g+1-j}\bra)_{1\leq i,j\leq g}$ is invertible}     
\end{array} \right\}.$$ For any $\bfitx_W=(W\subset V_p\otimes_{\Z_p}R)\in \adicFL^\times(R, R^+)$, there exists a unique basis $\{w_i^{\square}\}$ of $W$ such that \[(\bla w_i^{\square}, e_{2g+1-j}\bra)_{1\leq i,j\leq g}=\one_g.\] Therefore, there exist global sections $\bfitz_{i,j}\in \scrO_{\adicFL^\times}(\adicFL^\times)$ such that for any $\bfitx_{W}\in \adicFL^\times(R, R^+)$, $$w_i^{\square}=e_{i}+\sum_{j=1}^{g}\bfitz_{i,j}(\bfitx_W)e_{g+j}.$$ Since $\bla w_i^{\square}, w_j^{\square}\bra= 0$, we have \begin{align*}
    0 & = \bla w_i^{\square}, w_j^{\square}\bra\\
    & = \bla e_{i}, \sum_{k=1}^g\bfitz_{j,k}(\bfitx_W)e_{g+k}\bra + \bla \sum_{k=1}^g\bfitz_{i,k}(\bfitx_W)e_{g+k}, e_{j}\bra\\ 
    & = \bfitz_{j,g+1-i}(\bfitx_W)-\bfitz_{i,g+1-j}(\bfitx_W).
\end{align*} That is, the matrix $$\bfitz:=\begin{pmatrix}
\bfitz_{1,1} & \cdots & \bfitz_{1,g}\\
\vdots & & \vdots\\
\bfitz_{g,1} & \cdots & \bfitz_{g,g}
\end{pmatrix}$$ 
is symmetric with respect to the anti-diagonal. Moreover, we may use the matrix $\begin{pmatrix}\one_g & \bfitz(\bfitx_W)\end{pmatrix}$ (or just the matrix $\bfitz(\bfitx_W)$) to represent the element $\bfitx_W\in \adicFL^\times(R, R^+)$ because the basis $\{w_i^{\square}\}$ is represented by the matrix $$\begin{pmatrix}
1 & & & \bfitz_{1,1}(\bfitx_W) & \cdots & \bfitz_{1,g}(\bfitx_W)\\
& \ddots & & \vdots & & \vdots\\
& & 1 & \bfitz_{g,1}(\bfitx_W) & \cdots & \bfitz_{g,g}(\bfitx_W)
\end{pmatrix}=\begin{pmatrix}\one_g & \bfitz(\bfitx_W)\end{pmatrix}$$ with respect to the standard basis $e_1, \ldots, e_{2g}$ of $V_p$. 

In the rest of the paper, we take base change of the adic spaces $\adicFL$ and $\adicFL^{\times}$ to $\Spa(\C_p, \calO_{\C_p})$.

For every $w\in \Q_{>0}$, consider an open adic subspace $\adicFL_w^\times\subset \adicFL^\times$ defined by $$\adicFL^\times_w:=\left\{\bfitx\in \adicFL^\times: \max_{i, j}\inf_{t\in \Z_p}\{|\bfitz_{i,j}(\bfitx)-t|\}\leq p^{-w}\right\}.$$ 
For any algebraically closed complete nonarchimedean field $C$ containing $\Q_p$, let
$$\GSp_{2g, w}(C):=\left\{\bfgamma=\begin{pmatrix}\bfgamma_a & \bfgamma_b\\ \bfgamma_c & \bfgamma_d\end{pmatrix}\in\GSp_{2g}(C):\begin{array}{l}
    \bfgamma_a\in \GL_g(C),\text{ and}\\
    \max_{i, j}\inf_{h\in \Z_p}\{|(\bfgamma_a^{-1}\bfgamma_b)_{ij}-h|\}\leq p^{-w}
\end{array} \right\}$$
where $(\bfgamma_a^{-1}\bfgamma_b)_{ij}$ is the $(i,j)$-th entry of the matrix $\bfgamma_a^{-1}\bfgamma_b$. Then the $(C, \calO_C)$-points of $\adicFL^\times_w$ can be identified with the quotient $$\adicFL^\times_w(C, \calO_C)=P_{\Siegel}(C)\backslash\GSp_{2g, w}(C)$$
so that the natural inclusion $\adicFL^\times_w(C, \calO_C)\subset \adicFL(C, \calO_C)$ is induced by
$$\adicFL^\times_w(C, \calO_C)=P_{\Siegel}(C)\backslash\GSp_{2g, w}(C)\hookrightarrow P_{\Siegel}(C)\backslash\GSp_{2g}(C)=\adicFL(C, \calO_C).$$

Recall that there is a natural right action of $\GSp_{2g}(\Q_p)$ on $\adicFL$. The following lemma shows that $\adicFL_w^\times$ is stable under the action of the subgroup $\Iw_{\GSp_{2g}}\subset \GSp_{2g}(\Q_p)$.

\begin{Lemma}\label{Lemma: Iw_GSp stabilises FL_w^times}
The adic space $\adicFL_w^\times$ is stable under the right action of $\Iw_{\GSp_{2g}}$. Coordinate-wise, the action is given by $$\adicFL_w^{\times}\times \Iw_{\GSp_{2g}}\rightarrow \adicFL_w^{\times}, \quad \left(\bfitz,\begin{pmatrix}\bfgamma_a & \bfgamma_b\\ \bfgamma_c & \bfgamma_d\end{pmatrix} \right)\mapsto (\bfgamma_a+\bfitz\bfgamma_c)^{-1} (\bfgamma_b+\bfitz\bfgamma_d).$$
In particular, $\adicFL_w^{\times}$ is also stable under the right action of the subgroup $\Iw_{\GSp_{2g}}^+$.
\end{Lemma}
\begin{proof}
It follows from the definition that the right action of $\bfgamma\in\Iw_{\GSp_{2g}}$ indeed sends $\begin{pmatrix}\one_g & \bfitz(\bfitx_W)\end{pmatrix}$ to 
$$\begin{pmatrix}\one_g &  (\bfgamma_a+\bfitz(\bfitx_W)\bfgamma_c)^{-1} (\bfgamma_b+\bfitz(\bfitx_W)\bfgamma_d)\end{pmatrix} =(\bfgamma_a+\bfitz(\bfitx_W)\bfgamma_c)^{-1} \begin{pmatrix}\one_g & \bfitz(\bfitx_W)\end{pmatrix} \begin{pmatrix}\bfgamma_a & \bfgamma_b\\ \bfgamma_c & \bfgamma_d\end{pmatrix}.$$
It remains to show that, for every $\bfitx_W\in \adicFL^{\times}_w$, the matrix $(\bfgamma_a+\bfitz(\bfitx_W)\bfgamma_c)^{-1} (\bfgamma_b+\bfitz(\bfitx_W)\bfgamma_d)$ lands in $\adicFL^{\times}_w$. But this is straightforward.
\end{proof}

\subsection{Vector bundles on the flag variety}\label{subsection: vector bundles on the flag variety}
Let $\scrW_{\Fl}\subset\scrO_{\Fl}^{2g}$ be the universal maximal lagrangian over $\Fl$. The total space of $\scrW_{\Fl}$ can be naturally identified with $$\scrW_{\Fl}\simeq P_{\Siegel}\backslash (\bbA^g\times \GSp_{2g})$$ where \begin{enumerate}
    \item[$\bullet$] by viewing elements $\vec{v}\in \bbA^g$ as row vectors, $P_{\Siegel}$ acts on $\bbA^g$ from the left via $ \bfgamma * \vec{v}  = \vec{v}\cdot\bfgamma_a^{-1}$, for any $\bfgamma=\begin{pmatrix}\bfgamma_a & \bfgamma_b\\ \bfgamma_c & \bfgamma_d\end{pmatrix}\in P_{\Siegel}$;
    \item[$\bullet$] $P_{\Siegel}$ acts on $\GSp_{2g}$ via the left multiplication. 
\end{enumerate} 

Similarly, consider the linear dual $\scrW_{\Fl}^{\vee}$ of $\scrW_{\Fl}$. Then the total space of $\scrW_{\Fl}^{\vee}$ can be naturally identified with $$\scrW_{\Fl}^{\vee}\simeq P_{\Siegel}\backslash (\bbA^g\times \GSp_{2g})$$ where, by viewing elements $\vec{v}\in \bbA^g$ as column vectors, $P_{\Siegel}$ acts on $\bbA^g$ from the left via $ \bfgamma * \vec{v}  = \bfgamma_a\cdot\vec{v}$, for any $\bfgamma=\begin{pmatrix}\bfgamma_a & \bfgamma_b\\ \bfgamma_c & \bfgamma_d\end{pmatrix}\in P_{\Siegel}$. Under this identification, global sections of $\scrW_{\Fl}^{\vee}$ are identified with $$\left\{\text{algebraic functions }\phi: \GSp_{2g}\rightarrow \bbA^g: \phi(\bfgamma\bfalpha)=\bfgamma_a\cdot \phi(\bfalpha),\,\,\forall \bfgamma\in P_{\Siegel},\,\, \bfalpha\in\GSp_{2g}\right\}.$$
For every $i=1,\ldots, g$, we consider a global section $\bfits_i$ of $\scrW_{\Fl}^{\vee}$ defined by $$\bfits_i(\bfalpha):= \text{the $i$-th column of }\bfalpha_a$$ for all $\bfalpha=\begin{pmatrix}\bfalpha_a & \bfalpha_b\\ \bfalpha_c & \bfalpha_d\end{pmatrix}\in \GSp_{2g}$. If we write 
$$\bfits:=\begin{pmatrix}\bfits_1& \cdots& \bfits_g\end{pmatrix}\in (\scrW_{\Fl}^{\vee})^g$$ 
then we have $\bfits(\bfalpha)=\bfalpha_a$.

By passing to the adic space $\adicFL$ and restricting to $\adicFL_w^{\times}$, the (algebraic) sheaves $\scrW_{\Fl}$ and $\scrW_{\Fl}^{\vee}$ yield (analytic) sheaves $\scrW_{\adicFL_w^{\times}}$ and $\scrW_{\adicFL_w^{\times}}^{\vee}$ on $\adicFL_w^\times$. We still use $\bfits_i$'s to denote the restrictions on $\adicFL_w^{\times}$ of the corresponding algebraic sections. By definition, the sections $\bfits_i$'s are non-vanishing on $\adicFL_w^{\times}$ and hence $\bfits_i^{\vee}$'s are well-defined sections on $\scrW_{\adicFL_w^{\times}}$. We similarly write
$$\bfits^{\vee}:=\begin{pmatrix}\bfits_1^{\vee} \\ \vdots \\  \bfits_g^{\vee}\end{pmatrix}\in (\scrW_{\adicFL_w^{\times}})^g.$$
Moreover, the right action of $\Iw_{\GSp_{2g}}$ on $\adicFL_w^\times$ induces a right action of $\Iw_{\GSp_{2g}}$ on $\scrW_{\adicFL_w^{\times}}$. For later use, we would like to understand the behaviour of $\bfits_i^{\vee}$'s under this action.

\begin{Lemma}\label{Lemma: invariance of fake Hasse invariants}
For any $\bfgamma=\begin{pmatrix}\bfgamma_a &\bfgamma_b\\ \bfgamma_c & \bfgamma_d\end{pmatrix}\in \Iw_{\GSp_{2g}}$, we have
$$ \bfgamma^*(\bfits^{\vee}) = (\bfgamma_a + \bfitz\bfgamma_c)^{-1}\cdot \bfits^{\vee}.$$ 
\end{Lemma}
\begin{proof}
To prove the identity, it suffices to check on the level of $(C, \calO_C)$-points. Using the identification
$$\adicFL^\times_w(C, \calO_C)=P_{\Siegel}(C)\backslash\GSp_{2g, w}(C),$$
the sections of $\scrW_{\adicFL_w^{\times}}$ can be identified with
$$\left\{\text{analytic functions }\phi: \GSp_{2g,w}\rightarrow C^g: \phi(\bfgamma\bfalpha)=\phi(\bfalpha)\cdot\bfgamma_a^{-1},\,\,\forall \bfgamma\in P_{\Siegel}(C),\,\, \bfalpha\in\GSp_{2g,w}(C)\right\}$$ where elements in $C^g$ are viewed as row vectors. Under this identification, $\bfits^{\vee}$ sends $\bfalpha\in \GSp_{2g,w}(C)$ to $\bfalpha_a^{-1}$. Notice that a section $\phi:\GSp_{2g,w}(C)\rightarrow C^g$ of $\scrW_{\adicFL_w^{\times}}$ is determined by its restriction on $$\left\{\begin{pmatrix}\one_g & \bfitz \\ &\one_g \end{pmatrix}: \trans\bfitz\oneanti_g=\oneanti_g\bfitz,\,\,\max_{i, j}\inf_{h\in \Z_p}\{|\bfitz_{i,j}(\bfitx)-h|\}\leq p^{-w} \right\}.$$ Let $\bfalpha=\begin{pmatrix}\one_g & \bfitz\\ & \one_g\end{pmatrix}$. Then $\bfits^{\vee}(\bfalpha)=\one_g$ and $$(\bfgamma^*(\bfits^{\vee}))(\bfalpha)=\bfits^{\vee}(\bfalpha\bfgamma)=\bfits^{\vee} \left(\begin{pmatrix} \bfgamma_a+ \bfitz\bfgamma_c & \bfgamma_b+\bfitz\bfgamma_d\\\bfgamma_c& \bfgamma_d\end{pmatrix}\right)= (\bfgamma_a+\bfitz\bfgamma_c)^{-1}=(\bfgamma_a+\bfitz\bfgamma_c)^{-1}\cdot\bfits^{\vee}(\bfalpha) $$ as desired.
\end{proof}

An immediate corollary of the lemma above is the following:

\begin{Corollary}\label{Corollary: Iw-action on s}
For any $\bfgamma = \begin{pmatrix}\bfgamma_a & \bfgamma_b\\ \bfgamma_c & \bfgamma_d\end{pmatrix}\in \Iw_{\GSp_{2g}}$, we have \[
    \bfgamma^*(\bfits) = \bfits\cdot  (\bfgamma_a + \bfitz \bfgamma_c) 
\] 
where we view $\bfits$ as a section of $\scrW_{\adicFL_w^{\times}}^{\vee}$.
\end{Corollary}

\subsection{The Hodge--Tate period map and the \texorpdfstring{$w$}{w}-ordinary locus}\label{subsection: Hodge--Tate period map and the w-ordinary locus}

We briefly recall the well-known Hodge--Tate period map in the setup of (toroidally compactified) Siegel modular variety.

The Hodge--Tate period map (see \cite[\S 1]{Pilloni-Stroh-CoherentCohomologyandGaloisRepresentations} and \S \ref{subsection: perfectoid Siegel modular variety}) is a morphism of adic spaces
$$\pi_{\HT}:\overline{\calX}_{\Gamma(p^{\infty})}\rightarrow \adicFL.$$
On the level of points, and away from the boundary, the Hodge--Tate period map has the following explicit description. Suppose 
$C$ is an algebraically closed and complete extension of $\Q_p$ and $(A, \lambda)$ is a principally polarised abelian variety over $C$. The Hodge--Tate sequence of $A$ is
\[0\rightarrow \Lie A\rightarrow T_pA\otimes_{\Z_p}C\rightarrow \omega_{A^{\vee}}\rightarrow 0,\] where ${\omega}_{A^{\vee}}$ is the dual of the Lie algebra of the dual abelian variety $A^{\vee}$ and the second last map is induced from the Hodge--Tate map $\HT_A:T_pA\rightarrow \omega_{A^{\vee}}$. 
Here, we ignore the Tate twist by fixing a compatible system of $p$-power roots of unity $(\zeta_{p^n})_{n\in \Z_{\geq 1}}$ in $\C_p$. Notice that every point $\bfitx\in \calX_{\Gamma(p^{\infty})}(C, \calO_C)$ corresponds to a quadruple $(A, \lambda, \psi_N, \psi)$ where $(A, \lambda, \psi_N)$ is a principally polarised abelian variety over $C$ with a principal level $N$ structure and $\psi$ is a symplectic isomorphism $\psi: V_p\simeq T_p A$. Then $\pi_{\HT}$ sends $\bfitx$ to the maximal lagrangian $$\Lie A\subset T_pA\otimes_{\Z_p}C\overset{\psi^{-1}}{\cong} V_p\otimes_{\Z_p} C.$$

One can extend such an explicit description to the boundary points as well using the language of 1-motives.\footnote{The formal definition can be found in \cite[\S 1.2]{Stroh-TorComp} but intuitively one can think of them in the following way: over the boundary, the universal abelian variety degenerates into a semi-abelian variety $G$, which locally is an extension of an abelian scheme of dimension $g-a$ by a torus of rank $a$. The problem is that the Tate module of $G$ has rank $2g-a$. A $1$-motive is a complex $[Y \rightarrow G]$ where $Y$ is locally a lattice of rank $a$, the same $a$ as the toric rank of $G$. The key fact is that the $\mathrm{H}^1$ of the $1$-motive is an extension of the $\mathrm{H}^1$ of $G$ and the $\mathrm{H}^1$ of $Y$. In particular, even if the Tate module of $G$ doesn't have constant rank, the $\mathrm{H}^1$ of the $1$-motive does! And concretely one uses this as the extension of the Tate module of the universal abelian variety to the boundary.} The details are left to the interested readers.

\begin{Remark}\label{Remark: GSp2g-equivariance}
\normalfont On $\adicFL$, there is a $\GSp_{2g}(\Q_p)$-action given in \S \ref{subsection: flag varieties}. 
On $\calX_{\Gamma(p^{\infty})}$, there is also a $\GSp_{2g}(\Q_p)$-action described as follows. Let $\bfgamma\in \GSp_{2g}(\Q_p)$ and let $m\in\Z$ such that $p^m\bfgamma\in M_{2g}(\Z_p)$ and $p^{m-1}\bfgamma\not\in M_{2g}(\Z_p)$. Choose $k\in \Z_{\geq 0}$ sufficiently large such that the kernel of $p^m\bfgamma: A[p^k]\rightarrow A[p^k]$ stabilises. Let $H\subset A[p^k]$ denote the corresponding kernel. Then $\bfgamma$ sends $(A, \lambda, \psi_N, \psi)$ to $(A'=A/H, \lambda', \psi'_N, \psi')$ where
\begin{itemize}
\item $\lambda'$ is the induced polarisation on $A'$;
\item $\psi'_N$ is induced from $\psi_N$ via the isomorphism $A[N]\simeq A'[N]$;
\item $\psi'$ is given by the composition
$$V_p\rightarrow V_p\otimes_{\Z_p}\Q_p\xrightarrow[]{\psi}T_pA\otimes_{\Z_p}\Q_p\rightarrow T_pA'\otimes_{\Z_p}\Q_p$$
with the first map $V_p\rightarrow V_p\otimes_{\Z_p}\Q_p$ sending $\vec{v}$ to $p^{m}\vec{v}\cdot\bfgamma^{-1}$. One checks that the composition induces a symplectic isomorphism $V_p\simeq T_pA'$.
\end{itemize}
It turns out $\pi_{\HT}$ respects the $\GSp_{2g}(\Q_p)$-action on $\calX_{\Gamma(p^{\infty})}$ and $\adicFL$. In fact, in \cite{Scholze-2015}, Scholze showed that the $\GSp_{2g}(\Q_p)$-action extends to the minimal compactification and the Hodge--Tate period map for the minimal compactification (denoted by $\pi_{\HT}^{\min}$) is $\GSp_{2g}(\Q_p)$-equivariant. However, this is not the case after extending $\pi_{\HT}$ to the toroidal compactification since one needs to change the cone decomposition in the construction of the toroidal compactification for the action of an element $\bfgamma\in \GSp_{2g}(\Q_p)\smallsetminus \GSp_{2g}(\Z_p)$.  Nevertheless, one can still see from construction of $\pi_{\HT}$ that it is $\GSp_{2g}(\Z_p)$-equivariant (\cite[\S 1.14]{Pilloni-Stroh-CoherentCohomologyandGaloisRepresentations} or \S \ref{subsection: perfectoid Siegel modular variety}; see also \cite[\S 1.2.4]{BP-highercolemannots}). Moreover, $\pi_{\HT}$ factors through $\pi_{\HT}^{\min}$.
\end{Remark}

Let $\calG^{\univ}$ be the tautological semiabelian variety over $\overline{\calX}$ extending the universal abelian variety $\calA^{\univ}$ over $\calX$. Let $\pi: \calG^{\univ} \rightarrow \overline{\calX}$ be the structure morphism with identity section $e$ and let $$\underline{\omega}:=e^*\Omega^1_{\calG^{\univ}/\overline{\calX}}$$ which is a vector bundle of rank $g$ over $\overline{\calX}$. Pulling back along the projection $h: \overline{\calX}_{\Gamma(p^{\infty})}\rightarrow \overline{\calX}$, we obtain a vector bundle $$\underline{\omega}_{\Gamma(p^{\infty})}:=h^*\underline{\omega}$$ over $\overline{\calX}_{\Gamma(p^{\infty})}$.

\begin{Proposition}\label{Proposition: fake Hasse invariants and the coherent automorphic sheaf}
There is a natural isomorphism
$$\pi_{\HT}^*\scrW_{\adicFL}^{\vee}\cong \underline{\omega}_{\Gamma(p^{\infty})}.$$
\end{Proposition}
\begin{proof}
Let $\calA^{\univ}_{\Gamma(p^{\infty})}$ be the pullback of $\calA^{\univ}$ to $\calX_{\Gamma(p^{\infty})}$. Away from the boundary, we have a universal trivialisation $\psi^{\univ}:V_p\cong T_p\calA^{\univ}_{\Gamma(p^{\infty})}$. Let $\psi^{\univ, \vee}:V_p^{\vee}\cong T_p\calA^{\univ, \vee}_{\Gamma(p^{\infty})}$ be the dual trivialisation. The Hodge--Tate map on the universal abelian variety $\calA^{\univ}_{\Gamma(p^{\infty})}$ induces a map
$$\HT_{\Gamma(p^{\infty})}:V_p^{\vee}\overset{\psi^{\univ,\vee}}{\cong} T_p\calA^{\univ, \vee}_{\Gamma(p^{\infty})}\rightarrow \underline{\omega}_{\Gamma(p^{\infty})}|_{\calX_{\Gamma(p^{\infty})}}$$
which induces a surjection
$$\HT_{\Gamma(p^{\infty})}\otimes \id:V_p^{\vee}\otimes_{\Z_p}\scrO_{\calX_{\Gamma(p^{\infty})}}\twoheadrightarrow \underline{\omega}_{\Gamma(p^{\infty})}|_{\calX_{\Gamma(p^{\infty})}}.$$
According to \S \ref{subsection: perfectoid Siegel modular variety},
this surjection extends to a surjection
$$\HT_{\Gamma(p^{\infty})}\otimes \id:V_p^{\vee}\otimes_{\Z_p}\scrO_{\overline{\calX}_{\Gamma(p^{\infty})}}\twoheadrightarrow \underline{\omega}_{\Gamma(p^{\infty})}$$
on the entire perfectoid Siegel modular variety.

Consequently, the sheaf $\pi_{\HT}^* \scrW_{\adicFL}^{\vee}$, being the universal maximal Lagrangian quotient of $V_p^{\vee}\otimes_{\Z_p} \scrO_{\overline{\calX}_{\Gamma(p^{\infty})}}$, coincides with $\underline{\omega}_{\Gamma(p^{\infty})}$.
\end{proof}

Recall the sections $\bfits_i$ of $\scrW_{\adicFL}^{\vee}$ defined in \S \ref{subsection: vector bundles on the flag variety}. We define sections $\fraks_i\in \underline{\omega}_{\Gamma(p^{\infty})}$ by \begin{equation}\label{eq: fake Hasse invariants}
    \fraks_i:=\pi_{\HT}^*\bfits_i.
\end{equation} 
From the construction, one sees that
\begin{equation}\label{eq:fake Hasse invariants basis}
\fraks_i=\HT_{\Gamma(p^{\infty})}(e_i^{\vee})
\end{equation}
for all $i=1, \ldots, g$. These $\fraks_i$'s are examples of \textit{fake Hasse invariants} studied in \cite{Scholze-2015}. We also write $$\fraks:=\begin{pmatrix}\fraks_1& \cdots& \fraks_g\end{pmatrix}=\pi_{\HT}^*\bfits.$$ 

To wrap up the section, we introduce the notion of ``$w$-ordinary locus'' of the perfectoid Siegel modular variety. In particular, it is an open subset of $\overline{\calX}_{\Gamma(p^{\infty})}$ which contains the usual ordinary locus.

\begin{Definition}\label{Definition: w-ordinary}
For every $w\in \Q_{>0}$, define
\[\overline{\calX}_{\Gamma(p^{\infty}), w}:=\pi_{\HT}^{-1}(\adicFL_{w}^\times).\]
We also define
\[\overline{\calX}_{\Gamma(p^n), w}:=h_n(\overline{\calX}_{\Gamma(p^{\infty}), w}), \quad\overline{\calX}_{\Iw^+, w}:=h_{\Iw^+}(\overline{\calX}_{\Gamma(p^{\infty}), w}),\quad \overline{\calX}_{\Iw, w}:=h_{\Iw}(\overline{\calX}_{\Gamma(p^{\infty}), w}), \quad \text{and }\quad \overline{\calX}_w:=h(\overline{\calX}_{\Gamma(p^{\infty}), w}),\]
where $h_n: \overline{\calX}_{\Gamma(p^{\infty})} \rightarrow \overline{\calX}_{\Gamma(p^{n})}$, $h_{\Iw^+}: \overline{\calX}_{\Gamma(p^{\infty})}\rightarrow \overline{\calX}_{\Iw^+}$,  $h_{\Iw}: \overline{\calX}_{\Gamma(p^{\infty})}\rightarrow \overline{\calX}_{\Iw}$, and $h: \overline{\calX}_{\Gamma(p^{\infty})}\rightarrow \overline{\calX}$ are the natural projections. The subspaces $\overline{\calX}_{\Gamma(p^{\infty}), w}$, $\overline{\calX}_{\Gamma(p^n), w}$, $\overline{\calX}_{\Iw^+, w}$, $\overline{\calX}_{\Iw, w}$, and $\overline{\calX}_w$ are called the \textit{\textbf{$w$-ordinary loci}} of $\overline{\calX}_{\Gamma(p^{\infty})}$, $\overline{\calX}_{\Gamma(p^n)}$, $\overline{\calX}_{\Iw^+}$, $\overline{\calX}_{\Iw}$, and $\overline{\calX}$, respectively.
\end{Definition}

\begin{Remark}
\normalfont We point out that the $w$-ordinary loci defined above are analogues of the ``anti-canonical loci'' introduced in \cite{Scholze-2015} (also see \cite{BHW-2019}). They are different from the ``canonical loci '' used in \cite{CHJ-2017}. One can use the Atkin--Lehner operator (see Remark \ref{Remark: Atkin--Lehner operator}) to pass between the two types of loci.
\end{Remark}

We still denote by
$$\pi_{\HT}:\overline{\calX}_{\Gamma(p^{\infty}), w}\rightarrow \adicFL_{w}^\times$$
the restriction of the Hodge--Tate period map on the $w$-ordinary locus. It is equivariant under the right $\Iw_{\GSp_{2g}}$-actions on both sides.

Denote by $\frakz_{ij}:=\pi_{\HT}^*\bfitz_{ij}$ and $\frakz:=(\frakz_{i,j})_{1\leq i,j\leq g}=\pi_{\HT}^*\bfitz$. Let $\fraks_i^{\vee}:=\pi_{\HT}^*(\bfits_i^{\vee})$ and $$\fraks^{\vee}:=\begin{pmatrix}\fraks_1^{\vee} \\  \vdots \\  \fraks_g^{\vee}\end{pmatrix}=\pi_{\HT}^*(\bfits^{\vee}).$$
By Lemma \ref{Lemma: invariance of fake Hasse invariants} and Corollary \ref{Corollary: Iw-action on s}, we have
$$\bfgamma^*(\fraks^{\vee})  =(\bfgamma_a + \frakz\bfgamma_c)^{-1}\cdot \fraks^{\vee} $$ 
and 
\begin{equation}\label{eq: action on fake Hasse invariants}
    \bfgamma^*\fraks =  \fraks \cdot (\bfgamma_a + \frakz\bfgamma_c)
\end{equation}
for all $\bfgamma=\begin{pmatrix}\bfgamma_a &\bfgamma_b\\ \bfgamma_c & \bfgamma_d\end{pmatrix}\in \Iw_{\GSp_{2g}}$. We will need these sections $\fraks_i$'s and $\fraks_i^{\vee}$'s in \S \ref{subsection: admissibility} and \S \ref{subsection:comparison sheaf aip}.
\section{Overconvergent automorphic sheaves}\label{section:constructionsheaf}
In this section, we construct the overconvergent automorphic sheaves using the geometric objects introduced in the previous section. In particular, we generalise the ``perfectoid method'' which was originally adopted by Chojecki--Hansen--Johansson in \cite{CHJ-2017} to handle the compact Shimura curves over $\Q$. Notice that overconvergent automorphic sheaves are first introduced by Andreatta--Iovita--Pilloni in \cite{AIP-2015} using a different approach. At the end of the section we shall compare the two constructions (when $p>2g$).

\subsection{The perfectoid construction}\label{subsection: the perfectoid construction}
Let $\Alg_{(\Z_p, \Z_p)}$ be the category of complete sheafy $(\Z_p, \Z_p)$-algebras. We consider the functor $$\Alg_{(\Z_p, \Z_p)}\rightarrow \Sets, \quad (R, R^+)\mapsto \Hom_{\Groups}^{\cts}(T_{\GL_g, 0}, R^{\times}),$$ which is represented by the $(\Z_p, \Z_p)$-algebra $(\Z_p\llbrack T_{\GL_g,0}\rrbrack, \Z_p\llbrack T_{\GL_g, 0}\rrbrack)$. The \textit{\textbf{weight space}} is $$\calW:=\Spa(\Z_p\llbrack T_{\GL_g,0}\rrbrack,\Z_p\llbrack T_{\GL_g,0}\rrbrack)^{\rig},$$ where the superscript ``rig'' stands for taking the generic fibre. Every continuous group homomorphism $\kappa: T_{\GL_g, 0}\rightarrow R^{\times}$ can be expressed as $\kappa=(\kappa_{1}, ..., \kappa_{g})$ where each $\kappa_{i}:\Z_p^\times\rightarrow R^\times$ is a continuous group homomorphism. We write $\kappa^{\vee}:=(-\kappa_{g}, ..., -\kappa_{1})$ where $-\kappa_i$ is the inverse of $\kappa_i$. 

We adapt the terminologies of \say{small weights} and \say{affinoid weights} introduced in \cite{CHJ-2017} to our setting:

\begin{Definition}\label{Definition: weights}
\begin{enumerate}
    \item[(i)] A \textbf{small $\Z_p$-algebra} is a $p$-torsion free reduced ring which is also a finite $\Z_p\llbrack T_1, ..., T_d\rrbrack$-algebra for some $d\in \Z_{\geq 0}$. In particular, a small $\Z_p$-algebra is equipped with a canonical adic profinite topology and is complete with respect to the $p$-adic topology.
    \item[(ii)] A \textbf{small weight} is a pair $(R_{\calU}, \kappa_{\calU})$ where $R_{\calU}$ is a small $\Z_p$-algebra and $\kappa_{\calU}:T_{\GL_g,0}\rightarrow R^{\times}_{\calU}$ is a continuous group homomorphism such that $\kappa_{\calU}((1+p)\one_g)-1$ is topologically nilpotent in $R_{\calU}$ with respect to the $p$-adic topology. By the universal property of the weight space, we obtain a natural morphism
    \[\Spa(R_{\calU}, R_{\calU})^{\rig}\rightarrow \calW.\]
Occasionally, we abuse the terminology and call $\calU:=\Spa(R_{\calU}, R_{\calU})$ a small weight. For later use, we write $R_{\calU}^+:=R_{\calU}$ in this situation.
    \item[(iii)] An \textbf{affinoid weight} is a pair $(R_{\calU}, \kappa_{\calU})$ where $R_{\calU}$ is a reduced Tate algebra topologically of finite type over $\Q_p$ and $\kappa_{\calU}:T_{\GL_g,0}\rightarrow R^{\times}_{\calU}$ is a continuous group homomorphism. By the universal property of weight space, we obtain a natural morphism
     \[\Spa(R_{\calU}, R^{\circ}_{\calU})\rightarrow \calW.\]
Occassionally, we abuse the terminology and call $\calU:=\Spa(R_{\calU}, R^{\circ}_{\calU})$ an affinoid weight. For later use, we write $R_{\calU}^+=R_{\calU}^{\circ}$ in this situation.
    \item[(iv)] By a \textbf{weight}, we shall mean either a small weight or an affinoid weight.
\end{enumerate}
\end{Definition}

\begin{Remark}\label{Remark: convention on weights}
\normalfont For any $n\in \Z_{\geq 0}$, we view $n$ as a weight by identifying it with the character $$T_{\GL_g, 0}\rightarrow \Z_p^\times, \quad \diag(\bftau_1, ..., \bftau_g)\mapsto \prod_{i=1}^g\bftau_i^n.$$ Moreover, for any weight $\kappa=(\kappa_1, ..., \kappa_g)$, we write $\kappa+n $ for the weight $ (\kappa_1+n, ..., \kappa_g+n)$ defined by $$\diag(\bftau_1, ..., \bftau_g)\mapsto \prod_{i=1}^g\kappa_i(\bftau_i)\bftau_i^n.$$
\end{Remark}

We adopt the notation of \say{mixed completed tensor} used in \cite{CHJ-2017}:

\begin{Definition}\label{Definition: unadorned completed tensor}
Let $R$ be a small $\Z_p$-algebra.
\begin{enumerate}
\item[(i)] For any $\Z_p$-module $M$, we define \footnote{Our notation $\widehat{\otimes}'$ corresponds to the notation $\widehat{\otimes}$ in \cite[Definition 6.3]{CHJ-2017}. We make this change to distinguish from the one in Definition \ref{Definition: unadorned completed tensor} (ii).}
$$M\widehat{\otimes}' R:=\varprojlim _{j\in J} (M\otimes_{\Z_p}R/I_j)$$
where $\{I_j: j\in J\}$ runs through a cofinal system of neighborhood of $0$ consisting of $\Z_p$-submodules of $R$. If, in addition, $M$ is a $\Z_p$-algebra, then $M\widehat{\otimes}' R$ is also a $\Z_p$-algebra.
\item[(ii)] Let $B$ be a $\Q_p$-Banach space and let $B_0$ be an open and bounded $\Z_p$-submodule. We define the \textbf{mixed completed tensor}
$$B\widehat{\otimes}R:=(B_0\widehat{\otimes}'R)[\frac{1}{p}].$$
which is in fact independent of the choice of $B_0$. 
\end{enumerate}
\end{Definition}

\begin{Definition}\label{Definition: mixed completed tensor}
Let $(R_{\calU}, \kappa_{\calU})$ be a weight.
\begin{enumerate}
\item[(i)] For any $\Z_p$-module $M$, the term $M\widehat{\otimes}R^+_{\calU}$ will either stand for $M\widehat{\otimes}' R_{\calU}$ in the case of small weights (notice that $R_{\calU}=R^+_{\calU}$ in this case), or stand for the $p$-adically completed tensor over $\Z_p$ in the case of affinoid weights.
\item[(ii)]
For any $\Q_p$-Banach space $B$, the term $B\widehat{\otimes} R_{\calU}$ will either stand for the mixed completed tensor in the case of small weights, or stand for the usual $p$-adically completed tensor over $\Q_p$ in the case of affinoid weights. 
\end{enumerate}
\end{Definition}

\begin{Remark}\label{Remark: uniform Banach algebra structure 1}\normalfont 
For a uniform Banach $\Q_p$-algebra $B$ and any weight $(R_{\calU}, \kappa_{\calU})$, the tensor product $B\widehat{\otimes}R_{\calU}$ also admits a structure of uniform $\Q_p$-Banach algebra. In particular,
if $B^{\circ}$ is the unit ball of $B$ (with respect to the unique power-multiplicative Banach algebra norm), then the unit ball in $B\widehat{\otimes}R_{\calU} = B\widehat{\otimes}_{\Q_p}R_{\calU}^+[1/p]$ is given by $B^{\circ}\widehat{\otimes} R^+_{\calU}$. Here, note that $R_{\calU}^+[1/p]$ has a structure of a uniform Banach $\Q_p$-algebra given by the corresponding spectral norm (see \cite[pp. 202]{CHJ-2017}).
\end{Remark}

Next, we introduce the notion of ``$r$-analytic functions''.

\begin{Definition}\label{Definition: w-analytic functions} Let $r\in \Q_{>0}$ and $n\in \Z_{\geq 1}$. Let $B$ be a uniform $\C_p$-Banach algebra and let $B^{\circ}$ be the corresponding unit ball.
\begin{enumerate}
\item[(i)] A function $f: \Z_p^n\rightarrow B$ (resp., a function $f:(\Z_p^{\times})^n\rightarrow B$) is called \textbf{$r$-analytic} if for every $\underline{a}=(a_1, \ldots, a_n)\in \Z_p^n$ (resp., every $\underline{a}=(a_1, \ldots, a_n)\in (\Z_p^{\times})^n$), there exists a power series $f_{\underline{a}}\in B\llbrack T_1, \ldots, T_n\rrbrack$ which converges on the $n$-dimensional closed unit ball $\mathbf{B}^n(0, p^{-r})\subset \C_p^n$ of radius $p^{-r}$ such that 
$$f(x_1+a_1, \ldots, x_n+a_n)=f_{\underline{a}}(x_1, \ldots, x_n)$$
for all $x_i\in p^{\lceil r\rceil}\Z_p$, $i=1,\ldots, n$. Here $\lceil r\rceil$ stands for the smallest integer that is greater or equal to $r$.
\item[(ii)] Let $C^{r-\an}(\Z_p^n, B)$ (resp., $C^{r-\an}((\Z_p^{\times})^n, B)$) denote the set of $r$-analytic functions from $\Z_p^n$ (resp., $(\Z_p^{\times})^n$) to $B$.
\item[(iii)] Let $C^{r-\an}(\Z_p^n, B^{\circ})$ (resp., $C^{r-\an}((\Z_p^{\times})^n, B^{\circ})$) denote the subset of $C^{r-\an}(\Z_p^n, B)$ (resp., $C^{r-\an}((\Z_p^{\times})^n, B)$) consisting of those functions with value in $B^{\circ}$.
\end{enumerate}
\end{Definition}

\begin{Remark}\label{Remark: uniform Banach algebra structure 2}
\normalfont We claim that $C^{r-\an}(\Z_p^n, B)$ (resp., $C^{r-\an}((\Z_p^{\times})^n, B)$) admits a natural structure of uniform $\C_p$-Banach algebra. Indeed, express $\Z_p^n$ as the disjoint union of $p^{n\lceil r\rceil}$ closed balls of radius $p^{\lceil r\rceil}$, labelled by an index set $A$ of size $p^{n\lceil r\rceil}$. Then, for every $f\in \calC^{r-\an}(\Z_p^n, B)$, the restriction of $f$ on each closed ball (with label $a\in A$) is given by a power series $$f_a\in B\langle \frac{T_1}{p^r}, \ldots, \frac{T_n}{p^r}\rangle$$
where $B\langle \frac{T_1}{p^r}, \ldots, \frac{T_n}{p^r}\rangle$ stands for the subset of $B\llbrack T_1, \ldots, T_n\rrbrack$ which converges on the $n$-dimensional closed unit ball $\mathbf{B}^n(0, p^{-r})\subset \C_p^n$. Let $|\bullet|_B$ be the unique power-multiplicative norm on $B$. We can equip $B\langle \frac{T_1}{p^r}, \ldots, \frac{T_n}{p^r}\rangle$ with the following norm: for every $g=\sum_{\nu\in \Z_{\geq 0}^n} b_{\nu}T^{\nu}$, we put 
$$|g|:=\sup_{\nu\in \Z_{\geq 0}^n}|b_{\nu}|_B\cdot p^{-r |\nu|}.$$
Finally, if $f\in C^{r-\an}(\Z_p^n, B)$ is represented by $\{f_a\}_{a\in A}$, we put $|f|:=\sup_{a\in A} |f_a|$. This is indeed a uniform Banach norm with unit ball $\calC^{r-\an}(\Z_p^n, B^{\circ})$.
\end{Remark}

\begin{Definition}\label{Definition: w-analytic weight}
\begin{enumerate}
\item[(i)] A weight $(R_{\calU}, \kappa_{\calU})$ is called \textbf{$r$-analytic} if it is $r$-analytic when viewed as a function
$$\kappa_{\calU}:(\Z_p^{\times})^g\rightarrow R_{\calU}^{\times}\subset \C_p\widehat{\otimes}R_{\calU}$$
via the identification $T_{\GL_g,0}\cong (\Z_p^{\times})^g$.
\item[(ii)] For a weight $(R_{\calU}, \kappa_{\calU})$, we write $r_{\calU}$ for the smallest positive integer $r$ such that the weight is $r$-analytic.
\end{enumerate}
\end{Definition}

\begin{Remark}
\normalfont \begin{enumerate}
\item[(i)] It is well-known that every continuous character $\Z_p^{\times}\rightarrow R_{\calU}^{\times}$ is $r$-analytic for sufficiently large $r$. Moreover, if such a character is $r$-analytic, it necessarily extends to a character
$$\Z_p^{\times}(1+p^{r+1}\calO_{\C_p})\rightarrow (\calO_{\C_p}\widehat{\otimes}R_{\calU}^+)^{\times}\subset \C_p\widehat{\otimes}R_{\calU}.$$ See, for example, \cite[Proposition 2.6]{CHJ-2017}. 
\item[(ii)] If we write $\kappa_{\calU}=(\kappa_{\calU, 1}\ldots, \kappa_{\calU, g})$ with components $\kappa_{\calU,i}:\Z_p^{\times}\rightarrow R_{\calU}^{\times}$, then $\kappa_{\calU}$ is $r$-analytic if and only if all $\kappa_{\calU,i}$'s are $r$-analytic.
In this case, for any $w\in \Q_{>0}$ with $w>1+r_{\calU}$, $\kappa_{\calU}$ extends to a character
$$\kappa_{\calU}: T^{(w)}_{\GL_g, 0}\rightarrow (\calO_{\C_p}\widehat{\otimes}R_{\calU}^+)^{\times}\subset \C_p\widehat{\otimes}R_{\calU}.$$
\end{enumerate}
\end{Remark}

\begin{Definition}
Let $B$ be a uniform $\C_p$-Banach algebra.
\begin{enumerate}
\item[(i)] A function $\psi: U_{\GL_g, 1}^{\opp}\rightarrow B$ is called \textbf{$r$-analytic} if, under the (topological) identification
$$U^{\opp}_{\GL_g, 1}=\begin{pmatrix}1&&&\\p\Z_p&1&&\\\vdots&&\ddots&\\p\Z_p&\ldots&p\Z_p&1\end{pmatrix}\simeq \Z_p^{\frac{(g-1)g}{2}},$$
the function $$\psi:\Z_p^{\frac{(g-1)g}{2}}\rightarrow B$$
is $r$-analytic. Let $C^{r-\an}(U^{\opp}_{\GL_g,1}, B)$ denote the space of such functions. 
\item[(ii)] Let $(R_{\calU}, \kappa_{\calU})$ be an $r$-analytic weight. Using the decomposition $B_{\GL_g,0}=T_{\GL_g,0}U_{\GL_g,0}$, we extend $\kappa_{\calU}$ to a group homomorphism $\kappa_{\calU}: B_{\GL_g,0}\rightarrow R_{\calU}^{\times}$ by setting $\kappa_{\calU}|_{U_{\GL_g,0}}=1$. Define 
$$C^{r-\an}_{\kappa_{\calU}}(\Iw_{\GL_g}, B):=\left\{f: \Iw_{\GL_g}\rightarrow B: \begin{array}{l}
    f(\bfgamma\bfbeta)=\kappa_{\calU}(\bfbeta)f(\bfgamma),\,\,\forall \bfbeta\in B_{\GL_g, 0},\,\,\bfgamma \in \Iw_{\GL_g}\\
    f|_{U_{\GL_g, 1}^{\opp}}\text{ is $r$-analytic} 
\end{array}\right\}.$$ 
\item[(iii)] Let $C^{r-\an}_{\kappa_{\calU}}(\Iw_{\GL_g}, B^{\circ})$ denote the subset of $C^{r-\an}_{\kappa_{\calU}}(\Iw_{\GL_g}, B)$ consisting of those functions with value in $B^{\circ}$.
\end{enumerate}
\end{Definition}

\begin{Remark}\label{Remark: uniform Banach algebra structure 3}
\normalfont According to Remark \ref{Remark: uniform Banach algebra structure 2}, $C^{r-\an}(U^{\opp}_{\GL_g,1}, B)$ admits a structure of uniform $\C_p$-Banach algebra. Notice that an element in $C^{r-\an}_{\kappa_{\calU}}(\Iw_{\GL_g}, B)$ is determined by its restriction on $U^{\opp}_{\GL_g, 1}$. Consequently, $C^{r-\an}_{\kappa_{\calU}}(\Iw_{\GL_g}, B)$ admits a structure of uniform $\C_p$-Banach algebra via the identification $$C^{r-\an}_{\kappa_{\calU}}(\Iw_{\GL_g}, B)\cong C^{r-\an}(U^{\opp}_{\GL_g,1}, B).$$ In particular, $C^{r-\an}_{\kappa_{\calU}}(\Iw_{\GL_g}, B^{\circ})$ is the corresponding unit ball in $C^{r-\an}_{\kappa_{\calU}}(\Iw_{\GL_g}, B)$.
\end{Remark}

\begin{Remark}\label{Remark: extend to w-analytic nbhd}
\normalfont Let $\kappa_{\calU}$ be a weight and let $w\in \Q_{>0}$ with $w>r_{\calU}+1$. Recall that we have a decomposition $B_{\GL_g, 0}^{(w)}=T_{\GL_g,0}^{(w)}U_{\GL_g,0}^{(w)}$. Since $w>1+r_{\calU}$, $\kappa_{\calU}$ extends to a character on $T_{\GL_g,0}^{(w)}$, and hence to a character on $B_{\GL_g,0}^{(w)}$ by setting $\kappa_{\calU}|_{U_{\GL_g,0}^{(w)}}=0$.

We claim that every element $f$ in $C^{w-\an}_{\kappa_{\calU}}(\Iw_{\GL_g}, B)$ (resp., $C^{w-\an}_{\kappa_{\calU}}(\Iw_{\GL_g}, B^{\circ})$) naturally extends to a function $$f:\Iw_{\GL_g}^{(w)}\rightarrow B\quad(\text{resp., }\,\,f:\Iw_{\GL_g}^{(w)}\rightarrow B^{\circ})$$ such that $f(\bfgamma\bfbeta)=\kappa_{\calU}(\bfbeta)f(\bfgamma)$ for all $\bfbeta\in B_{\GL_g, 0}^{(w)}$ and $\bfgamma\in \Iw_{\GL_g}^{(w)}$. Indeed, we have a decomposition $$\Iw_{\GL_g}^{(w)}=U_{\GL_g, 1}^{\opp,(w)}T_{\GL_g,0}^{(w)}U_{\GL_g,0}^{(w)}.$$ Then for every $\bfnu\in U^{\opp,(w)}_{\GL_g,1}$, $\bftau\in T_{\GL_g,0}^{(w)}$, and $\bfnu'\in U_{\GL_g,0}^{(w)}$, we put $$f(\bfnu\bftau\bfnu')=\kappa_{\calU}(\bftau)f(\bfnu).$$ It is straightforward to check that $f$ is well-defined and satisfies the required condition. 
\end{Remark}

\begin{Definition}\label{Definition: strict Iwahori action on w-analytic representation for IwGL_g}
As a consequence of Remark \ref{Remark: extend to w-analytic nbhd}, given $w\in \Q_{>0}$ with $w>1+r_{\calU}$, there is a natural left action of $\Iw_{\GL_g}^{+, (w)}$ on $C^{w-\an}_{\kappa_{\calU}}(\Iw_{\GL_g}, B)$ and $C^{w-\an}_{\kappa_{\calU}}(\Iw_{\GL_g}, B^{\circ})$ (hence also a left action of $\Iw_{\GL_g}^+$) given by $$
    (\bfgamma\cdot f)(\bfgamma')=f(\trans\bfgamma\bfgamma')
$$
for all $\bfgamma\in \Iw^{+, (w)}_{\GL_g}$, $\bfgamma'\in \Iw_{\GL_g}$, and $f\in C^{w-\an}_{\kappa_{\calU}}(\Iw_{\GL_g}, B)$ (resp., $C^{w-\an}_{\kappa_{\calU}}(\Iw_{\GL_g}, B^{\circ})$). This left action is denoted by $\rho_{\kappa_{\calU}}: \Iw^{+,(w)}_{\GL_g}\rightarrow \Aut(C^{w-\an}_{\kappa_{\calU}}(\Iw_{\GL_g}, B))$ (resp., $\rho_{\kappa_{\calU}}: \Iw^{+,(w)}_{\GL_g}\rightarrow \Aut(C^{w-\an}_{\kappa_{\calU}}(\Iw_{\GL_g}, B^{\circ}))$).
\end{Definition}

We are ready to define the sheaf of overconvergent Siegel modular forms. 

\begin{Definition}\label{Definition: the sheaf of overconvergent Siegel forms}Let $(R_{\calU}, \kappa_{\calU})$ be a weight and let $w\in \Q_{>0}$ with $w> 1+r_{\calU}$.
\begin{enumerate}
\item[(i)] Let $\scrO_{\overline{\calX}_{\Gamma(p^{\infty}), w}}\widehat{\otimes}R_{\calU}$ be the sheaf on $\overline{\calX}_{\Gamma(p^{\infty}), w}$ given by $$\calY\mapsto \scrO_{\overline{\calX}_{\Gamma(p^{\infty}), w}}(\calY)\widehat{\otimes}R_{\calU}$$ for every affinoid open subset $\calY\subset \overline{\calX}_{\Gamma(p^{\infty}), w}$. This is in fact a sheaf of uniform $\C_p$-Banach algebra; i.e., $(\scrO_{\overline{\calX}_{\Gamma(p^{\infty}), w}}\widehat{\otimes}R_{\calU})(\calY)$ is a uniform $\C_p$-Banach algebra for every affinoid open $\calY$.

Similarly, let $\scrO^+_{\overline{\calX}_{\Gamma(p^{\infty}), w}}\widehat{\otimes}R^+_{\calU}$ be the sheaf on $\overline{\calX}_{\Gamma(p^{\infty}), w}$ given by $$\calY\mapsto \scrO^+_{\overline{\calX}_{\Gamma(p^{\infty}), w}}(\calY)\widehat{\otimes}R^+_{\calU}$$ for every affinoid open subset $\calY\subset \overline{\calX}_{\Gamma(p^{\infty}), w}$.
\item[(ii)] For any $r\in \Q_{>0}$ with $r>1+ r_{\calU}$, let $\scrC^{r-\an}_{\kappa_{\calU}}(\Iw_{\GL_g}, \scrO_{\overline{\calX}_{\Gamma(p^{\infty}), w}}\widehat{\otimes}R_{\calU})$ denote the sheaf on $\overline{\calX}_{\Gamma(p^{\infty}), w}$ given by 
$$\calY\mapsto C^{r-\an}_{\kappa_{\calU}}(\Iw_{\GL_g}, \scrO_{\overline{\calX}_{\Gamma(p^{\infty}), w}}(\calY)\widehat{\otimes}R_{\calU})$$ for every affinoid open subset $\calY\subset \overline{\calX}_{\Gamma(p^{\infty}), w}$. This is also a sheaf of uniform $\C_p$-Banach algebra.

The sheaf $\scrC^{r-\an}_{\kappa_{\calU}}(\Iw_{\GL_g}, \scrO^+_{\overline{\calX}_{\Gamma(p^{\infty}), w}}\widehat{\otimes}R^+_{\calU})$ is defined in the same way.

\item[(iii)] The \textbf{sheaf of $w$-overconvergent Siegel modular forms of strict Iwahori level and weight $\kappa_{\calU}$} is a subsheaf $\underline{\omega}_w^{\kappa_{\calU}}$ of $h_{\Iw^+, *}\scrC^{w-\an}_{\kappa_{\calU}}(\Iw_{\GL_g}, \scrO_{\overline{\calX}_{\Gamma(p^{\infty}), w}}\widehat{\otimes}R_{\calU})$ defined as follows. For every affinoid open subset $\calV\subset \overline{\calX}_{\Iw^+, w}$ with $\calV_{\infty} = h_{\Iw^+}^{-1}(\calV)$, we put $$\underline{\omega}_w^{\kappa_{\calU}}(\calV):=\left\{f\in C^{w-\an}_{\kappa_{\calU}}(\Iw_{\GL_g}, \scrO_{\overline{\calX}_{\Gamma(p^{\infty}), w}}(\calV_{\infty})\widehat{\otimes}R_{\calU}): \bfgamma^*f=\rho_{\kappa_{\calU}}(\bfgamma_a+\frakz\bfgamma_c)^{-1} f,\,\,\forall \bfgamma=\begin{pmatrix}\bfgamma_a & \bfgamma_b\\ \bfgamma_c & \bfgamma_d\end{pmatrix}\in \Iw_{\GSp_{2g}}^+\right\}.$$ 
Here, $\bfgamma^*f$ stands for the left action of $\bfgamma$ on $\scrO_{\overline{\calX}_{\Gamma(p^{\infty}), w}}$ induced by the natural right $\Iw_{\GSp_{2g}}^+$-action on $\overline{\calX}_{\Gamma(p^{\infty}), w}$ defined in \S \ref{subsection: Hodge--Tate period map and the w-ordinary locus}. 
    
Similarly, the \textbf{sheaf of integral $w$-overconvergent Siegel modular forms of strict Iwahori level and weight $\kappa_{\calU}$} is a subsheaf $\underline{\omega}_w^{\kappa_{\calU},+}$ of $h_{\Iw^+, *}\scrC^{w-\an}_{\kappa_{\calU}}(\Iw_{\GL_g}, \scrO^+_{\overline{\calX}_{\Gamma(p^{\infty}), w}}\widehat{\otimes}R^+_{\calU})$ defined as follows. For every affinoid open subset $\calV\subset \overline{\calX}_{\Iw^+, w}$ with $\calV_{\infty} = h_{\Iw^+}^{-1}(\calV)$, we put $$\underline{\omega}_w^{\kappa_{\calU},+}(\calV):=\left\{f\in C^{w-\an}_{\kappa_{\calU}}(\Iw_{\GL_g}, \scrO^+_{\overline{\calX}_{\Gamma(p^{\infty}), w}}(\calV_{\infty})\widehat{\otimes}R^+_{\calU}): \bfgamma^*f=\rho_{\kappa_{\calU}}(\bfgamma_a+\frakz\bfgamma_c)^{-1} f,\,\,\forall \bfgamma=\begin{pmatrix}\bfgamma_a & \bfgamma_b\\ \bfgamma_c & \bfgamma_d\end{pmatrix}\in \Iw_{\GSp_{2g}}^+\right\}.$$

\item[(iv)] The \textbf{space of $w$-overconvergent Siegel modular forms of strict Iwahori level and weight $\kappa_{\calU}$} is defined to be $$M_{\Iw^+, w}^{\kappa_{\calU}}:=H^0(\overline{\calX}_{\Iw^+, w},\, \underline{\omega}_w^{\kappa_{\calU}}).$$ We similarly define the \textbf{space of integral $w$-overconvergent Siegel modular forms of strict Iwahori level and weight $\kappa_{\calU}$} to be $$M_{\Iw^+, w}^{\kappa_{\calU},+}:=H^0(\overline{\calX}_{\Iw^+, w}, \,\underline{\omega}_w^{\kappa_{\calU}, +}).$$
    
\item[(v)] Taking limit with respect to $w$, the \textbf{space of overconvergent Siegel modular forms of strict Iwahori level and weight $\kappa_{\calU}$} is 
$$M_{\Iw^+}^{\kappa_{\calU}}:=\lim_{w\rightarrow\infty}M^{\kappa_{\calU}}_{\Iw, w}.$$
Similarly, the \textbf{space of integral overconvergent Siegel modular forms of strict Iwahori level and weight $\kappa_{\calU}$} is 
$$M^{\kappa_{\calU}, +}_{\Iw^+}:=\lim_{w\rightarrow\infty}M^{\kappa_{\calU},+}_{\Iw^+, w}.$$

\item[(vi)] Recall that $\calZ_{\Iw^+} = \overline{\calX}_{\Iw^+}\smallsetminus \calX_{\Iw^+}$ is the boundary divisor. The \textbf{sheaf of $w$-overconvergent Siegel cuspforms of strict Iwahori level and weight $\kappa_{\calU}$} is defined to be the subsheaf $\underline{\omega}_{w, \cusp}^{\kappa_{\calU}} = \underline{\omega}_{w}^{\kappa_{\calU}}(-\calZ_{\Iw^+})$ of $\underline{\omega}_{w}^{\kappa_{\calU}}$ consisting of sections that vanish along $\calZ_{\Iw^+}$.

A $w$-overconvergent Siegel modular form of strict Iwahori level and weight $\kappa_{\calU}$ is called \textbf{cuspidal} if it is an element of  \[
        S_{\Iw^+, w}^{\kappa_{\calU}}:= H^0(\overline{\calX}_{\Iw^+, w},\, \underline{\omega}_{w, \cusp}^{\kappa_{\calU}}).
\] 

Moreover, by taking limit with respect to $w$, the \textbf{space of overconvergent Siegel cuspforms of strict Iwahori level and weight $\kappa_{\calU}$} is defined to be $$S^{\kappa_{\calU}}_{\Iw^+}:=\lim_{w\rightarrow\infty}S^{\kappa_{\calU}}_{\Iw^+, w}.$$
\end{enumerate}
\end{Definition}

\begin{Remark}\label{Remark: well-defined}
\normalfont Notice that, in Definition \ref{Definition: the sheaf of overconvergent Siegel forms} (iii), for every $\bfitx\in \overline{\calX}_{\Gamma(p^{\infty}), w}(\C_p, \calO_{\C_p})$ and any $\begin{pmatrix}\bfgamma_a & \bfgamma_b \\ \bfgamma_c & \bfgamma_d\end{pmatrix}\in \Iw_{\GSp_{2g}}^+$, we have $\bfgamma_a+\frakz(\bfitx)\bfgamma_c\in \Iw^{+, (w)}_{\GL_g}$. Hence, $\rho_{\kappa_{\calU}}(\bfgamma_a+\frakz\bfgamma_c)$ is well-defined.
\end{Remark}

To simplify the notation, we defined a ``twisted'' left action of $\Iw^+_{\GSp_{2g}}$ on $\scrC^{w-\an}_{\kappa_{\calU}}(\Iw_{\GL_g}, \scrO_{\overline{\calX}_{\Gamma(p^{\infty}),w}}\widehat{\otimes}R_{\calU})$ by
$$\bfgamma. f:= \rho_{\kappa_{\calU}}(\bfgamma_a+\frakz \bfgamma_c)\bfgamma^*f.$$
Then sections of $\underline{\omega}^{\kappa_{\calU}}_w$ are precisely the $\Iw^+_{\GSp_{2g}}$-invariant sections of $h_{\Iw^+, *}\scrC^{w-\an}_{\kappa_{\calU}}(\Iw_{\GL_g}, \scrO_{\overline{\calX}_{\Gamma(p^{\infty}),w}}\widehat{\otimes}R_{\calU})$ under the twisted action.

\begin{Remark}
\normalfont The sheaf $\underline{\omega}_w^{\kappa_{\calU}}$ is functorial in the weight $(R_{\calU}, \kappa_{\calU})$. Given a map of weights $R_{\calU} \rightarrow R_{\calU'}$ and $w > \mbox{max}\left\{1+r_{\calU}, 1+r_{\calU'}\right\}$, we obtain a natural map $\underline{\omega}_w^{\kappa_{\calU}} \rightarrow \underline{\omega}_w^{\kappa_{\calU'}}$ induced from
\[C^{w-\an}_{\kappa_{\calU}}(\Iw_{\GL_g}, \scrO_{\overline{\calX}_{\Gamma(p^{\infty}), w}}\widehat{\otimes}R_{\calU}) \rightarrow C^{w-\an}_{\kappa_{\calU'}}(\Iw_{\GL_g}, \scrO_{\overline{\calX}_{\Gamma(p^{\infty}), w}}\widehat{\otimes}R_{\calU'}).\]  
\end{Remark}

\subsection{Hecke operators}\label{subsection: Hecke operators on the overconvergent automorphic forms}
In this subsection, we spell out how the Hecke operators act on the overconvergent Siegel modular forms. The Hecke operators at the primes dividing the tame level $N$ are not considered in this paper.

Throughout this subsection, let $(R_{\calU}, \kappa_{\calU})$ be a weight and $w>1+r_{\calU}$.\\

\noindent\textbf{Hecke operators outside $Np$.} We define the Hecke operators outside $Np$ using correspondences. Let $\ell$ be a rational prime that does not divide $Np$. For every $\bfgamma\in \GSp_{2g}(\Q_{\ell})\cap M_{2g}(\Z_{\ell})$, consider the moduli space $X_{\bfgamma, \Iw^+}$ over $X_{\Iw^+}$ parameterising \textit{isogenies of type} $\bfgamma$. More precisely, $X_{\bfgamma, \Iw^+}$ is the moduli space of quintuple $$(A, \lambda, \psi_N, \{C_i: i=1, \ldots, g\}, L)$$ where $(A, \lambda, \psi_N, \{C_i: i=1, \ldots, g\})\in X_{\Iw^+}$ and $L\subset A$ is a subgroup of finite order such that the isogeny $(A, \lambda)\rightarrow (A/L, \lambda')$ is of type $\bfgamma$ in the sense of \cite[Chapter VII, \S 3]{Faltings-Chai}, where $\lambda'$ stands for the induced principal polarisation. According to \emph{loc. cit.}, for every isogeny of type $\bfgamma$, its dual isogeny is also of type $\bfgamma$. In particular, the assignment $$(A, \lambda, \psi_N,  \{C_i: i=1, \ldots, g\}, L)\mapsto (A'=A/L, \lambda', \psi'_N,  \{C'_i: i=1, \ldots, g\}, L')$$ defines an isomorphism $\Phi_{\bfgamma}: X_{\bfgamma, \Iw^+}\xrightarrow[]{\sim}X_{\bfgamma, \Iw^+}$, where 
\begin{itemize}
\item $\lambda'$ is the induced polarisation on $A'$;
\item $\psi'_N$ and $C'_i$'s are induced from $\psi_N$ and $C_i$'s, respectively, via the isomorphisms $A[N]\simeq A'[N]$ and $A[p]\simeq A'[p]$;
\item $L'$ is defined by the dual isogeny of $(A,\lambda)\rightarrow (A',\lambda')$.
\end{itemize}

There are two finite \'etale projections 
$$
\begin{tikzcd}
& X_{\bfgamma, \Iw^+}\arrow[ld, "\pr_1"']\arrow[rd, "\pr_2"] \\
X_{\Iw^+} & & X_{\Iw^+}
\end{tikzcd}$$ 
where $\pr_1$ is the forgetful map and $\pr_2$ sends the quintuple $(A, \lambda, \psi_N, \{C_i: i=1, \ldots, g\}, L)$ to the quintuple $(A'=A/L, \lambda', \psi'_N, \{C'_i: i=1, \ldots, g\})$ described as above. Clearly, we have $\pr_1=\pr_2\circ \Phi_{\bfgamma}$.

Let $\calX_{\bfgamma, \Iw^+}$ be the adic space associated with $X_{\bfgamma, \Iw^+}$ by taking analytification. We obtain finite \'etale morphisms $\pr_1, \pr_2:\calX_{\bfgamma, \Iw^+}\rightrightarrows \calX_{\Iw^+}$ as well as an isomorphism $\Phi_{\bfgamma}:\calX_{\bfgamma, \Iw^+}\rightarrow \calX_{\bfgamma, \Iw^+}$. We further pass to the $w$-ordinary loci. More precisely, let $\calX_{\bfgamma, \Iw^+, w}$ denote the preimage of $\calX_{\Iw^+, w}$ under the projection $\pr_1$. Notice that $\Phi_{\bfgamma}$ preserves $\calX_{\bfgamma, \Iw^+, w}$ as the isogeny $(A, \lambda)\rightarrow (A', \lambda')$ induces a symplectic isomorphism $T_pA\cong T_pA'$. Hence, we obtain finite \'etale morphisms 
\begin{equation}\label{eq: correspondence-strict-Iw-level}
\begin{tikzcd}
& \calX_{\bfgamma, \Iw^+,w}\arrow[ld, "\pr_1"']\arrow[rd, "\pr_2"] \\
\calX_{\Iw^+,w} & & \calX_{\Iw^+,w}
\end{tikzcd}
\end{equation}
and an isomorphism $\Phi_{\bfgamma}:\calX_{\bfgamma,\Iw^+,w}\xrightarrow[]{\sim} \calX_{\bfgamma,\Iw^+,w}$. We still have $\pr_1=\pr_2\circ \Phi_{\bfgamma}$. 

In order to define the Hecke operator, we shall first construct a natural isomorphism
$$\varphi_{\bfgamma}: \pr_2^*\underline{\omega}^{\kappa_{\calU}}_w\xrightarrow[]{\sim}\pr_1^*\underline{\omega}^{\kappa_{\calU}}_w.$$
Here we have abused the notation and still write $\underline{\omega}^{\kappa_{\calU}}_w$ for its restriction to $\calX_{\Iw^+, w}$.

Indeed, pulling back the diagram (\ref{eq: correspondence-strict-Iw-level}) along the projection $h_{\Iw^+}:\calX_{\Gamma(p^{\infty}), w}\rightarrow \calX_{\Iw^+, w}$, we obtain finite \'etale morphisms
$$
\begin{tikzcd}
& \calX_{\bfgamma, \Gamma(p^{\infty}), w}\arrow[ld, "\pr_{1, \infty}"']\arrow[rd, "\pr_{2, \infty}"] \\
\calX_{\Gamma(p^{\infty}),w} & & \calX_{\Gamma(p^{\infty}),w}
\end{tikzcd}
$$
between perfectoid spaces and an $\Iw_{\GSp_{2g}}^+$-equivariant isomorphism $\Phi_{\bfgamma, \infty}:\calX_{\bfgamma,\Gamma(p^{\infty}),w}\xrightarrow[]{\sim} \calX_{\bfgamma,\Gamma(p^{\infty}),w}$. The isomorphism $\Phi_{\bfgamma, \infty}$ induces an isomorphism
$$\Phi^*_{\bfgamma, \infty}:\pr_{2,\infty}^*\scrO_{\calX_{\Gamma(p^{\infty})}, w}\xrightarrow[]{\sim}\pr_{1,\infty}^*\scrO_{\calX_{\Gamma(p^{\infty})}, w}.$$
It then induces an isomorphism
$$\Phi^*_{\bfgamma, \infty}:\scrC^{w-\an}_{\kappa_{\calU}}(\Iw_{\GL_g}, \pr_{2,\infty}^*\scrO_{\calX_{\Gamma(p^{\infty})}, w}\widehat{\otimes}R_{\calU})\xrightarrow[]{\sim}\scrC^{w-\an}_{\kappa_{\calU}}(\Iw_{\GL_g}, \pr_{1,\infty}^*\scrO_{\calX_{\Gamma(p^{\infty})}, w}\widehat{\otimes}R_{\calU})$$
by taking the identity on $R_{\calU}$.

Recall that $\frakz$ is the pullback of the coordinate $\bfitz$ via the Hodge--Tate period map $\pi_{\HT}: \calX_{\Gamma(p^{\infty}), w}\rightarrow \adicFL^{\times}_w$. Let $\frakz':=\pr^*_{1, \infty}\frakz$ and $\frakz'':=\pr^*_{2, \infty}\frakz$. Since $\Phi_{\bfgamma, \infty}$ induces an isomorphism on the $p$-adic Tate module, we have $\frakz'=\frakz''$. Consequently, a section $f$ of $\scrC^{w-\an}_{\kappa_{\calU}}(\Iw_{\GL_g}, \pr_{2,\infty}^*\scrO_{\calX_{\Gamma(p^{\infty})}, w}\widehat{\otimes}R_{\calU})$ satisfies 
$$\bfgamma^*f=\rho_{\kappa_{\calU}}(\bfgamma_a+\frakz''\bfgamma_c)^{-1} f$$ for all $\bfgamma=\begin{pmatrix}\bfgamma_a & \bfgamma_b\\ \bfgamma_c & \bfgamma_d\end{pmatrix}\in \Iw_{\GSp_{2g}}^+$, if and only if the section $\Phi^*_{\bfgamma, \infty}(f)$ of $\scrC^{w-\an}_{\kappa_{\calU}}(\Iw_{\GL_g}, \pr_{1,\infty}^*\scrO_{\calX_{\Gamma(p^{\infty})}, w}\widehat{\otimes}R_{\calU})$ satisfies
$$\bfgamma^*(\Phi^*_{\bfgamma, \infty}(f))=\rho_{\kappa_{\calU}}(\bfgamma_a+\frakz'\bfgamma_c)^{-1} (\Phi^*_{\bfgamma, \infty}(f))$$ for all $\bfgamma\in \Iw^+_{\GSp_{2g}}$. This yields the desired isomorphism
$$\varphi_{\bfgamma}: \pr_2^*\underline{\omega}^{\kappa_{\calU}}_w\xrightarrow[]{\sim}\pr_1^*\underline{\omega}^{\kappa_{\calU}}_w.$$

Given this, we consider the composition 
$$
    \begin{tikzcd}
    T_{\bfgamma}: & H^0(\calX_{\Iw^+, w}, \underline{\omega}_w^{\kappa_{\calU}})\arrow[r, "\pr_2^*"] & H^0(\calX_{\bfgamma, \Iw^+, w}, \pr_2^*\underline{\omega}_w^{\kappa_{\calU}}) \arrow[ld, out=-10, in=170, "\varphi_{\bfgamma}"']\\
    & H^0(\calX_{\bfgamma, \Iw^+, w}, \pr_1^*\underline{\omega}_w^{\kappa_{\calU}}) \arrow[r, "\Tr\pr_1"] & H^0(\calX_{\Iw^+, w}, \underline{\omega}_w^{\kappa_{\calU}}). \end{tikzcd}
$$

Finally, we have to extend the construction to the boundary. In fact, we shall prove that the sections of $\underline{\omega}^{\kappa_{\calU}}_w$ on $\overline{\calX}_{\Iw^+,w}$ are precisely the \emph{bounded} sections of $\underline{\omega}^{\kappa_{\calU}}_w$ over the open part $\calX_{\Iw^+,w}$, at least when $g \geq 2$. %By definition, a section of $\underline{\omega}^{\kappa_{\calU}}_w$ on $\overline{\calX}_{\Iw^+,w}$ can be viewed a section of $\scrC^{w-\an}_{\kappa_{\calU}}(\Iw_{\GL_g}, \scrO_{\overline{\calX}_{\Gamma(p^{\infty}),w}}\widehat{\otimes}R_{\calU})$, which is a sheaf of uniform $\C_p$-Banach algebra by Definition \ref{Definition: the sheaf of overconvergent Siegel forms} (ii). In particular, the notion of bounded section of $\underline{\omega}^{\kappa_{\calU}}_w$ on $\overline{\calX}_{\Iw^+,w}$ is well-defined.

\begin{Lemma}\label{Lemma: extend to the boundary}
Suppose $g\geq 2$. Every bounded section of $\underline{\omega}^{\kappa_{\calU}}_w$ on $\calX_{\Iw^+,w}$ uniquely extends to a section of $\overline{\calX}_{\Iw^+,w}$. In particular, the definition of overconvergent Siegel modular forms of weight $\kappa_{\calU}$ is independent of the choice of the polyhedral cone decomposition in the toroidal compactification. 
\end{Lemma}

\begin{proof}
By the discussion in \S \ref{subsection: admissibility} below, for a sufficiently large $n$, every section of $\underline{\omega}_{w}^{\kappa_{\calU}}$ can be viewed as a section of some auxiliary sheaf $\underline{\omega}_{n,w}^{\kappa_{\calU}}$ on $\overline{\calX}_{\Gamma(p^n), w}$. Moreover, by Proposition \ref{Proposition: omega is admissible}, there is a torsor $\adicIW_{w}^+$ over $\overline{\calX}_{\Gamma(p^n), w}$ such that every section of $\underline{\omega}_{n,w}^{\kappa_{\calU}}$ can be viewed as an element in $\scrO_{\adicIW_w^+}\widehat{\otimes}R_{\calU}$. 

If $R_{\calU}$ is a small weight, by choosing a pseudo-basis $(e_i)_{i\in I}$ of $R_{\calU}$ in the sense of \cite[Proposition 6.2]{CHJ-2017}, we can identify 
\[\scrO_{\adicIW_w^+}\widehat{\otimes}R_{\calU}\simeq \prod_{i\in I}\scrO_{\adicIW_w^+}\] 
using Proposition 6.4 of \emph{loc. cit.}. Hence, a bounded section of $\underline{\omega}_{w}^{\kappa_{\calU}}$ on $\calX_{\Iw^+,w}$ can be identified with a collection of bounded functions on the open part of $\adicIW_w^+$ (\emph{i.e.}, the part away from the boundary of $\overline{\calX}_{\Gamma(p^n), w}$) indexed by $I$. By applying \cite[Theorem 1.6]{Lutkebohmert} to $\adicIW_w^+$, the result follows.

If $R_{\calU}$ is an affinoid weight, a bounded section of $\underline{\omega}_{w}^{\kappa_{\calU}}$ on $\calX_{\Iw^+,w}$ can be identified with a bounded function on the open part of $\adicIW_w^+\times_{\Spa(\Q_p, \Z_p)} \Spa(R_{\calU}, R^+_{\calU})$ (\emph{i.e.}, the part away from the boundary of $\overline{\calX}_{\Gamma(p^n), w}$). Once again, applying \cite[Theorem 1.6]{Lutkebohmert} to $\adicIW_w^+
\times_{\Spa(\Q_p, \Z_p)} \Spa(R_{\calU}, R^+_{\calU})$ does the job.
\end{proof}

Thanks to Lemma \ref{Lemma: extend to the boundary} and the fact that 
$$\Phi^*_{\bfgamma, \infty}:\scrC^{w-\an}_{\kappa_{\calU}}(\Iw_{\GL_g}, \pr_{2,\infty}^*\scrO_{\calX_{\Gamma(p^{\infty})}, w}\widehat{\otimes}R_{\calU})\xrightarrow[]{\sim}\scrC^{w-\an}_{\kappa_{\calU}}(\Iw_{\GL_g}, \pr_{1,\infty}^*\scrO_{\calX_{\Gamma(p^{\infty})}, w}\widehat{\otimes}R_{\calU})$$
sends bounded sections to bounded sections, we know that $T_{\bfgamma}$ extends to the boundary. We arrive at the Hecke operator $$T_{\bfgamma}: M_{\Iw^+, w}^{\kappa_{\calU}}=H^0(\overline{\calX}_{\Iw^+,w}, \underline{\omega}^{\kappa_{\calU}}_w)\rightarrow H^0(\overline{\calX}_{\Iw^+,w}, \underline{\omega}^{\kappa_{\calU}}_w)= M_{\Iw^+, w}^{\kappa_{\calU}}.$$

For $g=1$, Lemma \ref{Lemma: extend to the boundary} is not true (think of the $j$-invariant) but as the toroidal compactification coincides with the minimal, we can extend the projections $\pr_1$  and $\pr_2$ to finite maps 

\begin{equation}\label{eq: correspondence-strict-Iw-level g=1}
\begin{tikzcd}
& \overline{\calX}_{\bfgamma, \Iw^+,w}\arrow[ld, "\pr_1"']\arrow[rd, "\pr_2"] \\
\overline{\calX}_{\Iw^+,w} & & \overline{\calX}_{\Iw^+,w}
\end{tikzcd}
\end{equation}
and the Hecke operator $T_{\bfgamma}$ naturally extends to the boundary
$$T_{\bfgamma}: M_{\Iw^+, w}^{\kappa_{\calU}}=H^0(\overline{\calX}_{\Iw^+,w}, \underline{\omega}^{\kappa_{\calU}}_w)\rightarrow H^0(\overline{\calX}_{\Iw^+,w}, \underline{\omega}^{\kappa_{\calU}}_w)= M_{\Iw^+, w}^{\kappa_{\calU}}.$$ 

\noindent\textbf{Hecke operators at $p$.} For $1\leq i\leq g$, we consider matrices $\bfu_{p,i}\in \GSp_{2g}(\Q_p)\cap M_{2g}(\Z_p)$ defined by
$$
\bfu_{p,i}:=\begin{pmatrix} \one_i \\ & p\one_{g-i}\\ & & p\one_{g-i}\\ & & & p^2\one_i\end{pmatrix}$$
for $1\leq i\leq g-1$, and
$$\bfu_{p,g}:=\begin{pmatrix}\one_g\\ & p\one_g\end{pmatrix}.$$
For later use, we write $$\bfu_{p, i}=\begin{pmatrix}\bfu_{p,i}^{\square} & \\ & \bfu_{p, i}^{\blacksquare}\end{pmatrix}$$ where $\bfu_{p, i}^{\square}$ and $\bfu_{p, i}^{\blacksquare}$ are the corresponding $g\times g$ diagonal matrices.

Notice that the $\bfu_{p,i}$-action on $\overline{\calX}_{\Gamma(p^{\infty})}$ preserves $\overline{\calX}_{\Gamma(p^{\infty}),w}$. This can be checked at the infinite level via local coordinates; \emph{i.e.,} the action of $\bfu_{p, i}$ on $\bfitz$ is given by $$\bfitz.\bfu_{p,i}=\bfu_{p,i}^{\square, -1}\bfitz\bfu_{p,i}^{\blacksquare}=\left\{\begin{array}{cl}
    \begin{pmatrix}
    p\bfitz_{1,1} & \cdots & p\bfitz_{1,g-i} & p^2\bfitz_{1,g+1-i} & \cdots & p^2\bfitz_{1,g}\\
    \vdots & & \vdots & \vdots && \vdots\\
    p\bfitz_{i,1} & \cdots & p\bfitz_{i,g-i} & p^2\bfitz_{i,g+1-i} & \cdots & p^2\bfitz_{i,g}\\
    \bfitz_{i+1,1} & \cdots & \bfitz_{i+1,g-i} & p\bfitz_{i+1,g+1-i} & \cdots & p\bfitz_{i+1,g}\\
    \vdots && \vdots & \vdots && \vdots\\
    \bfitz_{g,1} & \cdots & \bfitz_{g,g-i} & p\bfitz_{g,g+1-i} & \cdots & p\bfitz_{g,g}
    \end{pmatrix}, & \text{ if }i=1, ..., g-1 \\ \\
    p\bfitz, & \text{ if }i=g
\end{array}\right..$$ In particular, when $i=g$, the $\bfu_{p,g}$-action actually sends $\overline{\calX}_{\Gamma(p^{\infty}), w}$ into $\overline{\calX}_{\Gamma(p^{\infty}), w+1}$.

Recall the twisted left action of $\Iw^+_{\GSp_{2g}}$ on $\scrC^{w-\an}_{\kappa_{\calU}}(\Iw_{\GL_g}, \scrO_{\overline{\calX}_{\Gamma(p^{\infty}),w}}\widehat{\otimes}R_{\calU})$, given by the formula
$$\bfgamma. f:= \rho_{\kappa_{\calU}}(\bfgamma_a+\frakz \bfgamma_c)\bfgamma^*f.$$

\begin{Definition}\label{Definition:Upi}
\begin{enumerate}
\item[(i)] For $f\in \scrC^{w-\an}_{\kappa_{\calU}}(\Iw_{\GL_g}, \scrO_{\overline{\calX}_{\Gamma(p^{\infty}),w}}\widehat{\otimes}R_{\calU})$, we define $$\bfu_{p,i}.f\in \scrC^{w-\an}_{\kappa_{\calU}}(\Iw_{\GL_g}, \scrO_{\overline{\calX}_{\Gamma(p^{\infty}),w}}\widehat{\otimes}R_{\calU})$$ by
$$\bfu_{p,i}.f(\bfgamma'):=\bfu_{p,i}^*f(\bfu_{p,i}^{\square}\bfgamma'_0\bfu_{p,i}^{\square,-1}\bfbeta'_0)$$ where $\bfgamma'=\bfgamma'_0\bfbeta'_0\in \Iw_{\GL_g}$ with $\bfgamma'_0\in U^{\opp}_{\GL_g,1}$ and $\bfbeta'_0\in B_{\GL_g,0}$.
\item[(ii)] Suppose $f\in \scrC^{w-\an}_{\kappa_{\calU}}(\Iw_{\GL_g}, \scrO_{\overline{\calX}_{\Gamma(p^{\infty}),w}}\widehat{\otimes}R_{\calU})$ satisfies $$\bfgamma^*f=\rho_{\kappa_{\calU}}(\bfgamma_a+\frakz\bfgamma_c)^{-1}f$$ for all $\bfgamma\in \Iw^+_{\GSp_{2g}}$; i.e., $\bfgamma.f=f$. Pick a decomposition of the double coset
$$\Iw^+_{\GSp_{2g}}\bfu_{p,i}\Iw^+_{\GSp_{2g}}=\bigsqcup_{j=1}^m \bfdelta_{ij}\bfu_{p,i}\Iw^+_{\GSp_{2g}}$$
with $\bfdelta_{i,j}\in \Iw^+_{\GSp_{2g}}$. Define
$$U_{p,i}(f):=p^{\nu_i}\sum_{j=1}^m \bfdelta_{i,j}.(\bfu_{p,i}.f)\in \scrC^{w-\an}_{\kappa_{\calU}}(\Iw_{\GL_g}, \scrO_{\overline{\calX}_{\Gamma(p^{\infty}),w}}\widehat{\otimes}R_{\calU}),$$ 
where $\nu_i = -(g-i)(g+1)$ for $i = 1, ..., g-1$ and $\nu_g = \frac{-g(g+1)}{2}$. Here, we follow the normalisation as in \cite[\S 6.2]{AIP-2015}.
\end{enumerate}
\end{Definition}

Similarly as in \cite[\S 2.2]{Hansen-PhD}, the action of $\bfu_{p,i}$'s extends to an action of the semigroup $\Delta$ generated by the double cosets $[\Iw_{\GSp_{2g}}^+ \bfitu_{p,i} \Iw_{\GSp_{2g}}^+]$'s. If $\{\bfdelta_{i,j}'\}_{j=1}^m$ is another set of representatives for the double coset $[\Iw_{\GSp_{2g}}^+ \bfitu_{p,i} \Iw_{\GSp_{2g}}^+]$, up to re-labelling, we may assume \[
    \bfdelta_{ij}'\bfu_{p,i} = \bfdelta_{ij}\bfu_{p,i}\bfgamma_j
\] for some $\bfgamma_j\in \Iw_{\GSp_{2g}}^+$. Then, given a section $f\in \scrC^{w-\an}_{\kappa_{\calU}}(\Iw_{\GL_g}, \scrO_{\overline{\calX}_{\Gamma(p^{\infty}),w}}\widehat{\otimes}R_{\calU})$ satisfying \[
    \bfgamma^* f = \rho_{\kappa_{\calU}}(\bfgamma_a + \frakz\bfgamma_c)^{-1}f
\]
for any $\bfgamma\in \Iw_{\GSp_{2g}}^+$, we have \[
    \bfdelta_{ij}' . (\bfitu_{p,i} f) = \bfdelta_{ij} . (\bfitu_{p,i} \bfgamma_j. f) = \bfdelta_{ij}. (\bfitu_{p,i} f),
\]
which shows that the definition of $U_{p,i}$'s is independent to the choice of the representatives $\{\bfdelta_{ij}\}_{j=1}^m$.

\begin{Lemma}
Suppose $f\in \scrC_{\kappa_{\calU}}^{w-\an}(\Iw_{\GL_g}, \scrO_{\overline{\calX}_{\Gamma(p^{\infty}), w}}\widehat{\otimes} R_{\calU})$ such that $\bfgamma . f = f$. Then, the section $U_{p,i}(f)\in \scrC^{w-\an}_{\kappa_{\calU}}(\Iw_{\GL_g}, \scrO_{\overline{\calX}_{\Gamma(p^{\infty}),w}}\widehat{\otimes}R_{\calU})$ satisfies $\bfgamma.(U_{p,i}(f))=U_{p,i}(f)$ for all $\bfgamma\in \Iw^+_{\GSp_{2g}}$.
\end{Lemma}

\begin{proof}
We have $$\bfgamma.(U_{p,i}(f))= p^{\nu_i}\sum_{j=1}^m \bfgamma.(\bfdelta_{ij}.(\bfu_{p,i}(f)))= p^{\nu_i}\sum_{j=1}^m (\bfgamma\bfdelta_{ij}).(\bfu_{p,i}(f)).$$
The last term indeed computes $U_{p,i}(f)$ because $\{\bfgamma\bfdelta_{ij}: 1\leq j \leq m\}$ is also a valid set of representatives.
\end{proof}

Consequently, we arrive at the Hecke operator $$U_{p,i}: M_{\Iw^+, w}^{\kappa_{\calU}}=H^0(\overline{\calX}_{\Iw^+,w}, \underline{\omega}^{\kappa_{\calU}}_w)\rightarrow H^0(\overline{\calX}_{\Iw^+,w}, \underline{\omega}^{\kappa_{\calU}}_w)= M_{\Iw^+, w}^{\kappa_{\calU}}.$$

\begin{Definition}\label{Definition: Hecke algebra}
The \textbf{Hecke algebra outside $Np$} is defined to be
$$\bbT^{Np}:=\Z_p\left[T_{\bfgamma};  \bfgamma\in \GSp_{2g}(\Q_{\ell})\cap M_{2g}(\Z_{\ell}), \,\,\ell\nmid Np\right]$$
and the \textbf{total Hecke algebra} is defined to be
$$\bbT:= \bbT^{Np}\otimes_{\Z_p}\Z_p[U_{p,i}; i=1, \ldots, g].$$
\end{Definition}
We now define  $U_{p}:=\prod_{i=1}^{g}U_{p,i}$ and conclude with the following proposition
\begin{Proposition} The operator $U_p$  is a compact operator on $M_{\Iw^+, w}^{\kappa_{\calU}}$. 
\end{Proposition}
\begin{proof}
Note that the action of $\bfu_{p,g} $ on $\bfitz$ is given by $p \bfitz$ and that, by definition, the action of $\prod_{i=1}^{g}\bfu_{p,i}$ on $ \scrC^{w-\an}_{\kappa_{\calU}}(\Iw_{\GL_g}, \scrO_{\overline{\calX}_{\Gamma(p^{\infty}),w}}\widehat{\otimes}R_{\calU})$ factors through the inclusion  \[
 \scrC^{(w-1)-\an}_{\kappa_{\calU}}(\Iw_{\GL_g}, \scrO_{\overline{\calX}_{\Gamma(p^{\infty}),w}}\widehat{\otimes}R_{\calU}) \hookrightarrow \scrC_{\kappa_{\calU}}^{w-an}(\Iw_{\GL_g}, \scrO_{\overline{\calX}_{\Gamma(p^{\infty}), w}}\widehat{\otimes}R_{\calU}).
\]  This means that $U_{p}$ factors as 
\[
U_{p}: H^0(\overline{\calX}_{\Iw^+,w}, \underline{\omega}^{\kappa_{\calU}}_w)\rightarrow H^0(\overline{\calX}_{\Iw^+,w+1}, \underline{\omega}^{\kappa_{\calU}}_w){\rightarrow}   H^0(\overline{\calX}_{\Iw^+,w+1}, \underline{\omega}^{\kappa_{\calU}}_{w-1}){\rightarrow} H^0(\overline{\calX}_{\Iw^+,w}, \underline{\omega}^{\kappa_{\calU}}_w),
\]
where the first arrow is the natural restriction map. 

To show the desired result, note that it is known that restrictions of the structure sheaf of $\overline{\calX}_{\Iw^+,w}$ are compact operators. Moreover, by the discussion in \cite[\S 2.2]{Hansen-PhD}, the injection of $(w-1)$-analytic functions into $w$-analytic functions is compact. The assertion then follows by combining these two facts.
\end{proof}

\begin{Remark}\label{Remark: cuspforms are stable under Hecke actions}
\normalfont Note that the subspace $S_{\Iw^+, w}^{\kappa_{\calU}}\subset M_{\Iw^+, w}^{\kappa_{\calU}}$ of $w$-overconvergent Siegel cuspforms of weight $\kappa_{\calU}$ is stable under the action of $\bbT$. Moreover, as $U_p$ is a compact operator on $M_{\Iw^+, w}^{\kappa_{\calU}}$, it is also a compact operator on $S_{\Iw^+, w}^{\kappa_{\calU}}$.
\end{Remark}

\subsection{Admissibility}\label{subsection: admissibility}
Throughout this subsection, let $(R_{\calU}, \kappa_{\calU})$ be a weight and $w>1+r_{\calU}$. Whenever $(R_{\calU}, \kappa_{\calU})$ is a small weight (as we will explcitly point out), we fix an ideal $\fraka_{\calU}\subset R_{\calU}$ defining the profinite adic topology on $R_{\calU}$ and we assume $p\in \fraka_{\calU}$.

The purpose of this subsection is to show that, when $(R_{\calU}, \kappa_{\calU})$ is a small weight, the overconvergent automorphic sheaf $\underline{\omega}^{\kappa_{\calU}}_w$ can be identified with the $G$-invariants of an \emph{admissible Kummer \'etale Banach sheaf} in the sense of Definition \ref{Definition: admissible Banach sheaf}, where $G$ is a finite group. Such a description allows us to apply Corollary \ref{Corollary: generalised projection formula with invariants} to the sheaf $\underline{\omega}^{\kappa_{\calU}}_w$. This will be used in \S \ref{subsection: OES}.

Firstly, we introduce the notion of $w$-compatibility inspired by \cite[\S 4.5]{AIP-2015}.

\begin{Definition}\label{Definition: w-compatible}
Let $R$ be a flat $\calO_{\C_p}$-algebra and suppose $M$ is a free $R$-module of rank $g$. We write $R_w:=R\otimes_{\calO_{\C_p}}\calO_{\C_p}/p^w$ and $M_w:=M\otimes_R R_w$. Let $\underline{\bfitm}:=(\bfitm_1, \ldots, \bfitm_g)$ be an $R_w$-basis for $M_w$. We denote by $\Fil_{\bullet}^{\underline{\bfitm}}$ the full flag
$$0\subset \langle\bfitm_1\rangle\subset \langle \bfitm_1, \bfitm_2\rangle\subset\cdots\subset\langle \bfitm_1, \ldots, \bfitm_g\rangle$$
of the free $R_w$-module $M_w$. Namely, $\Fil_i^{\underline{\bfitm}}=\langle \bfitm_1, \ldots, \bfitm_i\rangle$ for all $i=1, \ldots, g$.
\begin{enumerate}
\item[(i)] A full flag $\Fil_{\bullet}$ of the free $R$-module $M$ is called \textbf{\textit{$w$-compatible with $\underline{\bfitm}$}} if 
$$\Fil_{i}\otimes_{R}R_w=\Fil_i^{\underline{\bfitm}}$$
for all $i=1, \ldots, g$.
\item[(ii)] Suppose $\Fil_{\bullet}$ is a $w$-compatible full flag as in (i). Consider a collection $\{v_i: i=1, \ldots, g\}$ where each $v_i$ is an $R$-basis for $\Fil_i/\Fil_{i-1}$. Then $\{v_i: i=1, \ldots, g\}$ is called \textbf{\textit{$w$-compatible with $\underline{\bfitm}$}} if 
$$v_i\mod (p^wM+\Fil_{i-1})=\bfitm_i\mod \Fil_{i-1}^{\underline{\bfitm}}$$
for all $i=1, \ldots, g$.
\end{enumerate}
\end{Definition} 

Pick a positive integer $n>\sup\{w, \frac{g}{p-1}\}$. Recall from \S \ref{subsection: perfectoid Siegel modular variety} the locally free $\scrO^+_{\overline{\calX}_{\Gamma(p^n)}}$-module $\underline{\omega}^{\mathrm{mod},+}_{\Gamma(p^n)}$ over $\overline{\calX}_{\Gamma(p^n)}$. Also recall the Hodge--Tate map
$$\HT_{\Gamma(p^n)}:V^{\vee}\otimes_{\Z}(\Z/p^n\Z)\rightarrow \underline{\omega}^{\mathrm{mod}, +}_{\Gamma(p^n)}/p^n\underline{\omega}^{\mathrm{mod}, +}_{\Gamma(p^n)}$$
over $\overline{\calX}_{\Gamma(p^n)}$. Restricting to the $w$-ordinary locus $\overline{\calX}_{\Gamma(p^n),w}$ and composing with a natural projection, we obtain
$$\HT_{\Gamma(p^n),w}:V^{\vee}\otimes_{\Z}(\Z/p^n\Z)\rightarrow \underline{\omega}^{\mathrm{mod}, +}_{\Gamma(p^n),w}/p^n\underline{\omega}^{\mathrm{mod}, +}_{\Gamma(p^n),w}\twoheadrightarrow \underline{\omega}^{\mathrm{mod}, +}_{\Gamma(p^n),w}/p^w\underline{\omega}^{\mathrm{mod}, +}_{\Gamma(p^n),w}$$
where $\underline{\omega}^{\mathrm{mod}, +}_{\Gamma(p^n),w}$ is the restriction of $\underline{\omega}^{\mathrm{mod}, +}_{\Gamma(p^n)}$ on $\overline{\calX}_{\Gamma(p^n),w}$.

\begin{Lemma}
The sheaf $\underline{\omega}^{\mathrm{mod}, +}_{\Gamma(p^n),w}/p^w\underline{\omega}^{\mathrm{mod}, +}_{\Gamma(p^n),w}$ is a free $\scrO^+_{\overline{\calX}_{\Gamma(p^n),w}}/p^w$-module of rank $g$ generated by the basis $\HT_{\Gamma(p^n),w}(e_1^{\vee})$, ..., $\HT_{\Gamma(p^n),w}(e_g^{\vee})$.
\end{Lemma}

\begin{proof}
Notice that $\underline{\omega}^{\mathrm{mod}, +}_{\Gamma(p^n),w}/p^w\underline{\omega}^{\mathrm{mod}, +}_{\Gamma(p^n),w}$ is locally free of rank $g$. It follows from the definition of $w$-ordinary locus that $\HT_{\Gamma(p^n),w}(e_1^{\vee})$, ..., $\HT_{\Gamma(p^n),w}(e_g^{\vee})$ span $\underline{\omega}^{\mathrm{mod}, +}_{\Gamma(p^n),w}/p^w\underline{\omega}^{\mathrm{mod}, +}_{\Gamma(p^n),w}$. Hence they must form a set of free generators.
\end{proof}

We consider an adic space $\adicIW^+_w$ over $\overline{\calX}_{\Gamma(p^n),w}$ parameterising certain $w$-compatible objects. Let $\Iw_{\GL_g}(\Z/p^n\Z)$ denote the preimage of $B_{\GL_g}(\Z/p\Z)$ under the surjection $\GL_g(\Z/p^n\Z)\xrightarrow[]{\mathrm{mod}\,\,p} \GL_g(\Z/p\Z)$. For every affinoid open $\mathcal{Y}=\Spa(R, R^+)\subset \overline{\calX}_{\Gamma(p^n),w}$ on which $\underline{\omega}^{\mathrm{mod}, +}_{\Gamma(p^n),w}$ is free, the set $\adicIW^+_w(\mathcal{Y})$ consists of pairs $$(\Fil_{\bullet}, \{w_i: i=1, \ldots, g\})$$ where, for some $\bfsigma\in \Iw_{\GL_g}(\Z/p^n\Z)$,
\begin{enumerate}
\item[(i)] $\Fil_{\bullet}$ is a full flag of the free $R^+$-module $\underline{\omega}^{\mathrm{mod}, +}_{\Gamma(p^n),w}(\mathcal{Y})$, which is $w$-compatible with $$\left(\HT_{\Gamma(p^n),w}(e_1^{\vee}),\ldots, \HT_{\Gamma(p^n),w}(e_g^{\vee})\right) \cdot\bfsigma;$$
\item[(ii)] Each $w_i$ is an $R^+$-basis for $\Fil_i/\Fil_{i-1}$ which is $w$-compatible with $$\left(\HT_{\Gamma(p^n),w}(e_1^{\vee}),\ldots, \HT_{\Gamma(p^n),w}(e_g^{\vee})\right) \cdot\bfsigma.$$
\end{enumerate}

Let $\pi: \adicIW^+_w\rightarrow \overline{\calX}_{\Gamma(p^n),w}$ denote the natural projection. There is a natural action of $\Iw_{\GL_g}^{(w)}$ on $\adicIW^+_w$ with the subgroup $U_{\GL_g,0}^{(w)}$ acting trivially. In particular, $\adicIW^+_w$ admits a natural action of $B_{\GL_g,0}^{(w)}/U_{\GL_g,0}^{(w)}=T_{\GL_g,0}^{(w)}$. We construct two auxiliary sheaves $\widetilde{\underline{\omega}}^{\kappa_{\calU},+}_{n,w}$ and $\widetilde{\underline{\omega}}^{\kappa_{\calU}}_{n,w}$. 

\begin{Definition}\label{definition: auxiliary sheaves}
\begin{enumerate}
\item[(i)] The sheaf $\widetilde{\underline{\omega}}^{\kappa_{\calU},+}_{n,w}$ over $\overline{\calX}_{\Gamma(p^n),w}$ is defined to be
$$\widetilde{\underline{\omega}}^{\kappa_{\calU},+}_{n,w}:=\left(\pi_*\scrO^+_{\adicIW^+_w}\widehat{\otimes}R_{\calU}\right)[\kappa_{\calU}^{\vee}];$$
i.e., the subsheaf of $\pi_*\scrO^+_{\adicIW^+_w}\widehat{\otimes}R_{\calU}$ consisting of those sections on which $T_{\GL_g,0}$-acts through the character $\kappa_{\calU}^{\vee}$.
\item[(ii)] The sheaf
$$\widetilde{\underline{\omega}}^{\kappa_{\calU}}_{n,w}:=\left(\pi_*\scrO_{\adicIW^+_w}\widehat{\otimes}R_{\calU}\right)[\kappa_{\calU}^{\vee}]$$
is defined similarly.
\end{enumerate}
\end{Definition}

\begin{Remark}\label{remark: no difference}
\normalfont Since $\kappa_{\calU}$ is $w$-analytic, the character $\kappa_{\calU}^{\vee}:T_{\GL_g,0}\rightarrow R_{\calU}^{\times}$ extends to a character on $T^{(w)}_{\GL_g,0}$. It turns out, in Definition \ref{definition: auxiliary sheaves}, there is no difference between taking $\kappa_{\calU}^{\vee}$-eigenspaces with respect to $T_{\GL_g,0}$- or $T^{(w)}_{\GL_g,0}$-actions.
\end{Remark}

\begin{Lemma}\label{Lemma: omega is projective Banach sheaf}
Let $(R_{\calU}, \kappa_{\calU})$ be a small weight. Then the sheaf $\widetilde{\underline{\omega}}^{\kappa_{\calU}}_{n,w}$ is a projective Banach sheaf of $\scrO_{\overline{\calX}_{\Gamma(p^n),w}}\widehat{\otimes}R_{\calU}$-modules in the sense of Definition \ref{Definition: Banach sheaf} (ii). Moreover, $\widetilde{\underline{\omega}}^{\kappa_{\calU},+}_{n,w}$ is an integral model of $\widetilde{\underline{\omega}}^{\kappa_{\calU}}_{n,w}$ in the sense of Definition \ref{Definition: Banach sheaf} (iv).
\end{Lemma}

\begin{proof}
Let $\{\calV_{n, i} : i\in I\}$ be an affinoid open covering of $\overline{\calX}_{\Gamma(p^n), w}$ such that $\underline{\omega}_{\Gamma(p^n)}^{\mathrm{mod}, +}|_{\calV_{n, i}}$ is free, for every $i\in I$. By choosing a basis for $\underline{\omega}_{\Gamma(p^n)}^{\mathrm{mod}, +}|_{\calV_{n, i}}$, we can identify \[
    \widetilde{\underline{\omega}}_{n, w}^{\kappa_{\calU}, +}|_{\calV_{n, i}} \simeq \scrO^+_{\calV_{n, i}} \langle T_{st}: 1\leq s<t\leq g\rangle \widehat{\otimes} R_{\calU}
\] 
which is the $p$-adic completion of a free $\scrO^+_{\calV_{n, i}} \widehat{\otimes} R_{\calU}$-module, as desired.
\end{proof}

We also consider the associated $p$-adically completed sheaves on the Kummer \'etale site.

\begin{Definition}
Let 
$$
\widetilde{\underline{\omega}}^{\kappa_{\calU},+}_{n,w,\ket}:= \varprojlim_m\left(\widetilde{\underline{\omega}}^{\kappa_{\calU},+}_{n,w} \bigotimes_{\scrO^+_{\overline{\calX}_{\Gamma(p^n),w}}}\scrO^+_{\overline{\calX}_{\Gamma(p^n),w}, \ket}/p^m\right)
$$
and let
$$
\widetilde{\underline{\omega}}^{\kappa_{\calU}}_{n,w,\ket}:=\widetilde{\underline{\omega}}^{\kappa_{\calU},+}_{n,w,\ket}[\frac{1}{p}].
$$
\end{Definition}

If $(R_{\calU}, \kappa_{\calU})$ is a small weight, by Lemma \ref{Lemma: omega is projective Banach sheaf} and Corollary \ref{Corollary: Banach induce Kummer etale Banach}, $\widetilde{\underline{\omega}}^{\kappa_{\calU}}_{n,w,\ket}$ is a projective Kummer \'etale Banach sheaf of $\scrO_{\overline{\calX}_{\Gamma(p^n),w}, \ket}\widehat{\otimes}R_{\calU}$-modules in the sense of Definition \ref{Definition: Kummer etale Banach sheaf} (ii). Moreover, $\widetilde{\underline{\omega}}^{\kappa_{\calU},+}_{n,w,\ket}$ is an integral model of $\widetilde{\underline{\omega}}^{\kappa_{\calU}}_{n,w,\ket}$ in the sense of Definition \ref{Definition: Kummer etale Banach sheaf} (iv). In fact, we show that the Kummer \'etale Banach sheaf $\widetilde{\underline{\omega}}^{\kappa_{\calU}}_{n,w,\ket}$ is \emph{admissible}.

\begin{Lemma}\label{Lemma: omega is admissible Kummer etale Banach sheaf}
Let $(R_{\calU}, \kappa_{\calU})$ be a small weight. Then the sheaf $\widetilde{\underline{\omega}}^{\kappa_{\calU}}_{n,w,\ket}$ is an admissible Kummer \'etale Banach sheaf of $\scrO_{\overline{\calX}_{\Gamma(p^n),w,\ket}}\widehat{\otimes}R_{\calU}$-modules (in the sense of Definition \ref{Definition: admissible Banach sheaf}) with integral model $\widetilde{\underline{\omega}}^{\kappa_{\calU},+}_{n,w,\ket}$.
\end{Lemma}

\begin{proof}
The proof is inspired by the discussion in \cite[\S 8.1]{AIP-2015}. We provide a sketch of proof. 

To simplify the notation, we write $\scrF^+=\widetilde{\underline{\omega}}_{n, w, \ket}^{\kappa_{\calU}, +}$ and $\scrF=\widetilde{\underline{\omega}}_{n, w, \ket}^{\kappa_{\calU}}$. We also write $\scrF^+_m:=\scrF^+/\fraka_{\calU}^m$, for every $m\in \Z_{\geq 1}$.

Let $\mathfrak{U}=\{\calV_{n, i} : i\in I\}$ be an open affinoid covering for $\overline{\calX}_{\Gamma(p^n), w}$ such that $\underline{\omega}_{\Gamma(p^n)}^{\mathrm{mod}, +}|_{\calV_{n, i}}$ is free, for every $i\in I$. We equip each $\calV_{n,i}$ the induced log structure from $\overline{\calX}_{\Gamma(p^n), w}$. By choosing a basis for $\underline{\omega}_{\Gamma(p^n)}^{\mathrm{mod}, +}|_{\calV_{n, i}}$, we can identify \[
    \scrF^+|_{\calV_{n, i}} \simeq \scrO^+_{\calV_{n, i}}\langle T_{st}: 1\leq s<t\leq g\rangle \widehat{\otimes} R_{\calU}
\] 
which is the $p$-adic completion of a free $\scrO^+_{\calV_{n, i}} \widehat{\otimes} R_{\calU}$-module.
Modulo $\fraka_{\calU}^m$, we obtain
\[
\scrF^+_m|_{\calV_{n, i}} \simeq \left(\scrO^+_{\calV_{n, i}} \otimes_{\Z_p} (R_{\calU}/\fraka_{\calU}^m)\right)[T_{st}: 1\leq s<t\leq g].
\]
For any $d\in \Z_{\geq 0}$, consider the subsheaf $(\scrF^+_m|_{\calV_{n,i}})^{\leq d}\subset \scrF^+_m|_{\calV_{n,i}}$ consisting of those polynomials of degree $\leq d$, and consider
$$
\scrF^+_{m,d}:=\ker\left( \prod_{i\in I}(\scrF^+_m|_{\calV_{n,i}})^{\leq d}\rightarrow\prod_{i, j\in I} \scrF^+_m|_{\calV_{n,i}\cap \calV_{n,j}}\right).
$$
Then each $\scrF^+_{m,d}$ is a coherent $\scrO^+_{\overline{\calX}_{\Gamma(p^n), w, \ket}}\otimes_{\Z_p} (R_{\calU}/\fraka_{\calU}^m)$-module and we have $\scrF^+_m=\varinjlim_d \scrF^+_{m,d}$, as desired.
\end{proof}

Next, we are going to relate the overconvergent automorphic sheaves $\underline{\omega}^{\kappa_{\calU},+}_w$ and $\underline{\omega}^{\kappa_{\calU}}_w$ with the auxiliary sheaves $\widetilde{\underline{\omega}}^{\kappa_{\calU},+}_{n,w}$ and $\widetilde{\underline{\omega}}^{\kappa_{\calU}}_{n,w}$. To this end, we need two intermediate sheaves $\underline{\omega}_{n, w}^{\kappa_{\calU},+}$ and $\underline{\omega}_{n, w}^{\kappa_{\calU}}$ over $\overline{\calX}_{\Gamma(p^n),w}$.

\begin{Definition}\label{Definition: overconvergent automorphic sheaf at Gamma(p^n)}
Let $h_{\Gamma(p^n)}: \overline{\calX}_{\Gamma(p^{\infty}),w}\rightarrow \overline{\calX}_{\Gamma(p^n),w}$ be the natural projection. 
\begin{enumerate}
\item[(i)] The subsheaf $\underline{\omega}_{n,w}^{\kappa_{\calU}}$ of $h_{\Gamma(p^n), *}\scrC^{w-\an}_{\kappa_{\calU}}(\Iw_{\GL_g}, \scrO_{\overline{\calX}_{\Gamma(p^{\infty}), w}}\widehat{\otimes}R_{\calU})$ is defined as follows. For every affinoid open subset $\calV\subset \overline{\calX}_{\Gamma(p^n), w}$ with $\calV_{\infty} = h_{\Gamma(p^n)}^{-1}(\calV)$, we put $$\underline{\omega}_{n,w}^{\kappa_{\calU}}(\calV):=\left\{f\in C^{w-\an}_{\kappa_{\calU}}(\Iw_{\GL_g}, \scrO_{\overline{\calX}_{\Gamma(p^{\infty}), w}}(\calV_{\infty})\widehat{\otimes}R_{\calU}): \bfgamma^*f=\rho_{\kappa_{\calU}}(\bfgamma_a+\frakz\bfgamma_c)^{-1} f,\,\,\forall \bfgamma=\begin{pmatrix}\bfgamma_a & \bfgamma_b\\ \bfgamma_c & \bfgamma_d\end{pmatrix}\in \Gamma(p^n)\right\}.$$ 
\item[(ii)] The subsheaf $\underline{\omega}_{n,w}^{\kappa_{\calU},+}$ of $h_{\Gamma(p^n), *}\scrC^{w-\an}_{\kappa_{\calU}}(\Iw_{\GL_g}, \scrO^+_{\overline{\calX}_{\Gamma(p^{\infty}), w}}\widehat{\otimes}R_{\calU})$ is defined as follows. For every affinoid open subset $\calV\subset \overline{\calX}_{\Gamma(p^n), w}$ with $\calV_{\infty} = h_{\Gamma(p^n)}^{-1}(\calV)$, we put $$\underline{\omega}_{n,w}^{\kappa_{\calU},+}(\calV):=\left\{f\in C^{w-\an}_{\kappa_{\calU}}(\Iw_{\GL_g}, \scrO^+_{\overline{\calX}_{\Gamma(p^{\infty}), w}}(\calV_{\infty})\widehat{\otimes}R_{\calU}): \bfgamma^*f=\rho_{\kappa_{\calU}}(\bfgamma_a+\frakz\bfgamma_c)^{-1} f,\,\,\forall \bfgamma=\begin{pmatrix}\bfgamma_a & \bfgamma_b\\ \bfgamma_c & \bfgamma_d\end{pmatrix}\in \Gamma(p^n)\right\}.$$ 
\end{enumerate}
\end{Definition}

\begin{Remark}\label{remark: invariants}
\normalfont Let $h_n: \overline{\calX}_{\Gamma(p^n),w}\rightarrow \overline{\calX}_{\Iw^+,w}$ denote the natural projection. Then the overconvergent Siegel modular sheaf $\underline{\omega}^{\kappa_{\calU}}_w$ can be identified as the $\Iw^+_{\GSp_{2g}}/\Gamma(p^n)$-invariants of the sheaf $h_{n,*}\underline{\omega}_{n,w}^{\kappa_{\calU}}$ with respect to the ``twisted'' action
$\bfgamma.f:=\rho_{\kappa_{\calU}}(\bfgamma_a+\frakz\bfgamma_c) \bfgamma^*f$
for every $\bfgamma=\begin{pmatrix}\bfgamma_a & \bfgamma_b\\ \bfgamma_c & \bfgamma_d\end{pmatrix}\in \Gamma(p^n)$ and $f\in \underline{\omega}_{n,w}^{\kappa_{\calU}}$. A similar result holds for the integral sheaf $\underline{\omega}^{\kappa_{\calU},+}_w$.
\end{Remark}

\begin{Proposition}\label{Proposition: omega is admissible}
There is a natural isomorphism of $\scrO^+_{\overline{\calX}_{\Gamma(p^n),w}}\widehat{\otimes}R_{\calU}$-modules $\Psi^+: \underline{\omega}^{\kappa_{\calU},+}_{n,w}\simeq \widetilde{\underline{\omega}}^{\kappa_{\calU},+}_{n,w}$. Inverting $p$, we obtain a natural isomorphism of $\scrO_{\overline{\calX}_{\Gamma(p^n),w}}\widehat{\otimes}R_{\calU}$-modules $\Psi: \underline{\omega}^{\kappa_{\calU}}_{n,w}\simeq \widetilde{\underline{\omega}}^{\kappa_{\calU}}_{n,w}$.
\end{Proposition}

\begin{proof}
As a preparation, consider the pullback diagram \[
    \begin{tikzcd}
        \adicIW_{w, \infty}^+ \arrow[r]\arrow[d, "\pi_{\infty}"'] & \adicIW_w^+\arrow[d, "\pi"]\\
        \overline{\calX}_{\Gamma(p^{\infty}),w} \arrow[r, "h_{\Gamma(p^n)}"] & \overline{\calX}_{\Gamma(p^n),w}
    \end{tikzcd}
\]
in the category of adic spaces. To show the existence of such a pullback, it suffices to notice that $\adicIW_w^+$ is locally isomorphic to $\overline{\calX}_{\Gamma(p^n),w}$ times finitely many copies of $\mathbf{B}(0,1)^{\frac{g(g+1)}{2}}$ where $\mathbf{B}(0,1)$ stands for the closed unit ball over $\C_p$. By \cite[Proposition 6.3.3 (3)]{Scholze-Weinstein-Berkeley}, the fibre product $\overline{\calX}_{\Gamma(p^{\infty}),w}\times_{\Spa(\C_p, \calO_{\C_p})} \mathbf{B}(0,1)^{\frac{g(g+1)}{2}}$ exists and is a sousperfectoid space.

For every affinoid open $\calV\subset \overline{\calX}_{\Gamma(p^n),w}$ and $\calV_{\infty}:=h_{\Gamma(p^n)}^{-1}\calV$, the desired isomorphism $\Psi^+$ will be established via a sequence of isomorphisms 
$$ \Psi^+: \underline{\omega}_{n,w}^{\kappa_{\calU},+}(\calV) \xrightarrow[\Psi_1]{\sim} \omega^{(1)} \xrightarrow[\Psi_2]{\sim} \omega^{(2)}\xrightarrow[\Psi_3]{\sim} \widetilde{\underline{\omega}}_{n,w}^{\kappa_{\calU},+}(\calV),
$$
where 
$$
\omega^{(1)} := \left\{ f\in C_{\kappa_{\calU}^{\vee}}^{w-\an}(\Iw_{\GL_g}, \scrO^+_{\calV_{\infty}}(\calV_{\infty})\widehat{\otimes}R_{\calU}): \bfgamma^* f = \rho_{\kappa_{\calU}^{\vee}}(\bfgamma_a^{\ddagger} + \frakz \bfgamma_c^{\ddagger}) f, \quad \forall \bfgamma = \begin{pmatrix} \bfgamma_a & \bfgamma_b \\ \bfgamma_c & \bfgamma_d\end{pmatrix} \in \Gamma(p^n)\right\}
$$
and
$$
\omega^{(2)} := \left\{f\in \pi_{\infty, *}\scrO^+_{\adicIW_{w, \infty}^+}(\calV_{\infty})\widehat{\otimes}R_{\calU}: \begin{array}{l}
        \bfgamma^* f = f, \quad \bftau^*f = \kappa_{\calU}^{\vee}(\bftau)f,\\
        \forall (\bfgamma, \bftau)\in \Gamma(p^n)\times T_{\GL_g, 0}
    \end{array}\right\}.
$$
Here, for any $\bfdelta \in M_g$, we write $\bfdelta^{\ddagger} := \oneanti_g \trans\bfdelta \oneanti_g$, which can be viewed as the ``transpose with respect to the anti-diagonal''. Notice that $\frakz^{\ddagger} = \frakz$. \\

\noindent\textbf{Construction of $\Psi_1$.}
Observe that there is an isomorphism of $\scrO^+_{\calV_{\infty}}(\calV_{\infty})\widehat{\otimes} R_{\calU}$-modules \[
    \Psi_1: C_{\kappa_{\calU}}^{w-\an}(\Iw_{\GL_g}, \scrO^+_{\calV_{\infty}}(\calV_{\infty})\widehat{\otimes} R_{\calU}) \rightarrow C_{\kappa_{\calU}^{\vee}}^{w-\an}(\Iw_{\GL_g}, \scrO^+_{\calV_{\infty}}(\calV_{\infty})\widehat{\otimes} R_{\calU})
\] 
defined by 
$$\Psi_1(f)(\bfgamma'):=f(\oneanti_g \trans\bfgamma'^{-1}\oneanti_g)$$
for all $f\in C_{\kappa_{\calU}}^{w-\an}(\Iw_{\GL_g}, \scrO^+_{\calV_{\infty}}(\calV_{\infty})\widehat{\otimes} R_{\calU})$ and $\bfgamma'\in \Iw_{\GL_g}$.

We claim that $\Psi_1$ induces an isomorphism $\underline{\omega}_{n,w}^{\kappa_{\calU},+}(\calV)\simeq \omega^{(1)}$. It suffices to check that if $\bfgamma^* f = \rho_{\kappa_{\calU}}(\bfgamma_a + \frakz\bfgamma_c)^{-1} f$ for every $\bfgamma = \begin{pmatrix}\bfgamma_a & \bfgamma_b\\ \bfgamma_c & \bfgamma_d\end{pmatrix} \in \Gamma(p^n)$, then $\bfgamma^*(\Psi_1(f)) = \rho_{\kappa_{\calU}^{\vee}}(\bfgamma_a^{\ddagger} + \frakz\bfgamma_c^{\ddagger})\Psi_1(f)$. Indeed, for any $\bfgamma'\in \Iw_{\GL_g}$, we have \begin{align*}
    \bfgamma^*(\Psi_1(f))(\bfgamma') & = \rho_{\kappa_{\calU}}(\bfgamma_a + \frakz\bfgamma_c)^{-1} f(\oneanti_g \trans\bfgamma'^{-1}\oneanti_g)\\
    & = f\left(\trans(\bfgamma_a + \frakz \bfgamma_c)^{-1}\oneanti_g \trans\bfgamma'^{-1}\oneanti_g\right)\\
    & = f\left(\oneanti_g \oneanti_g\trans(\bfgamma_a + \frakz \bfgamma_c)^{-1}\oneanti_g \trans\bfgamma'^{-1}\oneanti_g\right)\\
    & = f\left(\oneanti_g \trans(\oneanti_g \bfgamma_a \oneanti_g + \oneanti_g\frakz\oneanti_g \oneanti_g \bfgamma_c \oneanti_g)^{-1}\trans\bfgamma'^{-1}\oneanti_g\right)\\
    & = f\left(\oneanti_g \trans(\trans(\bfgamma_a^{\ddagger}+\frakz\bfgamma_c^{\ddagger})\bfgamma')^{-1}\oneanti_g\right)\\
    & = \rho_{\kappa_{\calU}^{\vee}}(\bfgamma_a^{\ddagger} + \frakz\bfgamma_c^{\ddagger}) \Psi_1(f)(\bfgamma').
\end{align*}

\noindent\textbf{Construction of $\Psi_2$.}
To construct $\Psi_2$, consider $\fraks^{\ddagger} = \begin{pmatrix}\fraks_g & \cdots & \fraks_1\end{pmatrix}\in \underline{\omega}_{\Gamma(p^{\infty})}(\calV_{\infty})^g$. Recall that $\fraks = \begin{pmatrix}\fraks_1& \cdots & \fraks_g\end{pmatrix}$ and thus \[
    \fraks^{\ddagger} = \fraks\oneanti_g.
\] Moreover, for any $\bfgamma = \begin{pmatrix}\bfgamma_a & \bfgamma_b \\ \bfgamma_c & \bfgamma_d\end{pmatrix} \in \Gamma(p^n)$, we have $\bfgamma^* \fraks = \fraks\cdot(\bfgamma_a + \frakz\bfgamma_c)$ by (\ref{eq: action on fake Hasse invariants}). Hence \[
    \bfgamma^* \fraks^{\ddagger}  = (\bfgamma^* \fraks) \oneanti_g = \fraks(\bfgamma_a + \frakz \bfgamma_c)\oneanti_g= \fraks \oneanti_g \oneanti_g (\bfgamma_a + \frakz \bfgamma_c)\oneanti_g = (\fraks \oneanti_g) \trans(\bfgamma_a^{\ddagger} + \frakz \bfgamma_c^{\ddagger})= \fraks^{\ddagger} \trans(\bfgamma_a^{\ddagger} + \frakz \bfgamma_c^{\ddagger}).
\]

Let $\Fil_{\bullet}^{\ddagger}$ be the full flag of the free $\scrO^+_{\calV_{\infty}}(\calV_{\infty})$-module $\underline{\omega}_{\Gamma(p^{\infty})}(\calV_{\infty})$ given by \[
    \Fil_{\bullet}^{\ddagger} = 0 \subset \langle \fraks_g \rangle \subset \langle \fraks_g, \fraks_{g-1} \rangle\subset \cdots \langle \fraks_g, \ldots, \fraks_1\rangle
\] and let $w_i^{\ddagger}$ be the image of $\fraks_{g+1-i}$ in $\Fil_{i}^{\ddagger}/\Fil_{i-1}^{\ddagger}$, for all $i=1, \ldots, g$. Then $(\Fil_{\bullet}^{\ddagger}, \{w_i^{\ddagger}\})$ defines a global section of $\pi_{\infty}^{-1}(\calV_{\infty})$. We obtain a surjection \[\Iw_{\GL_g}^{(w)}\rightarrow \pi_{\infty}^{-1}(\calV_{\infty}),\quad \bfgamma' \mapsto (\Fil_{\bullet}^{\ddagger}, \{w_i^{\ddagger}\}) \cdot \bfgamma'\]
with kernel $U_{\GL_g,0}^{(w)}$. This induces an isomorphism
\begin{align*}
    \Phi: \pi_{\infty, *}\scrO^+_{\adicIW_{w, \infty}^+}(\calV_{\infty})\widehat{\otimes} R^+_{\calU} & \xrightarrow{\sim} \left\{
        \text{analytic functions }
        g: \Iw_{\GL_g}^{(w)} \rightarrow \scrO^+_{\calV_{\infty}}(\calV_{\infty})\widehat{\otimes}R_{\calU} \textrm{ such that }g|_{U_{\GL_g,0}^{(w)}}=1 \right\}\\
    f & \mapsto \left(\bfgamma'\mapsto f((\Fil_{\bullet}^{\ddagger}, \{w_i^{\ddagger}\}) \cdot \bfgamma')\right).
\end{align*} We claim that if $\bfgamma^* f = f$ for any $\bfgamma = \begin{pmatrix}\bfgamma_a & \bfgamma_b \\ \bfgamma_c & \bfgamma_d\end{pmatrix}\in \Gamma(p^n)$, then $\bfgamma^*\Phi(f) = \rho_{\kappa_{\calU}^{\vee}}(\bfgamma_a^{\ddagger} + \frakz\bfgamma_c^{\ddagger})\Phi(f)$. Indeed, for any $\bfgamma' \in \Iw_{\GL_g}^{(w)}$, we have \begin{align*}
    (\bfgamma^*\Phi(f)) (\bfgamma') & = (\bfgamma^* f)(\bfgamma^*(\Fil_{\bullet}^{\ddagger}, \{w_i^{\ddagger}\}) \cdot \bfgamma')\\ 
    & = f\left((\Fil_{\bullet}^{\ddagger}, \{w_i^{\ddagger}\})\cdot  \trans(\bfgamma_a^{\ddagger} + \frakz\bfgamma_c^{\ddagger})\cdot\bfgamma'\right)\\
    & = \rho_{\kappa_{\calU}^{\vee}}(\bfgamma_a^{\ddagger} + \frakz\bfgamma_c^{\ddagger}) \Phi(f)(\bfgamma'),
\end{align*} where the second equation follows from the identity $\bfgamma^* \fraks^{\ddagger} = \fraks^{\ddagger} \trans(\bfgamma_a^{\ddagger} + \frakz\bfgamma_c^{\ddagger})$.

On the other hand, we can identify $\omega^{(1)}$ with the set of analytic functions
$$f: \Iw_{\GL_g}^{(w)} \rightarrow \scrO^+_{\calV_{\infty}}(\calV_{\infty})\widehat{\otimes}R_{\calU} $$
satisfying
\begin{itemize}
\item $f(\bfupsilon \bftau \bfnu) = \kappa_{\calU}^{\vee}(\bftau)f(\bfupsilon)$ for all $(\bfupsilon, \bftau, \bfnu)\in U_{\GL_g, 1}^{\opp, (w)}\times T_{\GL_g, 0}^{(w)}\times U_{\GL_g,0}^{(w)}$;
\item $\bfgamma^* f = \rho_{\kappa_{\calU}^{\vee}}(\bfgamma_a^{\ddagger}+ \frakz\bfgamma_c^{\ddagger})f$ for all $\bfgamma = \begin{pmatrix}\bfgamma_a & \bfgamma_b \\ \bfgamma_c & \bfgamma_d\end{pmatrix}\in \Gamma(p^n)$.
\end{itemize}

Therefore, putting $\Psi_2 := \Phi^{-1}$, one obtains the desired isomorphism \[
    \Psi_2: \omega^{(1)} \xrightarrow{\sim} \omega^{(2)}.
\]

\noindent\textbf{Construction of $\Psi_3$.} By the construction of $\underline{\omega}_{n,w}^{\kappa_{\calU}, +}$ and Lemma \ref{Lemma: structure sheaves at the infinite level and the structure sheaves at the Iwahori level}, one immediately obtains an identification of $\omega^{(2)}$ with $\widetilde{\underline{\omega}}_{n,w}^{\kappa_{\calU},+}(\calV)$. We simply take $\Psi_3$ to be this identification. 

Putting everything together, the composition $\Psi^+=\Psi_3\circ\Psi_2\circ\Psi_1$ yields an isomorphism
$$\Psi^+:\underline{\omega}_{n,w}^{\kappa_{\calU},+}(\calV)\xrightarrow{\sim} \widetilde{\underline{\omega}}_{n,w}^{\kappa_{\calU},+}(\calV).$$
It is also straightforward to check that the construction is functorial in $\calV$. By gluing, we arrive at an isomorphism
$$\Psi^+:\underline{\omega}_{n,w}^{\kappa_{\calU},+}\xrightarrow{\sim} \widetilde{\underline{\omega}}_{n,w}^{\kappa_{\calU},+}.$$
\end{proof}

Consider the $p$-adically completed sheaf of $\scrO_{\overline{\calX}_{\Iw^+,w,\ket}}$-modules associated with $\underline{\omega}^{\kappa_{\calU}}_w$; namely, let
$$
\underline{\omega}^{\kappa_{\calU},+}_{w,\ket}:=\varprojlim_m\left(\underline{\omega}^{\kappa_{\calU},+}_{w}\otimes_{\scrO^+_{\overline{\calX}_{\Iw^+,w}}}\scrO^+_{\overline{\calX}_{\Iw^+,w,\ket}}/p^m\right)
$$
and
$$\underline{\omega}^{\kappa_{\calU}}_{w,\ket}:=\underline{\omega}^{\kappa_{\calU},+}_{w,\ket}[\frac{1}{p}].$$

By Remark \ref{remark: invariants} and Proposition \ref{Proposition: omega is admissible}, $\underline{\omega}^{\kappa_{\calU}}_w$ can be identified with the sheaf of $\Iw^+_{\GSp_{2g}}/\Gamma(p^n)$-invariants of $h_{n,*}\widetilde{\underline{\omega}}_{n,w}^{\kappa_{\calU}}$. Hence, $\underline{\omega}^{\kappa_{\calU}}_{w,\ket}$ can be identified with the sheaf of $\Iw^+_{\GSp_{2g}}/\Gamma(p^n)$-invariants of $h_{n,*}\widetilde{\underline{\omega}}_{n,w, \ket}^{\kappa_{\calU}}$. When $(R_{\calU}, \kappa_{\calU})$ is a small weight, $h_{n,*}\widetilde{\underline{\omega}}_{n,w, \ket}^{\kappa_{\calU}}$ is an admissible Kummer \'etale Banach sheaf of $\scrO_{\overline{\calX}_{\Iw^+,w,\ket}}\widehat{\otimes}R_{\calU}$-modules by Lemma \ref{Lemma: omega is admissible Kummer etale Banach sheaf} and Lemma \ref{Lemma: pushforward along finite Kummer etale map}. Consequently, such a description allows us to apply Corollary \ref{Corollary: generalised projection formula with invariants} to the sheaf $\underline{\omega}^{\kappa_{\calU}}_{w,\ket}$. This will be used in the proof of Lemma \ref{Lemma: Leray spectral sequence for the automorphic sheaf}.

\subsection{Classical Siegel modular forms}\label{subsection: classical forms}
In this subsection, we show that the space of $w$-overconvergent Siegel modular forms does contain all of the classical Siegel modular forms.

Let $k = (k_1, ..., k_g)\in \Z_{\geq 0}^g$ be a dominant weight and consider $k^{\vee}=(-k_g, \ldots, -k_1)$. Recall the vector bundle $\underline{\omega}_{\Iw^+} = h_{\Iw^+}^*\underline{\omega}$, where $h_{\Iw^+}: \overline{\calX}_{\Iw^+} \rightarrow \overline{\calX}$ is the natural projection. Let \[
    \calM:= \Isom_{\overline{\calX}_{\Iw^+}}(\scrO^g_{\overline{\calX}_{\Iw^+}}, \underline{\omega}_{\Iw^+})
\]be the $\GL_g$-torsor over $\overline{\calX}_{\Iw^+}$ together with the structure morphism $\vartheta: \calM \rightarrow \overline{\calX}_{\Iw^+}$. Then the sheaf $\underline{\omega}_{\Iw^+}^{k}$ of \textit{\textbf{classical Siegel modular forms of weight $k$ (of strict Iwahori level)}} is defined to be 
\[
    \underline{\omega}_{\Iw^+}^k := \vartheta_*\scrO_{\calM}[k^{\vee}];
\] 
namely, the subsheaf of $\vartheta_*\scrO_{\calM}$ on which $T_{\GL_g}$ acts through the character $k^{\vee}$. The \textit{\textbf{space of classical Siegel modular forms of weight $k$ (of strict Iwahori level)}} is defined to be
$$M^{k, \mathrm{cl}}_{\Iw^+}:=H^0(\overline{\calX}_{\Iw^+}, \underline{\omega}_{\Iw^+}^k)$$
equipped with naturally defined Hecke operators.

\begin{Remark}\label{Remark: integral classical sheaf}
\normalfont One can also define the sheaf of integral classical Siegel modular forms by
$$\underline{\omega}_{\Iw^+}^{k,+} := \vartheta_*\scrO^+_{\calM}[k^{\vee}].$$
But we do not need this in the current subsection.
\end{Remark}

Restricting to the $w$-ordinary locus, we may consider the sheaf $\underline{\omega}_{\Iw^+}^{k}|_{\overline{\calX}_{\Iw^+, w}}$. Repeating the strategy as in the proof of Proposition \ref{Proposition: omega is admissible}, we arrive at the following explicit description of $\underline{\omega}_{\Iw^+}^{k}|_{\overline{\calX}_{\Iw^+, w}}$.

\begin{Definition}\label{Definition: algebraic functions}
\begin{enumerate}
\item[(i)] Let $P(\GL_g, \mathbb{A}^1)$ denote the $\Q_p$-vector space of maps $\GL_g\rightarrow \mathbb{A}^1$ between algebraic varieties over $\Q_p$.
\item[(ii)] For every uniform $\C_p$-Banach algebra $B$, define $$P(\GL_g, B):=P(\GL_g, \mathbb{A}^1)\widehat{\otimes}_{\Q_p}B$$
and let $P_k(\GL_g, B)$ denote the subspace of $P(\GL_g, B)$ consisting of those $f:\GL_g \rightarrow B$ such that $f(\bfgamma\bfbeta) =k(\bfbeta)f(\bfgamma)$ for all $\bfgamma\in \GL_{g}$ and  $\bfbeta\in B_{\GL_{g}}$.
\item[(iii)] There is a natural left action of $\GL_g$ on $P_k(\GL_g, B)$ given by
$$(\bfgamma.f)(\bfgamma')=f(\trans \bfgamma\bfgamma')$$
for all $\bfgamma, \bfgamma'\in \GL_g$ and $f\in P_k(\GL_g, B)$. This left action is denoted by $$\rho_k: \GL_g\rightarrow \Aut(P_k(\GL_g, B)).$$
\end{enumerate}
\end{Definition}

\begin{Proposition}\label{Proposition: explicit description of classical modular sheaf}
For any affinoid open $\calV\subset \overline{\calX}_{\Iw^+, w}$ with preimage $\calV_{\infty}$ in $\overline{\calX}_{\Gamma(p^{\infty}), w}$, we have a natural identification
$$
        \underline{\omega}_{\Iw^+}^{k}(\calV) = \left\{f \in P_k(\GL_g, \scrO_{\overline{\calX}_{\Gamma(p^{\infty}), w}}(\calV_{\infty})) :  \bfgamma^* f = \rho_{k}(\bfgamma_a+\frakz\bfgamma_c)^{-1}f, \,\,\,\forall \bfgamma = \begin{pmatrix}\bfgamma_a & \bfgamma_b\\ \bfgamma_c & \bfgamma_d\end{pmatrix} \in \Iw_{\GSp_{2g}}^+ 
        \right\}.
$$
In particular, there is a natural injection
\end{Proposition}
\begin{equation}\label{eq: alg. sheaf into overconvergent sheaf}
    \underline{\omega}_{\Iw^+}^k|_{\overline{\calX}_{\Iw^+, w}} \hookrightarrow \underline{\omega}_w^{k}.
\end{equation} 

\begin{proof}
For the first statement, the strategy in the proof of Proposition \ref{Proposition: omega is admissible} applies verbatim, except that we consider the torsor $\calM$ in place of $\adicIW^+_w$. The details are left to the reader. The inclusion $\underline{\omega}_{\Iw^+}^k|_{\overline{\calX}_{\Iw^+, w}} \hookrightarrow \underline{\omega}_w^{k}$ follows from the natural inclusion $P_k(\GL_g, \scrO_{\overline{\calX}_{\Gamma(p^{\infty}), w}}(\calV_{\infty}))\hookrightarrow C^{w-\an}_k(\Iw_{\GL_g}, \scrO_{\overline{\calX}_{\Gamma(p^{\infty}), w}}(\calV_{\infty}))$.
\end{proof}

The following result shows that the space of classical forms naturally injects into the space of overconvergent modular forms.

\begin{Lemma}\label{Lemma: injection of classical forms}
The Hecke-equivariant composition of maps
$$M^{k, \mathrm{cl}}_{\Iw^+}=H^0(\overline{\calX}_{\Iw^+}, \underline{\omega}_{\Iw^+}^k)\xrightarrow[]{\Res} H^0(\overline{\calX}_{\Iw^+, w}, \underline{\omega}_{\Iw^+}^k) \hookrightarrow M_{\Iw^+, w}^{k}
$$
is injective. 
\end{Lemma}

\begin{proof}
It suffices to show that 
$$\Res: H^0(\overline{\calX}_{\Iw^+}, \underline{\omega}_{\Iw^+}^k)\rightarrow H^0(\overline{\calX}_{\Iw^+, w}, \underline{\omega}_{\Iw^+}^k)$$
is injective; namely, given any global section $f$ of $\underline{\omega}_{\Iw^+}^k$ that vanishes on $\overline{\calX}_{\Iw^+, w}$, we have to show that $f = 0$ on every irreducible component of $\overline{\calX}_{\Iw^+}$.

For every algebraic variety $Y$ over $\C_p$, we know that $Y$ is irreducible if and only if the associated adic space $\calY$ over $\Spa(\C_p, \calO_{\C_p})$ is irreducible (see \cite[Theorem 2.3.1]{conrad-conn} and \cite[\S 1.1.11.(c)]{Huber-2013}). In particular, the irreducible components of $\overline{\calX}_{\Iw^+}$ coincide with the irreducible components of $\overline{X}_{\Iw^+}$. As $\overline{X}_{\Iw^+}$ is a compactification of $X_{\Iw^+}$, its irreducible components correspond to the irreducible components of $X_{\Iw^+}$. Under the identification \[
    X_{\Iw^+}(\C) = \GSp_{2g}(\Q)\backslash \GSp_{2g}(\A_f)\times \bbH_g/\Iw_{\GSp_{2g}}^+\Gamma(N),
\] \cite[\S 2]{Deligne-Shimura} provides the following description of the irreducible components of $\overline{\calX}_{\Iw^+}$:
\[
\pi_0(\overline{\calX}_{\Iw^+}) =  {\Q_{>0}}\backslash \mathbb{G}_m(\A_f)/ \varsigma\left(\Gamma(N) \Iw_{\GSp_{2g}}^+\right).
\]
where $\varsigma$ is the character of similitude involved in the definition of $\GSp_{2g}$. There is a similar description for $\pi_0(\overline{\calX})$. Note that $\pi_0(\overline{\calX}_{\Iw^+})$ is the same as  $\pi_0(\overline{\calX})$ because $\Iw_{\GSp_{2g}}^+$ and $\GSp_{2g}(\Z_p)$ have the same image via $\varsigma$. In particular, since every irreducible component in $\pi_0(\overline{\calX})$ contains an ordinary point, every irreducible component of $\overline{\calX}_{\Iw^+}$ intersects $\overline{\calX}_{\Iw^+, w}$.

By definition, $f$ can be viewed as a global section of the structure sheaf of $\calM$. Let $\calC$ be any irreducible component of $\overline{\calX}_{\Iw^+}$, it remains to show that $f$ vanishes on $\calM \times_{\overline{\calX}_{\Iw^+}}\calC$. Indeed, observe that $\mathcal{M}\times_{\overline{\calX}_{\Iw^+}} {\mathcal{C}} $ is irreducible and $f$ vanishes on $\calM \times_{\overline{\calX}_{\Iw}}(\calC \cap \overline{\calX}_{\Iw^+, w})$. Hence, the desired vanishing follows from \cite[Proposition 0.1.13]{Berthelot-rigid_cohomology} which states that a rigid analytic function vanishing on an open subset of an irreducible rigid analytic variety is identically zero.
\end{proof}

\subsection{The construction \`{a} la Andreatta--Iovita--Pilloni}\label{subsection: construction of AIP}
The sheaves $\underline{\omega}^{\kappa_{\calU}}_w$ constructed in \S \ref{subsection: the perfectoid construction} are analogues of the overconvergent automorphic sheaves constructed by Andreatta-Iovita-Pilloni in \cite{AIP-2015}. It is a natural question whether these two constructions coincide. In this subsection, we recall the construction in \cite{AIP-2015}. Later in \S \ref{subsection:comparison sheaf aip}, we will present a comparison result.

Choose $v\in \Q_{>0}\cap [0,\frac{1}{2})$ and let $n$ be a positive integer such that $v<\frac{1}{2p^{n-1}}$. Consider the open subset $$\overline{\calX}(v):=\{\bfitx\in \overline{\calX}: |\widetilde{\Ha}(\bfitx)|\geq p^{-v}\}\subset \overline{\calX},$$ where $\widetilde{\Ha}$ is a fixed lift of the Hasse invariant. (We point out that, for those $\bfitx$ at the boundary,  the Hasse invariant of $\bfitx$ is defined to be the Hasse invariant of the abelian part of the semiabelian scheme associated with $\bfitx$.) Thanks to \cite[Proposition 4.1.3]{AIP-2015}, for every $1\leq m\leq n$, there is a universal canonical subgroup $\calH_m$ of level $m$ of the tautological semiabelian variety over $\overline{\calX}(v)$. Let $\underline{\omega}_v$ denote the restriction of $\underline{\omega}$ on $\overline{\calX}(v)$.

We also consider the following finite covers of $\overline{\calX}(v)$:
\begin{enumerate}
\item[$\bullet$] Let $$\overline{\calX}_1(p^n)(v):=\Isom_{\overline{\calX}(v)}((\Z/p^n\Z)^g, \calH_n^{\vee})$$ be the adic space over $\overline{\calX}(v)$ which parameterises trivialisations of $\calH_n^{\vee}$. Notice that the group $\GL_g(\Z/p^n\Z)$ naturally acts on $\overline{\calX}_1(p^n)(v)$ from the right by permuting the trivialisations. 
\item[$\bullet$] Let $$\overline{\calX}_1(v):=\Isom_{\overline{\calX}(v)}((\Z/p\Z)^g, \calH_1^{\vee})$$ be the adic space over $\overline{\calX}(v)$ which parameterises trivialisations of $\calH_1^{\vee}$. 
\item[$\bullet$] The group $\GL_g(\Z/p\Z)$ naturally acts on $\overline{\calX}_1(v)$ from the right by permuting the trivialisations. By taking the quotient
$$\overline{\calX}_{\Iw}(v):=\overline{\calX}_1(v)/B_{\GL_g}(\Z/p\Z),$$ we obtain an adic space $\overline{\calX}_{\Iw}(v)$ over $\overline{\calX}(v)$ which parameterises full flags $\Fil_{\bullet}\calH_1^{\vee}$ of $\calH_1^{\vee}$.
\item[$\bullet$] Let $\underline{\omega}_{n,v}$ be the pullback of $\underline{\omega}_v$ along $\overline{\calX}_1(p^n)(v)\rightarrow \overline{\calX}(v)$. 
\item[$\bullet$] Let $\underline{\omega}_{\Iw,v}$ be the pullback of $\underline{\omega}_v$ along $\overline{\calX}_{\Iw}(v)\rightarrow \overline{\calX}(v)$.
\end{enumerate}

In order to proceed, we need to introduce formal models of aforementioned geometric objects: 
\begin{enumerate}
\item[$\bullet$] Recall that $\overline{\frakX}$ is the formal completion of $\overline{X}_0$ along the special fibre. Let $\widetilde{\frakX}(v)$ be the blowup of $\overline{\frakX}$ along the ideal $(\widetilde{\Ha}, p^v)$. Let $\overline{\frakX}(v)$ be the $p$-adic completion of the normalisation of the largest open formal subscheme of $\widetilde{\frakX}(v)$ where the ideal $(\widetilde{\Ha}, p^v)$ is generated by $\widetilde{\Ha}$. Then $\overline{\frakX}(v)$ is a formal model of $\overline{\calX}(v)$. 
\item[$\bullet$] Let $\overline{\frakX}_1(p^n)(v)$ be the normalisation of $\overline{\frakX}(v)$ in $\overline{\calX}_1(p^n)(v)$. The group $\GL_g(\Z/p^n\Z)$ naturally acts on $\overline{\frakX}_1(p^n)(v)$.
\item[$\bullet$] Let $\overline{\frakX}_1(v)$ be the normalisation of $\overline{\frakX}(v)$ in $\overline{\calX}_1(v)$. The group $\GL_g(\Z/p\Z)$ naturally acts on $\overline{\frakX}_1(v)$.
\item[$\bullet$] Let $\overline{\frakX}_{\Iw}(v)$ be the normalisation of $\overline{\frakX}(v)$ in $\overline{\calX}_{\Iw}(v)$. We can identify $\overline{\frakX}_{\Iw}(v)$ with the quotient $\overline{\frakX}_1(v)/B_{\GL_g}(\Z/p\Z)$.
\item[$\bullet$] Let $\frakG_v^{\univ}$ be the tautological semiabelian scheme over $\overline{\frakX}(v)$ with the structure morphism 
$\pi: \frakG_v^{\univ}\rightarrow \overline{\frakX}(v)$ and the identity section $e$.
Define 
$$\underline{\Omega}_v:=e^*\Omega^1_{\frakG_v^{\univ}/\overline{\frakX}(v)}.$$
\item[$\bullet$] Let $\underline{\Omega}_{n,v}$ be the pullback of $\underline{\Omega}_v$ along $\overline{\frakX}_1(p^n)(v)\rightarrow \overline{\frakX}(v)$. 
\item[$\bullet$] Let $\underline{\Omega}_{\Iw,v}$ be the pullback of $\underline{\Omega}_v$ along $\overline{\frakX}_{\Iw}(v)\rightarrow \overline{\frakX}(v)$.
\end{enumerate}

Now suppose $w\in \Q_{>0}$ lies in the interval $\left(n-1+\frac{v}{p-1}, n-\frac{vp^n}{p-1}\right]$. Let $$\psi_n^{\univ}: (\Z/p^n\Z)^g\cong \calH_n^{\vee}$$ denote the universal trivialisation of $\calH_n^{\vee}$ over $\overline{\frakX}_1(p^n)(v)$. Then \cite[Proposition 4.3.1]{AIP-2015} yields a locally free $\scrO_{\overline{\frakX}_1(p^n)(v)}$-submodule $\scrF\subset \underline{\Omega}_{n,v}$ of rank $g$, equipped with a map
$$\HT_{n,v,w}:  (\Z/p^n\Z)^g\overset{\psi_n^{\univ}}{\cong} \calH_n^{\vee}\rightarrow \scrF\otimes_{\scrO_{\overline{\frakX}_1(p^n)(v)}}\scrO_{\overline{\frakX}_1(p^n)(v)}/p^w$$
which induces an isomorphism
$$\HT_{n,v,w}\otimes\id: (\Z/p^n\Z)^g\otimes_{\Z}\scrO_{\overline{\frakX}_1(p^n)(v)}/p^w\cong \scrF\otimes_{\scrO_{\overline{\frakX}_1(p^n)(v)}}\scrO_{\overline{\frakX}_1(p^n)(v)}/p^w.$$
More precisely, locally on $\overline{\frakX}_1(p^n)(v)$, consider the family version of the Hodge-Tate map 
$$\HT_n: (\Z/p^n\Z)^g\overset{\psi_n^{\univ}}{\cong}\calH_n^{\vee}\rightarrow \omega_{\calH_n}$$
studied in \cite[\S 4]{AIP-2015}. Let $\epsilon_1, \ldots, \epsilon_g$ be the standard $(\Z/p^n\Z)$-basis for $(\Z/p^n\Z)^g$ and let $\widetilde{\HT}_n(\epsilon_i)$ be lifts of $\HT_n(\epsilon_i)$ from $\omega_{\calH_n}$ to $\underline{\Omega}_{n,v}$. Then $\scrF$ is generated by $\widetilde{\HT}_n(\epsilon_1), \ldots, \widetilde{\HT}_n(\epsilon_g)$. It turns out this local construction glues to a locally free $\scrO_{\overline{\frakX}_1(p^n)(v)}$-module of rank $g$.

In \cite[\S 4.5]{AIS-2015}, Andreatta--Iovita--Pilloni constructs a formal scheme $\mathfrak{IW}^+_{w,v}$ over $\overline{\frakX}_1(p^n)(v)$ which parameterises certain $w$-compatible objects. More precisely, $\mathfrak{IW}^+_{w,v}$ is the formal schemes over $\overline{\frakX}_1(p^n)(v)$ such that for every affine open subset $\Spf R\subset \overline{\frakX}_1(p^n)(v)$ on which $\scrF$ is free, $\mathfrak{IW}^+_{w,v}(R)$ consists of pairs $(\Fil_{\bullet}, \{w_i: i=1,\ldots, g\})$ where $\Fil_{\bullet}$ is a full flag of $\scrF$, each $w_i$ is an $R$-basis for $\Fil_i/\Fil_{i-1}$, and both $\Fil_{\bullet}$ and $\{w_i: i=1, \ldots, g\}$ are $w$-compatible with $\HT_{n,v,w}(\epsilon_1), \ldots, \HT_{n,v,w}(\epsilon_g)$ in the sense of Definition \ref{Definition: w-compatible}.

Now we go back to the generic fibres. Let $\adicIW_{w,v}^+$ be the adic space associated with the formal scheme $\mathfrak{IW}^+_{w,v}$ over $\Spa(\C_p, \calO_{\C_p})$. Then we have a chain of morphisms of adic spaces
$$\pi^{\AIP}: \adicIW_{w,v}^+\rightarrow \overline{\calX}_1(p^n)(v)\rightarrow\overline{\calX}_1(v)\rightarrow \overline{\calX}_{\Iw}(v).$$ As pointed out in \cite[\S 5.2.2]{AIS-2015}, there is a natural action of $B_{\GL_g,0}^{(w)}$ acting on $\adicIW_{w,v}^+$.

Finally, we are ready to define the overconvergent automorphic sheaves of Andreatta--Iovita--Pilloni. For a $w$-analytic weight $(R_{\calU}, \kappa_{\calU})$, the $T_{\GL_g,0}$-character $\kappa_{\calU}^{\vee}$ extends to a character of $T_{\GL_g,0}^{(w)}$ and further extends to a character of $B_{\GL_g,0}^{(w)}$ by setting $\kappa_{\calU}^{\vee}|_{U_{\GL_g,0}^{(w)}}=1$.

\begin{Definition}\label{definition: automorphic sheaves of AIP}
Let $(R_{\calU}, \kappa_{\calU})$ be a $w$-analytic weight. 
\begin{enumerate}
\item[(i)] Andreatta--Iovita--Pilloni's \textbf{sheaf of $w$-analytic $v$-overconvergent Siegel modular forms of weight $\kappa_{\calU}$ (of Iwahori level)} is defined to be \footnote{The notations $\underline{\omega}_{w, v}^{\kappa_{\calU}, \AIP}$, $M^{\kappa_{\calU}, \AIP}_{\Iw, w,v}$, and $M^{\kappa_{\calU}, \AIP}_{\Iw}$ correspond to the notations $\omega^{\dagger\, \kappa_{\calU}}_w$, $M^{\dagger\,\kappa_{\calU}}_w(\calX_{\Iw}(p)(v))$, and $M^{\dagger\,\kappa_{\calU}}(\calX_{\Iw}(p))$, respectively, in \cite{AIP-2015}.}
$$\underline{\omega}_{w, v}^{\kappa_{\calU}, \AIP}:=\pi_{*}^{\AIP}\scrO_{\adicIW_{w,v}^+}[\kappa_{\calU}^{\vee}],$$ 
where $\pi_{*}^{\AIP}\scrO_{\adicIW_{w,v}^+}[\kappa_{\calU}^{\vee}]$ stands for the subsheaf of $\pi_{*}^{\AIP}(\scrO_{\adicIW_{w,v}^+}\widehat{\otimes}R_{\calU})$ consisting of sections on which $B_{\GL_g,0}^{(w)}$ acts via the character $\kappa_{\calU}^{\vee}$. 
\item[(ii)] Andreatta--Iovita--Pilloni's \textbf{space of $w$-analytic $v$-overconvergent Siegel modular forms of weight $\kappa_{\calU}$ (of Iwahori level)} is 
$$M^{\kappa_{\calU}, \AIP}_{\Iw, w,v}:=H^0(\overline{\calX}_{\Iw}(v), \underline{\omega}_{w, v}^{\kappa_{\calU}, \AIP}).$$
\item[(iii)]The \textbf{space of locally analytic overconvergent Siegel modular forms of weight $\kappa_{\calU}$ (of Iwahori level)} is
$$M^{\kappa_{\calU}, \AIP}_{\Iw}:=\lim_{\substack{v\rightarrow 0\\ w\rightarrow\infty}}M^{\kappa_{\calU}, \AIP}_{\Iw, w, v}.$$
\item[(iv)] Recall that $\calZ_{\Iw} = \overline{\calX}_{\Iw}\smallsetminus \calX_{\Iw}$ is the boundary divisor. Andreatta--Iovita--Pilloni's \textbf{sheaf of $w$-analytic $v$-overconvergent Siegel cuspforms of weight $\kappa_{\calU}$ (of Iwahori level)} is defined to be the subseaf $\underline{\omega}_{w, v, \cusp}^{\kappa_{\calU}, \AIP} = \underline{\omega}_{w, v}^{\kappa_{\calU}, \AIP}(-\calZ_{\Iw})$ of $\underline{\omega}_{w, v}^{\kappa_{\calU}, \AIP}$ consisting of sections that vanish along $\calZ_{\Iw}$. 

Andreatta--Iovita--Pilloni's \textbf{space of $w$-analytic $v$-overconvergent Siegel cuspforms of weight $\kappa_{\calU}$ (of Iwahori level)} is defined to be \[
    S_{\Iw, w, v}^{\kappa_{\calU}, \AIP} := H^0(\overline{\calX}_{\Iw}(v), \underline{\omega}_{w, v, \cusp}^{\kappa_{\calU}, \AIP}),
\] and the \textbf{space of locally analytic overconvergent Siegel cuspforms of weight $\kappa_{\calU}$ (of Iwahori level)} is \[
    S_{\Iw}^{\kappa_{\calU}, \AIP} := \lim_{\substack{v\rightarrow 0 \\ w \rightarrow \infty}} S_{\Iw, w, v}^{\kappa_{\calU}, \AIP}.
\]
\end{enumerate}
\end{Definition}

\subsection{The \texorpdfstring{$w$}{w}-ordinary loci and the (pseudo-)canonical subgroups}\label{subsection: pseudocanonical subgroups}
In \S \ref{subsection:comparison sheaf aip}, we will prove the comparison between our perfectoid construction of the overconvergent Siegel modular forms and the construction of Andreatta--Iovita--Piloni. Immediate from the definitions, one observes the incompatibility of the underlying adic spaces used in the two constructions. That is, we employ the $w$-ordinary locus in the perfectoid construction while the authors of \cite{AIP-2015} make use of the ``$v$-locus'' $\overline{\calX}_{\Iw}(v)$. Therefore, as a preparation for the comparison result, we have to first compare these two different loci. Due to a technical reason (see Remark \ref{Remark: technical assumption on p>2g}), we assume $p>2g$ in this subsection. 

As a starter, we introduce open subsets $\adicFL^\times_{\can}$ and $\adicFL^\times_{\can, w}$ of the adic flag variety $\adicFL$. They are variants of $\adicFL^\times$ and $\adicFL^\times_w$ introduced in \S \ref{subsection: flag varieties}.

Consider the open subset $\adicFL^{\times}_{\can}\subset \adicFL$ whose $(R, R^+)$-points are $$\adicFL^{\times}_{\can}(R, R^+)=\left\{(W\subset V_p\otimes_{\Z_p}R)\in \adicFL(R, R^+):\begin{array}{l}
    \text{there exists a basis $\{w_i\}$ of $W$ such that}  \\
    \text{the matrix $(\bla w_i, e_j\bra)_{1\leq i,j\leq g}$ is invertible}     
\end{array} \right\}.$$
By the same argument as in \S \ref{subsection: flag varieties}, elements in $\adicFL^{\times}_{\can}$ can be represented by homogeneous coordinates
\[
   \begin{pmatrix} \bfitz' &\one_g \end{pmatrix}= \begin{pmatrix}
        \bfitz'_{1, 1} & \cdots & \bfitz'_{1,g} & 1 & & \\
        \vdots & & \vdots & & \ddots & \\
        \bfitz'_{g, 1} & \cdots & \bfitz'_{g,g} & & & 1
    \end{pmatrix},
\] In fact, $\adicFL^\times_{\can}$ is the translate of $\adicFL^\times$ by the longest Weyl element of the Weyl group of $\GSp_{2g}$. For any $w\in \Q_{>0}$, we then define $\adicFL^{\times}_{\can, w} \subset \adicFL_{\can}^\times$ to be \[
    \adicFL^{\times}_{\can, w} := \left\{\bfitx\in \adicFL_{\can}^\times: \max_{i,j}\inf_{t\in p\Z_p}\{|\bfitz_{i,j}(\bfitx) - t|\leq p^{-w}\}\right\}.
\] Similar to Definition \ref{Definition: w-ordinary}, we put \begin{align*}
    \overline{\calX}_{\Gamma(p^{\infty}), \can, w} & := \pi_{\HT}^{-1}(\adicFL^{\times}_{\can, w}),\\
    \overline{\calX}_{\Iw^+, \can, w} & := h_{\Iw^+}(\overline{\calX}_{\Gamma(p^{\infty}), \can, w}),\\
    \overline{\calX}_{\Iw, \can, w} & := h_{\Iw}(\overline{\calX}_{\Gamma(p^{\infty}), \can, w}),\\
    \overline{\calX}_{\can, w} &:= h(\overline{\calX}_{\Gamma(p^{\infty}), \can, w}).
\end{align*} These are referred as the \textbf{\textit{canonical $w$-ordinary loci}}.

To proceed, we also have to clarify the notion of ``$v$-locus'' at the strict Iwahori level. Recall from \S \ref{subsection: construction of AIP} that, for any $v\in \Q_{>0}\cap [0,\frac{1}{2})$, $\overline{\calX}_1(p^n)(v)$ (resp., $\overline{\calX}_1(v)$; resp., $\overline{\calX}_{\Iw}(v)$) is the adic space over $\overline{\calX}(v)$ which parameterises trivialisations of $\calH_n^{\vee}$ (resp., trivialisations of $\calH_1^{\vee}$; resp., full flags of $\calH_1^{\vee}$). In particular, $\overline{\calX}_1(v)$ is equipped with a natural right action of $\GL_g(\Z/p\Z)$ permuting the trivialisations. Consider the quotient
$$\overline{\calX}_{\Iw^+}(v):=\overline{\calX}_1(v)/T_{\GL_g}(\Z/p\Z)$$
which is an adic space over $\overline{\calX}(v)$ parametersing the ``strict Iwahori structures'' of $\calH_1^{\vee}$; namely, it parameterises %full flags $\Fil_{\bullet}\calH_1^{\vee}$ of $\calH_1^{\vee}$ together with 
collections of subgroups $\{D_i: i=1, \ldots, g\}$ of $\calH_1^{\vee}$ of order $p$ such that \[ D_i \cap D_j = 0\] for all $i\neq j$.
%$$\Fil_i\calH_1^{\vee}=\langle D_1, \ldots, D_i\rangle$$
for all $i=1, \ldots, g$. There is a chain of natural projections among these $v$-loci
$$\overline{\calX}_1(p^n)(v)\rightarrow\overline{\calX}_1(v)\rightarrow\overline{\calX}_{\Iw^+}(v)\rightarrow\overline{\calX}_{\Iw}(v)\rightarrow\overline{\calX}(v).$$ %One can identify $\overline{\calX}_{\Iw}(v)$ as the quotient of $\overline{\calX}_{\Iw^+}(v)$ by the finite group $U_{\GL_g}(\Z/p\Z)$.

The main result of this subsection is the following:

\begin{Theorem}\label{Theorem: cofinal system of w- and v-loci}
Given $\Gamma \in \{\Iw^+, \Iw\}$, the system of canonical $w$-ordinary loci $\{\overline{\calX}_{\Gamma, \can, w} : w\in \Q_{>0}\}$ and the system of $v$-loci $\{\overline{\calX}_{\Gamma}(v): v\in \Q_{>0}\cap [0, 1/2)\}$ are mutually cofinal. More precisely, \begin{enumerate}
    \item[(i)] For any given $v\in \Q_{>0}\cap [0, 1/2)$, there exists sufficiently large $w\in \Q_{>0}$ such that $\overline{\calX}_{\Gamma, \can, w}\subset \overline{\calX}_{\Gamma}(v)$.
    \item[(ii)] For any given $w\in \Q_{>0}$, there exists sufficiently small $v\in \Q_{>0}\cap [0, 1/2)$ such that $\overline{\calX}_{\Gamma}(v) \subset \overline{\calX}_{\Gamma, \can, w}$.
\end{enumerate}
\end{Theorem}

\begin{Remark}\label{Remark: Atkin--Lehner operator}
\normalfont To go back to the $w$-ordinary loci from the canonical ones, we use the \emph{Atkin--Lehner operator} \[
    \AL: \adicFL_{w-1}^\times \rightarrow \adicFL^{\times}_{\can, w}, \quad \begin{pmatrix}\one_g & \bfitz\end{pmatrix} \mapsto \begin{pmatrix}\one_g & \bfitz\end{pmatrix}\begin{pmatrix} & \one_g\\ -p\one_g & \end{pmatrix} = \begin{pmatrix}-p\bfitz & \one_g\end{pmatrix}
\] 
for $w\in \Q_{>1}$. This is an isomorphism with inverse $\AL^{-1}: \adicFL^{\times}_{\can, w} \rightarrow \adicFL_{w-1}^\times$ given by right multiplication by the matrix $\begin{pmatrix} & -\frac{1}{p}\one_g\\ \one_g\end{pmatrix}$. It induces an isomorphism $\AL:\overline{\calX}_{\Gamma, w-1}\xrightarrow[]{\sim}\overline{\calX}_{\Gamma, \can, w}$. Therefore, as an immediate corollary of Theorem \ref{Theorem: cofinal system of w- and v-loci}, the systems $\{\overline{\calX}_{\Gamma, w}: w\in \Q_{>0}\}$ and $\{\AL^{-1}\overline{\calX}_{\Gamma}(v): v\in \Q_{>0}\cap [0, 1/2)\}$ are mutually cofinal. 
\end{Remark}

To prove Theorem \ref{Theorem: cofinal system of w- and v-loci}, we follow the strategy in \cite[\S 2.3]{CHJ-2017}. However, we have to generalise their study of pseudocanonical subgroups to the case of semiabelian schemes with constant toric rank.

Let $C$ be an algebraically closed complete nonarchimedean field containing $\Q_p$ and let $\calO_C$ be its ring of integers. Suppose the valuation $v_p$ on $C$ is normalised so that $v_p(p)=1$. Let $G$ be a semiabelian scheme over $\calO_C$ of dimension $g$ with constant toric rank $r \leq g$. That is, $G$ sits inside an extension \[
    0 \rightarrow T \rightarrow G \rightarrow A \rightarrow 0,
\] where $T$ is a torus of rank $r$ over $\calO_C$ and $A$ is an abelian scheme of dimension $g-r$ over $\calO_C$. (We say that $G$ is \textit{\textbf{principally polarised}} if $A$ is principally polarised.) One sees that the $p$-adic Tate module $T_p G := \varprojlim_{n} G[p^n](C)$ is isomorphic to $\Z_p^{2g-r}$. 

Recall the Hodge--Tate complex over $\calO_C$ \[
    0 \rightarrow \Lie G \rightarrow T_p G\otimes_{\Z_p}\calO_C \rightarrow \omega_{G^{\vee}} \rightarrow 0,
\]  where ${\omega}_{G^{\vee}}$ is the dual of the Lie algebra $\Lie G^{\vee}$ of the dual semiabelian scheme $G^{\vee}$, and the second last map is induced from the Hodge--Tate map $\HT_G:T_pG\rightarrow \omega_{G^{\vee}}$. By \cite[Th\'eor\`eme II. 1.1]{Fargues-Genestier-Lafforgue}, the cohomology of this complex is killed by $p^{1/(p-1)}$.  

\begin{Definition}\label{Definition: subspace of Vp}
Recall that $V_p=V\otimes_{\Z} \Z_p\simeq \Z_p^{2g}$ is equipped with the standard basis $e_1, \ldots, e_{2g}$ together with a symplectic pairing. For every $0\leq r \leq g$, let $V_{p,r}$ denote the $\Z_p$-submodule of $V_p$ spanned by $e_{r+1}, e_{r+2}, \ldots, e_{2g-r}$, equipped with the induced symplectic pairing. We also write $V'_{p,r}$ to be $\Z_p$-submodule of $V_p$ spanned by $e_1, \ldots, e_{2g-r}$ and write $W_{p,r}$ to be the one spanned by $e_1, \ldots, e_r$. There is an obvious split exact sequence
$$0\rightarrow W_{p,r}\rightarrow V'_{p,r}\rightarrow V_{p,r}\rightarrow 0.$$
\end{Definition}

\begin{Definition}\label{Definition: w-ordinary semiabelian scheme with constant toric rank}
Let $G$ be a principally polarised semiabelian scheme over $\calO_C$ of dimension $g$ with constant toric rank $r \leq g$. \begin{enumerate}
    \item[(i)] An isomorphism $\alpha: V'_{p,r} \xrightarrow[]{\sim} T_p G$ is called a \textbf{trivialisation} of $T_p G$ if it is part of a commutative diagram \[
        \begin{tikzcd}
            V_{p,r}\arrow[r,"\sim"] & T_pA\\
            V'_{p,r} \arrow[r,"\sim"]\arrow[u, two heads] & T_pG\arrow[u, two heads]\\
            W_{p,r} \arrow[u, hook]\arrow[r,"\sim"] & T_p T\arrow[u, hook]
        \end{tikzcd}
        \] where \begin{itemize}
            \item the vertical arrows on the left are the ones as in Definition \ref{Definition: subspace of Vp};
            \item the vertical arrows on the right are induced from the exact sequence $0\rightarrow T\rightarrow G\rightarrow A\rightarrow 0$;
            \item the top arrow preserves the symplectic pairings.
        \end{itemize}
    \item[(ii)] A trivialisation $\alpha: V'_{p,r} \rightarrow T_p G$ is \textbf{$w$-ordinary} if $\HT_G(\alpha(e_i))\in p^w {\omega}_{G^{\vee}}$ for all $i=1, ..., g$. 
    \item[(iii)] We say that $G$ is \textbf{$w$-ordinary} if it admits a $w$-ordinary trivialisation. 
\end{enumerate}
\end{Definition}

\begin{Remark}\label{Remark: w-ordinary semiabelian schemes}
\normalfont \begin{enumerate}
    \item[(i)] From the definition, if $G$ is $w$-ordinary, it is $w'$-ordinary for any $w'>w$. It is also clear that $G$ is ordinary if and only if it is $w$-ordinary for all $w\in \Q_{>0}$. 
    \item[(ii)] One sees from the definition that classical points in $\overline{\calX}_{\Gamma(p^{\infty}), \can, w}$ correspond to principally polarised semiabelian schemes $G$ together with a $w$-ordinary trivialisation. 
\end{enumerate} 
\end{Remark}

\begin{Lemma}\label{Lemma: pseudocanonical subgroup}
Let $G$ be a $w$-ordinary semiabelian scheme (of dimension $g$ with constant toric rank $r$) over $\calO_C$ and let $n\in \Z_{\geq 1}$ such that $n < w+1$. The Hodge-Tate map $\HT_G$ induces a map
$$G[p^n](C) \rightarrow (\image\HT_{G})/p^{\min\{n, w\}} (\image \HT_{G}).
$$ Then the schematic closure of the kernel of this map defines a flat subgroup scheme $H_n \subset G[p^n]$ whose generic fibre is isomorphic to $(\Z/p^n\Z)^{g}$. Moreover, if $\alpha$ is a $w$-ordinary trivialisation of $T_p G$, then $H_n(C)$ is generated by $\alpha(e_1)$, ..., $\alpha(e_g)$. Here we have abused the notations and still use $\alpha(e_i)$'s to denote their images in $G[p^n](C)$.
\end{Lemma}

\begin{proof}
Since the Hodge--Tate complex is exact after inverting $p$, the image of $\Lie G$ in $T_p G\otimes_{\Z_p}\calO_C$ is a rank $g$ sub-lattice in the kernel of $T_pG\otimes_{\Z_p}\calO_C\rightarrow \omega_{G^{\vee}}$. Hence, the kernel of $\HT_{G}: T_p G \rightarrow {\omega}_{G^{\vee}}$ has rank at most $g$.

On the other hand, there is a commutative diagram \[
    \begin{tikzcd}
        T_p G\arrow[d]\arrow[r, "\HT_{G}"] &  {\omega}_{G^{\vee}}\arrow[d]\\
        G[p^n](C)\arrow[r, "\HT_{G[p^n]}"] & {\omega}_{G[p^n]^{\vee}}
    \end{tikzcd},
\] where the right vertical arrow is induced from the natural identification ${\omega}_{G[p^n]^{\vee}} = {\omega}_{G^{\vee}}/p^n{\omega}_{G^{\vee}}$. Consequently, $\ker\HT_{G[p^n]}$ also has rank at most $g$.

Let $\alpha$ be a $w$-ordinary trivialisation of $T_p G$. Since $n<w+1$, the kernel of the composition \[
    T_p G \xrightarrow{\HT_{G}} {\omega}_{G^{\vee}} \rightarrow {\omega}_{G^{\vee}}/p^n{\omega}_{G^{\vee}}
\] necessarily contains $\alpha(e_i)$, for all $i=1, ..., g$. Since $\alpha(e_i)$'s are $\Z_p$-linearly independent, their images in $G[p^n](C)$ are $(\Z/p^n\Z)$-linearly independent and hence generate $\ker\HT_{G[p^n]}$. Consequently, $H_n$ is precisely the schematic closure in $G[p^n]$ of the subgroup of $G[p^n](C)$ generated by $\{\alpha(e_i): i=1, ..., g\}$. Flatness of $H_n$ follows from the flatness of $G$.
\end{proof}

\begin{Definition}
The subgroup scheme $H_n$ defined in Lemma \ref{Lemma: pseudocanonical subgroup} is called the \textbf{pseudocanonical subgroup of level $n$}. When $n= 1$, we simply call $H_1$ the \textbf{pseudocanonical subgroup} of $G$.
\end{Definition}

\begin{Lemma}\label{Lemma: pseudocanonical subgroups behave like canonical subgroups}
Let $m\leq n$ be positive integers and let $w\in \Q_{>0}$ such that $w>n$. Let $G$ be a semiabelian scheme (of dimension $g$ with constant toric rank $r$) over $\calO_C$. Suppose $G$ is $w$-ordinary. Then, $G/H_{m}$ is $(w-m)$-ordinary, and for any $m'\in \Z$ with $m < m' \leq  n$, we have $H'_{m'-m} = H_{m'}/H_{m}$, where $H'_{m'-m}$ is the pseudocanonical subgroup of $G/H_{m}$ of level $m'-m$.
\end{Lemma}

\begin{proof}
The proof is the same as in \cite[Lemma 2.11]{CHJ-2017} as long as we use the matrix $\diag(p^m\one_{g}, \one_{g-r})$ in place of $\diag(1, p^m)$. Notice that the ``$p^m$'' factor appears at the bottom right corner in \textit{loc. cit.} because they work with a slightly different action of $\GL_2(\Q_p)$. 
\end{proof}

Before stating the next lemma, let us recall the notion of the \emph{degree} of a finite flat group scheme over $\calO_C$ studied in \cite{Fargues-canonical}. If $M$ is a $p$-power torsion $\calO_C$-module of finite presentation, we can write $$M\simeq \bigoplus_{i=1}^l \calO_C/a_i\calO_C$$
for some $a_i\in \calO_C$, $i=1,\ldots, l$. Then the degree of $M$ is defined to be $\deg M:= \sum_{i=1}^{l}v_p(a_i)$. Now, if $H$ is a finite flat group scheme over $\calO_C$ and let $\omega_H$ denote the $\calO_C$-module of invariant differentials on $H$, then we define the \textbf{\textit{degree}} of $H$ to be $\deg H:=\deg \omega_H$.

\begin{Lemma}\label{Lemma: Oort--Tate theory}
Let $G$ be a $w$-ordinary semiabelian scheme (of dimension $g$ with constant toric rank $r$) over $\calO_C$ and let $\alpha$ be a $w$-ordinary trivialisation. Let ${\omega}_{H_1}$ be the dual of $\Lie H_1$ and let $\omega_{H_1^{\vee}}$ be the dual of $\Lie H_1^{\vee}$. For $i=1, ..., g$, let $H_{1, i}$ be the schematic closure in $H_1$ of the subgroup generated by $\alpha(e_i)$. Then \begin{enumerate}
    \item[(i)] Each $H_{1, i}$ is isomorphic to $\Spec(\calO_C[X]/(X^p - a_{i}X))$ for some $a_i\in \calO_C$. The dual $H_{1, i}^{\vee}$ is isomorphic to $\Spec(\calO_C[X]/(X^p-b_iX))$ with $a_ib_i = p$. 
    \item[(ii)] We have isomorphisms ${\omega}_{H_1} \simeq \bigoplus_{i=1}^g \calO_C/a_i\calO_C$ and 
${\omega}_{H_1^{\vee}} \simeq \bigoplus_{i=1}^g \calO_C/b_i\calO_C$. In particular, we have $\deg H_1 = \sum_{i=1}^{g}v_p(a_i)$ and $\deg H_1^{\vee}=\sum_{i=1}^{g}v_p(b_i)=g-\sum_{i=1}^{g}v_p(a_i)$.
    \item[(iii)] Under the identification ${\omega}_{H_1^{\vee}} \simeq \bigoplus_{i=1}^g \calO_C/b_i\calO_C$, the image of the (linearised) Hodge--Tate map \[
        H_1(C) \otimes_{\Z_p} \calO_C \rightarrow {\omega}_{H_1^{\vee}}
    \] is equal to $\bigoplus_{i=1}^g c_i\calO_C/b_i\calO_C$ for some $c_i\in \calO_C$ such that $v_p(c_i) = v_p(a_i)/(p-1)$, $i=1, \ldots, g$.
\end{enumerate}
\end{Lemma}
\begin{proof}
Since each $H_{1, i}$ is a finite flat group scheme over $\calO_{C}$ of degree $p$, the assertion follows from classical Oort--Tate theory. See, for example, \cite[\S 6.5, Lemme 9]{Fargues-canonical}. 
\end{proof}

Recall from \cite[\S 3.1]{AIP-2015} that the \emph{Hodge height} of $G$ is defined to be the ``truncated'' $p$-adic valuation of the Hasse invariant of $G$. See \emph{loc. cit.} for details.

\begin{Lemma}\label{Lemma: pseudocanonical = canonical}
Let $G$ be a $w$-ordinary semiabelian scheme (of dimension $g$ with constant toric rank $r$) over $\calO_C$. Suppose $\frac{(2g-1)p}{2g(p-1)} < w \leq 1$. \footnote{The inequalities are valid because of the assumption $p>2g$ at the beginning of the subsection.} Then $H_1$ coincides with the canonical subgroup of $G$. Moreover, the Hodge height of $G$ is smaller than $1/2$. 
\end{Lemma}

\begin{proof}
We follow the strategy of the proof of \cite[Lemma 2.14]{CHJ-2017}. Consider the commutative diagram \[
    \begin{tikzcd}
        0 \arrow[r] & H_1(C) \arrow[r]\arrow[d, "\HT_{H_1}"] & G[p](C)\arrow[d, "\HT_{G[p]}"]\\
        0 \arrow[r] & {\omega}_{H_1^{\vee}} \arrow[r] & {\omega}_{G[p]^{\vee}}
    \end{tikzcd}
\]  with exact rows. Notice that we have an identification ${\omega}_{G[p]^{\vee}} = {\omega}_{G^{\vee}}/p{\omega}_{G^{\vee}}$. Let $\alpha$ be a $w$-ordinary trivialisation of $T_pG$. According to Lemma \ref{Lemma: pseudocanonical subgroup}, $\alpha(e_1), \ldots, \alpha(e_g)$ form a basis for $H_1(C)$. Also, by definition, we have $\HT_{G[p]}(\alpha(e_i))\in p^w {\omega}_{G[p]^{\vee}}$. 

Now, with respect to the generators $\alpha(e_1) \ldots, \alpha(e_g)$ of $H_1(C)$, the map $\omega_{H_1^{\vee}} \rightarrow \omega_{G[p]^{\vee}}$ can be identified with the inclusion \[
    \bigoplus_{i=1}^g \calO_C/b_i\calO_C \rightarrow (\calO_C/p\calO_C)^g, \quad (x_1, ..., x_g)\mapsto (a_1x_1, ..., a_gx_g).
\] 
Therefore, we see that \[
    a_i\HT_{H_1}(\alpha(e_i)) = \HT_{G[p]}(\alpha(e_i)) \in p^w {\omega}_{G[p]^{\vee}}.
\]  
By Lemma \ref{Lemma: Oort--Tate theory} (iii), we know that $\HT_{H_1}(\alpha(e_i))$ has valuation $v_p(a_i)/(p-1)$. This implies \[
    w \leq v_p(a_i) + \frac{v_p(a_i)}{p-1} = \frac{p v_p(a_i)}{p-1}. 
\] Consequently, we have \[
    \deg H_1 = \sum_{i=1}^g v_p(a_i) \geq \frac{gw(p-1)}{p} > \frac{g(p-1)}{p}\cdot \frac{(2g-1)p}{2g(p-1)} = \frac{2g-1}{2} = g - \frac{1}{2}.
\] It follows from \cite[Proposition 3.1.2]{AIP-2015} that $H_1$ is exactly the canonical subgroup of $G$ and the Hodge height of $G$ is less than $\frac{1}{2}$.
\end{proof}

\begin{Remark}\label{Remark: technical assumption on p>2g}
\normalfont The lemma might hold without the assumption $p>2g$ as long as one can produce finer estimates on the degree and the Hodge height. However, we do not attempt to find these better estimates.
\end{Remark}

\begin{Proposition}\label{Proposition: pseudocanonical = canonical}
Let $G$ be a $w$-ordinary semiabelian scheme (of dimension $g$ with constant toric rank $r$) over $\calO_C$. Suppose $\frac{(2g-1)p}{2g(p-1)}+n-1 < w \leq n$, then $H_{n}$ coincides with the canonical subgroup of $G$ of level $n$. In this case, the Hodge height of $G$ is less than $\frac{1}{2p^{n-1}}$.
\end{Proposition}
\begin{proof}
The proof follows from induction. The case for $n=1$ is precisely Lemma \ref{Lemma: pseudocanonical = canonical}. 

Assume that the statement is affirmative for $n-1$. By Lemma \ref{Lemma: pseudocanonical subgroups behave like canonical subgroups}, $G/H_1$ is $(w-1)$-ordinary and we have $\frac{(2g-1)p}{2g(p-1)}+n-2 < w-1 \leq n-1$. The induction hypothesis implies that the pseudocanonical subgroup $H_{n}/H_1$ of of level $n-1$ of $G/H_1$ is the canonical subgroup of level $n-1$ and that the Hodge height of $G/H_1$ is less than $\frac{1}{2p^{n-2}}$.

However, $H_1$ coincides with the canonical subgroup of $G$ by Lemma \ref{Lemma: pseudocanonical = canonical}. Hence, by \cite[Th\'eor\`em 6 (4)]{Fargues-canonical} (see also \cite[Theorem 3.1.1 (5)]{AIP-2015}), we see that the Hodge height of $G$ is bounded by $\frac{1}{2p^{n-1}}$ and that $H_n$ is the canonical subgroup of level $n$ of $G$.
\end{proof}

\begin{Corollary}\label{Corollary: w-locus injects into v-locus}
Let $n\in \Z_{\geq 1}$ and suppose $w\in \Q_{>0}$ such that $\frac{(2g-1)p}{2g(p-1)}<w\leq n$. Then there exists $v\in \Q_{>0}\cap [0, \frac{1}{2p^{n-1}})$ and a natural inclusion $\overline{\calX}_{\can, w}\hookrightarrow \overline{\calX}(v)$.
\end{Corollary}

\begin{proof}
It suffices to work with $(C, \calO_C)$-points for an algebraically closed complete nonarchimedean field $C$ containing $\Q_p$. (Notice that the classical points determine these adic spaces by \cite[(1.1.11)]{Huber-2013}). Let $\bfitx\in \overline{\calX}_{\can, w}(C, \calO_C)$. By the properness of $\overline{\calX}$, the point $\bfitx$ extends to an $\calO_C$-point $\tilde{\bfitx}$ of $\overline{\frakX}$. One can associate with $\tilde{\bfitx}$ a 1-motive $\widetilde{M}_{\tilde{\bfitx}}=[Y\rightarrow \widetilde{G}_{\tilde{\bfitx}}]$ where $\widetilde{G}_{\tilde{\bfitx}}$ is a semiabelian scheme (of dimension $g$ with constant toric rank) over $\calO_C$ and $Y$ is a free $\Z$-module of finite rank (see, for example, \cite{Stroh-TorComp}). 

From the definition of the Hodge--Tate period map (see \S \ref{subsection: perfectoid Siegel modular variety} for a quick review), we see that $\widetilde{G}_{\tilde{\bfitx}}$ is $w$-ordinary. By Proposition \ref{Proposition: pseudocanonical = canonical}, the Hodge height of $\widetilde{G}_{\tilde{\bfitx}}$ is smaller than $\frac{1}{2p^{n-1}}$. This means $\bfitx\in \overline{\calX}(v)(C, \calO_C)$ for some $v<\frac{1}{2p^{n-1}}$ and so we are done. %Finally, by the compactness of $\overline{\calX}_w$, we know that $\overline{\calX}_w\hookrightarrow \overline{\calX}(v)$ for some $v<\frac{1}{2p^{n-1}}$.
\end{proof}

Recall that, for any $v\in \Q_{>0}\cap [0, \frac{1}{2})$, $\calH_1$ is the universal canonical subgroup of the tautological semiabelian variety over $\overline{\calX}(v)$. Let $w>\frac{(2g-1)p}{2g(p-1)}$ and pick $v$ so that $\overline{\calX}_{\can, w}\hookrightarrow \overline{\calX}(v)$ as in Corollary \ref{Corollary: w-locus injects into v-locus}. We still write $\calH_1$ for its pullback to $\overline{\calX}_{\can, w}$.

In this case, consider
\[
   \overline{\calX}_{1, \can, w} := \Isom_{\overline{\calX}_{\can, w}}((\Z/p\Z)^g, \calH_1^{\vee});
\] namely, the adic space over $\overline{\calX}_{\can, w}$ which parameterises trivialisations of $\calH_1^{\vee}$. The group $\GL_g(\Z/p\Z)$ naturally acts on $\overline{\calX}_{1, \can, w}$ by permuting the trivialisations.

\begin{Lemma}\label{Lemma: alternative definition of w-locus of (strict) Iwahori level variety}
For $w>\frac{(2g-1)p}{2g(p-1)}$, there are natural identifications
\[
    \overline{\calX}_{1, \can, w}/B_{\GL_g}(\Z/p\Z) = \overline{\calX}_{\Iw, \can, w}\quad \text{ and }\quad \overline{\calX}_{1, \can, w}/T_{\GL_g}(\Z/p\Z) = \overline{\calX}_{\Iw^+, \can, w}.
\]
\end{Lemma}
\begin{proof}
We only give the proof for the first identity. The second one is similar and left to the readers.

We first focus on the part away from the boundary. Let $\calX_{\can, w} = \overline{\calX}_{\can, w} \cap \calX$ and let $\calA_{w}^{\univ}$ be the universal abelian variety over $\calX_{\can, w}$. 

The key observation is that any trivialisation $\psi:(\Z/p\Z)^g \rightarrow \calH_{1}^{\vee}$ induces a full flag $\Fil_{\bullet}^{\psi} \calA^{\univ}_{w}[p]$ on $\calA^{\univ}_{w}[p]$. Indeed, let $\epsilon_1, \ldots, \epsilon_g$ denote the standard basis for $(\Z/p\Z)^g$ and let $\Fil_{\bullet}^{\psi}\calH_1^{\vee}$ be the full flag of $\calH_1^{\vee}$ given by \[
    0\subset \langle \psi(\epsilon_1) \rangle \subset \langle \psi(\epsilon_1), \psi(\epsilon_2) \rangle \subset \cdots \subset \langle \psi(\epsilon_1), ..., \psi(\epsilon_g)\rangle.
\] 
Consider the natural projection \[
    \pr: \calA_w^{\univ}[p] \xrightarrow{\sim} \calA_w^{\univ}[p]^{\vee} \twoheadrightarrow \calH_1^{\vee}
\] where the first isomorphism is induced from the principal polarisation. Then the desired full flag $\Fil_{\bullet}^{\psi}\calA_w^{\univ}[p]$ is given by \[
    \Fil_{i}^{\psi}\calA_w^{\univ}[p] := \left\{\begin{array}{ll}
        \pr^{-1}\Fil_{i-g}^{\psi}\calH_1^{\vee}, & i > g \\
        (\pr^{-1}\Fil_{g-i}^{\psi}\calH_{1}^{\vee})^{\perp}, & i \leq g 
    \end{array}\right..
\] 
Moreover, if two such $\psi$'s induce the same $\Fil_{\bullet}^{\psi}\calH_1^{\vee}$, then the associated $\Fil_{\bullet}^{\psi}\calA_w^{\univ}[p]$ coincide. Hence, the assignment $\psi\mapsto \Fil_{\bullet}^{\psi} \calA^{\univ}_{w}[p]$ induces a natural inclusion $\calX_{1, \can,  w}/B_{\GL_g}(\Z/p\Z) \subset \calX_{\Iw, \can, w}$ away from the boundary.

Conversely, using the $w$-ordinarity, one sees that the universal full flag $\Fil_{\bullet}\calA_w^{\univ}[p]$ on $\calX_{\Iw, \can, w}$ induces a full flag $\Fil_{\bullet}\calH_1^{\vee}$ of $\calH_1^{\vee}$ given by 
\[
    \Fil_{i}\calH_{1}^{\vee} = \pr\left((\Fil_{g-i}\calA_{w}^{\univ}[p])^{\perp}\right)
\] for $i=1, ..., g$.
This yields the opposite inclusion away from the boundary. 

In order to extend to the boundary, one considers the 1-motives on the boundary strata and same argument as above applies verbatim. The details are left to the reader.
\end{proof}

Finally, we prove Theorem \ref{Theorem: cofinal system of w- and v-loci}.

\begin{proof}[Proof of Theorem \ref{Theorem: cofinal system of w- and v-loci}] 
\begin{enumerate}
    \item[(i)] We may assume $v=\frac{1}{2p^{n-1}}$ for some sufficiently large $n$. In this case, we can take any $\frac{(2g-1)p}{2g(p-1)}+n-1 < w \leq n$. Indeed, by Corollary \ref{Corollary: w-locus injects into v-locus}, we have a Cartesian diagram
     \[
        \begin{tikzcd}
            \overline{\calX}_{1, \can, w}\arrow[r, hook]\arrow[d] & \overline{\calX}_1(v)\arrow[d]\\
        \overline{\calX}_{\can, w}\arrow[r, hook] & \overline{\calX}(v)
        \end{tikzcd}
    \]
    where the top arrow is equivariant under the action of $\GL_g(\Z/p\Z)$. Taking the quotient by either $B_{\GL_g}(\Z/p\Z)$ or $T_{\GL_g}(\Z/p\Z)$, and applying Lemma \ref{Lemma: alternative definition of w-locus of (strict) Iwahori level variety}, we obtain the desired inclusions.
        
    \item[(ii)] We may assume $n-1<w<n$ for some sufficiently large $n$. Pick $v\in \Q_{>0}\cap [0, \frac{1}{2p^{n-1}})$ such that $w\in \left(n-1+\frac{v}{p-1}, n-\frac{vp^n}{p-1}\right]$. Applying \cite[Proposition 3.2.1]{AIP-2015}, on the level of classical points, we obtain a natural inclusion $\overline{\calX}(v)(C, \calO_C) \hookrightarrow \overline{\calX}_{\can, w}(C,\calO_C)$ and hence an inclusion $\overline{\calX}(v)\hookrightarrow \overline{\calX}_{\can, w}$. There is a Cartesian diagram \[
        \begin{tikzcd}
            \overline{\calX}_1(v)\arrow[r, hook]\arrow[d] & \overline{\calX}_{1,\can,  w}\arrow[d]\\
            \overline{\calX}(v)\arrow[r, hook] & \overline{\calX}_{\can, w}
        \end{tikzcd}
    \] 
    Once again, applying Lemma \ref{Lemma: alternative definition of w-locus of (strict) Iwahori level variety} and taking the corresponding quotients yield the desired inclusions.
\end{enumerate}
\end{proof}

\subsection{Comparison of the two constructions}\label{subsection:comparison sheaf aip} 
In this section, we still assume $p>2g$. The aim of this subsection is to prove the following theorem which compares the overconvergent automorphic sheaf $\underline{\omega}_w^{\kappa_{\calU}}$ constructed in \S \ref{subsection: the perfectoid construction} and the sheaf $\underline{\omega}_{w, v}^{\kappa_{\calU}, \AIP}$ of Andreatta--Iovita--Pilloni.

For any $v\in \Q_{>0}\cap [0, \frac{1}{2})$, let $h_{\diamond}:\overline{\calX}_{\Iw^+}(v)\rightarrow \overline{\calX}_{\Iw}(v)$ denote the natural projection. 

Suppose $n>\textrm{max}\{1, \frac{g}{p-1}\}$ and let $v\in \Q_{>0}\cap [0, \frac{1}{2p^{n-1}})$, $w\in\Q_{>0}\cap (n-1+\frac{v}{p-1}, n-\frac{vp^n}{p-1}]$. According to Theorem \ref{Theorem: cofinal system of w- and v-loci} (ii), there is a natural inclusion $\overline{\calX}_{\Iw^+}(v) \hookrightarrow \overline{\calX}_{\Iw^+,\can, w}$. On the other hand, by Remark \ref{Remark: Atkin--Lehner operator}, we have an isomorphism $\AL:\overline{\calX}_{\Iw^+, w-1}\xrightarrow[]{\sim}\overline{\calX}_{\Iw^+, \can, w}$ induced by the Atkin-Lehner operator. These combined induces an inclusion $\AL^{-1}\overline{\calX}_{\Iw^+}(v) \hookrightarrow \overline{\calX}_{\Gamma, w-1}$.

\begin{Theorem}\label{Theorem: comparison with AIP}
Let $n,v,w$ be as above and let $(R_{\calU}, \kappa_{\calU})$ be a weight such that $w>2+r_{\calU}$. Then, over $\AL^{-1}\overline{\calX}_{\Iw^+}(v)$, there is a canonical isomorphism of sheaves $$\Psi:\underline{\omega}_{w-1}^{\kappa_{\calU}}|_{\AL^{-1}\overline{\calX}_{\Iw^+}(v)}\xrightarrow{\sim} \AL^*h_{\diamond}^*\underline{\omega}_{w, v}^{\kappa_{\calU}, \AIP}.$$
\end{Theorem}

Recall that the space of overconvergent Siegel modular forms of weight $\kappa_{\calU}$ of strict Iwahori level (see Definition \ref{Definition: the sheaf of overconvergent Siegel forms} (v)) is defined to be
$$M^{\kappa_{\calU}}_{\Iw^+}=\varinjlim_{w\rightarrow \infty} M^{\kappa_{\calU}}_{\Iw^+,w}$$
where $$M^{\kappa_{\calU}}_{\Iw^+,w}=H^0(\overline{\calX}_{\Iw^+, w}, \,\,\underline{\omega}_w^{\kappa_{\calU}}).$$
We can also extend the notion of overconvergent Siegel modular forms of Andreatta--Iovita--Pilloni to the case of strict Iwahori level.

\begin{Definition}\label{Definition: AIP's Siegel modular forms for strict Iwahori level}
Let $(R_{\calU}, \kappa_{\calU})$ be a weight.
\begin{enumerate}
\item[(i)] Let $v\in \Q_{>0}\cap [0, 1/2)$ and $w\in \Q_{>0}$. Suppose $\kappa_{\calU}$ is $w$-analytic. The \textbf{space of $w$-analytic $v$-overconvergent Siegel modular forms of weight $\kappa_{\calU}$ (of strict Iwahori level)} of Andreatta--Iovita--Pilloni is defined to be
$$M^{\kappa_{\calU}, \AIP}_{\Iw^+,w,v}:=H^0(\overline{\calX}_{\Iw^+}(v), h_{\diamond}^*\underline{\omega}_{w, v}^{\kappa_{\calU}, \AIP}).$$
\item[(ii)] The \textbf{space of locally analytic overconvergent Siegel modular forms of weight $\kappa_{\calU}$ (of strict Iwahori level)} of Andreatta--Iovita--Pilloni is defined to be
$$M^{\kappa_{\calU}, \AIP}_{\Iw^+}:=\lim_{\substack{v\rightarrow 0\\ w\rightarrow\infty}}M^{\kappa_{\calU}, \AIP}_{\Iw^+,w,v}.$$
\item[(iii)] Similarly, the \textbf{space of $w$-analytic $v$-overconvergent Siegel cuspforms of weight $\kappa_{\calU}$ (of strict Iwahori level)} of Andreatta--Iovita--Pilloni is defined to be \[
    S_{\Iw^+, w, v}^{\kappa_{\calU}, \AIP} := H^0(\overline{\calX}_{\Iw^+}(v), h_{\diamond}^* \underline{\omega}_{w, v, \cusp}^{\kappa_{\calU}, \AIP}),
\] and the \textbf{space of locally analytic overconvergent Siegel cuspforms of weight $\kappa_{\calU}$ (of strict Iwahori level)} of Andreatta--Iovita--Pilloni is defined to be \[
    S_{\Iw^+}^{\kappa_{\calU}, \AIP} := \lim_{\substack{v\rightarrow 0\\ w\rightarrow\infty}}S^{\kappa_{\calU}, \AIP}_{\Iw^+,w,v}.
\]
\end{enumerate}
\end{Definition}

Then we have the following immediate corollary of Theorem \ref{Theorem: comparison with AIP} and Theorem \ref{Theorem: cofinal system of w- and v-loci}.

\begin{Corollary}\label{Corollary: comparison with AIP}
There are canonical isomorphisms
$$M^{\kappa_{\calU}}_{\Iw^+}\cong M^{\kappa_{\calU}, \AIP}_{\Iw^+}\quad \text{ and }\quad S_{\Iw^+}^{\kappa_{\calU}} \cong S_{\Iw^+}^{\kappa_{\calU}, \AIP}.$$
\end{Corollary}

\begin{Remark}
\normalfont In fact, it will follow from the construction of $\Psi$ that the isomorphisms in Corollary \ref{Corollary: comparison with AIP} is also Hecke-equivariant. %, except that the $U_p$-operators are only equivariant up to $p$-power scalars. More precisely, the authors of \cite{AIP-2015} normalise their $U_p$-operators $U^{\mathrm{AIP}}_{p,i}$ by dividing by a certain power of $p$. Therefore, for all $i=1,\ldots,g-1$, our $U_{p,i}$ acts as $p^{i(g+1)}U^{\mathrm{AIP}}_{p,i}$, and $U_{p,g}$ acts as $p^{g(g+1)/2}U^{\mathrm{AIP}}_{p,g}$.
\end{Remark}

The rest of the subsection is dedicated to the proof of Theorem \ref{Theorem: comparison with AIP}. 

Let $n$, $v$, $w$, and $(R_{\calU}, \kappa_{\calU})$ be as in Theorem \ref{Theorem: comparison with AIP}. Recall that the $\scrO_{\overline{\calX}_{\Iw^+}(v)}$-module (resp., $\scrO_{\overline{\calX}_{\Iw}(v)}$-module) $\underline{\omega}_{\Iw^+, v}$ (resp., $\underline{\omega}_{\Iw, v}$) is locally free of rank $g$. Let $\calV' \subset \overline{\calX}_{\Iw}(v)$ be an affinoid open subset such that $\underline{\omega}_{\Iw, v}|_{\calV'}$ is free, and let $\calV \subset \overline{\calX}_{\Iw^+}(v)$ be the preimage of $\calV'$. To construct $\Psi$, it suffices to establish a canonical isomorphism
$$\Psi:\underline{\omega}_{w-1}^{\kappa_{\calU}}(\AL^{-1}\calV)\xrightarrow{\sim} h_{\diamond}^*\underline{\omega}_{w, v}^{\kappa_{\calU}, \AIP}(\calV)$$
for every such $\calV$, which is also functorial in $\calV$.

As a preparation, consider the pullback diagram \[
    \begin{tikzcd}
        \adicIW_{w,v, \infty}^+ \arrow[r]\arrow[d, "\pi_{\infty}^{\AIP}"'] & \adicIW_{w,v}^+\arrow[d, "\pi^{\AIP}"]\\
        \overline{\calX}_{\Gamma(p^{\infty})}(v) \arrow[r, "h_{\Iw}"] & \overline{\calX}_{\Iw}(v)
    \end{tikzcd}
\] where $\overline{\calX}_{\Gamma(p^{\infty})}(v)$ is the preimage of $\overline{\calX}_{\Iw}(v)$ under the natural projection $h_{\Iw}: \overline{\calX}_{\Gamma(p^{\infty})} \rightarrow \overline{\calX}_{\Iw}$. The existence of the pullback follows from the same argument as in the proof of Proposition \ref{Proposition: omega is admissible}. For later usage, we denote by $\calV_{\infty}$ (resp., $\calV^+_{\infty}$) the preimage of $\calV'$ in $\overline{\calX}_{\Gamma(p^{\infty})}(v)$ (resp., in $\adicIW_{w,v, \infty}^+$) under the projection $h_{\Iw}$ (resp., $h_{\Iw}\circ \pi_{\infty}^{\AIP}$).

In what follows, we provide an explicit moduli interpretation of $\adicIW_{w,v, \infty}^+$, in three steps.\\

\noindent\textbf{Step 1.} Observe that the natural projection $h_{\Iw}:\overline{\calX}_{\Gamma(p^{\infty})}(v) \rightarrow \overline{\calX}_{\Iw}(v)$ factors as \[
    h_{\Iw}: \overline{\calX}_{\Gamma(p^{\infty})}(v) \xrightarrow{h_1} \overline{\calX}_1(p^n)(v) \rightarrow \overline{\calX}_{\Iw}(v).
\] Indeed, away from the boundary, the map $h_1$ can be described as follows. Let $\calX_{\Gamma(p^{\infty})}(v)$ be the part of $\overline{\calX}_{\Gamma(p^{\infty})}(v)$ away from the boundary. For every point $(A, \lambda, \psi_N, \psi_{p^{\infty}})\in \calX_{\Gamma(p^{\infty})}(v)$, consider the dual trivialisation $$\psi_{p^{\infty}}^{\vee}: V^{\vee}_p\xrightarrow[]{\sim} T_pA^{\vee}.$$
Modulo $p^n$, we obtain a symplectic isomorphism $$\psi_{p^n}^{\vee}:V_p^{\vee}\otimes_{\Z_p}(\Z/p^n\Z)\xrightarrow[]{\sim} A[p^n]^{\vee}.$$
Then $h_1$ sends $(A, \lambda, \psi_N, \psi_{p^{\infty}})$ to $(A, \lambda, \psi_N, \psi)$ where $\psi$ is the composition $$\psi:(\Z/p^n\Z)^g\hookrightarrow V_p^{\vee}\otimes_{\Z_p}(\Z/p^n\Z)\xrightarrow[]{\psi_{p^n}^{\vee}} A[p^n]^{\vee}\twoheadrightarrow H_n^{\vee}$$ with the first arrow sending $\epsilon_i$ to $e^{\vee}_{g+1-i}\otimes 1$, for all $i=1,\ldots, g$, and the last arrow being the natural surjection. From the proof of Lemma \ref{Lemma: alternative definition of w-locus of (strict) Iwahori level variety}, we see that $\psi$ is indeed a trivialisation of $H_n^{\vee}$.

Using the language of 1-motives, this description of $h_1$ also extends to the boundary. The details are left to the readers.\\

\noindent\textbf{Step 2.} Recall that, in \S \ref{subsection: construction of AIP}, we defined a locally free $\scrO_{\overline{\frakX}_1(p^n)(v)}$-submodule $\scrF\subset \underline{\Omega}_{n, v}$ on $\overline{\frakX}_{1}(p^n)(v)$. Passing to the adic generic fibre, let $\underline{\omega}_{n, v}^+$ denote the sheaf of $\scrO^+_{\overline{\calX}_{1}(p^n)(v)}$-module on $\overline{\calX}_{1}(p^n)(v)$ associated with $\underline{\Omega}_{n, v}$. Then $\scrF$ can be identified with a locally free $\scrO_{\overline{\calX}_{1}(p^n)(v)}^+$-submodule of $\underline{\omega}_{n, v}^+$, which is still denoted by $\scrF$. Moreover, let $\scrF_{\infty}$ be the pullback of $\scrF$ to $\overline{\calX}_{\Gamma(p^{\infty})}(v)$ along $h_1$.

Recall as well the $\scrO_{\overline{\frakX}_{\Gamma(p^n)}}$-modules $\underline{\Omega}_{\Gamma(p^n)}^{\mathrm{mod}}\subset\underline{\Omega}_{\Gamma(p^n)}$ constructed in \S \ref{subsection: perfectoid Siegel modular variety}. Passing to the adic generic fibre, they induce
$\scrO^+_{\overline{\calX}_{\Gamma(p^{n})}}$-modules $\underline{\omega}_{\Gamma(p^n)}^{\mathrm{mod}, +} \subset \underline{\omega}_{\Gamma(p^{n})}^+$ on $\overline{\calX}_{\Gamma(p^{n})}$. Let $\underline{\omega}_{\Gamma(p^{\infty})}^{\mathrm{mod}, +}\subset\underline{\omega}_{\Gamma(p^{\infty})}^+$ be their pullbacks to $\overline{\calX}_{\Gamma(p^{\infty})}$ and let $\underline{\omega}_{\Gamma(p^{\infty}), v}^{\mathrm{mod}, +}\subset\underline{\omega}_{\Gamma(p^{\infty}), v}^+$ be their restrictions on $\overline{\calX}_{\Gamma(p^{\infty})}(v)$.

We claim that there is a natural inclusion \[
    \scrF_{\infty} \subset \underline{\omega}_{\Gamma(p^{\infty}), v}^{\mathrm{mod}, +}.
\] 
Indeed, recall the map 
$$\HT_n:(\Z/p^n\Z)^g\rightarrow \omega_{\calH_n}$$
on $\overline{\calX}_1(p^n)(v)$ constructed in \S \ref{subsection: construction of AIP}. Pulling back to $\overline{\calX}_{\Gamma(p^\infty)}(v)$, we obtain a map
$$\HT_{n,\infty}: (\Z/p^n\Z)^g\rightarrow \omega_{\calH_{n,\infty}}$$
where $\calH_{n,\infty}$ is the pullback of $\calH_n$ along the projection $\overline{\calX}_{\Gamma(p^\infty)}(v)\rightarrow \overline{\calX}_1(p^n)(v)$.
On the other hand, recall the map $\HT_{\Gamma(p^{\infty})}$ on $\overline{\calX}_{\Gamma(p^\infty)}$ constructed in \S \ref{subsection: perfectoid Siegel modular variety}. Restricting to $\overline{\calX}_{\Gamma(p^\infty)}(v)$ and modulo $p^n$, we obtain a map
$$\HT_{\Gamma(p^{\infty}),n,v}: V^{\vee}\otimes_{\Z}(\Z/p^n\Z)\rightarrow \underline{\omega}^{\mathrm{mod}, +}_{\Gamma(p^{\infty}),v}/p^n \underline{\omega}^{\mathrm{mod}, +}_{\Gamma(p^{\infty}),v}.$$

These maps fit into a commutative diagram
\[
    \begin{tikzcd}
        &&\underline{\omega}^{\mathrm{mod},+}_{\Gamma(p^{\infty}),v} \arrow[dd, two heads] \arrow[rr, hook] && \underline{\omega}_{\Gamma(p^{\infty}), v}^+\arrow[d, two heads]\\
        (\Z/p^n\Z)^{g}\arrow[rrrr, crossing over, "\HT_{n, \infty}"]\arrow[d, hook] &&& & \omega_{\calH_{n, \infty}}\arrow[d, two heads]\\
        V^{\vee}\otimes_{\Z}(\Z/p^n\Z)\arrow[rr, "\HT_{\Gamma(p^{\infty}),n,v}"] &&\underline{\omega}^{\mathrm{mod},+}_{\Gamma(p^{\infty}),v}/p^n \arrow[r, hook] & \underline{\omega}_{\Gamma(p^{\infty}), v}^+/p^n \arrow[r, equal] & \omega_{\calH_{n, \infty}}/p^n
    \end{tikzcd}.
\] where the left inclusion sends $\epsilon_i$ to $e_i^{\vee}\otimes 1$, for all $i=1, \ldots, g$. The equality at the bottom right corner follows from \cite[Proposition 3.2.1]{AIP-2015}. By definition, $\scrF_{\infty}$ is generated by the lifts of $\HT_{n,\infty}(\epsilon_i)$'s from $\omega_{\calH_{n,\infty}}$ to $\underline{\omega}_{\Gamma(p^{\infty}), v}^+$ and hence the desired inclusion follows.\\

\noindent\textbf{Step 3.} We are now able to describe the torsor. Recall that there is a universal full flag $\Fil^{\univ}_{\bullet}\calH_1^{\vee}$ of $\calH_1^{\vee}$ on $\overline{\calX}_{\Iw}(v)$. Pulling back to $\overline{\calX}_{\Gamma(p^{\infty})}(v)$, we obtain universal full flag $\Fil^{\univ}_{\bullet}\calH_{1, \infty}^{\vee}$ of $\calH_{1, \infty}^{\vee}$. There is a natural projection $\Theta: \calH_{n,\infty}^{\vee}\rightarrow \calH_{1,\infty}^{\vee}$. Moreover, the Hodge--Tate map on $\calH_{n,\infty}^{\vee}$ induces a map
$$\HT_{\calH_{n,\infty}^{\vee}}: \calH_{n,\infty}^{\vee}\rightarrow \scrF_{\infty}\otimes_{\scrO^+_{\overline{\calX}_{\Gamma(p^{\infty})}(v)}}\scrO^+_{\overline{\calX}_{\Gamma(p^{\infty})}(v)}/p^w.$$

Then, for every affinoid open $\calY = \Spa(R, R^+)\subset \overline{\calX}_{\Gamma(p^{\infty})}(v)$ on which $\scrF_{\infty}(\mathcal{Y})$ is free, the sections $\adicIW_{w, v, \infty}^+(\calY)$ parametrise triples $(\psi, \Fil_{\bullet}, \{w_i: i=1,\ldots, g\})$ where
\begin{itemize}
\item $\psi: (\Z/p^n\Z)^g\xrightarrow{\sim} \calH_{n, \infty}^{\vee}|_{\calY}$ is a trivialisation such that
$$\psi\langle \epsilon_1, \ldots, \epsilon_i\rangle=\Theta(\Fil_i^{\univ}\calH_{1,\infty}^{\vee})$$ for all $i=1, \ldots, g$.
\item $\Fil_{\bullet}$ is a full flag of the free $R^+$-module $\scrF_{\infty}(\calY)$, which is $w$-compatible with $$\HT_{\calH_{n,\infty}^{\vee}}(\psi(\epsilon_1)), \ldots, \HT_{\calH_{n,\infty}^{\vee}}(\psi(\epsilon_g))$$ in the sense of Definition \ref{Definition: w-compatible} (i).
\item Each $w_i$ is an $R^+$-basis for $\Fil_i/\Fil_{i-1}$, which is $w$-compatible with $\HT_{\calH_{n,\infty}^{\vee}}(\psi(\epsilon_1)), \ldots, \HT_{\calH_{n,\infty}^{\vee}}(\psi(\epsilon_g))$ in the sense of Definition \ref{Definition: w-compatible} (ii).
\end{itemize}

We are now ready to prove Theorem \ref{Theorem: comparison with AIP}.

\begin{proof}[Proof of Theorem \ref{Theorem: comparison with AIP}]
Let $\calV'$, $\calV$, and $\calV_{\infty}$ be as above. We want to construct an isomorphism
$$\Psi:\underline{\omega}_{w-1}^{\kappa_{\calU}}(\AL^{-1}\calV)\xrightarrow{\sim} h_{\diamond}^*\underline{\omega}_{w, v}^{\kappa_{\calU}, \AIP}(\calV).$$

Recall the auxiliary sheaves $\underline{\omega}_{n,w}^{\kappa_{\calU}}$ and $\widetilde{\underline{\omega}}_{n,w}^{\kappa_{\calU}}$ constructed in \S \ref{subsection: admissibility}. By Proposition \ref{Proposition: omega is admissible} and Remark \ref{remark: invariants}, we have isomorphisms
$$\underline{\omega}_{w-1}^{\kappa_{\calU}}(\AL^{-1}\calV)\cong \left(\underline{\omega}_{n,w-1}^{\kappa_{\calU}}(\AL^{-1}\calV)\right)^{\Iw^+_{\GSp_{2g}}/\Gamma(p^n)}\cong \left(\widetilde{\underline{\omega}}_{n,w-1}^{\kappa_{\calU}}(\AL^{-1}\calV)\right)^{\Iw^+_{\GSp_{2g}}/\Gamma(p^n)}$$
where the $\Iw^+_{\GSp_{2g}}/\Gamma(p^n)$-action on the middle term is the twisted action in Remark \ref{remark: invariants}, while the $\Iw^+_{\GSp_{2g}}/\Gamma(p^n)$-action on the last term is the natural action. Finally, by comparing the moduli interpretation of $\adicIW_{w, v, \infty}^+$ above with the moduli interpretation of $\adicIW_{w, \infty}^+$ in \S \ref{subsection: admissibility} and noticing that $\epsilon_i$ corresponds to $e_i^{\vee}$ via $\AL$, we arrive at an isomorphism
$$\left(\widetilde{\underline{\omega}}_{n,w-1}^{\kappa_{\calU}}(\AL^{-1}\calV)\right)^{\Iw^+_{\GSp_{2g}}/\Gamma(p^n)}\cong h_{\diamond}^*\underline{\omega}_{w, v}^{\kappa_{\calU}, \AIP}(\calV).$$ This finishes the proof.
\end{proof}

\begin{Remark}\label{Remark: comparison of the integral sheaf}
\normalfont In fact, the method above provides a strategy to compare our integral sheaf $\underline{\omega}_{w}^{\kappa_{\calU}, +}$ with the integral overconvergent automorphic sheaf constructed in \cite{AIP-2015}. The details are left to the interested readers.
\end{Remark}
\section{Overconvergent cohomology groups}\label{section:overconvergentcohomologies} 
In this section, we introduce the overconvergent cohomology groups that will later appear on the other side of the overconvergent Eichler--Shimura morphism. Our construction follows the standard constructions in the literature (see, for example, \cite{Hansen-PhD} and \cite{Johansson-Newton}). 

\subsection{Analytic functions and analytic distributions}\label{subsection:continuousfunctions}
Consider $$\T_0:=\left\{(\bfgamma, \bfupsilon)\in \Iw_{\GL_g}\times M_g(p\Z_p): \trans\bfgamma\oneanti_g\bfupsilon= \trans\bfupsilon\oneanti_g\bfgamma\right\}.$$ Notice that a pair $(\bfgamma, \bfupsilon)\in \Iw_{\GL_g}\times M_g(p\Z_p)$ lies in $\T_0$ if and only if there exist $\bfalpha_b, \bfalpha_d\in M_g(\Z_p)$ such that \[\begin{pmatrix}\bfgamma & \bfalpha_b\\ \bfupsilon & \bfalpha_d\end{pmatrix}\in \GSp_{2g}(\Q_p)\cap M_{2g}(\Z_p).\] 
In fact, there is a natural embedding
$$\T_0\hookrightarrow \Iw_{\GSp_{2g}}, \quad (\bfgamma, \bfupsilon)\mapsto \begin{pmatrix}\bfgamma & \\ \bfupsilon & \oneanti_g\trans\bfgamma^{-1}\oneanti_g\end{pmatrix}.$$
Also consider the subset $\T_{00}$ of $\T_0$ defined by $$\T_{00}:=\left\{(\bfgamma, \bfupsilon)\in \T_0: \bfgamma\in U_{\GL_g,1}^{\opp}\right\}.$$ We can identify $\T_{00}$ with $U_{\GSp_{2g}, 1}^{\opp}$ through the bijection $$\T_{00}\rightarrow U_{\GSp_{2g}, 1}^{\opp}, \quad (\bfgamma, \bfupsilon)\mapsto \begin{pmatrix}\bfgamma & \\ \bfupsilon & \oneanti_g\trans\bfgamma^{-1}\oneanti_g\end{pmatrix}.$$ Observe that $\T_0$ admits two natural actions:\begin{enumerate}
    \item[(i)] There is a right action of $\Iw_{\GL_g}$ given by $$\T_0\times \Iw_{\GL_g}\rightarrow \T_0, \quad ((\bfgamma, \bfupsilon), \bfgamma')\mapsto (\bfgamma\bfgamma', \bfupsilon\bfgamma').$$ To see that this is indeed a right action, we embed $\Iw_{\GL_g}$ into $\Iw_{\GSp_{2g}}$ through $\bfgamma'\mapsto \begin{pmatrix}\bfgamma' & \\ & \oneanti_g\trans\bfgamma'^{-1}\oneanti_g \end{pmatrix}$ and verify that 
    \[\begin{pmatrix}\bfgamma & * \\ \bfupsilon & *\end{pmatrix}\begin{pmatrix}\bfgamma' & \\  & \oneanti_g\trans\bfgamma'^{-1}\oneanti_g\end{pmatrix}=\begin{pmatrix}\bfgamma\bfgamma' & * \\ \bfupsilon\bfgamma' & *\end{pmatrix}\]
    
     \item[(ii)] There is a left action of $\Xi:=\begin{pmatrix}\Iw_{\GL_g} & M_g(\Z_p)\\ M_g(p\Z_p) & M_g(\Z_p)\end{pmatrix}\cap \GSp_{2g}(\Q_p)$ given by $$\Xi\times \T_0\rightarrow \T_0, \quad \left( \begin{pmatrix}\bfalpha_a & \bfalpha_b\\ \bfalpha_c & \bfalpha_d\end{pmatrix}, (\bfgamma, \bfupsilon)\right)\mapsto (\bfalpha_a\bfgamma+\bfalpha_b\bfupsilon, \bfalpha_c\bfgamma+\bfalpha_d\bfupsilon).$$ To see this is indeed a left action, it suffices to observe that $$\begin{pmatrix}\bfalpha_a & \bfalpha_b\\ \bfalpha_c & \bfalpha_d\end{pmatrix} \begin{pmatrix}\bfgamma & *\\ \bfupsilon & *\end{pmatrix}=\begin{pmatrix}\bfalpha_a\bfgamma+\bfalpha_b\bfupsilon & *\\ \bfalpha_c\bfgamma+\bfalpha_d\bfupsilon & *\end{pmatrix}.$$
\end{enumerate} 
Since $\Iw_{\GSp_{2g}}^+$ is a subset of $\Xi$, we also obtain a natural left action of $\Iw_{\GSp_{2g}}^+$ on $\T_0$.\\ 

Let $r\in \Z_{\geq 1}$ and let $(R_{\calU}, \kappa_{\calU})$ be an $r$-analytic weight. In what follows, we will study ``$r$-analytic'' functions on $U_{\GSp_{2g}, 1}^{\opp}$, $\T_{00}$, and $\T_0$. Let us fix a (topological) isomorphism $$\Z_p^{g^2}\simeq U_{\GSp_{2g}, 1}^{\opp}.$$ 

\begin{Definition}
\begin{enumerate}
\item[(i)] We say that a function $f: U_{\GSp_{2g}, 1}^{\opp}\rightarrow R^+_{\calU}$ is \textbf{$r$-analytic} if the composition
$$\Z_p^{g^2}\simeq U_{\GSp_{2g}, 1}^{\opp}\xrightarrow[]{f}R_{\calU}^+\hookrightarrow \C_p\widehat{\otimes}R_{\calU}$$
is $r$-analytic in the sense of Definition \ref{Definition: w-analytic functions} (i). 
\item[(ii)] We say that a function $f: \T_{00}\rightarrow R^+_{\calU}$ is \textbf{$r$-analytic} if it is $r$-analytic viewed as a function on $U^{\opp}_{\GSp_{2g}, 1}$, via the identification $\T_{00}\cong U^{\opp}_{\GSp_{2g}, 1}$.
\end{enumerate}
\end{Definition}

Before proceeding, we need the following statement.

\begin{Theorem}[Amice]\label{Theorem: Amice's basis}
Let $r\in \Q_{\geq 0}$.
For any $d\in \Z_{>0}$ and for any $i= (i_1, ..., i_d)\in \Z_{\geq 0}^d$, define the function \[
    e_i^{(r)}: \Z_p^d \rightarrow \Z_p, \quad (x_1, ..., x_d)\mapsto \prod_{t=1}^d \lfloor p^{-r}i_t\rfloor! \begin{pmatrix} x_t\\ i_t\end{pmatrix}.
\] Then, $\{e_i^{(r)}\}_i$ provides an othonormal basis for $C^{r-\an}(\Z_p^d, \Z_p)$.
\end{Theorem}
\begin{proof}
This is a reformulation of \cite[Chapter III, 1.3.8]{Larzard}, which is based on the work of Y. Amice \cite[\S 10]{Amice}. 
\end{proof}

Given an $r$-analytic weight $(R_{\calU}, \kappa_{\calU})$, we define \[
    A^{r, \circ}(\T_{00}, R_{\calU}):= C^{r-\an}(\T_{00}, \Z_p)\widehat{\otimes} R_{\calU}^+ \quad \text{ and }\quad A^r(\T_{00}, R_{\calU}):=A^{r, \circ}(\T_{00}, R_{\calU})[\frac{1}{p}].
\] By identifying $\T_{00}$ with $\Z_p^{g^2}$, Theorem \ref{Theorem: Amice's basis} implies that \[
    A^{r, \circ}(\T_{00}, R_{\calU}) \simeq \widehat{\oplus}_{i\in \Z_{\geq 0}^{g^2}} R_{\calU}^+ e_i^{(r)}
\] and so we view elements in $A^{r, \circ}(\T_{00}, R_{\calU})$ as functions from $\T_{00}$ to $R_{\calU}^+$. In other words, we have \[
    A^{r, \circ}(\T_{00}, R_{\calU}) =  \left\{ \sum_{i\in \Z_{\geq 0}^{g^2}} c_i e_i^{(r)}: c_i\in R_{\calU}^+\text{ and }c_i \rightarrow 0 \text{ $\fraka_{\calU}$-adically}\right\},
\] where $\fraka_{\calU} = pR_{\calU}^+$ if $(R_{\calU}, \kappa_{\calU})$ is an affinoid weight and $\fraka_{\calU}$ is an ideal of definition of the profinite topology on $R_{\calU}^+$ if $(R_{\calU}, \kappa_{\calU})$ is a small weight. By definition, these functions are $r$-analytic. In fact, if $(R_{\calU}, \kappa_{\calU})$ is an affinoid weight, we have the identification \[
    A^r(\T_{00}, R_{\calU}) = \left\{\text{$r$-analytic functions }f: \T_{00} \rightarrow R_{\calU}\right\}.
\]

On the other hand, define 
$$A^{r, \circ}_{\kappa_{\calU}}(\T_0, R_{\calU}):=\left\{f:\T_0\rightarrow R_{\calU}^+: \begin{array}{l}
    f(\bfgamma\bfbeta, \bfupsilon\bfbeta)=\kappa_{\calU}(\bfbeta)f(\bfgamma, \bfupsilon)\,,\,\forall (\bfgamma, \bfupsilon)\in \T_0, \,\,\bfbeta\in B_{\GL_g, 0}\\
    f|_{\T_{00}} \in A^{r, \circ}(\T_{00}, R_{\calU}) 
\end{array}\right\}$$
and
$$A^{r}_{\kappa_{\calU}}(\T_0, R_{\calU}):=A^{r, \circ}_{\kappa_{\calU}}(\T_0, R_{\calU})[\frac{1}{p}].$$
There is an identification $$A_{\kappa_{\calU}}^{r, \circ}(\T_0, R_{\calU})\xrightarrow{\sim} A^{r, \circ}(\T_{00}, R_{\calU}), \quad f\mapsto f|_{\T_{00}}.$$ 
Taking duals, we obtain the corresponding spaces of $r$-analytic distributions
$$
    D_{\kappa_{\calU}}^{r, \circ}(\T_0, R_{\calU}):=\Hom_{R_{\calU}^+}^{\cts}(A_{\kappa_{\calU}}^{r, \circ}(\T_0, R_{\calU}), R_{\calU}^+)
$$
and
$$
D_{\kappa_{\calU}}^{r}(\T_0, R_{\calU}):=D_{\kappa_{\calU}}^{r, \circ}(\T_0, R_{\calU})[\frac{1}{p}].
$$
The left action of $\Xi$ on $\T_0$ then induces a left action of $\Xi$ on both $D_{\kappa_{\calU}}^{r, \circ}(\T_0, R_{\calU})$ and $D_{\kappa_{\calU}}^{r}(\T_0, R_{\calU})$. Furthermore, if $r'\geq r$, there is a natural injection $A_{\kappa_{\calU}}^{r, \circ}(\T_0, R_{\calU})\hookrightarrow A_{\kappa_{\calU}}^{r', \circ}(\T_0, R_{\calU})$ which induces injections (see \cite[\S 2.2]{Hansen-PhD}) \begin{align*}
    D_{\kappa_{\calU}}^{r', \circ}(\T_0, R_{\calU})\hookrightarrow D_{\kappa_{\calU}}^{r, \circ}(\T_0, R_{\calU})\quad \text{ and }\quad D_{\kappa_{\calU}}^{r'}(\T_0, R_{\calU})\hookrightarrow D_{\kappa_{\calU}}^{r}(\T_0, R_{\calU}).
\end{align*} We then write \[
    D_{\kappa_{\calU}}^{\dagger}(\T_0, R_{\calU}) := \varprojlim_{r} D_{\kappa_{\calU}}^{r}(\T_0, R_{\calU}).
\]

Suppose now that $(R_{\calU}, \kappa_{\calU})$ is a small weight and take $r>1+r_{\calU}$ (see Definition \ref{Definition: w-analytic weight}). Fix an ideal $\fraka_{\calU}$ of $R_{\calU}$ defining the profinite topology on $R_{\calU}$. Similar to \cite[Proposition 3.1]{CHJ-2017}, $D_{\kappa_{\calU}}^{r, \circ}(\T_0, R_{\calU})$ admits a decreasing filtration $\Fil^{\bullet} D_{\kappa_{\calU}}^{r, \circ}(\T_0, R_{\calU})$ defined by $$\Fil^jD_{\kappa_{\calU}}^{r, \circ}(\T_0, R_{\calU}):=\ker\left(D_{\kappa_{\calU}}^{r, \circ}(\T_0, R_{\calU})\rightarrow D_{\kappa_{\calU}}^{r-1, \circ}(\T_0, R_{\calU})/\fraka_{\calU}^jD_{\kappa_{\calU}}^{r-1, \circ}(\T_0, R_{\calU})\right).$$ Write $$D_{\kappa_{\calU}, j}^{r, \circ}(\T_0, R_{\calU}):=D_{\kappa_{\calU}}^{r, \circ}(\T_0, R_{\calU})/\Fil^j D_{\kappa_{\calU}}^{r, \circ}(\T_0, R_{\calU})$$ for every $j\in \Z_{\geq 1}$. 

\begin{Lemma}
Given a small weight $(R_{\calU}, \kappa_{\calU})$ and $r>1+r_{\calU}$.
\begin{enumerate}
    \item[(i)] For any $j\in \Z_{\geq 0}$, $\Fil^j D_{\kappa_{\calU}}^{r,\circ}$ is $\Xi$-stable.
    \item[(ii)] For any $j\in \Z_{\geq 0}$, $D_{\kappa_{\calU}, j}^{r, \circ}(\T_0, R_{\calU})$ is a finitely abelian group. Therefore, $$D_{\kappa_{\calU}}^{r, \circ}(\T_0, R_{\calU})=\varprojlim_{j} D_{\kappa_{\calU}, j}^{r, \circ}(\T_0, R_{\calU}),$$ is a profinite flat $\Z_p$-module in the sense of \cite[Definition 6.1]{CHJ-2017}.
\end{enumerate}
\end{Lemma}
\begin{proof}
To show (i), one observes that \[
    \fraka_{\calU}^jD_{\kappa_{\calU}}^{r-1, \circ}(\T_0, R_{\calU}) = \left\{\mu \in D_{\kappa_{\calU}}^{r-1, \circ}(\T_0, R_{\calU}): \mu(f)\in \fraka_{\calU}^j, \,\, \forall f\in A_{\kappa_{\calU}}^{r-1, \circ}(\T_0, R_{\calU})\right\}
\] Since $A_{\kappa_{\calU}}^{r-1, \circ}(\T_0, R_{\calU})$ is stable under the action of $\Xi$, $\fraka_{\calU}^j D_{\kappa_{\calU}}^{r-1, \circ}(\T_0, R_{\calU})$ is stable under the action of $\Xi$. This then implies the desired result. 

The proof for (ii) is inspired by the discussion in \cite[\S 2.1]{Hansen-Iwasawa}. We first simplify the notation by writing $d = g^2$. From the construction, the collection $\{e_i^{(r)}\}_{i}$ provides an orthonormal basis for $A_{\kappa_{\calU}}^{r, \circ}(\T_0, R_{\calU})$, \emph{i.e.}, we have an isomorphism \[
    A_{\kappa_{\calU}}^{r, \circ}(\T_0, R_{\calU}) \simeq  \widehat{\oplus}_{i\in \Z_{\geq 0}^d} R_{\calU} e_i^{(r)}.
\] Consequently, we have an isomorphism \[
    D_{\kappa_{\calU}}^{r, \circ}(\T_0, R_{\calU}) \simeq \prod_{i\in \Z_{\geq 0}^d} R_{\calU}, \quad \mu\mapsto (\mu(e_{i}^{(r)}))_{i}.
\]

For any $i\in \Z_{\geq 0}^d$, write $c_{r, i} := \prod_{t} \frac{\lfloor p^{-(r-1)}i_t \rfloor !}{\lfloor p^{-r} i_t \rfloor !}$. Then, the natural injection $A_{\kappa_{\calU}}^{r-1, \circ}(\T_0, R_{\calU}) \hookrightarrow A_{\kappa_{\calU}}^{r, \circ}(\T_0, R_{\calU})$ is given by \[
    \widehat{\oplus}_i R_{\calU} e_{i}^{(r)} \rightarrow \widehat{\oplus}_i R_{\calU} e_{i}^{(r-1)}, \quad e_{i}^{(r)} \mapsto e_{i}^{(r)} = c_{r, i} e_{i}^{(r-1)}.
\] Hence, the natural inclusion $D_{\kappa_{\calU}}^{r, \circ}(\T_0, R_{\calU}) \hookrightarrow D_{\kappa_{\calU}}^{r-1, \circ}(\T_0, R_{\calU})$ is given by \[
    \prod_{i\in \Z_{\geq 0}^d}R_{\calU} \rightarrow \prod_{i\in \Z_{\geq 0}^d} R_{\calU}, \quad (\mu(e_i^{(r)}))_i\mapsto (c_{r, i} \mu(e_i^{(r-1)})).
\] Moreover, by Legendre's formula, we have $v_p(c_{r, i}) = \sum_{t=1}^d \lfloor p^{-r} i_t\rfloor$. Therefore, we see that \[
    D_{\kappa_{\calU}, j}^{r, \circ}(\T_0, R_{\calU}) \simeq \oplus_{\substack{i\in \Z_{\geq 0}^d\\ v_p(c_{r,i}) < j}} R_{\calU}/(\fraka_{\calU}^j, p^{j-v_p(c_{r, i})}).
\] Since this is a finite direct sum and each direct summand is a finite abelian group, we conclude that each $D_{\kappa_{\calU}, j}^{r, \circ}(\T_0, R_{\calU})$ is a finite abelian group. 

Finally, from the construction, we see that the natural map \[
    D_{\kappa_{\calU}}^{r, \circ}(\T_0, R_{\calU}) \rightarrow \varprojlim_{j} D_{\kappa_{\calU}, j}^{r, \circ}(\T_0, R_{\calU}),
\] has dense image. Since both sides are compact, this natural map is an isomorphism. 
\end{proof}

%\begin{Remark}
%\normalfont Note that we used the fact that $\kappa_{\calU}$ is a small weight so that each $D_{\kappa_{\calU}, j}^{r, \circ}(\T_0, R_{\calU})$ is a finitely generated $(\Z/p^j\Z)$-module. We do not know if a similar statement holds for affinoid weights.
%\end{Remark}

\subsection{Overconvergent cohomology groups}\label{subsection: overconvergent cohomology groups}
Fix a small weight $(R_{\calU}, \kappa_{\calU})$ and we consider the \'{e}tale site $\calX_{\Iw^+, \et}$. Recall that, for every $n\in\Z_{\geq 1}$, $\calX_{\Gamma(p^n)}$ is a finite \'{e}tale Galois cover over $\calX_{\Iw^+}$ with Galois group $\Iw_{\GSp_{2g}}^+/\Gamma(p^n)$, and hence $\varprojlim_n\calX_{\Gamma(p^n)}$ is a pro-\'{e}tale Galois cover of $\calX_{\Iw^+}$ with Galois group $\Iw_{\GSp_{2g}}^+$. For each $j\in \Z_{\geq 1}$, let $\scrD_{\kappa_{\calU}, j}^{r, \circ}$ be the locally constant sheaf on $\calX_{\Iw^+, \et}$ associated with $D_{\kappa_{\calU}, j}^{r, \circ}(\T_0, R_{\calU})$ via $$\pi_1^{\et}(\calX_{\Iw^+})\rightarrow \Iw_{\GSp_{2g}}^+\rightarrow \Aut\left(D_{\kappa_{\calU}, j}^{r, \circ}(\T_0, R_{\calU})\right).$$ We obtain an inverse system of \'etale locally constant sheaves $(\scrD_{\kappa_{\calU}, j}^{r, \circ})_{j\in \Z_{\geq 1}}$ on $\calX_{\Iw^+, \et}$. This allows us to consider the \'{e}tale cohomology groups \begin{align*}
    & H_{\et}^t(\calX_{\Iw^+}, \scrD_{\kappa_{\calU}}^{r, \circ}):=\varprojlim_{j}H_{\et}^t(\calX_{\Iw^+}, \scrD_{\kappa_{\calU}, j}^{r, \circ}),\\
    & H_{\et}^t(\calX_{\Iw^+}, \scrD_{\kappa_{\calU}}^{r}):=H_{\et}^t(\calX_{\Iw^+}, \scrD_{\kappa_{\calU}}^{r, \circ})[\frac{1}{p}]
\end{align*} for every $t\in \Z_{\geq 0}$. 

\begin{Remark}
\normalfont On the algebraic variety $X_{\Iw^+}$, one can define locally constant sheaves $\scrD_{\kappa_{\calU}, j}^{r, \circ}$ and \'etale cohomology groups $H_{\et}^t(X_{\Iw^+}, \scrD_{\kappa_{\calU}}^{r, \circ})$ and $H_{\et}^t(X_{\Iw^+}, \scrD_{\kappa_{\calU}}^{r})$ in the same way.
\end{Remark}

Recall the identification \[
    X_{\Iw^+}(\C) = \GSp_{2g}(\Q)\backslash \GSp_{2g}(\A_f)\times \bbH_g/\Iw_{\GSp_{2g}}^+\Gamma(N).
\] By taking the trivial $\GSp_{2g}(\Z_{\ell})$-action on $D_{\kappa_{\calU}}^r(\T_0, R_{\calU})$ for every prime number $\ell\neq p$ and letting $\Iw_{\GSp_{2g}}^+$ act on $D_{\kappa_{\calU}}^r(\T_0, R_{\calU})$ via the left action of $\Xi$, we see that $D_{\kappa_{\calU}}^r(\T_0, R_{\calU})$ defines a local system on the locally symmetric space $X_{\Iw^+}(\C)$. In particular, for every $t\in \Z_{\geq 0}$, we can consider the Betti cohomology group \[
    H^t(X_{\Iw^+}(\C), D_{\kappa_{\calU}}^r(\T_0, R_{\calU})).
\]

\begin{Proposition}\label{Proposition: Comparison theorem of cohomologies}
For every $t\in \Z_{\geq 0}$, there is a natural isomorphism $$H_{\et}^t(\calX_{\Iw^+}, \scrD_{\kappa_{\calU}}^r)\cong H^t(X_{\Iw^+}(\C), D_{\kappa_{\calU}}^r(\T_0, R_{\calU})).$$\end{Proposition}
\begin{proof}
For any $j\in \Z_{>0}$, we have isomorphisms \[
    H_{\et}^t(\calX_{\Iw^+}, \scrD_{\kappa_{\calU},j }^{r, \circ})\cong H_{\et}^t(X_{\Iw^+}, \scrD_{\kappa_{\calU}, j}^{r, \circ})\cong H^t(X_{\Iw^+}(\C), D_{\kappa_{\calU},j}^{r, \circ}(\T_0, R_{\calU})),
\] where \begin{enumerate}
    \item[$\bullet$] the first isomorphism follows from the comparison isomorphism between the \'{e}tale cohomology groups of an algebraic variety and the ones on the corresponding adic spaces (see \cite[Theorem 3.8.1]{Huber-2013});\footnote{ On the algebraic variety $X_{\Iw^+} = X_{\Iw^+, \C_p}$, the locally constant sheaves $\scrD_{\kappa_{\calU}, j}^{r, \circ}$ and \'etale cohomology groups $H_{\et}^t(X_{\Iw^+}, \scrD_{\kappa_{\calU}}^{r, \circ})$ and $H_{\et}^t(X_{\Iw^+}, \scrD_{\kappa_{\calU}}^{r})$ are defined analogously as those on $\calX_{\Iw^+}$.} and
    \item[$\bullet$] the second isomorphism follows from the fact that $\Iw_{\GSp_{2g}}^+$ acts continuously on the module $D_{\kappa_{\calU},j}^{r, \circ}(\T_0, R_{\calU})$ and the well-known Artin comparison between the \'{e}tale cohomology of a complex algebraic variety and the Betti cohomology of the associated complex manifold. 
\end{enumerate} Note that we have used the algebraic isomorphism $\C_p\simeq \C$ fixed at the beginning of the paper.

Taking limit and inverting $p$, we then arrive at the isomorphisms \[
    H_{\et}^t(\calX_{\Iw^+}, \scrD_{\kappa_{\calU}}^r) \cong H_{\et}^t(X_{\Iw^+}, \scrD_{\kappa_{\calU}}^r)\cong \left(\varprojlim_{j} H^t(X_{\Iw^+}(\C), D_{\kappa_{\calU}, j}^{r, \circ}(\T_0, R_{\calU}))\right)[1/p].
\] By applying \cite[Lemma 3.18]{Scholze_2013}, one deduced that \[
    \varprojlim_{j} H^t(X_{\Iw^+}(\C), D_{\kappa_{\calU}, j}^{r, \circ}(\T_0, R_{\calU})) = H^t(X_{\Iw^+}(\C), D_{\kappa_{\calU}}^{r, \circ}(\T_0, R_{\calU})).
\] The assertion then follows.
\end{proof}

Finally, we define the Hecke operators acting on $H_{\et}^t(\calX_{\Iw^+},\scrD_{\kappa_{\calU}}^r)$. We begin with a brief recollection of the Hecke operators on $H^t(X_{\Iw^+}(\C), D_{\kappa_{\calU}}^r(\T_0, R_{\calU}))$ studied in \cite{Hansen-PhD}. We refer the readers to \textit{loc. cit.} for a more detailed discussion.\\

\noindent\textbf{Hecke operators outside $pN$.} Let $\ell$ be a prime number not dividing $pN$. For any $\bfgamma \in \GSp_{2g}(\Q_{\ell})\cap M_{2g}(\Z_{\ell})$, consider a double coset decomposition \[
    \GSp_{2g}(\Z_{\ell}) \bfgamma \GSp_{2g}(\Z_{\ell}) = \bigsqcup_j \bfdelta_{j}\bfgamma\GSp_{2g}(\Z_{\ell})
\] for some $\bfdelta_j\in \GSp_{2g}(\Z_{\ell})$. If we take the trivial $\GSp_{2g}(\Q_{\ell})$-action on $D_{\kappa_{\calU}}^r(\T_0, R_{\calU})$, then the natural left action of $\GSp_{2g}(\Q_{\ell})$ on $X_{\Iw^+}(\C)$ induces the Hecke operator 
\begin{equation}\label{eq: Hecke on coh. away from Np}
    T_{\bfgamma}: H^t(X_{\Iw^+}(\C), D_{\kappa_{\calU}}^r(\T_0, R_{\calU})) \rightarrow H^t(X_{\Iw^+}(\C), D_{\kappa_{\calU}}^r(\T_0, R_{\calU})), \quad [\mu]\mapsto \sum_{j} (\bfdelta_j\bfgamma) .  [\mu].
\end{equation}

\noindent\textbf{Hecke operators at $p$.} For the Hecke operators at $p$, recall from \S \ref{subsection: Hecke operators on the overconvergent automorphic forms} the matrices \begin{align*}
    \bfu_{p,i} = \left\{\begin{array}{cc}
        \begin{pmatrix} \one_i \\ & p\one_{g-i}\\ & & p\one_{g-i}\\ & & & p^2\one_i\end{pmatrix}, & 1\leq i \leq g-1 \\
        \\
        \begin{pmatrix}\one_g\\ & p\one_g\end{pmatrix}, & i=g
    \end{array}\right.
\end{align*} and we write \[
    \bfu_{p,i} = \begin{pmatrix}\bfu_{p,i}^{\square} & \\ & \bfu_{p,i}^{\blacksquare}\end{pmatrix}.
\] 

For every $i=1, \ldots, g$, consider a $\bfu_{p,i}$-action on $\T_0$ defined as follows: for every $(\bfgamma, \bfupsilon)\in \T_0$, we put
\[
\bfu_{p,i} . (\bfgamma, \bfupsilon) = (\bfu_{p,i}^{\square} \bfgamma_0 \bfu_{p,i}^{\square, -1}, \bfu_{p,i}^{\blacksquare} \bfupsilon_0 \bfu_{p,i}^{\square, -1})\bfbeta
\]
where we write $(\bfgamma, \bfupsilon) = (\bfgamma_0, \bfupsilon_0)\bfbeta$ with $\bfgamma_0\in U_{\GL_g, 1}^{\opp}$ and $\bfbeta\in B_{\GL_g, 0}$. This then induces a $\bfu_{p,i}$-action on $D_{\kappa_{\calU}}^{r}(\T_0, R_{\calU})$.

Similar to \S \ref{subsection: Hecke operators on the overconvergent automorphic forms}, for every $i=1, \ldots, g$, choose a double coset decomposition \[
    \Iw_{\GSp_{2g}}^+ \bfu_{p,i} \Iw_{\GSp_{2g}}^+ = \bigsqcup_{j}\bfdelta_{ij}\bfu_{p,i}\Iw_{\GSp_{2g}}^+.
\]
with $\bfdelta_{ij}\in \Iw_{\GSp_{2g}}^+$. The natural left action of $\GSp_{2g}(\Q_p)$ on $X_{\Iw^+}(\C)$ together with the actions of $\Iw_{\GSp_{2g}}^+$ and $\bfu_{p,i}$ on $D_{\kappa_{\calU}}^r(\T_0, R_{\calU})$ induce the Hecke operator \begin{equation}\label{eq: Hecek on coh. at p}
    U_{p,i}: H^t(X_{\Iw^+}(\C), D_{\kappa_{\calU}}^r(\T_0, R_{\calU})) \rightarrow H^t(X_{\Iw^+}(\C), D_{\kappa_{\calU}}^r(\T_0, R_{\calU})), \quad [\mu]\mapsto \sum_{j} \bfdelta_{ij} . (\bfu_{p,i} . [\mu]).
\end{equation}
%Here, again, $\nu_i = -(g-i)(g+1)$ for $i = 1, ..., g-1$ and $\nu_g = \frac{-g(g+1)}{2}$.

\begin{Definition}\label{Definition: Hecke operators on overconvergent cohomology}
\begin{enumerate}
    \item[(i)] The Hecke operators $T_{\bfgamma}$ (for $\bfgamma\in \GSp_{2g}(\Q_{\ell})\cap M_{2g}(\Z_{\ell})$ with $\ell\nmid Np$) and $U_{p,i}$ (for $i=1, ..., g$) acting on the overconvergent cohomology groups $H_{\et}^t(\calX_{\Iw^+}, \scrD_{\kappa_{\calU}}^r)$ are defined to be the operators $T_{\bfgamma}$ and $U_{p,i}$ acting on $H^t(X_{\Iw^+}(\C), D_{\kappa_{\calU}}^r(\T_0, R_{\calU}))$ via the isomorphism in Proposition \ref{Proposition: Comparison theorem of cohomologies}.
    \item[(ii)] We define the operator $U_{p}$ as the composition $U_p = \prod_{i=1}^gU_{p,i}$.
\end{enumerate}
\end{Definition}
\section{The overconvergent Eichler--Shimura morphism}\label{section:EichlerShimura}
In this section, we establish the second main result of this paper; {\it i.e.,} the construction of the overconvergent Eichler--Shimura morphism for Siegel modular forms. Our approach is similar to the one in \cite{CHJ-2017}. However, one major difference between our situation and the one in \emph{loc. cit.} is that the Siegel modular variety is non-compact. As a remedy, we apply the theory of (pro-)Kummer \'etale topology on log adic spaces developed in \cite{Diao} and \cite{Diao-Lan-Liu-Zhu} to handle the boundaries of the compactifications. (See \S \ref{subsection: review of Kummer etale and pro-Kummer etale sites} for a brief review.)

\subsection{The Kummer \'{e}tale and the pro-Kummer \'{e}tale cohomology groups}\label{subsection: pro-Kummer etale cohomology groups}

Recall from \S \ref{subsection: Siegel modular varieties} that $\overline{\calX}_{\Gamma(p^n)}$, $\overline{\calX}_{\Iw^+}$, and $\overline{\calX}$ are endowed with the divisorial log structures defined by the boundary divisors. The corresponding sheaves of monoids are denoted by $\scrM_n$, $\scrM_{\Iw^+}$, and $\scrM$, respectively. In what follows, we shall construct a sheaf $\sheafOD_{\kappa_{\calU}}^r$ on the pro-Kummer \'{e}tale site $\overline{\calX}_{\Iw^+, \proket}$ which computes the overconvergent cohomology groups introduced in \S \ref{subsection: overconvergent cohomology groups}. 

Consider the natural morphism of sites \[\jmath_{\ket}: \calX_{\Iw^+, \et}\rightarrow \overline{\calX}_{\Iw^+, \ket}.\]
Recall that, for every small weight $(R_{\calU}, \kappa_{\calU})$ and any integer $r\geq 1+r_{\calU}$, there is an inverse system of \'{e}tale locally constant sheaves $(\scrD_{\kappa_{\calU}, j}^{r, \circ})_{j\in \Z_{\geq 1}}$ on $\calX_{\Iw^+, \et}$. Applying \cite[Corollary 4.6.7]{Diao}, we obtain an isomorphism $$\varprojlim_{j}H_{\et}^t(\calX_{\Iw^+}, \scrD_{\kappa_{\calU}, j}^{r, \circ})\cong \varprojlim_{j} H_{\ket}^{t}(\overline{\calX}_{\Iw^+}, \jmath_{\ket, *}\scrD_{\kappa_{\calU}, j}^{r, \circ})$$ for every $t\in \Z_{\geq 0}$. Write \[H_{\ket}^t(\overline{\calX}_{\Iw^+}, \scrD_{\kappa_{\calU}}^r):=\varprojlim_{j} H_{\ket}^{t}(\overline{\calX}_{\Iw^+}, \jmath_{\ket, *}\scrD_{\kappa_{\calU}, j}^{r, \circ})[1/p].\] By Proposition \ref{Proposition: Comparison theorem of cohomologies}, we arrive at isomorphisms $$H_{\ket}^t(\overline{\calX}_{\Iw^+}, \scrD_{\kappa_{\calU}}^{r})\cong H_{\et}^t(\calX_{\Iw^+}, \scrD_{\kappa_{\calU}}^{r})\cong H^t(X_{\Iw^+}(\C), D_{\kappa_{\calU}}^{r}(\T_0, R_{\calU})).$$ To simplify the notation, we introduce the following abbreviations. 

\begin{Definition}\label{Definition: simple notations for overconvergent cohomologies}
Let $(R_{\calU}, \kappa_{\calU})$ be a small weight and let $r\geq 1+r_{\calU}$. We set \begin{align*}
    \OC_{\kappa_{\calU}}^{r, \circ} & := \varprojlim_{j} H_{\ket}^{n_0}(\overline{\calX}_{\Iw^+}, \jmath_{\ket, *}\scrD_{\kappa_{\calU}, j}^{r, \circ}),\\
    \OC_{\kappa_{\calU}}^{r} & := \OC_{\kappa_{\calU}}^{r, \circ}[\frac{1}{p}]=H_{\ket}^{n_0}(\overline{\calX}_{\Iw^+}, \scrD_{\kappa_{\calU}}^r),\\
    \OC_{\kappa_{\calU}, \calO_{\C_p}}^{r, \circ} & := \varprojlim_{j}\left(H_{\ket}^{n_0}(\overline{\calX}_{\Iw^+}, \jmath_{\ket, *}\scrD_{\kappa_{\calU}, j}^{r, \circ})\otimes_{\Z_p}\calO_{\C_p}\right),\\
    \OC_{\kappa_{\calU}, \C_p}^{r} & := \OC_{\kappa_{\calU}, \calO_{\C_p}}^{r, \circ}[\frac{1}{p}].
\end{align*}
where $n_0=\dim_{\C_p}\calX_{\Iw^+}$.
\end{Definition}

Let \[\nu: \overline{\calX}_{\Iw^+, \proket}\rightarrow \overline{\calX}_{\Iw^+, \ket}\] be the natural projection of sites. Consider the sheaf $\sheafOD_{\kappa_{\calU}}^r$ on the pro-Kummer \'{e}tale site $\overline{\calX}_{\Iw^+, \proket}$ defined by $$\sheafOD_{\kappa_{\calU}}^r:=\left(\varprojlim_{j}\left(\nu^{-1}\jmath_{\ket, *}\scrD_{\kappa_{\calU}, j}^{r, \circ}\otimes_{\Z_p}\scrO_{\overline{\calX}_{\Iw^+, \proket}}^+\right)\right)[\frac{1}{p}].$$ 

\begin{Proposition}\label{Proposition: overconvergent cohomology computed by the pro-Kummer etale cohomology}
There is a $G_{\Q_p}$-equivariant isomorphism $$\OC_{\kappa_{\calU}, \C_p}^r\cong H_{\proket}^{n_0}(\overline{\calX}_{\Iw^+}, \sheafOD_{\kappa_{\calU}}^r).$$
\end{Proposition}
\begin{proof}
By \cite[Theorem 6.2.1 \& Corollary 6.3.4]{Diao}, there is an almost isomorphism $$\left(H^{n_0}_{\ket}(\overline{\calX}_{\Iw^+}, \jmath_{\ket, *}\scrD_{\kappa_{\calU}, j}^{r, \circ})\otimes_{\Z_p}\calO_{\C_p}\right)^a\cong H_{\proket}^{n_0}(\overline{\calX}_{\Iw^+}, \nu^{-1}\jmath_{\ket, *}\scrD_{\kappa_{\calU}, j}^{r, \circ}\otimes_{\Z_p}\scrO_{\overline{\calX}_{\Iw^+, \proket}}^+)^a.$$ It remains to establish an almost isomorphism $$\varprojlim_{j} H_{\proket}^{n_0}\left(\overline{\calX}_{\Iw^+}, \nu^{-1}\jmath_{\ket, *}\scrD_{\kappa_{\calU}, j}^{r, \circ}\otimes_{\Z_p}\scrO_{\overline{\calX}_{\Iw^+, \proket}}^+\right)^a\cong H_{\proket}^{n_0}\left(\overline{\calX}_{\Iw^+}, \varprojlim_j\left(\nu^{-1}\jmath_{\ket, *}\scrD_{\kappa_{\calU}, j}^{r, \circ}\otimes_{\Z_p}\scrO_{\overline{\calX}_{\Iw^+, \proket}}^+\right)\right)^a.$$ Indeed, observe that the higher inverse limit $R^i \varprojlim_j\left(\nu^{-1}\jmath_{\ket, *}\scrD_{\kappa_{\calU}, j}^{r, \circ}\otimes_{\Z_p}\scrO_{\overline{\calX}_{\Iw^+, \proket}}^+\right)$ almost vanishes for $i\geq 1$ by an almost version of \cite[Lemma 3.18]{Scholze_2013} and \cite[Proposition 6.1.11]{Diao}. This then allows us to commute the inverse limit with taking cohomology, hence the result.
\end{proof}

We now discuss the Hecke operators acting on $H_{\proket}^{n_0}(\overline{\calX}_{\Iw^+}, \sheafOD_{\kappa_{\calU}}^{r})$. For Hecke operators at $p$, we define $U_{p, i}$ as in \eqref{eq: Hecek on coh. at p}. For Hecke operators away from $Np$, we use correspondences. More precisely, for any prime number $\ell \nmid Np$ and any $\bfgamma\in \GSp_{2g}(\Q_{\ell}) \cap M_{2g}(\Z_{\ell})$, consider the correspondence \[
    \begin{tikzcd}
        & \calX_{\bfgamma, \Iw^+}\arrow[dl, "\pr_1"']\arrow[rd, "\pr_2"]\\
        \calX_{\Iw^+} && \calX_{\Iw^+}
    \end{tikzcd},
\] studied in \S \ref{subsection: Hecke operators on the overconvergent automorphic forms}. Similar to the construction in \S \ref{subsection: Hecke operators on the overconvergent automorphic forms}, one obtains an isomorphism \[
    \varphi_{\bfgamma} : \pr_2^*\sheafOD_{\kappa_{\calU}}^r|_{\calX_{\Iw^+}}  \xrightarrow{\sim} \pr_1^*\sheafOD_{\kappa_{\calU}}^r|_{\calX_{\Iw^+}} .
\] Consider the composition \[
    \begin{tikzcd}
    T_{\bfgamma}: &  H^{n_0}_{\proet}(\calX_{\Iw^+}, \sheafOD_{\kappa_{\calU}}^r|_{\calX_{\Iw^+}} )\arrow[r, "\pr_2^*"] & H^{n_0}_{\proet}(\calX_{\bfgamma, \Iw^+}, \pr_2^*\sheafOD_{\kappa_{\calU}}^r|_{\calX_{\Iw^+}} )\arrow[ld, out=-10, in=170, "\varphi_{\bfgamma}"']\\
    &  H^{n_0}_{\proet}(\calX_{\Iw^+}, \pr_1^*\sheafOD_{\kappa_{\calU}}^r|_{\calX_{\Iw^+}} )\arrow[r, "\Tr_{\pr_1}"] & H_{\proet}^{n_0}(\calX_{\Iw^+}, \sheafOD_{\kappa_{\calU}}^r|_{\calX_{\Iw^+}})
    \end{tikzcd}.
\] 
However, since $H_{\et}^{n_0}(\calX_{\Iw^+}, \scrD_{\kappa_{\calU}, j}^{r, \circ})\cong H_{\ket}^{n_0}(\overline{\calX}_{\Iw^+}, \jmath_{\ket, *}\scrD_{\kappa_{\calU}, j}^{r, \circ})$ for every $j$, we have an identification
$$H^{n_0}_{\proet}(\calX_{\Iw^+}, \sheafOD_{\kappa_{\calU}}^r|_{\calX_{\Iw^+}} )\cong H_{\proket}^{n_0}(\overline{\calX}_{\Iw^+}, \sheafOD_{\kappa_{\calU}}^r)$$ and hence an operator $T_{\bfgamma}$ on $H_{\proket}^{n_0}(\overline{\calX}_{\Iw^+}, \sheafOD_{\kappa_{\calU}}^r)$. The isomorphism in Proposition \ref{Proposition: overconvergent cohomology computed by the pro-Kummer etale cohomology} is then Hecke-equivariant.

\subsection{The overconvergent Eichler--Shimura morphism}\label{subsection: OES}
In this subsection, we construct the overconvergent Eichler--Shimura morphism by first constructing a morphism between sheaves on the pro-Kummer \'{e}tale site $\overline{\calX}_{\Iw^+, w, \proket}$.

Let $(R_{\calU}, \kappa_{\calU})$ be a small weight and let $w \geq r\geq 1+r_{\calU}$. Recall that we have defined a sheaf $\sheafOD_{\kappa_{\calU}}^{r}$ on the pro-Kummer \'{e}tale site $\overline{\calX}_{\Iw^+, \proket}$ in \S \ref{subsection: pro-Kummer etale cohomology groups}. The following lemma is an analogue of \cite[Lemma 4.5]{CHJ-2017}.

\begin{Lemma}\label{Lemma: Explicit description of the sheaf OD}
Let $\calV=\varprojlim_{n}\calV_n\rightarrow \overline{\calX}_{\Iw^+}$ be a pro-Kummer \'{e}tale presentation of a log affinoid perfectoid object in $\overline{\calX}_{\Iw^+, \proket}$. Let $\calV_{\infty}:=\calV\times_{\overline{\calX}_{\Iw^+}}\overline{\calX}_{\Gamma(p^{\infty})}$. (Here we have abused the notation and identify $\overline{\calX}_{\Gamma(p^{\infty})}$ with the object $\varprojlim_n \overline{\calX}_{\Gamma(p^n)}$ in $\overline{\calX}_{\Iw^+, \proket}$.) Then there is a natural isomorphism $$\sheafOD_{\kappa_{\calU}}^r(\calV)\cong\left(D_{\kappa_{\calU}}^{r, \circ}(\T_0, R_{\calU})\widehat{\otimes}_{\Z_p}\widehat{\scrO}_{\overline{\calX}_{\Iw^+, \proket}}(\calV_{\infty})\right)^{\Iw^+_{\GSp_{2g}}}.$$
\end{Lemma}
\begin{proof}
Recall that $\scrD_{\kappa_{\calU}, j}^{r, \circ}$ is the locally constant sheaf on $\calX_{\Iw^+, \et}$ induced by $$\pi_1^{\et}(\calX_{\Iw^+})\rightarrow \Iw^+_{\GSp_{2g}}\rightarrow \Aut\left(D_{\kappa_{\calU}, j}^{r, \circ}(\T_0, R_{\calU})\right).$$ 
Since $\overline{\calX}_{\Gamma(p^{\infty})}$ is a profinite Galois cover of $\overline{\calX}_{\Iw^+}$ with Galois group $\Iw^+_{\GSp_{2g}}$, one sees that $\nu^{-1}\jmath_{\ket, *}\scrD_{\kappa_{\calU}, j}^{r, \circ}$ becomes the constant local system associated with $D_{\kappa_{\calU}, j}^{r, \circ}(\T_0, R_{\calU})$ after restricting to the localised site $\overline{\calX}_{\Iw^+, \proket}/\overline{\calX}_{\Gamma(p^{\infty})}$. 

Applying \cite[Theorem 5.4.3]{Diao}, we obtain an almost isomorphism $$\left(D_{\kappa_{\calU}, j}^{r, \circ}(\T_0, R_{\calU})\otimes_{\Z_p}\widehat{\scrO}_{\overline{\calX}_{\Iw^+, \proket}}^{+}(\calV_{\infty})\right)^a\cong \left(\left(\nu^{-1}\jmath_{\ket,*}\scrD_{\kappa_{\calU}, j}^{r, \circ}\otimes_{\Z_p}\scrO_{\overline{\calX}_{\Iw^+, \proket}}^+\right)(\calV_{\infty})\right)^a.$$ By taking $\Iw^+_{\GSp_{2g}}$-invariants, we obtain almost isomorphisms \begin{align*}
    \left(\left(D_{\kappa_{\calU}, j}^{r, \circ}(\T_0, R_{\calU})\otimes_{\Z_p}\widehat{\scrO}_{\overline{\calX}_{\Iw^+, \proket}}^{+}(\calV_{\infty})\right)^{\Iw^+_{\GSp_{2g}}}\right)^a & \cong \left(\left(\left(\nu^{-1}\jmath_{\ket,*}\scrD_{\kappa_{\calU}, j}^{r, \circ}\otimes_{\Z_p}\scrO_{\overline{\calX}_{\Iw^+, \proket}}^+\right)(\calV_{\infty})\right)^{\Iw^+_{\GSp_{2g}}}\right)^a\\
    & = \left(\left(\nu^{-1}\jmath_{\ket,*}\scrD_{\kappa_{\calU}, j}^{r, \circ}\otimes_{\Z_p}\scrO_{\overline{\calX}_{\Iw^+, \proket}}^+\right)(\calV)\right)^{a}.
\end{align*} Finally, taking inverse limits over $j$ and inverting $p$, we conclude that \begin{align*}
    \sheafOD_{\kappa_{\calU}}^r(\calV) & = \left(\varinjlim_{j}\left(\nu^{-1}\jmath_{\ket, *}\scrD_{\kappa_{\calU}, j}^{r, \circ}\otimes_{\Z_p}\scrO_{\overline{\calX}_{\Iw^+, \proket}}^+\right)(\calV)\right)[\frac{1}{p}]\\
    & \cong \left(D_{\kappa_{\calU}}^{r, \circ}(\T_0, R_{\calU})\widehat{\otimes}_{\Z_p}\widehat{\scrO}_{\overline{\calX}_{\Iw^+, \proket}}(\calV_{\infty})\right)^{\Iw^+_{\GSp_{2g}}}.
\end{align*}
\end{proof}

To deal with the overconvergent automorphic sheaves, we recall the Kummer \'etale sheaves $\underline{\omega}^{\kappa_{\calU},+}_{w,\ket}$ and $\underline{\omega}^{\kappa_{\calU}}_{w,\ket}$ associated with $\underline{\omega}^{\kappa_{\calU},+}_{w}$ and $\underline{\omega}^{\kappa_{\calU}}_{w}$ defined by the end of \S \ref{subsection: admissibility}. Then we consider the $p$-adically completed pullback of them to the pro-Kummer \'etale site; namely,
\[
\widehat{\underline{\omega}}_{w}^{\kappa_{\calU},+}:=\varprojlim_m\left(\underline{\omega}^{\kappa_{\calU},+}_{w,\ket}\bigotimes_{\scrO^+_{\overline{\calX}_{\Iw^+, w,\ket}}}\scrO^+_{\overline{\calX}_{\Iw^+, w, \proket}}/p^m\right)
\]
and \[\widehat{\underline{\omega}}_{w}^{\kappa_{\calU}}:=\widehat{\underline{\omega}}_{w}^{\kappa_{\calU},+}[\frac{1}{p}].\]

\begin{Lemma}\label{Lemma: Leray spectral sequence for the automorphic sheaf}
There is a canonical Hecke- and $G_{\Q_p}$-equivariant morphism $$H_{\proket}^{n_0}(\overline{\calX}_{\Iw^+, w}, \widehat{\underline{\omega}}_w^{\kappa_{\calU}}) \rightarrow H^0(\overline{\calX}_{\Iw^+, w}, \underline{\omega}_w^{\kappa_{\calU}+g+1})(-n_0).$$
\end{Lemma}
\begin{proof}
By the discussion at the end of \S \ref{subsection: admissibility}, we have seen that $\underline{\omega}^{\kappa_{\calU}}_{w,\ket}$ can be identified with the sheaf of $\Iw^+_{\GSp_{2g}}/\Gamma(p^n)$-invariants of an admissible Kummer \'etale Banach sheaf of $\scrO_{\overline{\calX}_{\Iw^+, w,\ket}}\widehat{\otimes}R_{\calU}$-modules. Corollary \ref{Corollary: generalised projection formula with invariants} then yields a canonical isomorphism
$$\underline{\omega}_{w, \ket}^{\kappa_{\calU}}\otimes_{\scrO_{\overline{\calX}_{\Iw^+, w,\ket}}}R^i\nu_*\widehat{\scrO}_{\overline{\calX}_{\Iw^+, w, \proket}}\xrightarrow{\sim}R^i\nu_*\widehat{\underline{\omega}}_{w}^{\kappa_{\calU}}$$ for every $i\in \Z_{\geq 0}$. 
On the other hand, by Proposition \ref{Proposition: compatibility with completed tensor}, we have a canonical isomorphism $$R^i\nu_*\widehat{\scrO}_{\overline{\calX}_{\Iw^+, w, \proket}}\cong \Omega^{\log, i}_{\overline{\calX}_{\Iw^+, w, \ket}}(-i).$$ Combining the two isomorphisms, we obtain
$$
R^i\nu_*\widehat{\underline{\omega}}_{w}^{\kappa_{\calU}}\cong \underline{\omega}_{w, \ket}^{\kappa_{\calU}}\otimes_{\scrO_{\overline{\calX}_{\Iw^+, w,\ket}}} \Omega^{\log, i}_{\overline{\calX}_{\Iw^+, w, \ket}}(-i).
$$

Moreover, there is a Leray spectral sequence $$E_2^{j, i}=H_{\ket}^{j}(\overline{\calX}_{\Iw^+, w}, R^i\nu_*\widehat{\underline{\omega}}_w^{\kappa_{\calU}})\Rightarrow H_{\proket}^{j+i}(\overline{\calX}_{\Iw^+, w}, \widehat{\underline{\omega}}_w^{\kappa_{\calU}}).$$ The edge map yields a Galois-equivariant morphism 
$$H^{n_0}_{\proket}(\overline{\calX}_{\Iw^+, w}, \widehat{\underline{\omega}}_w^{\kappa_{\calU}})\rightarrow H^0_{\ket}(\overline{\calX}_{\Iw^+, w}, R^{n_0}\nu_*\widehat{\underline{\omega}}_w^{\kappa_{\calU}})\cong H^{0}_{\ket}(\overline{\calX}_{\Iw^+, w}, \underline{\omega}_{w, \ket}^{\kappa_{\calU}}\otimes_{\scrO_{\overline{\calX}_{\Iw^+, w,\ket}}} \Omega^{\log, n_0}_{\overline{\calX}_{\Iw^+, w, \ket}})(-n_0).$$ 

Finally, let $\pi_{\Iw^+}:\calG_{\Iw^+, w}^{\univ}\rightarrow \overline{\calX}_{\Iw^+, w}$ denote the universal semiabelian variety over $\overline{\calX}_{\Iw^+, w}$ with identity section $e$ and let $$\underline{\omega}_{\Iw^+, w}:=e^*\Omega^1_{\calG_{\Iw^+, w}^{\univ}/\overline{\calX}_{\Iw^+ ,w}}.$$ 
Note that $\underline{\omega}_{\Iw^+, w}$ agrees with $\underline{\omega}^k_{\Iw^+}|_{\overline{\calX}_{\Iw^+ ,w}}$ studied in \S \ref{subsection: classical forms} for $k=(1,0,\ldots, 0)$. The Kodaira--Spencer isomorphism \cite[Theorem 1.41 (4)]{LanKS} yields an isomorphism $$\Sym^2\underline{\omega}_{\Iw^+, w}\cong \Omega_{\overline{\calX}_{\Iw^+, w}}^{\log, 1}.$$ 
Hence, 
$$\Omega_{\overline{\calX}_{\Iw^+, w}}^{\log, n_0}\cong \bigwedge^{n_0}\left(\Sym^2\underline{\omega}_{\Iw^+, w}\right)=\underline{\omega}_{\Iw^+, w}^{g+1}\subset \underline{\omega}_w^{g+1}$$
where the last inclusion follows from Lemma \ref{Lemma: injection of classical forms}. We obtain an injection $$
H^{0}_{\ket}(\overline{\calX}_{\Iw^+, w}, \underline{\omega}_{w, \ket}^{\kappa_{\calU}}\otimes_{\scrO_{\overline{\calX}_{\Iw^+, w,\ket}}} \Omega^{\log, n_0}_{\overline{\calX}_{\Iw^+, w, \ket}})(-n_0) \hookrightarrow H_{\ket}^0(\overline{\calX}_{\Iw^+, w}, \underline{\omega}_{w,\ket}^{\kappa_{\calU}+g+1})(-n_0)=H^0(\overline{\calX}_{\Iw^+, w}, \underline{\omega}_w^{\kappa_{\calU}+g+1})(-n_0).$$
Note that, due to the normalisation of the Hecke operators, the Kodaira--Spencer isomorphism is Hecke-equivariant (see \cite[pp. 258]{Faltings-Chai}). 
\end{proof}

For any matrix $\bfsigma\in M_g(\calO_{\C_p})$ and $\mu\in D_{\kappa_{\calU}}^{r}(\T_0, R_{\calU})$, we define a function $f_{\mu, \bfsigma}\in C_{\kappa_{\calU}}^{w-\an}(\Iw_{\GL_g}, \C_p\widehat{\otimes}R_{\calU})$ as follows. For any $\bfgamma'\in \Iw_{\GL_g}$, we define
$$f_{\mu, \bfsigma}(\bfgamma'):= \int_{(\bfgamma, \bfupsilon)\in \T_0}e_{\kappa_{\calU}}^{\hst}(\trans\bfgamma'(\bfgamma+\bfsigma\bfupsilon))\quad d\mu,$$ where $e_{\kappa_{\calU}}^{\hst}$ sends a matrix $X=(X_{ij})_{1\leq i,j\leq g}$ in $\Iw^{(w)}_{\GL_g}$ to 
$$    e_{\kappa_{\calU}}^{\hst}(X)=\frac{\kappa_{\calU, 1}(X_{11})}{\kappa_{\calU, 2}(X_{11})}\times \frac{\kappa_{\calU, 2}(\det((X_{ij})_{1\leq i,j\leq 2}))}{\kappa_{\calU, 3}(\det((X_{ij})_{1\leq i,j\leq 2}))}\times \cdots \times \kappa_{\calU, g}(\det(X)).
$$

The following lemma justifies this definition.
\begin{Lemma}
\begin{enumerate}
    \item[(i)] For every $\bfsigma\in M_g(\calO_{\C_p})$ and $\bfgamma'\in \Iw_{\GL_g}$, the assignment $$(\bfgamma, \bfupsilon)\mapsto e_{\kappa_{\calU}}^{\hst}(\trans\bfgamma'(\bfgamma+\bfsigma\bfupsilon))$$ defines an element in $A^r_{\kappa_{\calU}}(\T_0, R_{\calU})$.
    \item[(ii)] For every $\bfgamma'\in \Iw_{\GL_g}$ and $\bfbeta\in B_{\GL_g, 0}$, we have $$f_{\mu, \bfsigma}(\bfgamma'\bfbeta)=\kappa_{\calU}(\bfbeta)f_{\mu, \bfsigma}(\bfgamma').$$
\end{enumerate} 
\end{Lemma}

\begin{proof}
This is straightforward.
\end{proof}

\begin{Remark}
\normalfont The function $e^{\hst}_{\kappa_{\calU}}$ is an analogue of the highest weight vector in an algebraic representation of $\GL_g$. Moreover, $e^{\hst}_{\kappa_{\calU}}$ has the following alternative interpretation: for every $\bfgamma\in \Iw^{(w)}_{\GL_g}$, if we write $\bfgamma=\bfnu\bftau\bfnu'$ with $\bfnu\in U^{\opp,(w)}_{\GL_g, 1}$, $\bftau\in T^{(w)}_{\GL_g, 0}$, and $\bfnu'\in U^{(w)}_{\GL_g, 0}$, then we have $e^{\hst}_{\kappa_{\calU}}(\bfgamma)=\kappa_{\calU}(\bftau)$.
\end{Remark}

We are ready to construct the desired morphism $\eta_{\kappa_{\calU}}:\sheafOD_{\kappa_{\calU}}^r\rightarrow\widehat{\underline{\omega}}_{w}^{\kappa_{\calU}}$ between sheaves on the pro-Kummer \'{e}tale site $\overline{\calX}_{\Iw^+, w, \proket}$. Indeed, it suffices to construct a map $\sheafOD_{\kappa_{\calU}}^r(\calV)\rightarrow \widehat{\underline{\omega}}_{w}^{\kappa_{\calU}}(\calV)$ for every log affinoid perfectoid object $\calV$ in $\overline{\calX}_{\Iw^+, \proket}$. By Lemma \ref{Lemma: Explicit description of the sheaf OD}, we have
$$\sheafOD_{\kappa_{\calU}}^r(\calV)\simeq\left(D_{\kappa_{\calU}}^{r, \circ}(\T_0, R_{\calU})\widehat{\otimes}_{\Z_p}\widehat{\scrO}_{\overline{\calX}_{\Iw^+, \proket}}(\calV_{\infty})\right)^{\Iw^+_{\GSp_{2g}}}$$
where $\calV_{\infty}:=\calV\times_{\overline{\calX}_{\Iw^+}}\overline{\calX}_{\Gamma(p^{\infty})}$. 

On the other hand, by definition, we know that $\widehat{\underline{\omega}}_{w}^{\kappa_{\calU}}(\calV)$ consists of $f\in C^{w-\an}_{\kappa_{\calU}}(\Iw_{\GL_g}, \widehat{\scrO}_{\overline{\calX}_{\Iw^+, \proket}}(\calV_{\infty})\widehat{\otimes}R_{\calU})$ satisfying $\bfalpha^*f=\rho_{\kappa_{\calU}}(\bfalpha_a+\frakz\bfalpha_c)^{-1}f$, for all $\bfalpha=\begin{pmatrix}\bfalpha_a & \bfalpha_b\\ \bfalpha_c & \bfalpha_d\end{pmatrix}\in \Iw_{\GSp_{2g}}^+$. This is equivalent to saying that $\widehat{\underline{\omega}}_{w}^{\kappa_{\calU}}(\calV)$ consists of $\Iw_{\GSp_{2g}}^+$-invariant elements $f\in C^{w-\an}_{\kappa_{\calU}}(\Iw_{\GL_g}, \widehat{\scrO}_{\overline{\calX}_{\Iw^+, \proket}}(\calV_{\infty})\widehat{\otimes}R_{\calU})$ with respect to the twisted $\Iw_{\GSp_{2g}}^+$-action 
$$\bfalpha.f:=\rho_{\kappa_{\calU}}(\bfalpha_a+\frakz\bfalpha_c)(\bfalpha^*f).$$

Consider the map
$$D_{\kappa_{\calU}}^{r, \circ}(\T_0, R_{\calU})\widehat{\otimes}_{\Z_p}\widehat{\scrO}_{\overline{\calX}_{\Iw^+, \proket}}(\calV_{\infty})\rightarrow C^{w-\an}_{\kappa_{\calU}}(\Iw_{\GL_g}, \widehat{\scrO}_{\overline{\calX}_{\Iw^+, \proket}}(\calV_{\infty})\widehat{\otimes}R_{\calU}), \quad \mu\otimes \delta\mapsto \delta f_{\mu, \frakz}.$$
We claim that this map is $\Iw_{\GSp_{2g}}^+$-equivariant, and hence taking the $\Iw_{\GSp_{2g}}^+$-invariants yields the desired map $\sheafOD_{\kappa_{\calU}}^r(\calV)\rightarrow\widehat{\underline{\omega}}_{w}^{\kappa_{\calU}}(\calV)$. Indeed, for any $\bfalpha=\begin{pmatrix}\bfalpha_a & \bfalpha_b\\ \bfalpha_c & \bfalpha_d\end{pmatrix}\in \Iw_{\GSp_{2g}}^+$ and any $\bfgamma'\in \Iw_{\GL_g}$, we have 
\begin{align*}
(\bfalpha^* \delta) f_{\bfalpha\cdot\mu, \frakz}(\bfgamma') & =  (\bfalpha^* \delta)\left(\int_{\T_0} e_{\kappa_{\calU}}^{\hst}(\trans\bfgamma'(\bfgamma+\frakz\bfupsilon))\quad d\bfalpha\cdot \mu\right) \\
    & = (\bfalpha^* \delta)\left(\int_{\T_0}e_{\kappa_{\calU}}^{\hst}\left(\trans\bfgamma'\left((\bfalpha_a\bfgamma+\bfalpha_b\bfupsilon)+\frakz(\bfalpha_c\bfgamma+\bfalpha_d\bfupsilon)\right)\right)\quad d\mu\right)\\
    & = (\bfalpha^* \delta) \left(\int_{\T_0}e_{\kappa_{\calU}}^{\hst}\left(\trans\bfgamma'((\bfalpha_a+\frakz\bfalpha_c)\bfgamma+(\bfalpha_b+\frakz\bfalpha_d)\bfupsilon)\right)\quad d\mu\right)\\
    & = (\bfalpha^* \delta)\left(\int_{\T_0}e_{\kappa_{\calU}}^{\hst}\left(\trans\bfgamma'(\bfalpha_a+\frakz\bfalpha_c)(\bfgamma+(\bfalpha_a+\frakz\bfalpha_c)^{-1}(\bfalpha_b+\frakz\bfalpha_d)\bfupsilon)\right)\quad d\mu\right) \\
    & = (\bfalpha^* \delta)\left(\int_{\T_0}e_{\kappa_{\calU}}^{\hst}\left(\trans(\trans(\bfalpha_a+\frakz\bfalpha_c)\bfgamma')(\bfgamma+(\frakz\cdot \bfalpha)\bfupsilon)\right)\quad d\mu\right) \\
    & = (\bfalpha^* \delta) \left(\rho_{\kappa_{\calU}}(\bfalpha_a+\frakz\bfalpha_c) \int_{\T_0}e_{\kappa_{\calU}}^{\hst}\left(\trans\bfgamma'(\bfgamma+(\frakz\cdot\bfalpha)\bfupsilon)\right)\quad d\mu\right)\\
    &= \bfalpha.(\delta f_{\mu, \frakz})(\bfgamma')
\end{align*}
as desired.

Putting everything together, we consider the composition
\begin{align*}
 \OC_{\kappa_{\calU}, \C_p}^{r} \cong & H_{\proket}^{n_0}(\overline{\calX}_{\Iw^+}, \sheafOD_{\kappa_{\calU}}^r)\xrightarrow[]{\Res}  H_{\proket}^{n_0}(\overline{\calX}_{\Iw^+,w}, \sheafOD_{\kappa_{\calU}}^r)\\
  \xrightarrow[]{\eta_{\kappa_{\calU}}} & H_{\proket}^{n_0}(\overline{\calX}_{\Iw^+, w}, \widehat{\underline{\omega}}_w^{\kappa_{\calU}}) \rightarrow H^0(\overline{\calX}_{\Iw^+, w}, \underline{\omega}_w^{\kappa_{\calU}+g+1})(-n_0)=M_{\Iw^+, w}^{\kappa_{\calU}+g+1}(-n_0)
 \end{align*}
where the second last morphism is given by Lemma \ref{Lemma: Leray spectral sequence for the automorphic sheaf}. We arrive at the \textbf{\textit{overconvergent Eichler--Shimura morphism}}
\begin{align*}
    \ES_{\kappa_{\calU}}:\OC_{\kappa_{\calU}, \C_p}^{r}\rightarrow M_{\Iw^+, w}^{\kappa_{\calU}+g+1}(-n_0).
\end{align*} 

\begin{Proposition}\label{Proposition: OES for sheaves on the pro-Kummer etale site}
The overconvergent Eichler--Shimura morphism\[
    \ES_{\kappa_{\calU}}: \OC_{\kappa_{\calU}, \C_p}^{r} \rightarrow M_{\Iw^+, w}^{\kappa_{\calU+g+1}}(-n_0)
\] is Hecke- and $G_{\Q_p}$-equivariant. 
\end{Proposition}
\begin{proof}
The Galois-equivariance follows immediately from Lemma \ref{Lemma: Leray spectral sequence for the automorphic sheaf}. For Hecke operators away from $Np$, notice that the operators $T_{\bfgamma}$'s on both sides are defined in the same way using correspondences. Hence, it is straightforward to verify the $T_{\bfgamma}$-equivariances. It remains to check the $U_{p,i}$-equivariance for all $i=1, ..., g$.

To this end, due to the $\Iw_{\GSp_{2g}}^+$-equivariance of $\eta_{\kappa_{\calU}}$, we only have to check the $\bfu_{p,i}$-equivariance. Indeed, for every $\bfgamma'=\bfgamma_0'\bfbeta_0'\in \Iw_{\GL_g}$ with $\bfgamma'_0\in U^{\opp}_{\GL_g,1}$ and $\bfbeta'_0\in B_{\GL_g,0}$, we have  
\begin{align*}
    (\bfu_{p,i}^*\delta) f_{\bfu_{p,i}\cdot\mu, \frakz}(\bfgamma') & = (\bfu_{p,i}^*\delta)\left(\kappa_{\calU}(\bfbeta_0')\int_{\T_0}e_{\kappa_{\calU}}^{\hst}(\trans\bfgamma'_0(\bfgamma+\frakz\bfupsilon))\quad d\bfu_{p,i}\cdot\mu\right) \\
    & = (\bfu_{p,i}^*\delta)\left(\kappa_{\calU}(\bfbeta_0')\int_{\T_0} \kappa_{\calU}(\bfbeta)e_{\kappa_{\calU}}^{\hst}\left(\trans\bfgamma_0'(\bfgamma_0+\frakz\bfupsilon_0)\right)\quad d\bfu_{p,i}\cdot\mu\right)  \\
    & = (\bfu_{p,i}^*\delta) \left(\kappa_{\calU}(\bfbeta_0')\int_{\T_0}\kappa_{\calU}(\bfbeta)e_{\kappa_{\calU}}^{\hst}\left(\trans\bfgamma_0'(\bfu_{p,i}^{\square}\bfgamma_0\bfu_{p,i}^{\square,-1}+\frakz\bfu_{p,i}^{\blacksquare}\bfupsilon_0\bfu_{p,i}^{\square, -1})\right)\quad d\mu\right) \\
    & = (\bfu_{p,i}^*\delta)\left(\kappa_{\calU}(\bfbeta_0')\int_{\T_0}\kappa_{\calU}(\bfbeta)e_{\kappa_{\calU}}^{\hst}\left(\trans\bfgamma_0'\bfu_{p,i}^{\square}(\bfgamma_0+\bfu_{p,i}^{\square, -1}\frakz\bfu_{p,i}^{\blacksquare}\bfupsilon_0)\bfu_{p,i}^{\square, -1}\right)\quad d\mu\right) \\
    & = (\bfu_{p,i}^*\delta)\left(\kappa_{\calU}(\bfbeta_0')\int_{\T_0}\kappa_{\calU}(\bfbeta) e_{\kappa_{\calU}}^{\hst}\left(\trans\bfgamma_0'\bfu_{p,i}^{\square}(\bfgamma_0+(\frakz\cdot\bfu_{p,i})\bfupsilon_0)\bfu_{p,i}^{\square, -1}\right)\quad d\mu\right) \\
    & = (\bfu_{p,i}^*\delta)\left(\kappa_{\calU}(\bfbeta_0')\int_{\T_0} \kappa_{\calU}(\bfbeta) e_{\kappa_{\calU}}^{\hst}\left(\bfu_{p,i}^{\square, -1}\trans\bfgamma'_0\bfu_{p, i}^{\square}(\bfgamma_0+(\frakz\cdot\bfu_{p,i})\bfupsilon_0)\right)\quad d\mu \right) \\
    & = (\bfu_{p,i}^*\delta)\left(\kappa_{\calU}(\bfbeta_0')\int_{\T_0}\kappa_{\calU}(\bfbeta) e_{\kappa_{\calU}}^{\hst}\left(\trans(\bfu_{p,i}^{\square}\bfgamma_0'\bfu_{p,i}^{\square, -1})(\bfgamma_0+(\frakz\cdot\bfu_{p,i})\bfupsilon_0)\right)\quad d\mu\right) \\
    & = \bfu_{p,i} . (\delta f_{\mu, \frakz}), 
\end{align*} 
where we have written $(\bfgamma, \bfupsilon)=(\bfgamma_0, \bfupsilon_0)\bfbeta$ for $(\bfgamma_0, \bfupsilon_0)\in \T_{00}$ and $\bfbeta\in B_{\GL_g,0}$.
The antepenultimate equation follows from the property of matrix determinants. 
\end{proof}

\begin{Remark}\label{Remark: OES for cuspforms}
\normalfont There is an analogue for compactly supported cohomology groups and overconvergent cuspforms. Let $r$, $w$, and $(R_{\calU}, \kappa_{\calU})$ be the same as before. On one hand, consider
$$\sheafOD_{\kappa_{\calU}}^{r, \cusp}:=\left(\varprojlim_{j}\left(\nu^{-1}\jmath_{\ket, !}\scrD_{\kappa_{\calU}, j}^{r, \circ}\otimes_{\Z_p}\scrO_{\overline{\calX}_{\Iw^+, \proket}}^+\right)\right)[\frac{1}{p}].$$
Since \[H_{\ket}^{n_0}(\overline{\calX}_{\Iw^+}, \jmath_{\ket, !}\scrD_{\kappa_{\calU}, j}^{r, \circ})=H^{n_0}_{\et, c}(\calX_{\Iw^+}, \scrD_{\kappa_{\calU}, j}^{r, \circ}),\] an analogue of Proposition \ref{Proposition: overconvergent cohomology computed by the pro-Kummer etale cohomology} implies that $\sheafOD_{\kappa_{\calU}}^{r, \cusp}$ computes \[
    \OC_{\kappa_{\calU}, \C_p}^{r, c}:=\left(\varprojlim_{j}H_{\et, c}^{n_0}(\calX_{\Iw^+}, \scrD_{\kappa_{\calU}, j}^{r, \circ})\otimes_{\Z_p}\calO_{\C_p}\right)[\frac{1}{p}].
\]

On the other hand, recall the sheaf $\underline{\omega}_{w, \cusp}^{\kappa_{\calU}}$ of $w$-overconvergent Siegel cuspforms of weight $\kappa_{\calU}$ and consider the $p$-adically completed pullback $\widehat{\underline{\omega}}_{w, \cusp}^{\kappa_{\calU}}$ to the pro-Kummer \'etale site. Repeating the construction above, we obtain a morphism $\eta_{\kappa_{\calU}}^{\cusp}: \sheafOD_{\kappa_{\calU}}^{r, \cusp}\rightarrow \widehat{\underline{\omega}}_{w, \cusp}^{\kappa_{\calU}}$ which induces a morphism
\[\ES_{\kappa_{\calU}}^{\cusp}:\OC_{\kappa_{\calU}, \C_p}^{r, c}\rightarrow H^0(\overline{\calX}_{\Iw^+, w}, \underline{\omega}_{w, \cusp}^{\kappa_{\calU}+g+1})(-n_0)
\]
rendering the following Galois- and Hecke-equivariant diagram commutative:
$$\begin{tikzcd}
\OC_{\kappa_{\calU}, \C_p}^{r}\arrow[r, "\ES_{\kappa_{\calU}}"] & H^0(\overline{\calX}_{\Iw^+, w}, \underline{\omega}_{w}^{\kappa_{\calU}+g+1})(-n_0)\\
\OC_{\kappa_{\calU}, \C_p}^{r, c}\arrow[r, "\ES_{\kappa_{\calU}}^{\cusp}"]\arrow[u] & H^0(\overline{\calX}_{\Iw^+, w}, \underline{\omega}_{w, \cusp}^{\kappa_{\calU}+g+1})(-n_0)\arrow[u, hook]
\end{tikzcd},$$ where the vertical arrow on the left is the natural map from the compactly supported cohomology group to the usual cohomology group. Let \begin{equation}\label{eq: definition of OCcusp}
    \OC_{\kappa_{\calU}, \C_p}^{r, \cusp}:=\image\left(\OC_{\kappa_{\calU}, \C_p}^{r, c}\rightarrow \OC_{\kappa_{\calU}, \C_p}^{r}\right).
\end{equation} We arrive at the \textbf{\textit{overconvergent Eichler--Shimura morphism for overconvergent Siegel cuspforms}} $$\ES_{\kappa_{\calU}}^{\cusp}: \OC_{\kappa_{\calU}, \C_p}^{r, \cusp}\rightarrow S_{\Iw^+, w}^{\kappa_{\calU}+g+1}(-n_0),$$ where
\[S_{\Iw^+, w}^{\kappa_{\calU}+g+1}:=H^0(\overline{\calX}_{\Iw^+, w}, \underline{\omega}_{w, \cusp}^{\kappa_{\calU}+g+1})\]
is the space of $w$-overconvergent Siegel cuspforms of strict Iwahori level and weight $\kappa_{\calU}+g+1$.

Lastly, we point out that, by construction, both $\ES_{\kappa_{\calU}}^{\cusp}$ and $\ES_{\kappa_{\calU}}$ are functorial in the small weights $(R_{\calU}, \kappa_{\calU})$.
\end{Remark}

\subsection{The image of the overconvergent Eichler--Shimura morphism at classical weights}\label{subsection:imageOES}

The aim of this last part of the section is to describe the image of the overconvergent Eichler--Shimura morphism at classical algebraic weights. Let $k=(k_1,\ldots, k_g) \in \Z_{\geq 0}^g$ be a dominant weight. Note that the character group of $T_{\GSp_{2g}}$ is isomorphic to $\Z^{g+1}$ through $$\Z^{g+1}\times T_{\GSp_{2g}}\rightarrow \bbG_m, \quad ((k_1, ..., k_g; k_0), \diag(\bftau_1, ..., \bftau_g, \bftau_0\bftau_g^{-1}, ..., \bftau_0\bftau_1^{-1}))\mapsto \prod_{i=0}^{g}\bftau_i^{k_i}.$$ We view $k=(k_1,..., k_g)\in \Z^g$ as a character of $T_{\GSp_{2g}}$ via $$\Z^g\hookrightarrow\Z^{g+1}, \quad (k_1, ..., k_g)\mapsto (k_1, ..., k_g;0).$$

We first introduce some algebraic representations. Consider the algebraic representation $$\V_{\GSp_{2g}, k}^{\alg}:=\left\{\phi: \GSp_{2g}\rightarrow \bbA^1: \begin{array}{l}
    \phi\text{ is a morphism of algebraic varieties over $\Q_p$ such that}  \\
    \phi(\bfgamma\bfbeta)=k(\bfbeta)\phi(\bfgamma),\,\,\forall \bfgamma\in \GSp_{2g} , \bfbeta\in B_{\GSp_{2g}}  
\end{array}\right\}$$ 
equipped with a left $\GSp_{2g}$-action given by
$$(\bfgamma\cdot\phi)(\bfgamma') = \phi(\trans\bfgamma\bfgamma')$$ for any $\bfgamma, \bfgamma'\in \GSp_{2g}$ and $\phi\in \V_{\GSp_{2g}, k}^{\alg}$. 

Let $\V_{\GSp_{2g}, k}^{\alg, \vee} $ be the linear dual of $\V_{\GSp_{2g}, k}^{\alg}$. We equip $\V_{\GSp_{2g}, k}^{\alg, \vee}$ with a left $\GSp_{2g}$-action induced from the following right $\GSp_{2g}$-action on $\V_{\GSp_{2g}, k}^{\alg}$:
$$(\phi\cdot\bfgamma)(\bfgamma') = \phi(\bfgamma\bfgamma')$$ for any $\bfgamma, \bfgamma'\in \GSp_{2g}$ and $\phi\in \V_{\GSp_{2g}, k}^{\alg}$. Observe that there is a natural morphism of $\GSp_{2g}$-representations 
$$
    \beta:\V_{\GSp_{2g}, k}^{\alg, \vee} \rightarrow \V_{\GSp_{2g}, k}^{\alg}, \quad \mu \mapsto \left(\bfgamma \mapsto \int_{\bfalpha\in \GSp_{2g}}e_{k}^{\hst}(\trans\bfgamma \bfalpha)\quad d\mu\right),
$$where $e_k^{\hst}$ is the ``highest weight vector'' inside $\V_{\GSp_{2g}, k}^{\alg}$; namely, for any matrix $X=(X_{ij})_{1\leq i,j\leq 2g}\in \GSp_{2g}$, we define
$$e^{\hst}_k(X)=X_{11}^{k_1-k_2}\times \det((X_{ij})_{1\leq i,j\leq 2})^{k_2-k_3}\times\cdots \times\det((X_{ij})_{1\leq i,j\leq g})^{k_g}.$$

Secondly, we consider the cohomology groups induced by these algebraic representations. Notice that the left $\GSp_{2g}$-actions on $\V_{\GSp_{2g}, k}^{\alg}$ and $\V_{\GSp_{2g}, k}^{\alg, \vee}$ induce \'etale $\Q_p$-local systems on $\calX_{\Iw^+}$ which we still denote by the same symbols. In particular, we can consider the cohomology group $H_{\et}^{n_0}(\calX_{\Iw^+}, \V_{\GSp_{2g}, k}^{\alg, \vee})$. Similar to \S \ref{subsection: OES}, we introduce the sheaves $\sheafOV_k$ and $\sheafOV_k^{\vee}$ on $\overline{\calX}_{\Iw^+, \proket}$ defined by
$$
 \sheafOV_{k} := \nu^{-1}j_{\ket, *}\V_{\GSp_{2g}, k}^{\alg}\otimes_{\Q_p} \widehat{\scrO}_{\overline{\calX}_{\Iw^+, \proket}}
$$
and
$$
 \sheafOV_{k}^{\vee} := \nu^{-1}j_{\ket, *}\V_{\GSp_{2g}, k}^{\alg, \vee}\otimes_{\Q_p} \widehat{\scrO}_{\overline{\calX}_{\Iw^+, \proket}}.
$$
By the same argument as in Proposition \ref{Proposition: overconvergent cohomology computed by the pro-Kummer etale cohomology}, we obtain a natural identification
$$
H_{\et}^{n_0}(\calX_{\Iw^+}, \V_{\GSp_{2g}, k}^{\alg, \vee})\otimes_{\Q_p}\C_p \cong H_{\proket}^{n_0}(\overline{\calX}_{\Iw^+}, \sheafOV_k^{\vee}).
$$
Moreover, if $\calV = \varprojlim_{n}\calV_n \rightarrow \overline{\calX}_{\Iw^+}$ is a pro-Kummer \'etale presentation of a log affinoid perfectoid object in $\overline{\calX}_{\Iw^+, \proket}$ and we let $\calV_{\infty} := \calV \times_{\overline{\calX}_{\Iw^+}}\overline{\calX}_{\Gamma(p^{\infty})}$, then, following the same argument as in the proof of Lemma \ref{Lemma: Explicit description of the sheaf OD}, we obtain identifications 
$$
\sheafOV_{k}(\calV) = \left(\V_{\GSp_{2g}, k}^{\alg}\otimes_{\Q_p} \widehat{\scrO}_{\overline{\calX}_{\Iw^+, \proket}}(\calV_{\infty})\right)^{\Iw_{\GSp_{2g}}^+}
$$
and
$$
\sheafOV_{k}^{\vee}(\calV) = \left(\V_{\GSp_{2g}, k}^{\alg, \vee}\otimes_{\Q_p} \widehat{\scrO}_{\overline{\calX}_{\Iw^+, \proket}}(\calV_{\infty})\right)^{\Iw_{\GSp_{2g}}^+}.
$$

\begin{Remark}\label{Remark: Hecke operators for alg. rep.}
\normalfont There are naturally defined Hecke operators on $H_{\et}^{n_0}(\calX_{\Iw^+}, \V_{\GSp_{2g}, k}^{\alg, \vee})$. More precisely, similar to Proposition \ref{Proposition: Comparison theorem of cohomologies}, we have \[
    H_{\et}^{n_0}(\calX_{\Iw^+}, \V_{\GSp_{2g}, k}^{\alg, \vee}) \cong H^{n_0}(X_{\Iw^+}(\C), \V_{\GSp_{2g}, k}^{\alg, \vee}),
\] where the right-hand side is the Betti cohomology of $X_{\Iw^+}(\C)$ with coefficients in $\V_{\GSp_{2g}, k}^{\alg, \vee}$. Hence, it suffices to define the Hecke operators acting on $H^{n_0}(X_{\Iw^+}(\C), \V_{\GSp_{2g}, k}^{\alg, \vee})$. They are defined as follows. \begin{enumerate}
    \item[$\bullet$] For any Hecke operator $T_{\bfgamma}$ away from $Np$, its action on $H^{n_0}(X_{\Iw^+}(\C), \V_{\GSp_{2g}, k}^{\alg, \vee})$ is defined by the same formula as (\ref{eq: Hecke on coh. away from Np}).
    \item[$\bullet$] For the $U_{p,i}$-action, let $\bfu_{p,i}$ act on $\GSp_{2g}(\Q_p)$ via conjugation \[
        \bfu_{p,i} . \bfgamma = \bfu_{p,i} \bfgamma \bfu_{p,i}^{-1}.
    \] Observe that if $\bfgamma\in B_{\GSp_{2g}}(\Q_p)$, then $\bfu_{p,i} . \bfgamma \in B_{\GSp_{2g}}(\Q_p)$ and the diagonal entries of $\bfgamma$ coincide with the diagonal entries of $\bfu_{p,i} . \bfgamma$. This action then induces a left $\bfu_{p,i}$-action on $\V_{\GSp_{2g}, k}^{\alg, \vee}$. The operator $U_{p,i}$ acting on $H^{n_0}(X_{\Iw^+}(\C), \V_{\GSp_{2g}, k}^{\alg, \vee})$ is defined by the same formula as (\ref{eq: Hecek on coh. at p}). 
\end{enumerate} 
\end{Remark}

We also consider the $p$-adically completed automorphic sheaf $\widehat{\underline{\omega}}_{\Iw^+}^{k}$ on $\overline{\calX}_{\Iw^+, \proket}$ defined by
$$
        \widehat{\underline{\omega}}_{\Iw^+}^{k} := \varprojlim_m\left(\underline{\omega}_{\Iw^+}^{k,+}\otimes_{\scrO^+_{\overline{\calX}_{\Iw^+}}}\scrO^+_{\overline{\calX}_{\Iw^+, \proket}}/p^m\right)[\frac{1}{p}]
$$
where $\underline{\omega}_{\Iw^+}^{k,+}$ is defined in Remark \ref{Remark: integral classical sheaf}. It follows from Proposition \ref{Proposition: explicit description of classical modular sheaf} that 
$$
        \widehat{\underline{\omega}}_{\Iw^+}^{k}(\calV) = \left\{f \in P_k(\GL_g, \widehat{\scrO}_{\overline{\calX}_{\Iw^+, w, \proket}}(\calV_{\infty})) :  \bfgamma^* f = \rho_{k}(\bfgamma_a+\frakz\bfgamma_c)^{-1}f, \,\,\,\forall \bfgamma = \begin{pmatrix}\bfgamma_a & \bfgamma_b\\ \bfgamma_c & \bfgamma_d\end{pmatrix} \in \Iw_{\GSp_{2g}}^+ 
        \right\}.
$$
for any log affinoid perfectoid object $\calV\in \overline{\calX}_{\Iw^+, \proket}$ and $\calV_{\infty} = \calV \times_{\overline{\calX}_{\Iw^+}}\overline{\calX}_{\Gamma(p^{\infty})}$.\\

\noindent\textbf{The algebraic Eichler-Shimura morphism}
Recall the Hodge--Tate morphism \[
    \HT_{\Gamma(p^{\infty})}: V_p \rightarrow \underline{\omega}_{\Gamma(p^{\infty})}
\] from \S \ref{subsection: perfectoid Siegel modular variety}. It follows from the definition that 
$$
        \underline{\omega}_{\Iw^+}^k = (\Sym^{k_1-k_2}\underline{\omega}_{\Iw^+})\otimes_{\scrO_{\overline{\calX}_{\Iw^+}}} (\Sym^{k_2-k_3}\wedge^2 \underline{\omega}_{\Iw^+}) \otimes_{\scrO_{\overline{\calX}_{\Iw^+}}} \cdots \otimes_{\scrO_{\overline{\calX}_{\Iw^+}}} (\Sym^{k_g}\det \underline{\omega}_{\Iw^+})
$$
and hence
$$
        \underline{\omega}_{\Gamma(p^{\infty})}^k = (\Sym^{k_1-k_2}\underline{\omega}_{\Gamma(p^{\infty})})\otimes_{\scrO_{\overline{\calX}_{\Gamma(p^{\infty})}}} (\Sym^{k_2-k_3}\wedge^2 \underline{\omega}_{\Gamma(p^{\infty})}) \otimes_{\scrO_{\overline{\calX}_{\Gamma(p^{\infty})}}} \cdots \otimes_{\scrO_{\overline{\calX}_{\Gamma(p^{\infty})}}} (\Sym^{k_g}\det \underline{\omega}_{\Gamma(p^{\infty})}).
$$

Let $V_{\mathrm{std}}$ denote the standard representation of $\GSp_{2g}$ over $\Q_p$, with standard basis $x_1, \ldots, x_{2g}$. There is an isomorphism of $\GSp_{2g}(\Q_p)$-representations $V_{\mathrm{std}}\simeq V_{\Q_p}:=V_p\otimes_{\Z_p}\Q_p$ sending $x_i$ to $e_{2g+1-i}$, for $i=1, \ldots, g$, and sending $x_i$ to $-e_{2g+1-i}$, for $i=g+1, \ldots, 2g$. If we write
$$
 V_{\mathrm{std}}^{k}:=(\Sym^{k_1-k_2}V_{\mathrm{std}}) \otimes_{\Q_p} (\Sym^{k_2-k_3}(\wedge^2 V_{\mathrm{std}}))\otimes_{\Q_p} \cdots \otimes_{\Q_p} (\Sym^{k_g}(\wedge^gV_{\mathrm{std}}))
$$
and
$$
 V_{\Q_p}^{k}:=(\Sym^{k_1-k_2}V_{\Q_p}) \otimes_{\Q_p} (\Sym^{k_2-k_3}(\wedge^2 V_{\Q_p}))\otimes_{\Q_p} \cdots \otimes_{\Q_p} (\Sym^{k_g}(\wedge^gV_{\Q_p})),
$$
the Hodge-Tate map induces a map $V_{\mathrm{std}}^{k}\simeq V_{\Q_p}^{k}\rightarrow \underline{\omega}_{\Iw^+}^k$. Moreover, it is well-known that $\V_{\GSp_{2g}, k}^{\alg}$ is an irreducible $\GSp_{2g}$-subrepresentation of $V_{\mathrm{std}}^{k}$ (see for example \cite[Lecture 17]{Fulton-Harris}). In particular, the highest weight vector $e^{\hst}_k$ in $\V_{\GSp_{2g}, k}^{\alg}$ corresponds to the element
$$
x_1^{k_1-k_2}\otimes (x_1\wedge x_2)^{k_2-k_3}\otimes\cdots \otimes (x_1\wedge \cdots \wedge x_g)^{k_g}
$$
in $V_{\mathrm{std}}^k$.

The composition
$$\V_{\GSp_{2g}, k}^{\alg, \vee}\xrightarrow[]{\beta}\V_{\GSp_{2g}, k}^{\alg}\hookrightarrow V_{\mathrm{std}}^k\simeq V_{\Q_p}^{k}\rightarrow  \underline{\omega}_{\Iw^+}^k$$ then induces a map \[
    \eta_{k}^{\alg}: \sheafOV_k^{\vee} \rightarrow \widehat{\underline{\omega}}_{\Iw^+}^k.
\] Eventually, we arrive at the \textbf{\textit{algebraic Eichler-Shimura morphism of weight $k$}} \begin{align*}
    \ES_k^{\alg}: H_{\et}^{n_0}(\calX_{\Iw^+}, \V_{\GSp_{2g}, k}^{\alg, \vee}) \otimes_{\Q_p}\C_p \simeq H_{\proket}^{n_0}(\overline{\calX}_{\Iw^+}, \sheafOV_k^{\vee}) \xrightarrow{\eta_k^{\alg}} H_{\proket}^{n_0}(\overline{\calX}_{\Iw^+}, \widehat{\underline{\omega}}_{\Iw^+}^k) \rightarrow H^0(\overline{\calX}_{\Iw^+}, \underline{\omega}_{\Iw^+}^{k+g+1})(-n_0),
\end{align*} where the last map follows from the same argument as in the proof of Lemma \ref{Lemma: Leray spectral sequence for the automorphic sheaf}. We remark that $\ES_k^{\alg}$ coincides with the one induced from \cite[Theorem VI. 6.2]{Faltings-Chai}. It is Hecke- and $G_{\Q_p}$-equivariant, and also surjective. 

In fact, over the $w$-ordinary locus $\overline{\calX}_{\Iw^+, w}$, the map $\eta_{k}^{\alg}$ has the following explicit description. 

\begin{Lemma}
\begin{enumerate}
\item[(i)] Let $\calV = \varprojlim_n \calV_n$ be a pro-Kummer \'etale presentation of a log affinoid perfectoid object in $\overline{\calX}_{\Iw^+, w, \proket}$ and let $\calV_{\infty} = \calV\times_{\overline{\calX}_{\Iw^+, w}}\overline{\calX}_{\Gamma(p^{\infty}), w}$. There is a well-defined $\GSp_{2g}(\Q_p)$-equivariant map
$$
\widetilde{\eta}_k^{\alg}: \V_{\GSp_{2g}, k}^{\alg, \vee} \otimes_{\Q_p} \widehat{\scrO}_{\overline{\calX}_{\Iw^+, w, \proket}}(\calV_{\infty})\rightarrow P_k(\GL_g, \widehat{\scrO}_{\overline{\calX}_{\Iw^+, w, \proket}}(\calV_{\infty}))
$$
defined by $\mu\otimes\delta \mapsto \delta f_{\mu, \frakz}^{\alg}$ where 
$$ f_{\mu, \frakz}^{\alg}(\bfgamma')=\int_{\bfalpha\in \GSp_{2g}} e_{k}^{\hst}\left(\begin{pmatrix}\trans\bfgamma' & \\ & \oneanti_g \bfgamma'^{-1}\oneanti_g\end{pmatrix}\begin{pmatrix}\one_g & \frakz\\ & \one_g\end{pmatrix} \bfalpha\right)\quad d\mu.$$
Here, the $\GSp_{2g}(\Q_p)$-action on the right hand side is given by
$$
\bfgamma.f:=\rho_k(\bfgamma_a+\frakz\bfgamma_c)(\bfgamma^*f)
$$
for every $\bfgamma=\begin{pmatrix} \bfgamma_a & \bfgamma_b\\ \bfgamma_c & \bfgamma_d\end{pmatrix}\in \GSp_{2g}(\Q_p)$ and $f\in P_k(\GL_g, \widehat{\scrO}_{\overline{\calX}_{\Iw^+, w, \proket}}(\calV_{\infty}))$.
\item[(ii)] The map $\eta_k^{\alg}$ is obtained from $\widetilde{\eta}^{\alg}_k$ by taking $\Iw^+_{\GSp_{2g}}$-invariants on both sides.
\end{enumerate}
\end{Lemma}

\begin{proof}
\begin{enumerate}
\item[(i)] Notice that $\widetilde{\eta}^{\alg}_k$ is the composition of $\beta$ with the map 
$$
\xi_k^{\alg}: \V_{\GSp_{2g}, k}^{\alg} \otimes_{\Q_p} \widehat{\scrO}_{\overline{\calX}_{\Iw^+, w, \proket}}(\calV_{\infty})\rightarrow P_k(\GL_g, \widehat{\scrO}_{\overline{\calX}_{\Iw^+, w, \proket}}(\calV_{\infty}))
$$
defined by $\phi\otimes \delta \mapsto \delta g_{\phi,\frakz}$
where $$g_{\phi,\frakz}(\bfgamma')=\phi\left(\begin{pmatrix}\one_g& \\ \trans\frakz&\one_g\end{pmatrix}
\begin{pmatrix}\bfgamma' & \\ & \oneanti_g \trans(\bfgamma')^{-1}\oneanti_g\end{pmatrix}\right)$$
for all $\bfgamma'\in \GL_g(\C_p)$.

Recall that $\beta$ is $\GSp_{2g}(\Q_p)$-equivariant. It remains to check that $\xi^{\alg}_k$ is $\GSp_{2g}(\Q_p)$-equivariant, which follows from a straightforward calculation.

\item[(ii)] It suffices to check that the $\Iw^+_{\GSp_{2g}}$-invariance of $\xi^{\alg}_k$ coincides with the map induced from the composition $\V_{\GSp_{2g}, k}^{\alg}\hookrightarrow V_{\mathrm{std}}^k\simeq V_{\Q_p}^{k}\rightarrow  \underline{\omega}_{\Iw^+}^k$. Notice that $\V_{\GSp_{2g}, k}^{\alg}$ is spanned by $\GSp_{2g}$-translations of the highest weight vector $e^{\hst}_k$. Therefore, we only need to check that $\xi^{\alg}_k(e^{\hst}_k\otimes 1)$ gives the correct element in $\widehat{\underline{\omega}}_{\Iw^+}^k$.

Indeed, since the Hodge--Tate map $V_p\rightarrow \underline{\omega}_{\Iw^+}$ sends $e_{2g+1-i}$ to $\fraks_i$, for $i=1, \ldots, g$, we see that the composition $\V_{\GSp_{2g}, k}^{\alg}\hookrightarrow V_{\mathrm{std}}^k\simeq V_{\Q_p}^{k}\rightarrow  \underline{\omega}_{\Iw^+}^k$ sends the highest weight vector $e^{\hst}_k$ to 
$$\fraks_1^{k_1-k_2}\otimes (\fraks_1\wedge \fraks_2)^{k_2-k_3}\otimes \cdots\otimes (\fraks_1\wedge\cdots\wedge \fraks_g)^{k_g}.$$

On the other hand, notice that the element $\fraks_1\wedge\cdots \wedge \fraks_t$ corresponds to the function $X = (X_{ij})_{1\leq i,j\leq g}\mapsto \det((X_{ij})_{1\leq i,j\leq t})$ in $P_k(\GL_g, \widehat{\scrO}_{\overline{\calX}_{\Iw^+, w, \proket}}(\calV_{\infty}))$. Therefore, $e^{\hst}_k$ is sent to the function
$$X \mapsto X_{11}^{k_1-k_2}\times \det((X_{ij})_{1\leq i,j\leq 2})^{k_2-k_3}\times\cdots \times\det((X_{ij})_{1\leq i,j\leq g})^{k_g}$$
in $P_k(\GL_g, \widehat{\scrO}_{\overline{\calX}_{\Iw^+, w, \proket}}(\calV_{\infty}))$. This element coincides with $\xi^{\alg}_k(e^{\hst}_k\otimes 1)$, as desired.
\end{enumerate}
\end{proof}

Recall the natural inclusion $M^{k, \mathrm{cl}}_{\Iw^+}=H^0(\overline{\calX}_{\Iw^+}, \underline{\omega}_{\Iw^+}^k)\hookrightarrow H^0(\overline{\calX}_{\Iw^+, w}, \underline{\omega}_w^k) = M_{\Iw^+, w}^{k}$ from Lemma \ref{Lemma: injection of classical forms}. The main result of this subsection is the following.

\begin{Theorem}\label{thm:imageclassicalweight}
Let $k = (k_1, ..., k_g) \in \Z_{\geq 0}^g$ be a dominant weight. Then the image of
$$\ES_{k}: \OC^r_{k, \C_p}  \longrightarrow M^{k+g+1}_{\Iw^+, w}(-n_0)$$
is contained in the space of the classical forms $M^{k+g+1, \cl}_{\Iw^+}(-n_0)$.
\end{Theorem}
\begin{proof}
Firstly, there is a natural map \[
\V_{\GSp_{2g}, k}^{\alg} \rightarrow A^{r}_{k}(\T_0, \Q_p)
\]
induced by the inclusion \[
    \T_0 \rightarrow \GSp_{2g}(\Q_p), \quad (\bfgamma, \bfupsilon)\mapsto \begin{pmatrix}\bfgamma & \\ \bfupsilon & \oneanti_g \trans\bfgamma^{-1}\oneanti_g\end{pmatrix}.
\] The dual of this map gives
\[
D_k^r(\T_0, \Q_p)\rightarrow \V^{\alg, \vee}_{\GSp_{2g},k}
\]
which then induces a map of sheaves \[
    \sheafOD_{k}^r \rightarrow \sheafOV_{k}^{\vee}
\] over $\overline{\calX}_{\Iw^+, w, \proket}$. Hence, the theorem follows once we show that the following diagram commutes
$$
\begin{tikzcd}
H_{\proket}^{n_0}(\overline{\calX}_{\Iw^+}, \sheafOD_{k}^r)\arrow[d]\arrow[r, "\ES_{k}"]  & H^0(\overline{\calX}_{\Iw^+, w}, \underline{\omega}_w^{k+g+1})(-n_0)   \\
 H_{\proket}^{n_0}(\overline{\calX}_{\Iw^+}, \sheafOV_{k}^{\vee}) \arrow[r, "\ES_k^{\alg}"] & H^0(\overline{\calX}_{\Iw^+}, \underline{\omega}_{\Iw^+}^{ k+g+1})(-n_0) \arrow[u]
\end{tikzcd}.
$$

Over $\overline{\calX}_{\Iw^+, w, \proket}$, it follows from the construction that we have a commutative diagram \[
    \begin{tikzcd}
        \sheafOD_{k}^r\arrow[r, "\eta_k"]\arrow[d] & \widehat{\underline{\omega}}_w^{k}\\
        \sheafOV_k^{\vee}\arrow[r, "\eta_k^{\alg}"] & \widehat{\underline{\omega}}_{\Iw^+}^{k}\arrow[u, hook]
    \end{tikzcd},
\] where the inclusion on the right-hand side is given by the inclusion (\ref{eq: alg. sheaf into overconvergent sheaf}). Consequently, there is a commutative diagram on the cohomology groups \[
    \begin{tikzcd}
        H_{\proket}^{n_0}(\overline{\calX}_{\Iw^+}\sheafOD_{k}^r)\arrow[r]\arrow[d, "\mathrm{Res}"] \arrow[ddd, bend right = 80, "\ES_{k}"'] & H_{\proket}^{n_0}(\overline{\calX}_{\Iw^+}, \sheafOV_{k}^{\vee})\arrow[d, "\mathrm{Res}"] \arrow[dddr, bend left = 40, "\ES_k^{\alg}"]\\
        H_{\proket}^{n_0}(\overline{\calX}_{\Iw^+, w}, \sheafOD_{k}^r)\arrow[r]\arrow[d, "\eta_k"] & H_{\proket}^{n_0}(\overline{\calX}_{\Iw^+, w}, \sheafOV_k^{\vee})\arrow[d, "\eta_{k}^{\alg}"]\\
        H_{\proket}^{n_0}(\overline{\calX}_{\Iw^+, w}, \widehat{\underline{\omega}}_{w}^{k})\arrow[d] & H_{\proket}^{n_0}(\overline{\calX}_{\Iw^+, w}, \widehat{\underline{\omega}}_{\Iw^+}^{k})\arrow[d]\arrow[l, hook']\\
        H^0(\overline{\calX}_{\Iw^+, w}, \underline{\omega}_{w}^{k+g+1})(-n_0) & H^0(\overline{\calX}_{\Iw^+, w}, \underline{\omega}_{\Iw^+}^{k+g+1})(-n_0)\arrow[l, hook'] & H^0(\overline{\calX}_{\Iw^+}, \underline{\omega}_{\Iw^+}^{k+g+1})(-n_0)\arrow[l, hook', "\mathrm{Res}"]
    \end{tikzcd}.
\] This finishes the proof.
\end{proof}

\begin{Corollary}\label{Corollary: OES is surjective on small slope part}
    Let $k=(k_1, ..., k_g)\in \Z_{\geq 0}^g$ be a dominant weight. Let $\OC_{k, \C_p}^{r, \mathrm{ss}}$ and $M_{\Iw^+, w}^{k+g+1, \mathrm{ss}}(-n_0)$ be the small-slope part of $\OC_{k, \C_p}^{r}$ and $M_{\Iw^+, w}^{k+g+1}(-n_0)$ respectively, i.e., the part on which the $U_{p,i}$-eigenvalues have slopes bounded as in \cite[Theorem 7.1.1]{AIP-2015} and the $U_p$-eigenvalue has slope bounded as in \cite[Definition 3.2.4]{Hansen-PhD}. Then, $\ES_k$ induces a surjective Hecke- and Galois-equivariant morphism \[
        \ES_k: \OC_{k, \C_p}^{r, \mathrm{ss}} \rightarrow M_{\Iw^+, w}^{k+g+1, \mathrm{ss}}(-n_0).
    \]
\end{Corollary}
\begin{proof}
    By the proof of Theorem \ref{thm:imageclassicalweight}, we have a Hecke- and Galois-equivariant commutative diagram \[
        \begin{tikzcd}
            \OC_{k, \C_p}^{r} \arrow[r] \arrow[d, "\ES_k"'] & H_{\proket}^{n_0}(\overline{\calX}_{\Iw^+}, \sheafOV_{k}^{\vee}) \arrow[d, "\ES_k^{\alg}"]\\
            M_{\Iw^+, w}^{k+g+1}(-n_0) & H^0(\overline{\calX}_{\Iw^+}, \underline{\omega}_{\Iw^+}^{k+g+1})(-n_0) \arrow[l]
        \end{tikzcd},
    \]
    where the horizontal arrows become isomorphism after passing to the small-slope parts by \cite[Theorem 3.2.5]{Hansen-PhD} and \cite[Theorem 7.1.1]{AIP-2015}. We conclude by noting that $\ES_k^{\alg}$ exhibits $H^0(\overline{\calX}_{\Iw^+}, \underline{\omega}_{\Iw^+}^{k+g+1})$ as a Hecke- and Galois-equivariant summand of $H_{\proket}^{n_0}(\overline{\calX}_{\Iw^+}, \sheafOV_{k}^{\vee})$ (\cite[Theorem VI. 6.2]{Faltings-Chai}), so we are done. 
\end{proof}
\section{The overconvergent Eichler--Shimura morphism on the cuspidal eigenvariety }\label{section:oncuspidaleigenvariety}
In this section, we glue the overconvergent Eichler--Shimura morphism over a suitable eigenvariety $\calE_0$. We begin with some preliminaries on the overconvergent cohomology groups in \S \ref{subsection: preliminaries on overconvergent cohomology}. The eigenvariety $\calE_{0}$ is constructed in \S \ref{subsection: cuspidal eigenvariety}. Finally in \S \ref{subsection: sheaves on the cuspidal eigenvariety}, we show that the overconvergent Eichler--Shimura morphism spreads out over $\calE_0$.

Throughout this section, we assume $p>2g$ so that we can apply results in \cite{AIP-2015} via the comparison in \S \ref{subsection:comparison sheaf aip}. On the other hand, we believe that the results in this section hold for smaller primes as well. In order to deal with these smaller primes, one would have to reprove several results in \cite{AIP-2015} in our context; \emph{e.g.}, the classicality result and the fact that $S_{\Iw^+}^{\kappa_{\calU}}$ has property (Pr) in the sense of \cite{Buzzard_2007}. We decide to leave these generalities to the reader in order to keep this paper within a reasonable length.

\subsection{Some preliminaries on overconvergent cohomology groups}\label{subsection: preliminaries on overconvergent cohomology} The purpose of this subsection is to review the basic constructions and properties of the overconvergent cohomology groups needed in latter subsections. Most of the materials are recorded from \cite{Hansen-PhD, CHJ-2017}. We do not claim any originality here. 

\begin{Definition}\label{Definition: open weights}
\begin{enumerate}
    \item[(i)] Let $(R_{\calU}, \kappa_{\calU})$ be a small weight. We say it is \textbf{open} if the natural map \[
        \calU^{\rig} = \Spa(R_{\calU}, R_{\calU})^{\rig} \rightarrow \calW
    \] is an open immersion.
    \item[(ii)] Let $(R_{\calU}, \kappa_{\calU})$ be an affinoid weight. We say it is \textbf{open} if the natural map \[
        \calU^{\rig} = \calU = \Spa(R_{\calU}, R_{\calU}^{\circ}) \rightarrow \calW
    \] is an open immersion. 
    \item[(iii)] A weight $(R_{\calU}, \kappa_{\calU})$ is called an \textbf{open weight} if it is either an small open weight or an affinoid open weight. 
\end{enumerate}
\end{Definition}

Given an open weight $(R_{\calU}, \kappa_{\calU})$ and an integer $r>1+r_{\calU}$, one considers the so called \emph{Borel--Serre chain complex} $C_{\bullet}(\Iw_{\GSp_{2g}}^+, A_{\kappa_{\calU}}^{r}(\T_0, R_{\calU}))$ (resp., \emph{Borel--Serre cochain complex} $C^{\bullet}(\Iw_{\GSp_{2g}}^+, D_{\kappa_{\calU}}^{r}(\T_0, R_{\calU}))$) which computes the Betti homolgy groups $H_t(X_{\Iw^+}(\C), A_{\kappa_{\calU}}^r(\T_0, R_{\calU}))$ (resp., Betti cohomology groups $H^t(X_{\Iw^+}(\C), D_{\kappa_{\calU}}^r(\T_0, R_{\calU}))$). The Borel--Serre chain complex is a finite complex as it is constructed by a fixed triangulation on the Borel--Serre compactification of the locally symmetric space $X_{\Iw^+}(\C)$. We write \begin{align*}
    & C_{\tol}^{\kappa_{\calU}, r} := \bigoplus_{t}C_t(\Iw^+_{\GSp_{2g}}, A_{\kappa_{\calU}}^r(\T_0, R_{\calU})),\\
    & C_{\kappa_{\calU}, r}^{\tol} := \bigoplus_{t}C^t(\Iw_{\GSp_{2g}}^+, D_{\kappa_{\calU}}^r(\T_0, R_{\calU})).
\end{align*} Then $C^{\kappa_{\calU}, r}_{\tol}$ is an ON-able $R_{\calU}[1/p]$-module as $A_{\kappa_{\calU}}^r(\T_0, R_{\calU})$ is ON-able (see \cite[\S 2.2, Remarks]{Hansen-PhD}). Moreover, there are naturally defined Hecke operators on $C_{\tol}^{\kappa_{\calU}, r}$ and the action of $U_p$ is compact (see [\textit{op. cit.}, \S 2.2]). We define $F_{\kappa_{\calU}, r}^{\oc}\in R_{\calU}[1/p]\llbrack T\rrbrack$ to be the Fredholm determinant of $U_p$ acting on $C^{\kappa_{\calU}, r}_{\tol}$. Notice that, for any $h\in \Q_{\geq 0}$, the existence of a slope-$\leq h$ decomposition of $C_{\tol}^{\kappa_{\calU}, r}$ is equivalent to the existence of a slope-$\leq h$ factorisation of $F_{\kappa_{\calU},r}^{\oc}$. Moreover, by [\textit{op. cit.}, Proposition 3.1.2], if $C^{\kappa_{\calU},r}_{\tol, \leq h}$ is the slope-$\leq h$ submodule of $C^{\kappa_{\calU},r}_{\tol}$ and suppose $\calU'=(R_{\calU'}, \kappa_{\calU'})$ is another open weight such that $\calU'^{\rig} \subset \calU^{\rig}$, there is a canonical isomorphism \[
    C_{\tol, \leq h}^{\kappa_{\calU}, r}\otimes_{R_{\calU}[\frac{1}{p}]}R_{\calU'}[\frac{1}{p}] \cong C_{\tol, \leq h}^{\kappa_{\calU'}, r}.
\]

\begin{Definition}
Let $\calU=(R_{\calU}, \kappa_{\calU})$ be an open weight and let $h\in \Q_{\geq 0}$. The pair $(\calU, h)$ is called a \textbf{slope datum} if $F_{\kappa_{\calU},r}^{\oc}$ admits a slope-$\leq h$ factorisation.
\end{Definition}

\begin{Proposition}\label{Proposition: functoriality of BS chain complex under slope decomposition} Let $(\calU, h)$ be a slope datum and let $(R_{\calU'}, \kappa_{\calU'})$ be an affinoid open weight such that $\calU' \subset \calU^{\rig}$. 
\begin{enumerate}
    \item[(i)] There is a canonical isomorphism \[
        H_{t}(X_{\Iw^+}(\C), A_{\kappa_{\calU}}^r(\T_0, R_{\calU}))_{\leq h} \otimes_{R_{\calU}[\frac{1}{p}]} R_{\calU'} \cong H_{t}(X_{\Iw^+}(\C), A_{\kappa_{\calU'}}^r(\T_0, R_{\calU'}))_{\leq h}
    \] for all $t\in \Z$, where the subscript ``$\leq h$'' stands for the slope-$\leq h$ submodule.
    \item[(ii)] The cochain complex $C_{\kappa_{\calU}, r}^{\tol}$ and the cohomology groups $H^t(X_{\Iw^+}(\C), D_{\kappa_{\calU}}^r(\T_0, R_{\calU}))$ admit slope-$\leq h$ decompositions. The corresponding slope-$\leq h$ submodules are denoted by $C_{\kappa_{\calU}, r}^{\tol, \leq h}$ and $H^t(X_{\Iw^+}(\C), D_{\kappa_{\calU}}^r(\T_0, R_{\calU}))^{\leq h}$, respectively. 
    \item[(iii)] There are canonical isomorphisms$$
        C_{\kappa_{\calU}, r}^{\tol, \leq h}\otimes_{R_{\calU}[\frac{1}{p}]}R_{\calU'}\cong C_{\kappa_{\calU'}, r}^{\tol, \leq h}$$
        and 
        $$ H^t(X_{\Iw^+}(\C), D_{\kappa_{\calU}}^r(\T_0, R_{\calU}))^{\leq h}\otimes_{R_{\calU}[\frac{1}{p}]}R_{\calU'} \cong H^t(X_{\Iw^+}(\C), D_{\kappa_{\calU'}}^r(\T_0, R_{\calU'}))^{\leq h}.$$
\end{enumerate}
\end{Proposition}
\begin{proof}
The proof follows verbatim as in the proofs of \cite[Proposition 3.3 \& Proposition 3.4]{CHJ-2017}.
\end{proof}

Notice that, if we vary the open weights $\calU$, the Fredholm determinants glue into a power series $F_{\calW}^{\oc}\in \scrO_{\calW}(\calW)\{\{T\}\}$. Here we drop the subscript ``$r$'' because the Fredholm determinant does not depend on $r$ according to \cite[Proposition 3.1.1]{Hansen-PhD}.

\subsection{The cuspidal eigenvariety}\label{subsection: cuspidal eigenvariety}
\noindent\textbf{The spectral variety.}
In the previous subsection, we obtained the Fredholm determinant $F_{\calW}^{\oc}\in \scrO_{\calW}(\calW)\{\{T\}\}$. On the other hand, given an affinoid weight $(R_{\calU}, \kappa_{\calU})$ and $w>1+r_{\calU}$, by \cite[Proposition 8.1.3.1]{AIP-2015} and Theorem \ref{Theorem: comparison with AIP} (see also [\textit{op. cit.}, Proposition 8.2.3.3]), the space of cuspforms $S_{\Iw^+,w}^{\kappa_{\calU}}=H^0(\overline{\calX}_{\Iw^+, w}, \underline{\omega}_{w, \cusp}^{\kappa_{\calU}})$ has property (Pr) in the sense of \cite{Buzzard_2007}; namely, it is a direct summand of a potentially ON-able $\C_p\widehat{\otimes}R_{\calU}$-Banach space. Also recall that $U_p$ acts compactly on the space of overconvergent Siegel modular forms. Therefore, we can define the Fredholm determinant $F_{\kappa_{\calU}, w}^{\mf}$ of $U_p$ acting on $S_{\Iw^+, w}^{\kappa_{\calU}}$. When we vary the affinoid weights, the Fredholm determinants glue together. After further taking the inductive limit over $w$, we arrive at a power series $F_{\calW}^{\mf}\in \scrO_{\calW}(\calW)\{\{T\}\}\widehat{\otimes}_{\Q_p}\C_p$.

Define $$F_{\calW}:=F_{\calW}^{\mf}F_{\calW}^{\oc}\in \scrO_{\calW}(\calW)\{\{T\}\}\widehat{\otimes}_{\Q_p}\C_p,$$ which is still a Fredholm series. Let $\mathbb{A}_{\C_p}^1$ denote the adic affine line over $(\C_p, \calO_{\C_p})$ with coordinate function $T$ and let $\bbA_{\calW}^1:=\calW\times_{\Spa(\Q_p, \Z_p)}\mathbb{A}_{\C_p}^1$. The \textit{\textbf{spectral variety}} $\calS$ is defined to be the zero locus of $F_{\calW}$ in $\bbA_{\calW}^1$.\\

\noindent\textbf{The eigenvarieties.} 

\begin{Definition}\label{Defintiion: slope data}
Let $\calU$ be an open weight so that $\calU^{\rig}\subset\calW$. Let $\calU^{\rig}_{\C_p}$ denote the base change of $\calU^{\rig}$ to $\Spa(\C_p, \calO_{\C_p})$ and consider the adic affine line $\bbA_{\calU}^1:=\calU^{\rig}\times_{\Spa(\Q_p, \Z_p)}\mathbb{A}_{\C_p}^1$ over $\calU^{\rig}_{\C_p}$. Moreover, for any $h\in \Q_{\geq  0}$, consider the closed ball $\mathbf{B}(0,p^h)$ of radius $p^h$ over $\C_p$ and define $\B_{\calU, h}:= \calU^{\rig} \times_{\Spa(\Q_p, \Z_p)}\mathbf{B}(0, p^h)$. Let $\calS_{\calU, h}:=\calS\cap \B_{\calU, h}$. We say that the pair $(\calU, h)$ is \textbf{slope-adapted} if  the natural map $\calS_{\calU, h}\rightarrow \calU^{\rig}_{\C_p}$ is finite flat.
\end{Definition}

Consider the collection $$\Cov(\calS)=\{\calS_{\calU, h}: (\calU, h)\text{ is slope-adapted}\}.$$ Let $\Cov_{\mathrm{aff}}(\calS)$ be a subcollection of $\Cov(\calS)$, consisting of those $\calS_{\calU, h}$ with $\calU$ being an affinoid weight. By \cite[Theorem 4.6]{Buzzard_2007} (see also \cite[Proposition 4.1.4]{Hansen-PhD}), $\Cov_{\mathrm{aff}}(\calS)$ forms an open covering for $\calS$ (and thus so is $\Cov(\calS)$). Using this covering, we define the following two coherent sheaves on $\calS$.

\begin{Definition}
\begin{enumerate}
\item[(i)] Recall from \S \ref{subsection:continuousfunctions} that $D_{\kappa_{\calU}}^{\dagger}(\T_0, R_{\calU})$ is defined to be the inverse limit of $D_{\kappa_{\calU}}^r(\T_0, R_{\calU})$ with respect to $r$. The coherent sheaf $\scrH_{\Par}^{\tol}$ on $\calS$ is defined by
$$\scrH_{\Par}^{\tol}(\calS_{\calU, h}):=\image\left(\bigoplus_{t}H_{c}^t(X_{\Iw^+}(\C), D_{\kappa_{\calU}}^{\dagger}(\T_0, R_{\calU}))^{\leq h}\rightarrow \bigoplus_{t}H^t(X_{\Iw^+}(\C), D_{\kappa_{\calU}}^{\dagger}(\T_0, R_{\calU}))^{\leq h}\right)\widehat{\otimes}_{\Q_p} \C_p$$
for all $\calS_{\calU, h}\in\Cov_{\mathrm{aff}}(\calS)$, where the map $$H_{c}^t(X_{\Iw^+}(\C), D_{\kappa_{\calU}}^{\dagger}(\T_0, R_{\calU}))^{\leq h}\rightarrow H^t(X_{\Iw^+}(\C), D_{\kappa_{\calU}}^{\dagger}(\T_0, R_{\calU}))^{\leq h}$$ is induced from the natural map from the compactly supported cohomology groups to the usual ones.

\item[(ii)] The coherent sheaf $\scrS_{\Iw^+}^{\dagger}$ on $\calS$ is defined by
$$\scrS_{\Iw^+}^{\dagger}(\calS_{\calU, h}):=S_{\Iw^+}^{ \kappa_{\calU}+g+1, \leq h}$$
for all $\calS_{\calU, h}\in\Cov_{\mathrm{aff}}(\calS)$, where the superscript ``$\leq h$'' stands for the slope-$\leq h$ part with respect to the $U_p$-operator.
\end{enumerate}
\end{Definition}

These are indeed well-defined coherent sheaves (see, for example, \cite[\S 4.3]{Hansen-PhD} and \cite[\S 8.1]{AIP-2015}, respectively). The Hecke algebra $\bbT$ acts on both of the coherent sheaves. The eigenvarieties we are interested in are the following.

\begin{Definition}
\begin{enumerate}
\item[(i)] For every $\calS_{\calU, h}\in \Cov_{\mathrm{aff}}(\calS)$, let $\bbT_{\calU, h}^{\oc}$ be the reduced $\scrO_{\calS_{\calU, h}}(\calS_{\calU, h})$-algebra generated by the image of $\bbT\rightarrow \End\left(\scrH_{\Par}^{\tol}(\calS_{\calU, h})\right)$. Let $\bbT_{\calU, h}^{\oc, \circ}$ be the integral closure of $\scrO_{\calS_{\calU, h}}(\calS_{\calU, h})^{\circ}$ inside $\bbT_{\calU, h}^{\oc}$. 
\item[(ii)] Let $\scrT_{\oc}$ be the coherent sheaf on $\calS$ defined by $\scrT_{\oc}(\calS_{\calU, h}):=\bbT_{\calU, h}^{\oc}$, and let $\scrT_{\oc}^{\circ}$ be the subsheaf of $\scrT_{\oc}$ defined by $\scrT_{\oc}^{\circ}(\calS_{\calU, h}):=\bbT_{\calU, h}^{\oc, \circ}$.
\item[(iii)] The \textbf{reduced cuspidal eigenvariety} $\calE_{0}^{\oc}$ is defined to be the relative adic space $\Spa_{\calS}(\scrT_{\oc}, \scrT_{\oc}^{\circ})$. 
\end{enumerate}
\end{Definition}

\begin{Definition}
\begin{enumerate}
\item[(i)] For every $\calS_{\calU, h}\in \Cov_{\mathrm{aff}}(\calS)$, let $\bbT_{\calU, h}^{\mf}$ be the reduced $\scrO_{\calS_{\calU, h}}(\calS_{\calU, h})$-algebra generated by the image of $\bbT\rightarrow \End\left(\scrS_{\Iw^+}^{\dagger}(\calS_{\calU, h})\right)$. Let $\bbT_{\calU, h}^{\mf,\circ}$ be the integral closure of $\scrO_{\calS_{\calU, h}}(\calS_{\calU, h})^{\circ}$ inside $\bbT_{\calU, h}^{\mf}$.
\item[(ii)] Let $\scrT_{\mf}$ be the coherent sheaf on $\calS$ defined by $\scrT_{\mf}(\calS_{\calU, h}):=\bbT_{\calU, h}^{\mf}$. Let $\scrT_{\mf}^{\circ}$ be the subsheaf of $\scrT_{\mf}$ defined by $\scrT_{\mf}^{\circ}(\calS_{\calU, h}):=\bbT_{\calU, h}^{\mf, \circ}$. 
\item[(iii)] The \textbf{equidimensional cuspidal eigenvariety} $\calE_0^{\mf}$ is defined to be the equidimensional locus of the  relative adic space $\Spa_{\calS}(\scrT_{\mf}, \scrT_{\mf}^{\circ})$.
\end{enumerate}
\end{Definition}

\begin{Remark}\label{Remark: eigenvarieties}
\normalfont Notice that $\calE_0^{\mf}$ is (the stricit Iwahori version of) the equidimensional cuspidal eigenvariety constructed in \cite{AIP-2015} after base change to $\C_p$. On the other hand, $\calE_{0}^{\oc}$ is the reduced cuspidal eigenvariety considered in \cite{Wu-2020} after base change to $\C_p$. We also point out that the cuspidal eigenvariety considered in \textit{op. cit.} is the cuspidal part of the eigenvariety for $\GSp_{2g}$ constructed in \cite{Johansson-Newton} (see also \cite{Hansen-PhD}).
\end{Remark}

\begin{Proposition}\label{Proposition: comparison of eigenvarieties}
There is a natural closed immersion $\calE_{0}^{\mf}\hookrightarrow \calE_{0}^{\oc}$.
\end{Proposition}
\begin{proof}
The strategy is to apply \cite[Theorem 5.1.2]{Hansen-PhD}. To this end, we need to find a \textit{very Zariski-dense} subset $\calS^{\cl}$ of $\calS$ such that for every $\bfitx\in \calS^{\cl}$ with dominate algebraic weight $k=(k_1, ..., k_g)\in \Z_{\geq 0}^{g}$ and any $Y\in \bbT$, we have $$\det\left(1-TY|\scrS_{\Iw^+, \bfitx}^{\dagger}\right)\,\,|\,\,\det\left(1-TY|\scrH_{\Par, \bfitx}^{\tol}\right).$$ 

By \cite[Theorem 3.2.5]{Hansen-PhD}, there exists an $h_{k}\in \R_{>0}$ such that for all $h\in \Q\cap [0, h_k]$, the canonical map $$H^{n_0}_{\Par}(X_{\Iw^+}(\C), D_{k}^{\dagger}(\T_0, \Q_p))^{\leq h}\rightarrow H_{\Par}^{n_0}(X_{\Iw^+}(\C), \V_{\GSp_{2g}, k}^{\alg, \vee})^{\leq h}$$ is an isomorphism. On the other hand, let 
$$\underline{\omega}_{\Iw^+, \cusp}^{k}:=\underline{\omega}_{\Iw^+}^{k}\otimes_{\scrO_{\overline{\calX}_{\Iw^+}}}\scrO_{\overline{\calX}_{\Iw^+}}(-\calZ_{\Iw^+})$$ be the sheaf of classical cuspidal Siegel modular forms of weight $k$ on $\overline{\calX}_{\Iw^+}$. The classicality theorem \cite[Theorem 7.1.1]{AIP-2015} provides an $a_k\in \Q_{>0}$ such that for all $h\in\Q\cap[0,a_k]$, the slope-$\leq h$ overconvergent Siegel cuspforms of weight $k$ are classical; namely, $$H^0(\overline{\calX}_{\Iw^+, w}, \underline{\omega}_{w, \cusp}^{k})^{\leq h}\subset H^0(\overline{\calX}_{\Iw^+}, \underline{\omega}_{\Iw^+, \cusp}^{k}).$$ Now, let $\ell_{k}=\min\{h_k, a_k\}$ and take $h\leq \ell_k$. Applying the \emph{generalised Eichler--Shimura morphism} in \cite[Theorem 3.8]{Hida_2002}, we obtain an injection from the space of slope-$\leq h$ overconvergent Siegel cuspforms of classical weight into the slope-$\leq h$ cohomology group with coefficient in the algebraic representation. Consequently, the desired very Zariski-dense subset of $\calS$ can be taken to be $$\calS^{\cl}=\bigcup_{\calS_{\calU, h}\in \Cov_{\mathrm{aff}}(\calS)}\{\bfitx\in \calS_{\calU, h}: \bfitx\text{ has classical weight }k\in \Z_{\geq 0}^{g}\text{ and } h\leq \ell_k\}$$
Finally, \cite[Theorem 5.1.2]{Hansen-PhD} yields the result. 
\end{proof}

Given Proposition \ref{Proposition: comparison of eigenvarieties}, we may identify $\calE_0^{\mf}$ with its image in $\calE_0^{\oc}$ and denote it by $\calE_0$ for simplicity. We have a diagram $$\begin{tikzcd}
\calE_0\arrow[r, "\pi"]\arrow[rr, bend right = 30, "\wt"'] & \calS\arrow[r, "\wt_{\calS}"] & \calW
\end{tikzcd}.$$

\subsection{Sheaves on the cuspidal eigenvariety}\label{subsection: sheaves on the cuspidal eigenvariety}

Let $(R_{\calU}, \kappa_{\calU})$ be a weight and $r$ be an integer with $r>1+r_{\calU}$. If $(R_{\calU}, \kappa_{\calU})$ is a small weight, recall $\OC_{\kappa_{\calU}, \C_p}^{r, \cusp}$ from \eqref{eq: definition of OCcusp}. If $(R_{\calU}, \kappa_{\calU})$ is an affinoid weight, define \[
    \OC_{\kappa_{\calU}, \C_p}^{r, \cusp} = \image\left( H_c^{n_0}(X_{\Iw^+}(\C), D_{\kappa_{\calU}}^{r}(\T_0, R_{\calU}))\widehat{\otimes}\C_p \rightarrow H^{n_0}(X_{\Iw^+}(\C), D_{\kappa_{\calU}}^{r}(\T_0, R_{\calU}))\widehat{\otimes}\C_p \right),
\]
where the morphism is the natural morphism from the compactly supported cohomology to the usual Betti cohomology. 
We also write \[
    \OC_{\kappa_{\calU}, \C_p}^{\dagger, \cusp} := \varprojlim_{r} \OC_{\kappa_{\calU}, \C_p}^{r, \cusp}.
\]

Suppose that $(R_{\calU}, \kappa_{\calU})$ is an small open weight and recall the overconvergent Eichler--Shimura morphism for overconvergent Siegel cuspforms \[
    \ES_{\kappa_{\calU}}^{\cusp}: \OC_{\kappa_{\calU}, \C_p}^{r, \cusp} \rightarrow S_{\Iw^+, w}^{\kappa_{\calU}+g+1}(-n_0).
\] If $(\calU, h)$ slope-adapted, then the Hecke-equivariance of $\ES_{\kappa_{\calU}}^{\cusp}$ induces a $\C_p\widehat{\otimes}R_{\calU}$-linear map \[
   \ES_{\kappa_{\calU}}^{\cusp, \leq h}: \OC_{\kappa_{\calU}, \C_p}^{r \cusp, \leq h}\rightarrow S_{\Iw^+, w}^{\kappa_{\calU}+g+1,\leq h}(-n_0). 
\] of finite projective $\C_p\widehat{\otimes}R_{\calU}$-modules. 

Now, if $\calU'\subset \calU^{\rig}$ is an affinoid weight, the $\C_p\widehat{\otimes}R_{\calU'}$-linear map $\ES_{\kappa_{\calU'}}^{\cusp}$ is defined to be the composition \begin{equation}\label{eq: ES for affinoid weights}
    \ES_{\kappa_{\calU'}}^{\cusp, \leq h}: \OC_{\kappa_{\calU'}, \C_p}^{r, \cusp, \leq h} \cong \OC_{\kappa_{\calU}, \C_p}^{r, \cusp, \leq h}\otimes_{R_{\calU}[\frac{1}{p}]}R_{\calU'} \rightarrow S_{\Iw^+, w}^{\kappa_{\calU}+g+1, \leq h}(-n_0) \rightarrow S_{\Iw^+, w}^{\kappa_{\calU'}+g+1, \leq h}(-n_0),
\end{equation} where the first isomorphism follows from Proposition \ref{Proposition: functoriality of BS chain complex under slope decomposition}. 

Recall the natural map $\pi:\calE_0\rightarrow\calS$ and let $\calE_{\calU, h}$ be the preimage of $\calS_{\calU, h}$. On the cuspidal eigenvariety $\calE_0$, we consider two coherent sheaves $\sheafOC^{\dagger}_{\cusp}$ and $\scrS_{\Iw^+}^{\dagger}(-n_0)$ given by 
$$
    \sheafOC^{\dagger}_{\cusp}(\calE_{\calU, h}):=\OC_{\kappa_{\calU}, \C_p}^{\dagger, \cusp, \leq h}
    $$
and    
$$\scrS_{\Iw^+}^{\dagger}(-n_0)(\calE_{\calU, h}):= S_{\Iw^+}^{\kappa_{\calU}+g+1, \leq h}(-n_0),
$$
for all $\calS_{\calU, h}\in \Cov_{\mathrm{aff}}(\calS)$. 

\begin{Theorem}\label{Theorem: OES over the eigenvariety}
There exists a morphism \[
    \EScal: \sheafOC_{\cusp}^{\dagger}\rightarrow \scrS_{\Iw^+}^{\dagger}(-n_0)
\] of coherent sheaves over $\calE_0$ such that if $(\calU, h)$ is a slope-adapted pair, then $\EScal(\calE_{\calU, h})$ is exactly the overconvergent Eichler--Shimura morphism for overconvergent Siegel cuspforms \[
    \ES_{\kappa_{\calU}}^{\cusp, \leq h}: \OC_{\kappa_{\calU}, \C_p}^{\dagger, \cusp, \leq h} \rightarrow S_{\Iw^+}^{\kappa_{\calU}+g+1, \leq h}(-n_0).
\]
\end{Theorem}
\begin{proof}
It follows from (\ref{eq: ES for affinoid weights}) and the functoriality of $\ES_{\kappa_{\calU}}^{\cusp}$ in small open weights $\calU$.
\end{proof}

Denote by $\sheafIm$ and $\sheafKer$ the image and the kernel of $\EScal$, respectively. We obtain a short exact sequence of sheaves on $\calE_0$ $$0\rightarrow \sheafKer\rightarrow \sheafOC_{\cusp}^{\dagger}\rightarrow \sheafIm\rightarrow 0.$$ We remind the readers that this short exact sequence is Galois- and Hecke-equivariant. 

Let $\calV = \Spa(R_{\calV}, R_{\calV}^+)$ be an affinoid open subspace of $\calE_0$ such that $\sheafKer(\calV)$, $\sheafOC_{\cusp}^{\dagger}(\calV)$, and $\sheafIm(\calV)$ are free and such that the sequence \[
    0\rightarrow \sheafKer(\calV)\rightarrow \sheafOC_{\cusp}^{\dagger}(\calV)\rightarrow \sheafIm(\calV)\rightarrow 0
\] is exact. Consider \[
    \scrH(\calV):=\Hom_{R_{\calV}}(\sheafIm(\calV), \sheafKer(\calV)).
\] Recall that we have the Sen operator $\varphi_{\Sen} = \varphi_{\Sen, \calV}$ associated with $\scrH(\calV)$, which was introduced in \cite{Sen-analytic} (see also \cite{Kisin-2003}).

The following result is an analogue to \cite[Theorem 6.1(c)]{AIS-2015}. 

\begin{Theorem}\label{Theorem: local spliting of OES} Let $\calV=\Spa(R_{\calV}, R_{\calV}^+)\subset \calE_0$ be an affinoid open subspace such that $\sheafKer(\calV)$, $\sheafOC_{\cusp}^{\dagger}(\calV)$, and $\sheafIm(\calV)$ are free and such that the sequence \[
    0\rightarrow \sheafKer(\calV)\rightarrow \sheafOC_{\cusp}^{\dagger}(\calV)\rightarrow \sheafIm(\calV)\rightarrow 0
\] is exact. Suppose $\varphi_{\Sen}$ is non-vanishing. Then the short exact sequence $$0\rightarrow \sheafKer\rightarrow \sheafOC_{\cusp}^{\dagger}\rightarrow \sheafIm\rightarrow 0$$ splits locally over $\calV$.
\end{Theorem}
\begin{proof}
We follow the same strategy as in \cite[Theorem 6.1(c)]{AIS-2015}. Observe that we have an isomorphism $H^1(G_{\Q_p}, \scrH(\calV)) \simeq \text{Ext}^1_{R_{\calV}[G_{\Q_p}]}(\sheafIm(\calV), \sheafKer(\calV))$. Thus, the $G_{\Q_p}$-equivariance of the short exact sequence defines a class in $H^1(G_{\Q_p}, \scrH(\calV))$. Then by \cite[Proposition 2.3]{Kisin-2003}, $\det(\varphi_{\Sen})\in R_{\calV}$ kills the cohomology group $H^1(G_{\Q_p}, \scrH(\calV))$. On the other hand, $\det(\varphi_{\Sen})$ is non-zero. Therefore, after localising at this element, the short exact sequence splits as a sequence of semilinear $G_{\Q_p}$-representations. Since the Galois-action commutes with the Hecke-actions, the splitting can be chosen to be Hecke-equivariant.  
\end{proof}

\subsection{Application to Galois representations}

We now give an application to the Galois representations associated with overconvergent Siegel modular forms. 

Let $f$ be a classical cuspidal Siegel eigenform of weight $(k_1+g+1, \ldots, k_g+g+1)$. We make the following two hypotheses:

\begin{hyp}[M1]\label{M1}
The $f$-isotypical part of $H^{\bullet}(X_{\Iw^+}(\C), \V_{\GSp_{2g}, k}^{\alg, \vee})\otimes_{\Q_p}\C_p$ is concentrated in degree $n_0$, it is included in the parabolic cohomology (\it{i.e.}, in the image of the compactly supported cohomology) and is $2^g$-dimensional.
\end{hyp}

\begin{hyp}[Et]\label{Et}
The weight map $\wt: \calE_{0} \rightarrow \calW$ is \'etale at the point $\bfitx_f$ corresponding to $f$.
\end{hyp}

These two hypotheses are conjectured to hold in great generality, even if there will be cases when they won't hold. For example, already for $\GSp_4$, there exist some CAP representations whose corresponding eigensystem in $H^3_{\et}$ is two-dimensional (cf. \cite[Hypothesis A (7)]{Weissauer-4dimGalRep}). 

Some positive partial results, always for $\GSp_4$, can  be obtained for forms of paramodular level thanks to Robert and Schimidt \cite{RS-NewformGSp4} who developed a (local) newform theory for representations of paramodular level.  If the representation $\pi$ is generic, meaning it admits a Whittaker model (see \cite[\S 0.5]{Soudry}), then it is known that $\pi$ satisfies strong multiplicity one \cite[Theorem 1.5]{Soudry}, meaning that if one consider another generic $\pi'$ such the local components $\pi_v$ and $\pi'_v$ are isomorphic for almost all $v$, then $\pi=\pi'$. Moreover, if the level is paramodular, we know by \cite[Theorem 4.5]{RosnerWissauer} that there is no non-generic automorphic representation isomorphic to $\pi$ almost everywhere. This means that if $\pi$ is paramodular, then it satisfies our ({M1}) assumption. 

It is a folklore expectation that if $\pi$ is generic and non-endoscopic, the same strong multiplicity one result among all representations (not only generic) should hold.

When ({M1}) holds, assuming  further that the $U_p$-eigenvalues are all distinct on the $f$-isotypical part, then one can often proceeds as in \cite[Proposition 8.3.2]{AIP-2015},  to obtain \'etaleness.\\

Recall that there is an integer $h_k$ (depending on $k$) such that for all $h\in \Q\cap [0, h_k]$, the canonical map $$H^{n_0}_{\Par}(X_{\Iw^+}(\C), D_{k}^{\dagger}(\T_0, \Q_p))^{\leq h}\rightarrow H_{\Par}^{n_0}(X_{\Iw^+}(\C), \V_{\GSp_{2g}, k}^{\alg, \vee})^{\leq h}$$ is an isomorphism.
We have the following theorem.
\begin{Theorem}\label{Theorem: Locally free}
Let $f$ as above, satisfying hypotheses (M1) and (Et). Suppose moreover that it is of finite slope for the $U_p$-operators, of slope $h \leq h_k$. There exists an affinoid neighbourhood $\mathcal{U}$ of $\wt(\bfitx_f) \in \calW$ and an affinoid neighbourhood $\mathcal{V}$ of $\bfitx_f$  such that
\[
H^{n_0}(X_{\Iw^+}(\C), D_{\kappa_{\calU}}^{\dagger}(\T_0, R_{\calU}))^{\leq h} \otimes \scrT_{\mf}(\pi(\mathcal{V}))
\]
is free of rank $2^g$ over $R_{\calU}$. 
\end{Theorem}
\begin{proof}
We point out that this strategy of proof goes back to Hida, and it has been formalised in \cite[\S 2.5]{BDJ}.
Choose an open neighborhood $\calU'$ containing $\wt(\bfitx_f)$ and consider the maximal ideal $\frakm_{f,\calU'}$ in $\bbT_{\calU', h}^{\mf, \circ}$ corresponding to the point $\bfitx_f$. 
By (M1) and the slope hypothesis, the cohomology is concentrated in one degree. Therefore, we can apply \cite[Lemma 2.10]{BDJ} to the (rigid) localisation
\[ H^{n_0}(X_{\Iw^+}(\C), D_{\kappa_{\calU'}}^{\dagger}(\T_0, R_{\calU'}))^{\leq h}_{\frakm_{f,\calU'}}
\]
to conclude that there exists an open neighborhood $\calU$ of $\wt(\bfitx_f)$ such that 
\[
H^{n_0}(X_{\Iw^+}(\C), D_{\kappa_{\calU'}}^{\dagger}(\T_0, R_{\calU'}))^{\leq h}_{R_{\calU'}}\otimes R_{\calU}
\]
 is free over $R_{\calU}$. The rank must be $2^g$ by \cite[Lemma 2.9]{BDJ} and again (M1). By the \'etaleness hypothesis, there is a neighbourhood $\mathcal{V}$ of $\bfitx_f$ such that $\wt: \mathcal{V} \rightarrow \calU$ is an isomorphism, and hence   $\scrT_{\mf}(\pi(\mathcal{V})) \cong R_{\calU}$.
\end{proof}

Recall that, by Theorem \ref{Proposition: Comparison theorem of cohomologies}, when $
\calU'$ is a small weight, there is an isomorphism
$$H_{\et}^t(\calX_{\Iw^+}, \scrD_{\kappa_{\calU'}}^r)\cong H^t(X_{\Iw^+}(\C), D_{\kappa_{\calU'}}^r(\T_0, R_{\calU'}))$$
which equips the right hand side with a continuous action of $\Gal(\overline{\Q}/\Q)$. This actions commutes with the Hecke actions. Let $\calU \subset \calU'$ be an open subspace satisfying the property in Theorem \ref{Theorem: Locally free}, we obtain a continuous Galois action on 
\[
H^{n_0}(X_{\Iw^+}(\C), D_{\kappa_{\calU}}^{\dagger}(\T_0, R_{\calU}))^{\leq h} \otimes \scrT_{\mf}(\pi(\mathcal{V})).
\]
We arrive at the following theorem:
\begin{Theorem}\label{thm:GaloisRepresentations}
Let $\calU$ and $\calV$ be as in Theorem \ref{Theorem: Locally free}. Consider the Galois representation  
\[
\rho_{\calU}: \Gal(\overline{\Q}/\Q)\rightarrow \mathrm{GL}_{R_{\calU}} \left(H^{n_0}(X_{\Iw^+}(\C), D_{\kappa_{\calU}}^{\dagger}(\T_0, R_{\calU}))^{\leq h} \otimes \scrT_{\mf}(\pi(\mathcal{V}))\right)  \cong \mathrm{GL}_{2^g}(R_{\calU})
\]
defined by the Galois module $H^{n_0}(X_{\Iw^+}(\C), D_{\kappa_{\calU}}^{\dagger}(\T_0, R_{\calU}))^{\leq h} \otimes \scrT_{\mf}(\pi(\mathcal{V}))$. Then $\rho_{\calU}$ interpolates the Spin Galois representations associated with classical Siegel modular forms constructed in \cite{KretShin}, for all classical points in $\Spa(\scrT_{\mf}(\pi(\mathcal{V}), {\scrT_{\mf}(\pi(\mathcal{V})}^{\circ})$. In particular if $f'$ is an overconvergent Siegel modular form corresponding to a maximal ideal $\frakm_{f',\calU}$ of $\scrT_{\mf}(\pi(\mathcal{V}))$, then the Galois representation
\[
\rho_{f'}: \Gal(\overline{\Q}/\Q)\rightarrow \mathrm{GL}_{2^g}(\scrT_{\mf}(\pi(\mathcal{V})) / {\frakm_{f',\calU}})
\]
induced by $\rho_{\calU}$ realises the Galois representation associated with $f'$.
\end{Theorem}

Notice that our construction of families of Galois representations doesn't involve pseudo-representations or determinants at all. Nonetheless, given that pseudo-representations are determined by the corresponding Hecke/Frobenius eigenvalues, our Galois representations indeed coincide (up to semi-simplifications) with the ones that one could abstractly construct via the theory of pseudo-representations.

We also point out that the results of \cite{KretShin} indicate that the image of the spin Galois representation is contained in $\mathrm{GSpin}_{2g+1}$. We are not able to prove that the image of our Galois representation falls in $\mathrm{GSpin}_{2g+1}(R_{\calU})$. However, it seems promising to study the image of our Galois representation in more detail using the pairing on $H_{\et}^t(\calX_{\Iw^+}, \scrD_{\kappa_{\calU}}^r)$ constructed in \cite{Wu-2020}.
 
\begin{appendix}
\section{Kummer \'{e}tale and pro-Kummer \'{e}tale sites of log adic spaces}\label{Section: Kummer etale and pro-Kummer etale sites of log adic spaces}
In order to study the boundaries of various toroidal compactifications of Siegel varieties, we adopt the language of \emph{logarithmic adic spaces} established in \cite{Diao}. The purpose of \S\ref{subsection: review of Kummer etale and pro-Kummer etale sites} is to review the basic notions of log adic spaces, as well as their Kummer \'etale and pro-Kummer \'etale topologies, for convenience of the readers who are not familiar with the language. In \S\ref{subsection: main result}, we present an explicit calculation of the sheaf $R^i\nu_*\widehat{\scrO}_{X_{\proket}}$ which plays an essential role in the construction of the overconvergent Eichler--Shimura morphisms in \S\ref{subsection: OES}. Finally, in \S \ref{subsection: generalised projective formula}, we introduce the notion of \emph{Kummer \'etale Banach sheaves} and prove a (generalised) projection formula for those Kummer \'etale Banach sheaves that are \emph{admissible}. \\

\noindent\textbf{Notation.} We warn the reader that, in this section, (log) adic spaces will no longer be written in calligraphic font as we deal with more general (log) adic spaces, not only those studied in the main body of the text. 

\subsection{Review of log adic spaces}\label{subsection: review of Kummer etale and pro-Kummer etale sites}
Let $k$ be a nonarchimedean field (\emph{i.e.,} a field complete with respect to a nonarchimedean norm $|\cdot|:k\rightarrow \R_{\geq 0}$) and let $\calO_k=\{x\in k\,:\, |x|\leq 1\}$.  

\begin{Definition}
Let $X$ be a locally noetherian adic space over $\Spa(k, \calO_k)$.
\begin{enumerate}
\item[(i)] A \textbf{pre-log structure} on $X$ is a pair $(\scrM_{X}, \alpha)$ where $\scrM_{X}$ is a sheaf of monoids on $X_{\et}$ and $\alpha:\scrM_{X}\rightarrow \scrO_{X_{\et}}$ is a morphism of sheaves of (multiplicative) monoids. It is called a \textbf{log structure} if the induced morphism $\alpha^{-1}(\scrO^{\times}_{X_{\et}})\rightarrow \scrO^{\times}_{X_{\et}}$ is an isomorphism. In this case, the triple $(X, \scrM_{X}, \alpha)$ is called a \textbf{log adic space}. If the context is clear, we simply say that $X$ is a log adic space. 
\item[(ii)] For a pre-log structure $(\scrM_{X}, \alpha)$ on $X$, the \textbf{associated log structure} is $({}^{a}\!\!\!\!\scrM_X, {}^{a}\alpha)$ where ${}^{a}\!\!\!\!\scrM_X$ is given by the pushout 
$$
\begin{tikzcd}
\alpha^{-1}(\scrO^{\times}_{X_{\et}})\arrow[r] \arrow[d] & \scrM_X \arrow[d]  \\
       \scrO^{\times}_{X_{\et}}  \arrow[r] & {}^{a}\!\!\!\!\scrM_X \arrow[ul, phantom, "\ulcorner", very near start]
\end{tikzcd}
$$
and ${}^{a}\alpha:{}^{a}\!\!\!\!\scrM_{X}\rightarrow \scrO_{X_{\et}}$ is the induced morphism.
\item[(iii)] A \textbf{morphism} $f: (Y, \scrM_Y, \alpha_Y)\rightarrow (X, \scrM_X, \alpha_X)$ of log adic spaces is a morphism $f: Y\rightarrow X$ of adic spaces together with a morphism of sheaves of monoids $f^{\sharp}:f^{-1}\scrM_X\rightarrow \scrM_Y$ such that the diagram
$$
\begin{tikzcd}
{f^{-1}\scrM_X} \arrow[r, "f^{\sharp}"] \arrow[d, "f^{-1}\alpha_X"'] & \scrM_Y \arrow[d, "\alpha_Y"]  \\
        {f^{-1}\scrO_{X_{\et}}} \arrow[r] & {\scrO_{Y_{\et}}} 
\end{tikzcd}
$$
commutes. Moreover, the log structure associated with the pre-log structure $f^{-1}\scrM_X\rightarrow f^{-1}\scrO_{X_{\et}}\rightarrow \scrO_{Y_{\et}}$ is called the \textbf{pullback log structure}, denoted by $f^*\scrM_X$. We say that $f$ is \textbf{strict} if $f^*\scrM_X\xrightarrow[]{\sim} \scrM_Y$.
\end{enumerate}
\end{Definition}

\begin{Definition}
\begin{enumerate}
\item[(i)] Let $(X, \scrM_{X}, \alpha)$ be a locally noetherian log adic space as above. Let $P$ be a monoid and let $P_{X}$ denote the associated constant sheaf of monoids on $X_{\et}$. A \textbf{chart of $X$ modeled on $P$} is a morphism of sheaves of monoids $\theta: P_{X}\rightarrow \scrM_{X}$ such that $\alpha(\theta(P_{X}))\subset \scrO^+_{X_{\et}}$ and such that the log structure associated with the pre-log structure $\alpha\circ\theta:P_{X}\rightarrow \scrO_{X_{\et}}$ is isomorphic to $\scrM_{X}$. We say that the chart is \textbf{fs} if $P$ is fine and saturated. 
\item[(ii)] A locally noetherian log adic space is called an \textbf{fs log adic space} if it \'etale locally admits charts modeled on fs monoids.
\item[(iii)] Let $f: (Y, \scrM_{Y}, \alpha_{Y})\rightarrow (X, \scrM_{X}, \alpha_{X})$ be a morphism between locally noetherian log adic spaces. A \textbf{chart} of $f$ consists of charts $\theta_{X}:P_{X}\rightarrow \scrM_{X}$ and $\theta_{Y}: Q_{Y}\rightarrow \scrM_{Y}$ and a homomorphism $u:P\rightarrow Q$ such that the diagram
 \[
        \xymatrix{ {P_{Y}} \ar^-u[r] \ar^-{\theta_{X}}[d] & {Q_{Y}} \ar^-{\theta_{Y}}[d] \\
        {f^{-1}\scrM_{X}} \ar^-{f^\sharp}[r] & {\scrM_{Y}} }
    \]
commutes. We say that the chart is \textbf{fs} if both $P$ and $Q$ are fs. When the context is clear, we simply say that $u:P\rightarrow Q$ is a chart of $f$.
\end{enumerate}
\end{Definition}

Below are two typical examples of locally noetherian fs log adic spaces. In this paper, all of the toroidally compactified Siegel varieties (equipped with logarithmic structures associated with the boundary divisors) have the form as in Example \ref{Example: normal crossings}.

\begin{Example}\label{Example: unit disc}
\normalfont Let $n>0$ be an integer. Consider the \textit{\textbf{$n$-dimensional unit disc}}
\[\bbD^n:=\Spa(k\langle T_1, \ldots, T_n\rangle, \calO_k\langle T_1, \ldots, T_n\rangle),\]
equipped with the log structure associated with the pre-log structure induced by \[\Z^n_{\geq 0}\rightarrow k\langle T_1, \ldots, T_n\rangle,\,\,\,\, (a_1, \ldots, a_n)\mapsto T_1^{a_1}\cdots T_n^{a_n}.\] Clearly, $\bbD^n$ is modeled on the fs chart $\Z^n_{\geq 0}$.\qed
\end{Example}

\begin{Example}\label{Example: normal crossings}
\normalfont Let $X$ be a smooth rigid analytic variety over $k$, viewed as an adic space over $\Spa(k, \calO_k)$ via \cite[(1.1.11)]{Huber-2013}. Let $D\subset X$ be a \textit{\textbf{normal crossings divisor}} in the sense of \cite[Example 2.3.17]{Diao}. Namely, $\iota:D\hookrightarrow X$ is a closed immersion such that, analytic locally, $X$ and $D$ are of the form $S\times \bbD^n$ and $S\times \{T_1\cdots T_n=0\}$, where $S$ is a smooth connected rigid analytic variety and $\iota$ is the pullback of the natural inclusion $\{T_1\cdots T_n=0\}\hookrightarrow \bbD^n$. We equip $X$ with the log structure 
\[\scrM_{X}=\{f\in \scrO_{X_{\et}}\,|\, f \textrm{ is invertible on }X\smallsetminus D\}\]
with $\alpha:\scrM_{X}\rightarrow \scrO_{X_{\et}}$ being the natural inclusion. This is the \textit{\textbf{divisorial log structure}} associated with the divisor $D$. This log structure agrees with the pullback of the log structure on $\bbD^n$ constructed in Example \ref{Example: unit disc}. \qed
\end{Example}

Log adic spaces in Example \ref{Example: normal crossings} are special cases of \emph{log smooth} ones. For later use, we recall the definition of log smoothness.

For any monoid $P$ and any commutative ring $T$, we write $T[P]$ for the associated monoid algebra.

\begin{Definition}
Let $X$ be a locally noetherian adic space over $\Spa(k, \calO_k)$ and let $P$ be a finitely generated monoid. For any affinoid open subspace $\Spa(R, R^+)\subset X$, let $(R\langle P\rangle, R^+\langle P\rangle)$ be the completion of $(R[P], R^+[P])$. By gluing the morphisms $\Spa(R\langle P\rangle, R^+\langle P\rangle)\rightarrow \Spa(R, R^+)$, we obtain a morphism $X\langle P\rangle\rightarrow X$. Moreover, we equip $X\langle P\rangle$ with the log structure modeled on the chart $P$; \emph{i.e.}, the one locally induced by $P\rightarrow R\langle P\rangle$.
\end{Definition}

\begin{Definition}\label{Definition: log smooth}
Let $f: Y\rightarrow X$ be a morphism between locally noetherian fs log adic spaces. We say that $f$ is \textbf{log smooth} if \'etale locally $f$ admits an fs chart $u:P\rightarrow Q$ such that
\begin{enumerate}
\item[(i)] the kernel and the torsion part of the cokernel of $u^{\mathrm{gp}}:P^{\mathrm{gp}}\rightarrow Q^{\mathrm{gp}}$ are finite groups of order invertible in $\scrO_X$; and
\item[(ii)] $f$ and $u$ induce a morphism $Y\rightarrow X\times_{X\langle P\rangle}X\langle Q\rangle$ of log adic spaces whose underlying morphism of adic spaces is \'etale.
\end{enumerate}
A locally noetherian fs log adic space $X$ is \textbf{log smooth} if the structure morphism $X\rightarrow \Spa(k, \calO_k)$ is log smooth, where $\Spa(k, \calO_k)$ is equipped with the trivial log structure.
\end{Definition}

Now we introduce the notion of Kummer \'etale morphisms and the Kummer \'etale site.

\begin{Definition}\label{Definition: Kummer etale morphism}
\begin{enumerate}
\item[(i)] An injective homomorphism $u:P\rightarrow Q$ of fs monoids is called \textbf{Kummer} if for every $a\in Q$, there exists some integer $n >0$ such that $na\in u(P)$.
\item[(ii)] A morphism (resp., finite morphism) $f: Y\rightarrow X$ of locally noetherian fs log adic spaces is called \textbf{Kummer \'etale} (resp., \textbf{finite Kummer \'etale}) if \'etale locally on $X$ and $Y$, $f$ admits an fs chart $u:P\rightarrow Q$ which is Kummer with $|Q^{\mathrm{gp}}/u^{\mathrm{gp}}(P^{\mathrm{gp}})|$ invertible on $\scrO_Y$, and such that $f$ and $u$ induce a morphism $Y\rightarrow X\times_{X\langle P\rangle} X\langle Q\rangle$ of log adic spaces whose underlying morphism of adic spaces is \'etale (resp., finite \'etale).
\item[(iii)] If a Kummer \'etale (resp., finite Kummer \'etale) morphism is strict, we say it is \textbf{strictly \'etale} (resp., \textbf{strictly finite \'etale}).
\end{enumerate}
\end{Definition}

\begin{Remark}
\normalfont By \cite[Lemma 4.1.10]{Diao}, if $f:Y\rightarrow X$ is a Kummer \'etale morphism between locally noetherian fs log adic spaces and if $X$ admits a chart modeled on a sharp fs monoid $P$, then, \'etale locally on $X$ and $Y$, the morphism $f$ admits a Kummer fs chart $P\rightarrow Q$ with $Q$ being sharp.
\end{Remark}

\begin{Definition} 
Let $X$ be a locally noetherian fs log adic space. The \textbf{Kummer \'etale site} $X_{\ket}$ (resp., \textbf{finite Kummer \'etale site} $X_{\textrm{fket}}$) of $X$ is defined as follows. The underlying category is the full subcategory of the category of locally noetherian fs log adic spaces consisting of objects that are Kummer \'etale (resp., finite Kummer \'etale) over $X$. The coverings are given by the topological coverings. 

The \textbf{structure sheaf} $\scrO_{X_{\ket}}$ (resp., \textbf{integral structure sheaf} $\scrO_{X_{\ket}}^{+})$ on $X_{\ket}$ is defined to be the presheaf sending $U\mapsto \scrO_U(U)$ (resp., $U\mapsto \scrO_U^+(U)$). We also write $\scrM_{X_{\ket}}$ for the presheaf sending $U\mapsto \scrM_U(U)$. By \cite[Theorem 4.3.1, Proposition 4.3.4]{Diao}, these presheaves are indeed sheaves. 
\end{Definition}

\begin{Proposition}[Corollary 4.4.18, \cite{Diao}]\label{Proposition: Kummer etale fundamental group}
Let $X$ be a connected locally noetherian fs log adic space and let $\xi$ be a log geometric point (see \cite[Definition 4.4.2]{Diao}). Then there is an equivalence of categories
\[F_{X}: X_{\text{fk\'et}}\xrightarrow[]{\sim} \pi_1^{\ket}(X, \xi)-\FSets\] sending $Y\mapsto Y_{\xi}:=\Hom_{X}(\xi, Y)$, where the $\pi_1^{\ket}(X, \xi)-\FSets$ denotes the category of finite sets equipped with a continuous action of the \textbf{Kummer \'etale fundamental group} $\pi_1^{\ket}(X, \xi)$.

For any two log geometric points $\xi$ and $\xi'$, the fundamental groups $\pi_1^{\ket}(X, \xi)$ and $\pi_1^{\ket}(X, \xi')$ are isomorphic. Hence, we may omit ``$\xi$'' from the notation whenever the context is clear.
\end{Proposition}

\begin{Lemma}\label{Kummer etale Galois cover}
Assume $k$ is of characteristic $0$. Let $X$ and $Y$ be locally noetherian fs log adic spaces whose underlying adic spaces are smooth connected rigid analytic varieties over $k$. Suppose the log structures on $X$ and $Y$ are the divisorial log structures associated with the normal crossing divisors $D\subset X$ and $E\subset Y$ as in Example \ref{Example: normal crossings}. Let $U=X\smallsetminus D$ and $V=Y\smallsetminus E$. Suppose we have a finite Kummer \'etale surjective morphism $f:Y\rightarrow X$ such that $f^{-1}(U)=V$ and that $f|_{V}:V\rightarrow U$ is a finite \'etale Galois cover with Galois group $G$. Then $f$ is a finite Kummer \'etale Galois cover with Galois group $G$.
\end{Lemma}

\begin{proof}
According to Proposition \ref{Proposition: Kummer etale fundamental group}, we have equivalences of categories
\[F_X:X_{\text{fk\'et}}\xrightarrow[]{\sim} \pi_1^{\ket}(X)-\FSets\]
and 
\[F_U: U_{\text{f\'et}}\xrightarrow[]{\sim} \pi_1^{\et}(U)-\FSets.\]
We have to show that $G$ is a finite quotient of $\pi_1^{\ket}(X)$ and, under the equivalence $F_X$, $Y$ corresponds to the finite set $G$ equipped with the natural $\pi_1^{\ket}(X)$-action.

By \cite[Proposition 4.2.1]{Diao} and \cite[Theorem 1.6]{Hansen-2020}, we have an equivalence of categories between $X_{\text{fk\'et}}$ and $U_{\text{f\'et}}$, under which $Y$ corresponds to $V$. It also induces a natural isomorphism $\pi_1^{\ket}(X)\cong \pi_1^{\et}(U)$ making the following diagram commutative. 
 \[
        \xymatrix{ {X_{\text{fk\'et}}} \ar^-{\sim}[r] \ar^{F_X}_{\sim}[d] & {U_{\text{f\'et}}} \ar^-{F_U}_{\sim}[d] \\
        {\pi_1^{\ket}(X)-\FSets} \ar^-{\sim}[r] & {\pi_1^{\et}(U)-\FSets} }
    \]
Since $V$ corresponds to the finite set $G$ under the equivalence $F_U$, we are done.
\end{proof}

Finally, we introduce the pro-Kummer \'etale site. For the rest of \S \ref{subsection: review of Kummer etale and pro-Kummer etale sites}, the nonarchimedean field $k$ is assumed to be an extension of $\Q_p$. 
\begin{Definition}
Let $X$ be a locally noetherian fs log adic space over $\Spa(k,\calO_k)$.
\begin{enumerate}
\item[(i)] The \textbf{pro-Kummer \'{e}tale site} $X_{\proket}$ of $X$ is defined as follows. The underlying category is the full subcategory of $\textrm{pro-}X_{\ket}$ consisting of cofiltered inverse limit $Y=\varprojlim_{i\in I}Y_i$ with $Y_i\in X_{\ket}$ such that the transition morphisms $Y_i\rightarrow Y_j$ are finite Kummer \'{e}tale and are surjective for sufficiently large $i$. Such an inverse limit if called a \textbf{pro-Kummer \'etale presentation} of $Y$. As for the coverings, we refer the readers to \cite[Definition 5.1.1, 5.1.2]{Diao} for details.
\item[(ii)] There is a natural projection of sites \[\nu: X_{\proket}\rightarrow X_{\ket}.\] The \textbf{structure sheaves} on $X_{\proket}$ are given by
$$
\scrO_{X_{\proket}}^+:=\nu^{-1}\scrO_{X_{\ket}}^+, \quad \scrO_{X_{\proket}}:=\nu^{-1}\scrO_{X_{\ket}}$$
and the \textbf{completed structure sheaves} are given by
$$
\quad \widehat{\scrO}_{X_{\proket}}^+:=\varprojlim_{n}\left(\scrO_{X_{\proket}}/p^n\right), \quad \widehat{\scrO}_{X_{\proket}}:=\widehat{\scrO}_{X_{\proket}}^+[1/p].
$$
We also write $\scrM_{X_{\proket}}:=\nu^{-1}(\scrM_{\ket})$ together with a natural morphism $\alpha:\scrM_{\proket}\rightarrow \scrO_{X_{\proket}}$.
\end{enumerate}
\end{Definition}
The pro-Kummer \'etale topology admits a convenient basis consisting of the \emph{log affinoid perfectoid objects}.

\begin{Definition}
An object $U$ in $X_{\proket}$ is called \textbf{log affinoid perfectoid} if it admits a pro-Kummer \'etale presentation $U = \varprojlim_{i \in I} U_i$ such that
\begin{enumerate}
 \item[(i)] There is an initial object $0 \in I$;
 \item[(ii)] Each $U_i= (\Spa(R_i, R_i^+)$ is affinoid and admits a chart modeled on a sharp fs monoid $P_i$ such that each transition morphism $U_j \rightarrow U_i$ is modeled on a Kummer chart $P_i \to P_j$;
\item[(iii)] The affinoid algebra $(R, R^+) := \big(\varinjlim_{i \in I} \, (R_i, R_i^{+})\big)^\wedge$ is a perfectoid affinoid algebra, where the completion is with respect to the $p$-adic topology;
\item[(iv)] The monoid $P := \varinjlim_{i \in I} P_i$ is \textbf{$n$-divisible}, for all $n \in \Z_{\geq 1}$. Namely, the $n$-th multiple map $[n]:P\rightarrow P$ is surjective for all $n\in\Z_{\geq 1}$.
    \end{enumerate}
    Such a presentation $U = \varprojlim_{i\in I} U_i$ is called a \textbf{perfectoid presentation} of $U$.
\end{Definition}

\begin{Proposition}[Proposition 5.3.12, \cite{Diao}]
The log affinoid perfectoid objects in $X_{\proket}$ form a basis of the pro-Kummer \'etale site.
\end{Proposition}

\begin{Proposition}[Theorem 5.4.3, \cite{Diao}]\label{Proposition: almost vanishing}
Let $U\in X_{\proket}$ be a log affinoid perfectoid object, with the associated perfectoid space $\widehat{U}=\Spa(R, R^+)$. Then
\begin{enumerate}
\item[(i)] For each $n\in \Z_{\geq 1}$, we have $\scrO^+_{X_{\proket}}(U)/p^n\cong R^+/p^n$, and it is canonically almost isomorphic to $(\scrO^+_{X_{\proket}}/p^n)(U)$.
\item[(ii)] For each $n\in \Z_{\geq 1}$ and $i\in \Z_{\geq 1}$, $H^i(U, \scrO^+_{X_{\proket}}/p^n)$ is almost equal to zero. Consequently, $H^i(U, \widehat{\scrO}^+_{X_{\proket}})$ is almost equal to zero.
\item[(iii)] We have $\widehat{\scrO}^+_{X_{\proket}}(U)\cong R^+$ and $\widehat{\scrO}_{X_{\proket}}(U)\cong R$. Moreover, $\widehat{\scrO}^+_{X_{\proket}}(U)$ is canonically isomorphic to the $p$-adic completion of $\scrO^+_{X_{\proket}}(U)$.
\end{enumerate}
\end{Proposition}

\begin{Example}\label{Example: basic example of profinite Galois cover}
\normalfont We recall the following example from \cite[\S 6]{Diao}. Let $P$ be a sharp fs monoid. Consider
\[\mathbb{E}:=\Spa(\C_p\langle P\rangle, \calO_{\C_p}\langle P\rangle)\]
equipped with the natural log structure modeled on chart $P$. (If $P=\Z_{\geq 0}^n$, then $\mathbb{E}$ is just the $n$-dimensional unit disc in Example \ref{Example: unit disc}.) For each $m\in \Z_{>0}$, let $\frac{1}{m}P$ denote the sharp fs monoid containing $P$ such that the inclusion $P\hookrightarrow \frac{1}{m}P$ is isomorphic to the $m$-th multiple map $[m]:P\rightarrow P$. Define
\[\mathbb{E}_m:=\Spa(\C_p\langle \frac{1}{m}P\rangle, \calO_{\C_p}\langle \frac{1}{m}P\rangle)\]
equipped with the natural log structure modeled on the chart $\frac{1}{m}P$.  If $m|m'$, there is a natural finite Kummer \'etale morphism $\mathbb{E}_{m'}\rightarrow \mathbb{E}_{m}$ modeled on the chart $\frac{1}{m}P\hookrightarrow \frac{1}{m'}P$. According to \cite[Proposition 4.1.6]{Diao}, the morphism $\mathbb{E}_m\rightarrow \mathbb{E}$ is actually finite Kummer \'etale Galois with Galois group 
$$
\Gamma_{/m}:=\Hom\big((\frac{1}{m}P)^{\mathrm{gp}}/P^{\mathrm{gp}}, \boldsymbol{\mu}_{\infty}\big),
$$
where $\boldsymbol{\mu}_{\infty}$ denotes the group of all roots of unity in $\C_p$. Let $P_{\Q_{\geq 0}}:=\varinjlim_m (\frac{1}{m}P)$. It turns out
\[\widetilde{\mathbb{E}}:=\varprojlim_m \mathbb{E}_m\in \mathbb{E}_{\proket}\] is a log affinoid perfectoid object, with associated perfectoid space
\[\widehat{\widetilde{\mathbb{E}}}=\Spa(\C_p\langle P_{\Q_{\geq 0}}\rangle, \calO_{\C_p}\langle P_{\Q_{\geq 0}}\rangle).\]\qed
\end{Example}

Following \cite[Definition 6.1.2]{Diao}, a pro-Kummer \'{e}tale cover $Y\rightarrow X$ is called a \textit{\textbf{Galois cover with (profinite) Galois group $G$}} if there exists a presentation $Y=\varprojlim_{i} Y_i$ such that each $Y_i\rightarrow X$ is a finite Kummer \'{e}tale cover with Galois group $G_i$ and $G\cong \varprojlim_i G_i$.

For example, $\widetilde{\mathbb{E}}$ is a Galois cover over $\mathbb{E}$ with profinite Galois group \[\Gamma\cong \varprojlim_m \Gamma_{/m}=\varprojlim_m \Hom((\frac{1}{m}P)^{\mathrm{gp}}/P^{\mathrm{gp}}, \boldsymbol{\mu}_{\infty}) \cong\Hom(P^{\mathrm{gp}}_{\Q\geq 0}/P^{\mathrm{gp}}, \boldsymbol{\mu}_{\infty})\]
(cf. \cite[(6.1.4)]{Diao}).

\subsection{Calculation of \texorpdfstring{$R^i\nu_*\widehat{\scrO}_{X_{\proket}}$}{Lg}}\label{subsection: main result}
In this subsection, we study the sheaves $R^i\nu_*\widehat{\scrO}_{X_{\proket}}$ following the calculations in \cite{Scholze-perfectoid-survey},  \cite{Scholze_2013}, and \cite{CHJ-2017}, but in the context of log adic spaces. Throughout this subsection, $X$ is an fs log adic space that is log smooth over $\Spa(\C_p, \calO_{\C_p})$ (cf. Definition \ref{Definition: log smooth}). We will omit the subscript ``prok\'et'' from $\scrO_{X_{\proket}}$, $\scrO^+_{X_{\proket}}$, $\widehat{\scrO}_{X_{\proket}}$, $\widehat{\scrO}^+_{X_{\proket}}$, etc., whenever this causes no confusion.

By definition, $R^i\nu_*\widehat{\scrO}_{X}$ (resp., $R^i\nu_*\widehat{\scrO}^+_{X}$) is the sheaf on $X_{\ket}$ associated with the presheaf 
\[U\mapsto H^i(U_{\proket}, \widehat{\scrO}_{X})\,\,\,\,\,\,\,\,\big(\textrm{resp., }U\mapsto H^i(U_{\proket}, \widehat{\scrO}^+_{X})\big).\]

For every $U\in X_{\ket}$, $U$ is log smooth over $\Spa(\C_p, \calO_{\C_p})$. By \cite[Proposition 3.1.10]{Diao}, \'etale locally on $U$ there exists a \textit{\textbf{toric chart}} $U\rightarrow \mathbb{E}=\Spa(\C_p\langle P\rangle, \calO_{\C_p}\langle P\rangle)$ for some sharp fs monoid $P$,  {\it i.e.,} a strictly \'etale morphism $U\rightarrow \mathbb{E}=\Spa(\C_p\langle P\rangle, \calO_{\C_p}\langle P\rangle)$ that is a composition of rational localisations and finite \'etale morphisms. For such toric charts, we are able to calculate $H^i(U_{\proket}, \widehat{\scrO}^+_{X})$ and $H^i(U_{\proket}, \widehat{\scrO}_{X})$ in an explicit way.

\begin{Lemma}\label{Lemma: Cartan-Leray}
Suppose $U\in X_{\ket}$ is equipped with a toric chart $U\rightarrow \mathbb{E}$ as above. 
\begin{enumerate}
\item[(i)] For every $i\in \Z_{\geq 0}$ and $m\in\Z_{\geq 1}$, there is a natural injection
\[H^i_{\cts}(\Gamma, \scrO^+_{X_{\ket}}(U)/p^m)^a\hookrightarrow H^i(U_{\proket}, \scrO^+_{X}/p^m)^a\]
with cokernel killed by $p$, where $\Gamma$ is equipped with the profinite topology, $\scrO^+_{X_{\ket}}(U)/p^m$ is equipped with the discrete topology, and $\Gamma$ acts trivially on $\scrO^+_{X_{\ket}}(U)/p^m$.
\item[(ii)] 
For every $i\in\Z_{\geq 0}$, there is a natural injection
\[H^i_{\cts}(\Gamma, \scrO^+_{X_{\ket}}(U))^a\hookrightarrow H^i(U_{\proket}, \widehat{\scrO}^+_{X})^a\]
with cokernel killed by $p$, where $\Gamma$ is equipped with the profinite topology, $\scrO^+_{X_{\ket}}(U)$ is equipped with the $p$-adic topology, and $\Gamma$ acts trivially on $\scrO^+_{X_{\ket}}(U)$. By inverting $p$, we obtain an isomorphism
\[H^i_{\cts}(\Gamma, \scrO_{X_{\ket}}(U))\xrightarrow[]{\sim} H^i(U_{\proket}, \widehat{\scrO}_{X}).\]
\item[(iii)]
For every $i\in \Z_{\geq 0}$ and $m\in \Z_{\geq 1}$, by choosing an isomorphism $P^{\mathrm{gp}}\simeq \Z^n$, there is a natural almost injection $$\bigwedge^i (\scrO^+_{U_{\ket}}/p^m)^n\hookrightarrow R^i\nu_* (\scrO^+_{U_{\proket}}/p^m)$$ whose cokernel is killed by $p$. This induces a natural almost injection $$\bigwedge^i (\scrO^+_{U_{\ket}})^n\hookrightarrow R^i\nu_* \widehat{\scrO}^+_{U_{\proket}}$$
with cokernel killed by $p$. Inverting $p$, we obtain an isomorphism
$$\bigwedge^i (\scrO_{U_{\ket}})^n\simeq R^i\nu_* \widehat{\scrO}_{U_{\proket}}.$$ In particular, the sheaf $R^i\nu_*\widehat{\scrO}_{X}$ is a locally free $\scrO_{X_{\ket}}$-module. 
\end{enumerate}
\end{Lemma}

\begin{proof}
\begin{enumerate}
\item[(i)] Recall the log affinoid perfectoid Galois cover $\widetilde{\mathbb{E}}\rightarrow \mathbb{E}$ (with profinite Galois group $\Gamma$) constructed in Example \ref{Example: basic example of profinite Galois cover}. Consider \[\widetilde{U}:=U\times_{\mathbb{E}}\widetilde{\mathbb{E}}\in X_{\proket}.\]
By \cite[Lemma 5.3.8]{Diao}, $\widetilde{U}$ is also log affinoid perfectoid and $\widetilde{U}\rightarrow U$ is a Galois cover with the same Galois group. We obtain the Cartan--Leray spectral sequence (see \cite[Remark 2.25]{CHJ-2017}) \[E_2^{i,j}= H_{\cts}^i(\Gamma, H^j(\widetilde{U}, \scrO_{X}^+/p^m))\Rightarrow H^{i+j}(U_{\proket}, \scrO_{X}^+/p^m).\] By Proposition \ref{Proposition: almost vanishing} (ii), $H^j(\widetilde{U}, \scrO_{X}^+/p^m)$ is almost zero for all $j\in \Z_{\geq 1}$. Therefore, we have an almost isomorphism 
\begin{equation}\label{eq: almost isomorphism Cartain-Leray}
H^i(U_{\proket}, \scrO_{X}^+/p^m)^a\simeq H_{\cts}^i(\Gamma, (\scrO_{X}^+/p^m)(\widetilde{U}))^a.
\end{equation}

On the other hand, by \cite[Lemma 6.1.7, Remark 6.1.8]{Diao}, the natural morphism 
\[H^i(\Gamma, (\scrO_{X_{\ket}}^+/p^m)(U))\rightarrow H^i(\Gamma, (\scrO_{X}^+/p^m)(\widetilde{U}))\]
is injective with cokernel killed by $p$ for all $i\in \Z_{\geq 0}$. Combining this with the almost isomorphism (\ref{eq: almost isomorphism Cartain-Leray}), we obtain the desired almost injection.

\item[(ii)] By an almost version of \cite[Lemma 3.18]{Scholze_2013} and Proposition \ref{Proposition: almost vanishing} (i) (ii), we see that the inverse system $\{\scrO_{X}^+/p^m: m\in \Z_{\geq 1}\}$ has almost vanishing higher inverse limits on the pro-Kummer \'{e}tale site. Therefore, we obtain almost isomorphisms
\[H^i(U_{\proket}, \widehat{\scrO}_{X}^+)^a\cong \varprojlim_m H^i(U_{\proket}, \scrO_{X}^+/p^m)^a \simeq \varprojlim_m H_{\cts}^i(\Gamma, (\scrO_{X}^+/p^m)(\widetilde{U}))^a.\]

On the other hand, for every $i\in \Z_{\geq 0}$, we claim that there is a natural isomorphism
\[H_{\cts}^i(\Gamma, \scrO_{X_{\ket}}^+(U))\cong \varprojlim_m H_{\cts}^i(\Gamma, (\scrO_{X_{\ket}}^+/p^m)(U)).\] Indeed, by the same arguments as in the proof of \cite[Theorem 2.7.5]{NSW-cohomology}, there is a short exact sequence \[0\rightarrow R^1\varprojlim_{m} H_{\cts}^{i-1}(\Gamma, (\scrO_{X_{\ket}}^+/p^m)(U))\rightarrow H_{\cts}^{i}(\Gamma, \scrO_{X_{\ket}}^+(U))\rightarrow \varprojlim_m H_{\cts}^{i}(\Gamma, (\scrO_{X_{\ket}}^+/p^m)(U))\rightarrow 0.\] 
It suffices to show that 
\[R^1\varprojlim_m H_{\cts}^{i-1}(\Gamma, (\scrO_{X_{\ket}}^+/p^m)(U))=0.\] 
Notice that $P^{\mathrm{gp}}$ is a finitely generated torsion-free abelian group. By choosing a $\Z$-basis of $P^{\mathrm{gp}}$, we obtain an isomorphism $\Gamma\cong \widehat{\Z}(1)^n$ of profinite groups which induces an isomorphism
\[
H^{i-1}_{\cts}(\Gamma, (\scrO^+_{X_{\ket}}/p^m)(U))\simeq \bigwedge^{i-1}(\scrO^+_{X_{\ket}}(U)/p^m)^n.
\]
Thus, for every $m'>m$, the transition map \[H_{\cts}^{i-1}(\Gamma, (\scrO_{X_{\ket}}^+/p^{m'})(U))\rightarrow H_{\cts}^{i-1}(\Gamma, (\scrO_{X_{\ket}}^+/p^{m})(U))\] is a surjection. Hence, the inverse system $\{H_{\cts}^{i-1}(\Gamma, (\scrO_{X_{\ket}}^+/p^m)(U)): m\in \Z_{>0}\}$ satisfies the Mittag-Leffler condition which yields the desired vanishing of $R^1\lim$.

Putting everything together, we obtain a natural injection
$$
H_{\cts}^i(\Gamma, \scrO_{X_{\ket}}^+(U))^a\cong \varprojlim_m H_{\cts}^i(\Gamma, (\scrO_{X_{\ket}}^+/p^m)(U))^a\hookrightarrow \varprojlim_m H_{\cts}^i(\Gamma, (\scrO_{X}^+/p^m)(\widetilde{U}))^a\cong H^i(U_{\proket}, \widehat{\scrO}_{X}^+)^a
$$
whose cokernel is killed by $p$, as desired.

\item[(iii)] We show that the restriction of $R^i\nu_*\widehat{\scrO}_X$ on $U_{\ket}$ is isomorphic to the free $\scrO_{U_{\ket}}$-module $\bigwedge^i (\scrO_{U_{\ket}})^n$. In fact, as a byproduct of the computation above, we have isomorphisms (depending on the fixed choice of the identification $\Gamma\cong \widehat{\Z}(1)^n$)
$$H^i_{\cts}(\Gamma, \scrO^+_{X_{\ket}}(U))\cong \varprojlim_m H_{\cts}^i(\Gamma, (\scrO_{X_{\ket}}^+/p^m)(U))\simeq \varprojlim_m\bigwedge^i(\scrO^+_{X_{\ket}}(U)/p^m)^n=\bigwedge^i(\scrO^+_{X_{\ket}}(U))^n.$$
Inverting $p$, we obtain an isomorphism
$$H^i(U_{\proket}, \widehat{\scrO}_X)\simeq H^i_{\cts}(\Gamma, \scrO_{X_{\ket}}(U))\simeq \bigwedge^i(\scrO_{X_{\ket}}(U))^n.$$

Consider $V\in U_{\ket}$ such that $V\rightarrow U$ admits a chart $P\rightarrow P'$ and such that $V\rightarrow U$ factors as 
$$V\rightarrow U'\times_{U'\langle P\rangle}U'\langle P'\rangle \rightarrow U' \rightarrow U$$
where 
\begin{itemize}
\item $U'\subset U$ is a strictly \'etale morphism which is a composition of finite \'etale morphisms and rational localisations;
\item $P\rightarrow P'$ is isomorphic to the $m$-th multiple map $[m]: P\rightarrow P$;
\item $V\rightarrow U'\times_{U'\langle P\rangle}U'\langle P'\rangle$ is a strictly \'etale morphism which is a composition of finite \'etale morphisms and rational localisations.
\end{itemize}

Notice that such a $V$ admits a toric chart $V\rightarrow \mathbb{E}'$ where $\mathbb{E}'=\Spa(\C_p\langle P'\rangle, \calO_{\C_p}\langle P'\rangle)$. Repeating the argument above, we arrive at an isomorphism
$$H^i_{\cts}(\Gamma', \scrO_{X_{\ket}}(V))\simeq H^i(V_{\proket}, \widehat{\scrO}_X)$$
where 
$$\Gamma':= \Hom (P'^{\mathrm{gp}}_{\Q\geq 0}/P'^{\mathrm{gp}}, \boldsymbol{\mu}_{\infty}).$$

In addition, the injection $P\rightarrow P'$ induces an injection $\Gamma'\rightarrow \Gamma$ which is isomorphic to multiplication by $m$. The fixed identification $\Gamma\cong \widehat{\Z}(1)^n$ then identifies $\Gamma'\rightarrow \Gamma$ with the $m$-th multiple map $[m]: \widehat{\Z}(1)^n\rightarrow \widehat{\Z}(1)^n$. We arrive at the following commutative diagram
$$\begin{tikzcd}
\bigwedge^i(\scrO_{X_{\ket}}(U))^n \arrow[d] \arrow[r, "\simeq"] & H^i_{\cts}(\Gamma, \scrO_{X_{\ket}}(U))\arrow[d] \arrow[r, "\simeq"] &H^i(U_{\proket}, \widehat{\scrO}_X) \arrow[d]
\\
\bigwedge^i(\scrO_{X_{\ket}}(V))^n \arrow[r, "\simeq"] & H^i_{\cts}(\Gamma', \scrO_{X_{\ket}}(V))\arrow[r, "\simeq"] &H^i(V_{\proket}, \widehat{\scrO}_X)
 \end{tikzcd}$$
 
Finally, let $\mathcal{B}_U$ denote the collection of such $V$'s. Notice that every Kummer map $P\rightarrow Q$ between sharp fs monoids factors through $[m]:P\rightarrow P$ for some $m\in \Z_{\geq 1}$. Hence, every $W\in U_{\ket}$ is covered by elements in $\mathcal{B}_U$. This is enough to conclude that the sheafification of the presheaf $W\rightarrow H^i(W_{\proket}, \widehat{\scrO}_X)$ on $U_{\ket}$ is isomorphic to the free sheaf $\bigwedge^i(\scrO_{U_{\ket}})^n$. This completes the proof. 
\end{enumerate}
\end{proof}

We also provide a coordinate-free description of $R^i\nu_*\widehat{\scrO}_{X}$. The following result is a logarithmic version of \cite[Proposition 3.23, Lemma 3.24]{Scholze-perfectoid-survey}.

\begin{Lemma}\label{Lemma: log analogue of Lemma 3.24 in Scholze's survey paper}
For every $n\in\Z_{\geq 1}$, let $\mu_{p^n}$ be the sheaf of $p^n$-th roots of unity on $X_{\ket}$. Consider the $\Z_p$-local system $\Z_p(1):=\varprojlim_{n}\mu_{p^n}$ on $X_{\ket}$ and let $\widehat{\Z}_p(1):=\nu^{-1}\Z_p(1)$ be the associated $\widehat{\Z}_p$-local system on $X_{\proket}$ (cf. \cite[Definition 6.3.2]{Diao}). The short exact sequence \[0\rightarrow \widehat{\Z}_p(1)\rightarrow \varprojlim_{x\mapsto x^p}\scrM_{X_{\proket}}\rightarrow \scrM_{X_{\proket}}\rightarrow 0\] induces a boundary map \[\scrM_{X_{\ket}}=\nu_*\scrM_{X_{\proket}}\rightarrow R^1\nu_*\widehat{\Z}_p(1).\] Then, there exists a unique $\scrO_{X_{\ket}}$-linear morphism $\Omega_{X_{\ket}}^{\log, 1}\rightarrow R^1\nu_*\widehat{\scrO}_{X}(1)$ such that the diagram 
  \[
        \xymatrix{{\scrM_{X_{\ket}}} \ar^-{}[r] \ar^{\mathrm{dlog}}[d] & {R^1\nu_*\widehat{\Z}_p(1)} \ar^-{}[d] \\
        {\Omega_{X_{\ket}}^{\log, 1}} \ar^-{}[r] & {R^1\nu_*\widehat{\scrO}_{X}(1)} }
    \]
is commutative, where $\Omega_{X_{\ket}}^{\log, 1}$  is the sheaf of log differentials defined in \cite[Definition 3.2.25]{Diao}.

Moreover, this morphism is an isomorphism. As a corollary, by taking cup product and exterior product, we obtain a canonical isomorphism $R^i\nu_*\widehat{\scrO}_{X}\cong \Omega_{X_{\ket}}^{\log,i}(-i)$ for every $i\geq 1$.
\end{Lemma}
\begin{proof}
The proof follows almost verbatim from the proof of \cite[Lemma 3.24]{Scholze-perfectoid-survey}, except that we have to replace the short exact sequence in \cite[Corollary 6.14]{Scholze_2013} by the short exact sequence in \cite[Corollary 2.4.5]{Diao-Lan-Liu-Zhu}. Here we only give a sketch of the proof.

Firstly, since the question is \'etale local, we may assume that $X$ admits a toric chart $X\rightarrow \Spa(\C_p\langle P\rangle, \calO_{\C_p}\langle P\rangle)$ for some sharp fs monoid $P$. Secondly, the desired $\scrO_{X_{\ket}}$-linear morphism, if exists, must be unique because $\Omega^{\log, 1}_{X_{\ket}}$ is a locally free $\scrO_{X_{\ket}}$-module generated by the image of $\textrm{dlog}$. It remains to show the existence.

Consider the map of short exact sequences
 \[
        \xymatrix{ 0 \ar^-{}[r] & {\widehat{\Z}_p(1)} \ar^-{}[r] \ar^{}[d] & {\varprojlim_{x\mapsto x^p}\scrM_{X_{\proket}}} \ar^-{}[r] \ar^-{}[d] & {\scrM_{X_{\proket}}}  \ar^-{}[r] \ar^-{\textrm{dlog}}[d] & 0 \\
        0 \ar^-{}[r] & {\widehat{\scrO}_{X}(1)} \ar^-{}[r]  & {\textrm{gr}^1 \scrO\!\mathbb{B}^+_{\textrm{dR}, \log,X}} \ar^-{}[r] & {\widehat{\scrO}_{X}\otimes_{\scrO_{X_{\ket}}} \Omega_{X_{\ket}}^{\log, 1}}  \ar^-{}[r]  & 0 }
    \]
    where the lower sequence is from \cite[Corollary 2.4.5]{Diao-Lan-Liu-Zhu}. The vertical map in the middle sends an element $a\in \varprojlim_{x\mapsto x^p}\calM_{X_{\proket}}$ to 
\[\log(\mathbf{e}^a):=-\sum_{n=1}^{\infty}\frac{1}{n}(1-\mathbf{e}^a)^n\in \textrm{Fil}^1\scrO\!\mathbb{B}^+_{\textrm{dR}, \log,X}\] and hence maps to $\textrm{gr}^1 \scrO\!\mathbb{B}^+_{\textrm{dR}, \log,X}$ (see \cite[\S 2.2]{Diao-Lan-Liu-Zhu} for the definition of $\mathbf{e}^a$). It is straightforward to check the commutativity of the diagram. This diagram then induces a diagram of boundary maps
  \[
        \xymatrix{{H^0(X_{\proket}, \scrM_{X_{\proket}})} \ar^-{}[r] \ar^{\textrm{dlog}}[d] & {H^1(X_{\proket}, \widehat{\Z}_p(1))} \ar^-{}[d] \\
        {H^0(X_{\proket}, \widehat{\scrO}_{X}\otimes_{\scrO_{X_{\ket}}} \Omega_{X_{\ket}}^{\log, 1})} \ar^-{}[r] & {H^1(X_{\proket}, \widehat{\scrO}_{X}(1))} }
    \]
    which provides the desired morphism. 
    
    To check that this map is an isomorphism, we fix an identification $P^{\mathrm{gp}}\simeq \Z ^n=\bigoplus_{j=1}^n \Z e_j$ which induces an isomorphism
 $$\Gamma:=\Hom(P^{\mathrm{gp}}_{\Q\geq 0}/P^{\mathrm{gp}}, \boldsymbol{\mu}_{\infty})\cong \widehat{\Z}(1)^n.$$ 
Notice that each $e_j$ can be written as a $\Z$-linear combination of elements in $P$; \emph{i.e.}, $e_j=\sum_{t=1}^m a_t p_t$ for some $a_t\in \Z$ and $p_t\in P$. We define $\mathrm{dlog}(e_j):=\sum_{t=1}^m a_t\mathrm{dlog}(p_t)$ where we have identify $p_t$ with its image in $\scrM_{X_{\ket}}$. One checks that $\mathrm{dlog}(e_j)$ is independent of the choice of the $\Z$-linear combination and the $\mathrm{dlog}(e_j)$'s form a basis for the free $\scrO_{X_{\ket}}$-module $\Omega^{\log, 1}_{X_{\ket}}$.
 
 On the other hand, by the computation in the proof of Lemma \ref{Lemma: Cartan-Leray}, the identification $\Gamma\simeq \widehat{\Z}(1)^n$ induces an isomorphism $R^1\nu_* \widehat{\scrO}_X\simeq \scrO_{X_{\ket}}^n=\bigoplus_{j=1}^n \scrO_{X_{\ket}}\epsilon_j$. Direct computation shows that the map $\Omega^{\log,1}_{X_{\ket}}\rightarrow R^1\nu_* \widehat{\scrO}_X$ sends $\mathrm{dlog}(e_j)$ to $\epsilon_j$, for every $j=1, \ldots, n$. This finishes the proof.
\end{proof}

To wrap up this subsection, we include a logarithmic analogue of \cite[Proposition 6.8]{CHJ-2017} which suggests that the calculation of $R^i\nu_*\widehat{\scrO}_X$ is compatible with the ``mixed completed tensor''. 

\begin{Proposition}\label{Proposition: compatibility with completed tensor}
Let $M$ be a profinite flat $\calO_K$-module in the sense of \cite[Definition 6.1]{CHJ-2017}. Then there is a canonical isomorphism
\[R^i\nu_*(\widehat{\scrO}_{X}\widehat{\otimes}M)\cong (R^i\nu_*\widehat{\scrO}_{X})\widehat{\otimes} M,\]
where $\widehat{\otimes}$ stands for the ``mixed completed tensor'' in the sense of \cite[Definition 6.6]{CHJ-2017}. Here, the mixed completed tensor on the right hand side is with respect to the subsheaf $\mathrm{Im}(R^i\nu_*\widehat{\scrO}^+_X\rightarrow R^i\nu_*\widehat{\scrO}_{X})\subset R^i\nu_*\widehat{\scrO}_{X}$.

Consequently, by Lemma \ref{Lemma: log analogue of Lemma 3.24 in Scholze's survey paper}, we have
\[R^i\nu_*(\widehat{\scrO}_{X}\widehat{\otimes}M)\cong \Omega^{\log, i}_{X_{\ket}}(-i)\widehat{\otimes}M.\]
\end{Proposition}

\begin{proof}
The proof follows verbatim as in the proof of \cite[Proposition 6.8]{CHJ-2017} as long as we replace \cite[Lemma 6.11(1)(2)]{CHJ-2017} by Lemma \ref{Lemma: Cartan-Leray}.
\end{proof}

\subsection{Banach sheaves and a (generalised) projection formula}\label{subsection: generalised projective formula} 
In this subsection, we introduce the notion of \emph{Banach sheaves} on the Kummer \'etale topology of a log adic space, generalising the ones studied in \cite[\S A]{AIP-2015} and \cite[\S 2]{Boxer--Pilloni--higherColeman}. Then, for certain \emph{admissible} Banach sheaves, we prove a projection formula which will be used in the main body of the paper.

Recall from Definition \ref{Definition: weights} that a \emph{small $\Z_p$-algebra} is a $p$-torsion free reduced ring $R$ which is also a finite $\Z_p\llbrack T_1, ..., T_d\rrbrack$-algebra for some $d\in \Z_{\geq 0}$. It is a profinite flat $\Z_p$-module in the sense of \cite[Definition 6.1]{CHJ-2017}. In particular, there exists a set of elements $\{e_\sigma: \sigma\in \Sigma\}$ in $R$ such that $R\simeq \prod_{\sigma\in \Sigma}\Z_pe_\sigma$ equipped with the product topology. This set of elements $\{e_{\sigma}: \sigma\in \Sigma\}$ is called a \emph{pseudo-basis} for $R$. Moreover, $R$ is equipped with an adic profinite topology and is complete with respect to the $p$-adic topology. 

Throughout this subsection, we keep the following notations:
\begin{itemize}
\item Let $R$ be a fixed small $\Z_p$-algebra and let $\fraka$ be a fixed ideal of definition containing $p$.
\item All (log) adic spaces are assumed to be reduced and quasi-separated. In particular, $X$ either stands for a locally noetherian reduced adic space over $(\C_p, \calO_{\C_p})$ or a locally noetherian reduced fs log adic space over $(\C_p, \calO_{\C_p})$. In the second case, we use $X_{\an}$ to denote the underlying adic space of $X$.
\item We adopt the notation of ``mixed completed tensors'' $-\,\widehat{\otimes}'R$ and $-\,\widehat{\otimes}R$ as in Definition \ref{Definition: unadorned completed tensor}.
\end{itemize}

\begin{Lemma}\label{Lemma: structure sheaf mixed completed tensor}
\begin{itemize}
\item[(i)] Let $X$ be a locally noetherian adic space over $(\C_p, \calO_{\C_p})$. Then the presheaf $\scrO^+_X\widehat{\otimes}' R$ (resp., $\scrO_X\widehat{\otimes} R$) sending any quasi-compact open subset $U\subset X$ to $\scrO^+_X(U)\widehat{\otimes}'R$ (resp., $\scrO_X(U)\widehat{\otimes}R$) is a sheaf. In particular, $\scrO_X\widehat{\otimes} R$ is a sheaf of Banach $\C_p$-algebras.
\item[(ii)] Let $X$ be a locally noetherian fs log adic space over $(\C_p, \calO_{\C_p})$. Then the presheaf $\scrO^+_{X_{\ket}}\widehat{\otimes}' R$ (resp., $\scrO_{X_{\ket}}\widehat{\otimes} R$) sending any quasi-compact $U\in X_{\ket}$ to $\scrO^+_{X_{\ket}}(U)\widehat{\otimes}'R$ (resp., $\scrO_{X_{\ket}}(U)\widehat{\otimes}R$) is a sheaf. In particular, $\scrO_{X_{\ket}}\widehat{\otimes} R$ is a sheaf of Banach $\C_p$-algebras.
\item[(iii)] Let $X$ be a locally noetherian fs log adic space over $(\C_p, \calO_{\C_p})$. Then the presheaf $\widehat{\scrO}^+_{X_{\proket}}\widehat{\otimes}' R$ (resp., $\widehat{\scrO}_{X_{\proket}}\widehat{\otimes} R$) sending any qcqs $U\in X_{\proket}$ to $\widehat{\scrO}^+_{X_{\proket}}(U)\widehat{\otimes}'R$ (resp., $\widehat{\scrO}_{X_{\proket}}(U)\widehat{\otimes}R$) is a sheaf. In particular, $\widehat{\scrO}_{X_{\proket}}\widehat{\otimes} R$ is a sheaf of Banach $\C_p$-algebras.
\end{itemize}
\end{Lemma}

\begin{proof}
Choosing a presentation $R\simeq \prod_{\sigma\in \Sigma} \Z_p e_\sigma$ and using \cite[Proposition 6.4]{CHJ-2017}, the statements reduce to the sheafiness of the corresponding structure presheaves.
\end{proof}

\begin{Definition}\label{Definition: Banach modules}
Let $B$ be a Banach $\Q_p$-algebra and let $B_0$ be an open and bounded $\Z_p$-submodule. 
\begin{enumerate}
\item[(i)] A topological $B$-module $M$ is called a \textbf{Banach $B$-module} if there exists an open bounded $B_0$-submodule $M_0$ which is $p$-adically complete and separated such that $M=M_0[1/p]$.
\item[(ii)] Let $J$ be an index set. Consider the $B$-module $B(J)$ consisting of sequences $\{b_j: j\in J\}$ which converge to 0 with respect to the filter in $J$ of the complement of the finite subsets of $J$.
Then $B(J)$ is a Banach $B$-module. Indeed, let $B_0(J)$ be the $p$-adic completion of the free $B_0$-module $\bigoplus_{j\in J} B_0$. Then we have $B(J)\simeq B_0(J)[1/p]$.
\item[(iii)] A topological $B$-module $M$ is called an \textbf{orthonormalisable Banach $B$-module} (or, \textbf{ON-able Banach $B$-module} for short) if there exists a topological isomorphism $M\simeq B(J)$ for some index set $J$. A topological $B$-module $M$ is called a \textbf{projective Banach $B$-module} if it is a direct summand (as a topological $B$-module) inside an orthonormalisable Banach $B$-module.
\end{enumerate}
\end{Definition}

\begin{Definition}\label{Definition: Banach sheaf}
Let $X$ be a locally noetherian adic space over $(\C_p, \calO_{\C_p})$.
\begin{enumerate}
\item[(i)] A sheaf of topological $\scrO_{X}\widehat{\otimes} R$-modules $\scrF$ is called a \textbf{Banach sheaf of $\scrO_{X}\widehat{\otimes} R$-modules} if 
\begin{itemize}
\item for every quasi-compact open subset $U\subset X$, $\scrF(U)$ is a Banach $\scrO_X(U)\widehat{\otimes} R$-module;
\item there exists an affinoid open covering $\mathfrak{U}=\{U_i: i\in I\}$ of $X$ such that for every $i\in I$ and every affinoid open subset $V\subset U_i$, the continuous restriction map
\[\scrF(U_i)\otimes_{\scrO_X(U_i)}\scrO_X(V)\rightarrow \scrF(V)\]
induces a topological isomorphism
\[\scrF(U_i)\widehat{\otimes}_{\scrO_X(U_i)}\scrO_X(V)\rightarrow \scrF(V)\]
\end{itemize}
where the completion is with respect to the $p$-adic topology. Such a covering $\mathfrak{U}$ is called an \textbf{atlas} of $\scrF$.
\item[(ii)] A sheaf $\scrF$ as in (i) is called a \textbf{projective Banach sheaf of $\scrO_X\widehat{\otimes} R$-modules} if there exists an atlas $\mathfrak{U}=\{U_i: i\in I\}$ such that $\scrF(U_i)$'s are projective Banach $\scrO_X(U_i)\widehat{\otimes}R$-modules. 
\item[(iii)] A morphism between Banach sheaves of $\scrO_X\widehat{\otimes}R$-modules is a continuous map of sheaves of topological $\scrO_X\widehat{\otimes}R$-modules.
\item[(iv)] Let $\scrF$ be a Banach sheaf of $\scrO_X\widehat{\otimes}R$-modules as in (i). An \textbf{integral model} of $\scrF$ is a subsheaf $\scrF^+$ of $\scrO^+_X\widehat{\otimes}'R$-modules such that
\begin{itemize}
\item for every quasi-compact open $U\subset X$, $\scrF^+(U)$ is open and bounded in $\scrF(U)$;
\item $\scrF=\scrF^+[1/p]$;
\item there exists an atlas $\mathfrak{U}=\{U_i:i\in I\}$ of $\scrF$ such that, for every $i\in I$ and every affinoid open subset $V\subset U_i$, the canonical map
\[\scrF^+(U_i)\widehat{\otimes}_{\scrO^+_X(U_i)}\scrO^+_X(V)\rightarrow \scrF^+(V)\]
is an isomorphism, where the completion is with respect to the $p$-adic topology.
\end{itemize}
\end{enumerate}
\end{Definition}

We are also interested in a Kummer \'etale version of Banach sheaves.

\begin{Definition}\label{Definition: Kummer etale Banach sheaf}
Let $X$ be a locally noetherian fs log adic space of $(\C_p, \calO_{\C_p})$.
\begin{enumerate}
\item[(i)] A sheaf of topological $\scrO_{X_{\ket}}\widehat{\otimes}R$-modules $\scrF$ is called a \textbf{Kummer \'etale Banach sheaf of} $\scrO_{X_{\ket}}\widehat{\otimes}R$-\textbf{modules} if
\begin{itemize}
\item for every quasi-compact open $U\in X_{\ket}$, $\scrF(U)$ is a Banach $\scrO_{X_{\ket}}(U)\widehat{\otimes} R$-module;
\item there exists an Kummer \'etale covering $\mathfrak{U}=\{U_i: i\in I\}$ of $X$ by affinoid $U_i$'s such that for every Kummer \'etale map $V\rightarrow U_i$ with affinoid $V$, the continuous restriction map
\[\scrF(U_i)\otimes_{\scrO_{X_{\ket}}(U_i)}\scrO_{X_{\ket}}(V)\rightarrow \scrF(V)\]
induces a topological isomorphism
\[\scrF(U_i)\widehat{\otimes}_{\scrO_{X_{\ket}}(U_i)}\scrO_{X_{\ket}}(V)\rightarrow \scrF(V)\]
\end{itemize}
where the completion is with respect to the $p$-adic topology. Such a covering $\mathfrak{U}$ is called a \textbf{Kummer \'etale atlas} of $\scrF$.

\item[(ii)] A sheaf as in (i) is called a \textbf{projective Kummer \'etale Banach sheaf of} $\scrO_{X_{\ket}}\widehat{\otimes} R$-\textbf{modules} if there exists a Kummer \'etale atlas $\mathfrak{U}=\{U_i: i\in I\}$ such that $\scrF(U_i)$'s are projective Banach $\scrO_{X_{\ket}}(U_i)\widehat{\otimes}R$-modules.
\item[(iii)] A morphism between Kummer \'etale Banach sheaves of $\scrO_{X_{\ket}}\widehat{\otimes}R$-modules is a continuous map of topological $\scrO_{X_{\ket}}\widehat{\otimes}R$-modules.
\item[(iv)] Let $\scrF$ be a Kummer \'etale Banach sheaf of $\scrO_{X_{\ket}}\widehat{\otimes}R$-modules as in (i). An \textbf{integral model} of $\scrF$ is a subsheaf $\scrF^+$ of $\scrO^+_{X_{\ket}}\widehat{\otimes}'R$-modules such that
\begin{itemize}
\item for every quasi-compact $U\in X_{\ket}$, $\scrF^+(U)$ is open and bounded in $\scrF(U)$;
\item $\scrF=\scrF^+[1/p]$;
\item there exists a Kummer \'etale atlas $\mathfrak{U}=\{U_i:i\in I\}$ of $\scrF$ such that, for every $i\in I$ and every affinoid $V\in U_{i,\ket}$, the canonical map
\[\scrF^+(U_i)\widehat{\otimes}_{\scrO^+_{X_{\ket}}(U_i)}\scrO^+_{X_{\ket}}(V)\rightarrow \scrF^+(V)\]
is an isomorphism, where the completion is with respect to the $p$-adic topology.
\end{itemize}
\end{enumerate}
\end{Definition}

Clearly, an analytic refinement of an atlas (resp., a Kummer \'etale refinement of a Kummer \'etale atlas) is also an atlas (resp., a Kummer \'etale atlas). Also notice that it is not true that a Banach sheaf (resp., Kummer \'etale Banach sheaf) on an affinoid adic space (resp., affinoid log adic space) is the sheaf associated with its global section. Nonetheless, we have the following result.

\begin{Lemma}\label{Lemma: Banach sheaf associated with global section}
Let $(A, A^+)$ be a complete reduced Tate algebra over $(\C_p, \calO_{\C_p})$ and let $M$ be a projective Banach $A\widehat{\otimes}R$-module.
\begin{enumerate}
\item[(i)] Let $X=\Spa(A, A^+)$ be the associated adic space. Then the presheaf $M\widehat{\otimes}_A \scrO_X$ sending an affinoid open subset $\Spa(B, B^+)\subset X$ to $M\widehat{\otimes}_A B$ is a sheaf.
\item[(ii)] Suppose $X=\Spa(A, A^+)$ is equipped with an fs log structure. Then the presheaf $M\widehat{\otimes}_A \scrO_{X_{\ket}}$ sending an affinoid open subset $\Spa(B, B^+)\in X_{\ket}$ to $M\widehat{\otimes}_A B$ is a sheaf.
\end{enumerate}
\end{Lemma}

\begin{proof}
It immediately reduces to the case where $M$ is an orthonormalisable Banach $A\widehat{\otimes}R$-module; \emph{i.e.}, $M\simeq (A\widehat{\otimes}R)(J)$ for some index set $J$. It then reduces to the case where $|J|=1$. Then the lemma follows from Lemma \ref{Lemma: structure sheaf mixed completed tensor}.
\end{proof}

As a corollary, one can associate a projective Kummer \'etale Banach sheaf with every projective Banach sheaf.

\begin{Corollary}\label{Corollary: Banach induce Kummer etale Banach}
Let $X$ be a locally noetherian fs log adic space over $(\C_p, \calO_{\C_p})$ and let $\scrF$ be a projective Banach sheaf of $\scrO_{X_{\an}}\widehat{\otimes}R$-modules with atlas $\mathfrak{U}=\{U_i: i\in I\}$. Suppose $\scrF$ admits an integral model $\scrF^+$. Consider the $p$-adically completed sheaf of $\scrO_{X_{\ket}}$-modules $\scrF_{\ket}$ associated with $\scrF$; namely, 
\[\scrF_{\ket}:=\left(\varprojlim_m \scrF^+\otimes_{\scrO^+_{X_{\an}}}\scrO^+_{X_{\ket}}/p^m\right)[\frac{1}{p}].\]
Then $\scrF_{\ket}$ is a projective Kummer \'etale Banach sheaf of $\scrO_{X_{\ket}}\widehat{\otimes}R$-modules with Kummer \'etale atlas $\mathfrak{U}=\{U_i: i\in I\}$, where each $U_i$ is equipped with the induced log structure from $X$. Moreover, for every affinoid $V\in U_{i, \ket}$, we have
\[\scrF_{\ket}(V)\cong \scrF(U_i)\widehat{\otimes}_{\scrO_{X_{\an}}(U_i)}\scrO_{X_{\ket}}(V).\]
\end{Corollary}

We need an easy lemma.

\begin{Lemma}\label{Lemma: a convenient basis}
Let $X$ be a locally noetherian fs log adic space over $(\C_p, \calO_{\C_p})$ and let $\mathfrak{U}=\{U_i:i\in I\}$ be a Kummer \'etale covering of $X$ by affinoid $U_i$'s. Consider the full subcategory $\calB_{\mathfrak{U}}$ of $X_{\ket}$ consisting of those affinoid $V\in X_{\ket}$ such that the map $V\rightarrow X$ factors through $V\rightarrow U_i \rightarrow X$ for some $i\in I$. Then $\calB_{\mathfrak{U}}$ forms a basis for the site $X_{\ket}$.
\end{Lemma}

\begin{proof}
We have to prove that every $U\in X_{\ket}$ admits a covering by such $V$'s and that $\calB_{\mathfrak{U}}$ is closed under fibred products. Both statements are clear.
\end{proof}

\begin{proof}[Proof of Corollary \ref{Corollary: Banach induce Kummer etale Banach}]
Let $\calB_{\mathfrak{U}}$ be the basis of $X_{\ket}$ as in Lemma \ref{Lemma: a convenient basis} associated with the covering $\mathfrak{U}=\{U_i: i\in I\}$. It suffices to show that the assignment $$V\mapsto \scrF(U_i)\widehat{\otimes}_{\scrO_{X_{\an}}(U_i)}\scrO_{X_{\ket}}(V),$$
for every $V\in \calB_{\mathfrak{U}}$ which factors through $V\rightarrow U_i\rightarrow X$, defines a sheaf on $\calB_{\mathfrak{U}}$. (Notice that this assigment is independent of the choice of $i$ and hence well-defined.) The sheafiness of this assignment follows from Lemma \ref{Lemma: Banach sheaf associated with global section} and the sheafiness of $\scrF$.
\end{proof}

In what follows, we are interested in those Kummer \'etale Banach sheaves that are ``admissible''. Let us first recall the notion of coherent sheaves on a ringed site.

\begin{Definition}\label{Definition: coherent sheaf on ringed space}
Let $(Z, \scrO_Z)$ be a ringed site. A sheaf of $\scrO_Z$-modules $\scrF$ is called a \textbf{coherent $\scrO_Z$-module} if there exists a covering $\mathfrak{U}=\{U_i: i\in I\}$ for $Z$ such that for every $i\in I$, there exist positive integers $m$, $n$, and an exact sequence of $\scrO_Z|_{U_i}$-modules
\[\bigoplus_{j=1}^m\scrO_Z|_{U_i}\rightarrow \bigoplus_{k=1}^n \scrO_Z|_{U_i}\rightarrow \scrF|_{U_i}\rightarrow 0.\]
In this situation, we say that $\scrF$ is a coherent $\scrO_Z$-module \textbf{subject to the covering} $\mathfrak{U}$.
\end{Definition}

We will apply this definition to the ringed site $(X_{\ket}, \scrO^+_{X_{\ket}}\otimes_{\Z_p} (R/\fraka^m))$.

\begin{Definition}\label{Definition: admissible Banach sheaf}
Let $X$ be a locally noetherian fs log adic space over $(\C_p, \calO_{\C_p})$ and let $\scrF$ be a projective Kummer \'etale Banach sheaf of $\scrO_{X_{\ket}}\widehat{\otimes}R$-modules. Suppose it admits an integral model $\scrF^+$ and, for every $m\in \Z_{\geq 1}$, we write $\scrF^+_m:=\scrF^+/\fraka^m$. We say that $\scrF$ is \textbf{admissible} if there exist
\begin{itemize}
\item a Kummer \'etale atlas $\mathfrak{U}=\{U_i: i\in I\}$ of $X$ such that each $\scrF^+(U_i)$ is the $p$-adic completion of a free $\scrO^+_{X_{\ket}}\widehat{\otimes}'R$-module; and
\item for every $m\in \Z_{\geq 1}$ and $d\in \Z_{\geq 1}$, a subsheaf $\scrF^+_{m,d}\subset \scrF^+_m$ which is a coherent $\scrO^+_{X_{\ket}}\otimes_{\Z_p} (R/\fraka^m)$-module subject to the covering $\mathfrak{U}$,
\end{itemize}
such that we have $\scrF^+\cong\varprojlim_m \scrF^+_m$ and $\scrF_m^+\cong \varinjlim_d \scrF^+_{m,d}$ for every $m\in \Z_{\geq 1}$.

Such a Kummer \'etale atlas is called an \textbf{admissible atlas} for $\scrF$.
\end{Definition}

\begin{Lemma}\label{Lemma: pushforward along finite Kummer etale map}
Let $h: Y\rightarrow X$ be a finite Kummer \'etale morphism between locally noetherian fs log adic spaces over $(\C_p, \calO_{\C_p})$. Suppose $\scrF$ is an admissible Kummer \'etale Banach sheaf of $\scrO_{Y_{\ket}}\widehat{\otimes}R$-modules. Then $h_*\scrF$ is an admissible Kummer \'etale Banach sheaf of $\scrO_{X_{\ket}}\widehat{\otimes}R$-modules.
\end{Lemma}

\begin{proof}
Suppose $\mathfrak{U}=\{U_i:i\in I\}$ is an admissible atlas for $\scrF$ on $Y$. By Definition \ref{Definition: Kummer etale morphism} and \cite[Proposition 4.1.6]{Diao}, the finite Kummer \'etale morphism $h:Y\rightarrow X$ is, Kummer \'etale locally on $X$, isomorphic to a direct sum of isomorphisms. Therefore, one can find an affinoid Kummer \'etale covering $\{V_j: j\in J\}$ of $X$ such that, for every $i\in I$ and $j\in J$, $U_i\times_X V_j$ is isomorphic to a disjoint union of finite copies of $U_i$'s. Consequently, the Kummer \'etale covering $\mathfrak{V}=\{U_i\times_X V_j: i\in I, j\in J\}$ is a desired admissible atlas for $h_*\scrF$.
\end{proof}

If $\scrF$ is a Kummer \'etale Banach sheaf of $\scrO_{X_{\ket}}\widehat{\otimes}R$-modules with an integral structure $\scrF^+$. We write
\[\widehat{\scrF}^+:=\varprojlim_m \left(\scrF^+\otimes_{\scrO^+_{X_{\ket}}}\scrO^+_{X_{\proket}}/p^m\right)\cong \varprojlim_m \left(\scrF^+\otimes_{(\scrO^+_{X_{\ket}}\widehat{\otimes}'R)}(\scrO^+_{X_{\proket}}\widehat{\otimes}'R)/p^m\right)\]
and $\widehat{\scrF}:=\widehat{\scrF}^+[1/p]$. They are sheaves of $\widehat{\scrO}^+_{X_{\proket}}\widehat{\otimes}'R$-modules and $\widehat{\scrO}_{X_{\proket}}\widehat{\otimes}R$-modules, respectively. 

Recall the natural projection of sites $\nu: X_{\proket}\rightarrow X_{\ket}$. The main result of this subsection is the following.

\begin{Proposition}[Generalised projection formula]\label{Proposition: generalised projection formula}
Let $X$ be a locally noetherian fs log adic space which is log smooth over $(\C_p, \calO_{\C_p})$ and let $\scrF$ be a projective Kummer \'etale Banach sheaf of $\scrO_{X_{\ket}}\widehat{\otimes}R$-modules. Suppose $\scrF$ is admissible. Then, for every $j\in \Z_{\geq 0}$, there is a natural isomorphism of Kummer \'etale Banach sheaves of $\scrO_{X_{\ket}}\widehat{\otimes} R$-modules
\[\scrF\otimes_{\scrO_{X_{\ket}}} R^j\nu_*\widehat{\scrO}_{X_{\proket}}\xrightarrow{\sim}R^j\nu_*\widehat{\scrF}.\]
\end{Proposition}

To prove the proposition, we need some preparations. 

\begin{Lemma}\label{Lemma: proj. formula lemma 1}
Let $X$ be a locally noetherian fs log adic space over $(\C_p, \calO_{\C_p})$. Let $\scrH$ be an $\widehat{\scrO}_{X_{\proket}}^+ \widehat{\otimes} R$-module and let $\scrH_m:= \scrH/\fraka^m$ for every $m\in \Z_{\geq 1}$. Suppose 
\begin{itemize}
    \item $\scrH = \varprojlim_m \scrH_m$; and
    \item for every $m\in \Z_{\geq 1}$, there exists a sequence of finite free $\widehat{\scrO}_{X_{\proket}}^+\otimes_{\Z_p} (R/\fraka^m)$-submodules $\{\scrH_{m, d}: d\in \Z_{\geq 0}\}$ of $\scrH_m$ such that $\scrH_m\cong \varinjlim_d \scrH_{m,d}$.
\end{itemize}
   Then, for every $j\in \Z_{\geq 0}$, the natural map \[
    R^j\nu_{*}\scrH \rightarrow \varprojlim_m R^j\nu_{*} \scrH_{m}
\] is an almost isomorphism.
\end{Lemma}
\begin{proof}
We have to show the almost vanishing of the higher inverse limit $R^j\varprojlim_{m}\scrH_m$. Applying an almost version of \cite[Lemma 3.18]{Scholze_2013}, it suffices to show that, for every log affinoid perfectoid object $U\in X_{\proket}$, there are almost isomorphisms 
$$R^1\varprojlim_{m} \scrH_m(U)^a = 0 
$$
and 
$$ H^j(U, \scrH_m)^a = 0$$
for every $j\in \Z_{\geq 0}$. The first almost vanishing follows from the Mittag-Leffler condition. To obtain the second almost isomorphism, observe that \[
    H^j(U, \scrH_m) \cong \varinjlim_{d} H^j(U, \scrH_{m, d}).
\] and each $H^j(U, \scrH_{m, d})$ is almost zero by \cite[Theorem 5.4.3]{Diao}.
\end{proof}

\begin{Lemma}\label{Lemma: proj. formula lemma 2}
Let $X$ be a locally noetherian fs log adic space which is log smooth over $(\C_p, \calO_{\C_p})$. If $\scrG$ is a projective Kummer \'etale Banach sheaf of $\scrO_{X_{\ket}}\widehat{\otimes}R$-modules, then, for every $j\in \Z_{\geq 0}$, the sheaf $R^j\nu_*\widehat{\scrG}$ is also a projective Kummer \'etale Banach sheaf of $\scrO_{X_{\ket}}\widehat{\otimes}R$-modules.
\end{Lemma}

\begin{proof}
By considering a Kummer \'etale atlas for $\scrG$ and writing $R\simeq \prod_{\sigma\in \Sigma} \Z_p e_{\sigma}$, we immediately reduce to the case where 
\begin{itemize}
\item $X$ is affinoid and admits a toric chart $X\rightarrow \mathbb{E}=\Spa(\C_p\langle P\rangle, \calO_{\C_p}\langle P\rangle)$ for some sharp fs monoid $P$; 
\item $R=\Z_p$ and $\fraka=(p)$;
\item $\scrG$ is globally projective; \emph{i.e.}, $\scrG(X)$ is a projective Banach $\scrO_{X_{\ket}}(X)$-module and for every 
affinoid $U\in X_{\ket}$, we have a natural isomorphism
\[\scrG(X)\widehat{\otimes}_{\scrO_{X_{\ket}}(X)}\scrO_{X_{\ket}}(U)\xrightarrow[]{\sim} \scrG(U).\]
\end{itemize}
We further reduce to the case where $\scrG$ is globally orthonormalisable; namely, $\scrG\simeq \scrO_{X_{\ket}}(J)$ for some index set $J$. Let $\scrG^+$ be the $p$-adic completion of the free $\scrO^+_{X_{\ket}}$-module $\bigoplus_{J}\scrO^+_{X_{\ket}}$ and let $\scrG^+_m:=\scrG^+/p^m\simeq \bigoplus_{J} \scrO^+_{X_{\ket}}/p^m$. By Lemma \ref{Lemma: proj. formula lemma 1}, we have a natural almost isomorphism
$$R^j\nu_*\widehat{\scrG}^+\xrightarrow[]{\sim} \varprojlim_m R^j \nu_*\widehat{\scrG}^+_m$$
where $\widehat{\scrG}^+_m=\widehat{\scrG}^+/p^m\simeq \bigoplus_{J} \scrO^+_{X_{\proket}}/p^m$.

We claim that, in this case, the sheaf $R^j\nu_*\widehat{\scrG}$ is isomorphic to $\left(\bigwedge^j (\scrO_{X_{\ket}})^n\right)(J)$ for some $n\in \Z_{\geq 1}$. For this, we follow the strategy as in the proof of Lemma \ref{Lemma: Cartan-Leray}.

In order to be consistent with the notations in Lemma \ref{Lemma: Cartan-Leray}, we write $U=X$. Consider the collection $\mathcal{B}_U$ used in the proof of Lemma \ref{Lemma: Cartan-Leray}. In particular, for every $V\in \mathcal{B}_U$, the map $V\rightarrow U$ admits a Kummer chart $P\rightarrow P'$ which is isomorphic to the $m$-th multiple map $[m]:P\rightarrow P$. Moreover, the injection $P\rightarrow P'$ induces an injection $\Gamma'\rightarrow \Gamma$. If we fix an identification $\Gamma\cong \widehat{\Z}(1)^n$, the injection $\Gamma'\rightarrow \Gamma$ can be identified with the $m$-th multiple map $[m]:\widehat{\Z}(1)^n\rightarrow \widehat{\Z}(1)^n$.

By the calculation in Lemma \ref{Lemma: Cartan-Leray}, we obtain an almost injection
$$\left(\bigwedge^j (\scrO^+_{X_{\ket}}/p^m(V))^n\right)^a \simeq H^j(\Gamma, \scrO^+_{X_{\ket}}/p^m(V))^a\hookrightarrow H^j_{\proket}(V, \scrO^+_X/p^m)^a$$
with cokernel killed by $p$. Taking direct sum and then inverse limit, we obtain an almost injection
$$
\varprojlim_m \bigoplus_J \left(\bigwedge^j (\scrO^+_{X_{\ket}}/p^m(V))^n\right)^a \hookrightarrow \varprojlim_m \bigoplus_J H^j_{\proket}(V, \scrO^+_X/p^m)
$$
with cokernel killed by $p$. Inverting $p$, we obtain an isomorphism
$$
\left(\bigwedge^j (\scrO^+_{X_{\ket}}(V))^n\right)(J)\simeq \varprojlim_m \bigoplus_J H^j_{\proket}(V, \scrO^+_X/p^m).
$$
However, note that the sheaf $$R^j\nu_*\widehat{\scrG}\cong \left(\varprojlim_m R^j \nu_*\widehat{\scrG}^+_m\right)[\frac{1}{p}]$$
is just the sheafification of $W\mapsto \varprojlim_m \bigoplus_J H^j_{\proket}(W, \scrO^+_X/p^m)$. Consequently, $R^j\nu_*\widehat{\scrG}$ coincides with the sheaf $\left(\bigwedge^j (\scrO^+_{X_{\ket}})^n\right)(J)$ which is clearly an ON-able Banach sheaf of $\scrO_{X_{\ket}}\widehat{\otimes} R$-modules.
\end{proof}

\begin{proof}[Proof of Proposition \ref{Proposition: generalised projection formula}]
We split the proof into three steps.\\

\noindent\textbf{Step 1.} We first verify that both $\scrF\otimes_{\scrO_{X_{\ket}}} R^j\nu_*\widehat{\scrO}_{X_{\proket}}$ and $R^j\nu_* \widehat{\scrF}$ are projective Kummer \'etale Banach sheaf of $\scrO_{X_{\ket}}\widehat{\otimes} R$-modules.

Indeed, the statement for $\scrF\otimes_{\scrO_{X_{\ket}}} R^j\nu_*\widehat{\scrO}_{X_{\proket}}$ follows from the locally finite free-ness of $R^j\nu_*\widehat{\scrO}_{X_{\proket}}$ (cf. Lemma \ref{Lemma: Cartan-Leray}) and the statement for $R^j\nu_* \widehat{\scrF}$ follows from Lemma \ref{Lemma: proj. formula lemma 2}. In fact, we can be more precise. Consider an affinoid Kummer \'etale covering $\mathfrak{U}=\{U_i: i\in I\}$ satisfying:
\begin{itemize}
\item $\mathfrak{U}$ is an admissible atlas of $\scrF$;
\item each $U_i$ admits a toric chart $U_i\rightarrow \Spa(\C_p\langle P_i\rangle, \calO_{\C_p}\langle P_i\rangle)$ for some sharp fs monoid.
\end{itemize}
Then, by the proof of Lemma \ref{Lemma: Cartan-Leray} and Lemma \ref{Lemma: proj. formula lemma 2}, we see that $\mathfrak{U}$ is a Kummer \'etale atlas for both $\scrF\otimes_{\scrO_{X_{\ket}}} R^j\nu_*\widehat{\scrO}_{X_{\proket}}$ and $R^j\nu_* \widehat{\scrF}$. (In fact, they are both orthonormalisable on each $U_i$.) For the rest of the proof, we fix such a cover $\mathfrak{U}$.\\

\noindent\textbf{Step 2.} We construct two natural morphisms
$$
\Psi: \scrF^+\bigotimes_{\scrO^+_{X_{\ket}}\widehat{\otimes}'R} R^j\nu_*\left(\widehat{\scrO}^+_{X_{\proket}}\widehat{\otimes}'R\right)\rightarrow \varprojlim_{m}R^j\nu_{*}\widehat{\scrF}_m^+
$$
and
$$
 \Theta: R^j\nu_*\widehat{\scrF}^+ \rightarrow \varprojlim_{m}R^j\nu_{*}\widehat{\scrF}_m^+.
$$
where $$\widehat{\scrF}^+_m=\scrF^+_m\otimes_{\scrO^+_{X_{\ket}}}\widehat{\scrO}^+_{X_{\proket}}=\widehat{\scrF}^+/\fraka^m.$$
There is clearly such a map $\Theta$. It remains to construct $\Psi$.

For every $m\in \Z_{\geq 1}$ and $d\in \Z_{\geq 0}$, we write \[
    \widehat{\scrF}_{m, d}^+ := \scrF_{m, d}^+ \bigotimes_{\scrO_{X_{\ket}}^+\otimes_{\Z_p} (R/\fraka^m)} \left(\widehat{\scrO}_{X_{\proket}}^+ \otimes_{\Z_p} ( R/\fraka^m)\right).
\] 
By the usual projection formula for ringed sites (see, for example, \cite[01E6]{stacks-project}), we obtain a canonical morphism
$$
    \Psi_{m, d}: \scrF^+_{m, d} \bigotimes_{\scrO_{X_{\ket}}^+\otimes_{\Z_p} (R/\fraka^m)} R^j\nu_* \left(\widehat{\scrO}_{X_{\proket}}^+ \otimes_{\Z_p} ( R/\fraka^m)\right) \rightarrow R^j\nu_* \widehat{\scrF}_{m, d}^+.
$$
Taking direct limit with respect to $d$, followed by taking inverse limit with respect to $m$, we obtain a canonical morphism \[
    \Psi': \varprojlim_m\left(\scrF^+_m \bigotimes_{\scrO_{X_{\ket}}^+\otimes_{\Z_p} (R/\fraka^m)} R^j\nu_* \left(\widehat{\scrO}_{X_{\proket}}^+ \otimes_{\Z_p} (R/\fraka^m)\right) \right) \rightarrow \varprojlim_mR^j\nu_* \widehat{\scrF}_{m}^+.
\]
On the other hand, we have natural morphisms
\begin{align*}
\Psi'': &\scrF^+\bigotimes_{\scrO^+_{X_{\ket}}\widehat{\otimes}'R} R^j\nu_*\left(\widehat{\scrO}^+_{X_{\proket}}\widehat{\otimes}'R\right) \rightarrow \scrF^+\bigotimes_{\scrO^+_{X_{\ket}}\widehat{\otimes}'R} \varprojlim_m \left(R^j\nu_*\left(\widehat{\scrO}^+_{X_{\proket}}\otimes_{\Z_p} (R/\fraka^m)\right)\right)\\
=& (\varprojlim_m\scrF^+_m)\bigotimes_{\scrO^+_{X_{\ket}}\widehat{\otimes}'R} \varprojlim_m \left(R^j\nu_*\left(\widehat{\scrO}^+_{X_{\proket}}\otimes_{\Z_p} (R/\fraka^m)\right)\right)\\
\rightarrow& \varprojlim_m\left(\scrF^+_m\bigotimes_{\scrO^+_{X_{\ket}}\otimes_{\Z_p} (R/\fraka^m)}  R^j\nu_*\left(\widehat{\scrO}^+_{X_{\proket}}\otimes_{\Z_p} (R/\fraka^m)\right)\right)
\end{align*}
Composing with $\Psi'$, we obtain the desired morphism
$$
\Psi: \scrF^+\bigotimes_{\scrO^+_{X_{\ket}}\widehat{\otimes}'R} R^j\nu_*\left(\widehat{\scrO}^+_{X_{\proket}}\widehat{\otimes}'R\right)\rightarrow \varprojlim_{m}R^j\nu_{*}\widehat{\scrF}_m^+.
$$

\noindent\textbf{Step 3.} For simplicity, we write $\scrG_i$, $\scrG^+_i$, and $\scrG^+_{i,m}$ for $\scrF|_{U_i}$, $\scrF^+|_{U_i}$, and $\scrF^+_{i,m}|_{U_i}$, respectively. Since $\scrG^+_{i,m}$ is a free $\scrO^+_{U_{i,\ket}}\otimes (R/\fraka^m)$-module, we can express $\scrG^+_{i,m}$ as a filtered direct limit of finite free submodules $\scrG^+_{i,m,\alpha}$.

We can repeat the construction in Step 2 to $\scrG^+_i$, $\scrG^+_{i,m}$, and $\scrG^+_{i,m,\alpha}$. In particular, we obtain maps
$$ \Psi_i: \scrG^+_i \bigotimes_{\scrO_{U_{i,\ket}}^+ \widehat{\otimes}' R} R^j\nu_*\left(\widehat{\scrO}_{U_{i,\proket}}^+\widehat{\otimes}' R\right) \rightarrow \varprojlim_{m} R^j\nu_*\widehat{\scrG}_{i,m}^+$$
and
$$\Theta_i: R^j\nu_*\widehat{\scrG}^+_i \rightarrow \varprojlim_{m}R^j\nu_{*}\widehat{\scrG}_{i,m}^+$$
where $$\widehat{\scrG}^+_{i,m}=\scrG^+_{i,m}\otimes_{\scrO^+_{U_{i,\ket}}}\widehat{\scrO}^+_{U_{i,\proket}}=\widehat{\scrG}^+_i/\fraka^m.$$

Moreover, we have a commutative diagram
$$ \begin{tikzcd}
        \scrF^+|_{U_i} \bigotimes_{(\scrO^+_{U_{i, \ket}}\widehat{\otimes}' R)}R^j\nu_*\left(\widehat{\scrO}_{U_{i, \proket}}^+\widehat{\otimes}' R\right)\arrow[r, "\Psi|_{U_i}"]\arrow[d, "\cong "'] & \varprojlim_{m} R^j\nu_* \widehat{\scrF}_m^+|_{U_i}\arrow[d, "\cong"] & R^j\nu_*\widehat{\scrF}^+|_{U_i} \arrow[l, "\Theta|_{U_i}"']\arrow[d, "\cong"]\\
        \scrG_{i}^+ \bigotimes_{(\scrO^+_{U_{i, \ket}}\widehat{\otimes}' R)}R^j\nu_*\left(\widehat{\scrO}_{U_{i, \proket}}^+\widehat{\otimes}' R\right)\arrow[r, "\Psi_i"] & \varprojlim_{m} R^j\nu_* \widehat{\scrG}_{i,m}^+ & R^j\nu_*\widehat{\scrG}_i^+\arrow[l, "\Theta_i"']
    \end{tikzcd}
    $$
   The square on the left is commutative because the cofiltered systems $\{\scrF^+_{m,d}|_{U_i}\}$ and $\{\scrG^+_{i,m,\alpha}\}$ are cofinal to each other. By Lemma \ref{Lemma: proj. formula lemma 1}, $\Theta_i=\Theta|_{U_i}$ is an almost isomorphism. This implies that $\Theta[1/p]$ is an isomorphism of projective Kummer \'etale Banach sheaves of $\scrO_{X_{\ket}}\widehat{\otimes}R$-modules.

We claim that $\Psi_i$ also becomes an isomorphism after inverting $p$. By construction, $\Psi_i$ factors as the composition of
$$
\Psi''_i:\scrG_{i}^+ \bigotimes_{\scrO^+_{U_{i, \ket}}\widehat{\otimes}' R}R^j\nu_*\left(\widehat{\scrO}^+_{U_{i, \proket}}\widehat{\otimes}' R\right)\rightarrow \varprojlim_m\left(\scrG^+_{i,m}\bigotimes_{\scrO^+_{U_{i,\ket}}\otimes_{\Z_p}(R/\fraka^m)}\left(R^j\nu_*\widehat{\scrO}^+_{U_{i,\proket}}\otimes_{\Z_p} (R/\fraka^m)\right)\right)
$$
and a canonical isomorphism
\begin{align*}
\Psi'_i: & \varprojlim_m\left(\scrG^+_{i,m}\bigotimes_{\scrO^+_{U_{i,\ket}}\otimes_{\Z_p}(R/\fraka^m)}\left(R^j\nu_*\widehat{\scrO}^+_{U_{i,\proket}}\otimes_{\Z_p} (R/\fraka^m)\right)\right)\\ = &\varprojlim_m\left(\varinjlim_{\alpha}\scrG^+_{i,m,\alpha}\bigotimes_{\scrO^+_{U_{i,\ket}}\otimes_{\Z_p}(R/\fraka^m)}\left(R^j\nu_*\widehat{\scrO}^+_{U_{i,\proket}}\otimes_{\Z_p} (R/\fraka^m)\right)\right)\\
\cong & \varprojlim_m\varinjlim_{\alpha}R^j\nu_*\left(\scrG^+_{i,m,\alpha}\bigotimes_{\scrO^+_{U_{i,\ket}}\otimes_{\Z_p}(R/\fraka^m)}\left(\widehat{\scrO}^+_{U_{i,\proket}}\otimes_{\Z_p} (R/\fraka^m)\right)\right)\\
=&\varprojlim_m\varinjlim_{\alpha} R^j\nu_* \widehat{\scrG}^+_{i,m,\alpha}=\varprojlim_mR^j\nu_*\widehat{\scrG}^+_{i,m}
\end{align*}
where the second isomorphism follows from the fact that each $\scrG^+_{i,m,\alpha}$ is a finite free $\scrO^+_{U_{i,\ket}}\otimes_{\Z_p} (R/\fraka^m)$-module.

It remains to prove that $\Psi''_i$ becomes an isomorphism after inverting $p$. Recall that $U_i$ admits a toric chart $U_i\rightarrow \Spa(\C_p\langle P\rangle, \calO_{\C_p}\langle P\rangle)$ for some sharp fs monoid $P$. By choosing an identification $\Gamma:=\Hom(P^{\mathrm{gp}}_{\Q_{\geq 0}}/P^{\mathrm{gp}}, \boldsymbol{\mu}_{\infty})\simeq \widehat{\Z}(1)^n$, Lemma \ref{Lemma: Cartan-Leray} yields an isomorphism $R^j\nu_* \widehat{\scrO}_{U_{i,\proket}}\simeq \bigwedge^j (\scrO_{U_{i,\ket}})^n$.

On one hand, by Proposition \ref{Proposition: compatibility with completed tensor}, we have
\begin{align*}
\scrG_{i}^+ \bigotimes_{\scrO^+_{U_{i, \ket}}\widehat{\otimes}' R}R^j\nu_*\left(\widehat{\scrO}^+_{U_{i, \proket}}\widehat{\otimes}' R\right)[\frac{1}{p}] & =\scrG_{i} \bigotimes_{\scrO_{U_{i, \ket}}\widehat{\otimes} R}R^j\nu_*\left(\widehat{\scrO}_{U_{i, \proket}}\widehat{\otimes} R\right)\\
& \simeq \scrG_{i} \bigotimes_{\scrO_{U_{i, \ket}}\widehat{\otimes} R}\left((R^j\nu_*\widehat{\scrO}_{U_{i, \proket}})\widehat{\otimes} R\right)\\
& =\scrG_i \bigotimes_{\scrO_{U_{i,\ket}}}R^j\nu_*\widehat{\scrO}_{U_{i,\proket}}\simeq \scrG_i \bigotimes_{\scrO_{U_{i,\ket}}} \bigwedge^j (\scrO_{U_{i,\ket}})^n.
\end{align*}

On the other hand, if we write $R/\fraka^m\simeq \bigoplus_{\sigma\in \Sigma_m} \Z/p^{\sigma}$, then we have
$$R^j\nu_*\widehat{\scrO}^+_{U_{i,\proket}}\otimes_{\Z_p} (R/\fraka^m)\simeq \bigoplus_{\sigma\in \Sigma_m} R^j\nu_* (\scrO^+_{U_{i,\proket}}/p^{\sigma}).$$
By Lemma \ref{Lemma: Cartan-Leray} (iii), there is an almost injection
$$\bigwedge^j (\scrO^+_{U_{i,\ket}}/p^{\sigma})^n\hookrightarrow R^j\nu_*(\scrO^+_{U_{i,\proket}}/p^{\sigma})$$ 
whose cokernel is killed by $p$. This yields an almost injection
\begin{align*}
&\scrG^+_i\bigotimes_{\scrO^+_{U_{i,\ket}}}\left(\bigwedge^j(\scrO^+_{U_{i,\ket}})^n\right)=\varprojlim_m\scrG^+_{i,m}\bigotimes_{\scrO^+_{U_{i,\ket}}\otimes_{\Z_p}(R/\fraka^m)}\left(\bigwedge^j (\scrO^+_{U_{i,\ket}}\otimes_{\Z_p}(R/\fraka^m))^n\right)\\
=& \varprojlim_m\scrG^+_{i,m}\bigotimes_{\scrO^+_{U_{i,\ket}}\otimes_{\Z_p}(R/\fraka^m)}\left(\bigoplus_{\sigma\in \Sigma_m}\bigwedge^j (\scrO^+_{U_{i,\ket}}/p^{\sigma})^n\right)\hookrightarrow \varprojlim_m\scrG^+_{i,m}\bigotimes_{\scrO^+_{U_{i,\ket}}\otimes_{\Z_p}(R/\fraka^m)}\left(\bigoplus_{\sigma\in \Sigma_m} R^j\nu_*(\scrO^+_{U_{i,\proket}}/p^{\sigma})\right)\\
\simeq & \varprojlim_m\scrG^+_{i,m}\bigotimes_{\scrO^+_{U_{i,\ket}}\otimes_{\Z_p}(R/\fraka^m)}\left(R^j\nu_*\widehat{\scrO}^+_{U_{i,\proket}}\otimes_{\Z_p} (R/\fraka^m)\right)
\end{align*}
with cokernel killed by $p$. 

Consequently, both sides of $\Psi''_i$ are isomorphic to $\scrG_i\bigotimes_{\scrO_{U_{i,\ket}}}\left(\bigwedge^j(\scrO_{U_{i,\ket}})^n\right)$ after inverting $p$, and one checks that $\Psi''_i[1/p]$ is just the identity map on $\scrG_i\bigotimes_{\scrO_{U_{i,\ket}}}\left(\bigwedge^j(\scrO_{U_{i,\ket}})^n\right)$. This finishes the proof.
\end{proof}

\begin{Corollary}\label{Corollary: generalised projection formula with invariants}
Let $X$ be a locally noetherian fs log adic space which is log smooth over $(\C_p, \calO_{\C_p})$. Let $\scrF$ be an admissible projective Kummer \'etale Banach sheaf of $\scrO_{X_{\ket}}\widehat{\otimes} R$-modules, with the corresponding integral structure $\scrF^+$. Suppose $\scrF^+$ is equipped with an $\scrO^+_{X_{\ket}}\widehat{\otimes}'R$-linear action of a finite group $G$. This induces an $\scrO_{X_{\ket}}\widehat{\otimes}R$-linear action of $G$ on $\scrF$. Then the subsheaf of $G$-invariants $\scrF^G$ also satisfies the generalised projection formula. More precisely, we have a natural isomorphism\[
    \scrF^G \otimes_{\scrO_{X_{\ket}}}R^i\nu_*\widehat{\scrO}_{X_{\proket}}\xrightarrow{\sim} R^i\nu_*\widehat{\scrF^G}
\]
\end{Corollary}
\begin{proof}
By Proposition \ref{Proposition: generalised projection formula}, we have an isomorphism \[
    \scrF\otimes_{\scrO_{X_{\ket}}}R^i\nu_*\widehat{\scrO}_{X_{\proket}}\xrightarrow{\sim}R^i\nu_*\widehat{\scrF}.
\] Taking the $G$-invariants, we obtain an isomorphism
\[
    \scrF^G\otimes_{\scrO_{X_{\ket}}}R^i\nu_*\widehat{\scrO}_{X_{\proket}}\xrightarrow{\sim} \left(R^i\nu_*\widehat{\scrF}\right)^G.
\] %The reason why it is still an isomorphism as is that $G$ is a finite group, so $G$-invariants is an exact functor; we can check this for every $U \in X_{\ket}$: \[ H^i(U_{\proket},{\widehat{\scrF}})^G={\left(\scrF(U) \widehat{\otimes}_{(\scrO_{X_{\ket}}\widehat{\otimes} R)} \left(H^i(U_{\proket},\widehat{\scrO}_{X_{\proket}}\widehat{\otimes}R) \right) \right)}^G.\]

It remains to show $\left(R^i\nu_*{\widehat{\scrF}}\right)^G \cong  R^i\nu_*{\widehat{\scrF^G}}$. Indeed, consider the following commutative diagram
$$ \begin{tikzcd}
      \scrO_{X_{\proket}}[G]-\Mod \arrow[r, "\nu_*"]\arrow[d, "(-)^G"'] &  \scrO_{X_{\ket}}[G]-\Mod\arrow[d, "(-)^G"] \\
       \scrO_{X_{\proket}}-\Mod \arrow[r, "\nu_*"] & \scrO_{X_{\ket}}-\Mod
    \end{tikzcd}
    $$
Notice that the higher right derived functors of both of the vertical arrows vanish as $G$ is a finite group and the base field is of characteristic zero. Now, applying the standard Grothendieck spectral sequence argument to both compositions $\nu_*\circ (-)^G$ and $(-)^G\circ \nu_*$, we obtain the desired commutativity of $R^i\nu_*$ and $(-)^G$.
\end{proof}
\section{Toroidal compactifications of the Siegel modular varieties}\label{section: boundary}
In this section, we study the toroidal compactifications of the Siegel modular varieties following \cite{Stroh-TorComp} and \cite{Pilloni-Stroh-CoherentCohomologyandGaloisRepresentations}. In particular, Pilloni and Stroh construct the (toroidally compactified) perfectoid Siegel modular variety of infinite level (\`a la Scholze in \cite{Scholze-2015}) by introducing the \emph{modified integral structures} of the toroidal compactifications on the finite levels.

This section is organised as follows. In \S \ref{subsection: boundary strata}, we study the notion of toroidal compactification of Siegel modular varieties at finite level. Then, in \S \ref{subsection: perfectoid Siegel modular variety}, we recall the construction of the perfectoid Siegel modular variety of infinite level and the associated Hodge--Tate period map by following \cite{Pilloni-Stroh-CoherentCohomologyandGaloisRepresentations}. We also show that this perfectoid object serves as a pro-Kummer \'etale Galois cover over the Siegel modular varieties at the finite levels. In order to be consistent with the notations in the main text of this paper, our notations are slightly different from the ones in \cite{Stroh-TorComp} and \cite{Pilloni-Stroh-CoherentCohomologyandGaloisRepresentations}. 

Throughout this section, we fix the following notations: \begin{enumerate}
    \item[$\bullet$] Let $V=\Z^{2g}$ and $V_p=\Z_p^{2g}$, equipped with the symplectic pairings defined in \S \ref{subsection: Algebraic and p-adic groups}. We denote by $\frakC$ the collection of all totally isotropic direct summands of $V$.
    \item[$\bullet$] For any totally isotropic direct summand $V'\subset V$, let $C(V/V'^{\perp})$ denote the cone of symmetric bilinear forms on $(V/V'^{\perp})\otimes_{\Z}\R$ which are positive semi-definite and whose kernel is defined over $\Q$.
    \item[$\bullet$] Observe that if $V', V''\in \frakC$ such that $V'\subset V''$, there is a natural inclusion $C(V/V'^{\perp})\subset C(V/V''^{\perp})$. We define $$\calC:=\big(\bigsqcup_{V'\in \frakC}C(V/V'^{\perp})\big)/\sim$$ where the equivalence relation is given by the aforementioned inclusions.
    \item[$\bullet$] Let $\frakS$ be a fixed $\GSp_{2g}(\Z)$-admissible smooth rational polyhedral cone decomposition of $\calC$ (see \cite[Definition 3.2.3.1]{Stroh-TorComp}). This means $\frakS$ consists of a smooth rational polyhedral cone decomposition of $C(V/V'^{\perp})$ (in the sense of \cite[Chapter IV, \S 2]{Faltings-Chai}) for every $V'\in \frakC$ such that 
    \begin{enumerate}
    \item[(i)] The decomposition of $C(V/V'^{\perp})$ coincides with the restriction of the decomposition of $C(V/V''^{\perp})$ whenever $V'\subset V''$, and
    \item[(ii)] $\frakS$ is $\GSp_{2g}(\Z)$-invariant and $\frakS/\GSp_{2g}(\Z)$ is a finite set.
    \end{enumerate}
    \item[$\bullet$] For every $n\in \Z_{\geq 1}$, let 
    $$\Gamma(p^n)=\{\bfgamma\in \GSp_{2g}(\Z_p): \bfgamma\equiv \one_{2g}\mod p^n\}$$
    as in \S \ref{subsection: Algebraic and p-adic groups}. Let us abuse the notation and write $\Gamma(p^0):=\GSp_{2g}(\Z_p)$. 
    \item[$\bullet$] For simplicity, let $\Iw$ and $\Iw^+$ denote the $p$-adic groups $\Iw_{\GSp_{2g}}$ and $\Iw^+_{\GSp_{2g}}$ as in \S \ref{subsection: Algebraic and p-adic groups}, respectively.
    \item[$\bullet$] For the rest of the section, let $\Gamma$ denote either $\Gamma(p^n)$ (for some $n\geq 0$), $\Iw$, or $\Iw^+$ which indicates the level structures of the Siegel modular varieties that we concern. We also write $$\widetilde{\Gamma}:=\GSp_{2g}(\Z)\cap\Gamma.$$ 
    \end{enumerate}

\subsection{Toroidal compactifications and boundary strata}\label{subsection: boundary strata}
Let $N\geq 3$ be a fixed integer coprime to $p$. Let $X_0$ be the moduli scheme over $\calO_{\C_p}$ of principally polarised abelian schemes of dimension $g$ equipped with a principal $N$-level structure. The fixed choice of polyhedral cone decomposition $\frakS$ gives rise to a toroidal compactification $\overline{X}_0$ (see, for example, \cite[Chapter IV, \S 4]{Faltings-Chai} or \cite[\S 3.2]{Stroh-TorComp}). Let $X$ and $\overline{X}$ be the base change of $X_0$ and $\overline{X}_0$ to $\C_p$, respectively. We view $\overline{X}$ as an fs log scheme equipped with the divisorial log structure defined by the boundary divisor.

Let $\Gamma$ denote either $\Gamma(p^n)$ (for some $n\geq 0$), $\Iw$, or $\Iw^+$. Let $X_{\Gamma}$ be the finite \'etale cover of $X$ parameterising $\Gamma$-level structure, as defined in Definition \ref{Definition: Siegel modular varieties of (strict) Iwahoris level}. More precisely, 
\begin{enumerate}
\item[(i)] $X_{\Gamma(p^n)}$ parameterises $(A, \lambda, \psi_N, \psi_{p^n})$ where $(A, \lambda)$ is a principally polarised abelian variety over $\C_p$ and $$\psi_N:V\otimes_{\Z}(\Z/N\Z)\xrightarrow[]{\sim} A[N]$$ and $$\psi_{p^n}:V\otimes_{\Z}(\Z/p^n\Z)\xrightarrow[]{\sim}A[p^n]$$
are symplectic isomorphisms.
\item[(ii)] $X_{\Iw}$ parameterises $(A, \lambda, \psi_N, \Fil_{\bullet}A[p])$ where $(A, \lambda, \psi_N)$ is as in (i) and $\Fil_{\bullet}A[p]$ is a full flag of $A[p]$ that satisfies
$$(\Fil_{\bullet}A[p])^{\perp}\cong \Fil_{2g-\bullet}A[p]$$
with respect to the Weil pairing.
\item[(iii)] $X_{\Iw^+}$ parameterises $(A, \lambda, \psi_N, \{C_i:i=1, \ldots, g\})$ where $(A, \lambda, \psi_N)$ is as in (i) and $\{C_i:i=1, \ldots, g\}$ is a collection of subgroups $C_i\subset A[p]$ of order $p$ such that
$$C_i \cap C_j = 0$$
for all $i\neq j$.
\end{enumerate}

We know that $X_{\Gamma(p^n)} \rightarrow X$ (resp., $X_{\Gamma(p)} \rightarrow X_{\Iw}$; resp., $X_{\Gamma(p)} \rightarrow X_{\Iw^+}$) is Galois with Galois group $\GSp_{2g}(\Z/p^n\Z)$ (resp., $B_{\GSp_{2g}}(\Z/p\Z)$; resp., $T_{\GSp_{2g}}(\Z/p\Z)$). %$B_{\GSp_{2g}}^+(\Z/p\Z) = \left\{ \left(\substack{\bfgamma_a \,\, \bfgamma_b\\ \quad \,\, \bfgamma_d}\right)\in B_{\GSp_{2g}}(\Z/p\Z): \bfgamma_a \text{ is diagonal}\right\}$). 
The goal of this subsection is to construct the \emph{toroidal compactification} $\overline{X}_{\Gamma}$ of $X_{\Gamma}$ determined by the fixed polyhedral decomposition $\frakS$. It is an fs log scheme satisfying the following properties:
\begin{enumerate}
    \item[(Tor1)] $\overline{X}_{\Gamma}$ is finite Kummer \'etale over $\overline{X}$; 
    \item[(Tor2)] There is a cartesian diagram \[
        \begin{tikzcd}
            X_{\Gamma} \arrow[r, hook]\arrow[d] & \overline{X}_{\Gamma}\arrow[d]\\
            X\arrow[r, hook] & \overline{X}
        \end{tikzcd}
    \] and that the log structure on $\overline{X}_{\Gamma}$ is the divisorial log structure defined by the divisor $Z_{\Gamma}:=\overline{X}_{\Gamma}\smallsetminus X_{\Gamma}$;
    \item[(Tor3)] \begin{enumerate}
        \item[(i)] If $\Gamma = \Gamma(p^n)$, then \[
            \overline{X}_{\Gamma} \rightarrow \overline{X}
        \] is Galois with Galois group $\GSp_{2g}(\Z/p^n\Z)$. 
        \item[(ii)] If $\Gamma = \Iw$, then \[
            \overline{X}_{\Gamma(p)} \rightarrow \overline{X}_{\Iw}
        \] is Galois with Galois group $B_{\GSp_{2g}}(\Z/p\Z)$. 
        \item[(iii)] If $\Gamma = \Iw^+$, then \[
            \overline{X}_{\Gamma(p)} \rightarrow \overline{X}_{\Iw^+}
        \] is Galois with Galois group $T_{\GSp_{2g}}(\Z/p\Z)$.
    \end{enumerate}
\end{enumerate}

The construction of the toroidal compactification in the case $\Gamma=\Gamma(p^n)$ is well-known. For completeness, we briefly review the construction of $\overline{X}_{\Gamma(p^n)}$ following \cite{Pilloni-Stroh-CoherentCohomologyandGaloisRepresentations}.

Notice that every $\sigma\in \frakS$ necessarily lives in the interior of $C(V/V'^{\perp})$ for a unique $V'\in \frakC$ of some rank $r\leq g$. We have the following diagram from \cite[4.1.A]{Pilloni-Stroh-CoherentCohomologyandGaloisRepresentations}: 
$$
\begin{tikzcd}
M_{V', n}\arrow[r]\arrow[rd] & M_{V', n, \sigma}\arrow[r]\arrow[d] & M_{V', n, \frakS}\arrow[ld]\\
& B_{V', n}\arrow[d]\\ & X_{V', n}
\end{tikzcd}.
$$ 
We briefly describe the objects in the diagram and refer to \cite[Appendice A]{Pilloni-Stroh-CoherentCohomologyandGaloisRepresentations} for details:\begin{enumerate}
\item[$\bullet$] Let $X_{0,V'}$ be the moduli scheme parameterising principally polarised abelian schemes over $\calO_{\C_p}$ of dimension $g-r$ equipped with a principal $N$-level structure. Let $X_{V'}$ denote the base change of $X_{0,V'}$ to $\C_p$.
\item[$\bullet$] Let $X_{V', n}$ be the finite \'etale cover of $X_{V'}$ parameterising principal $p^n$-level structures. Over $X_{V', n}$, there is a universal abelian variety $A_{V'}$. 
\item[$\bullet$] Roughly speaking, the algebraic variety $B_{V', n}$ over $X_{V', n}$ parameterises semiabelian varieties with ``principal $N$- and $p^n$-level structures'' where the semiabelian variety is an extension of $A_{V'}$ by the torus $T_{V'}:=V'\otimes_{\Z}\mathbb{G}_m$. In particular, over $B_{V', n}$, there is a universal semiabelian variety
$$0\rightarrow T_{V'}\rightarrow G_{V'}\rightarrow A_{V'}\rightarrow 0$$
together with a universal isogeny of semiabelian varieties
$$\begin{tikzcd}
T_{V'}\arrow[d, "\text{id}"] \arrow[r]& G_{V'}\arrow[d] \arrow[r] & A_{V'}\arrow[d, "p^n"]\\
T_{V'}\arrow[r] &G_{V'} \arrow[r] & A_{V'}
\end{tikzcd}$$
whose kernel induces a natural inclusion $A_{V'}[p^n]\subset G_{V'}[p^n]$. This yields a decomposition
$$G_{V'}[p^n]\simeq (V'/p^nV'\otimes\mu_{p^n})\oplus A_{V'}[p^n].$$

\item[$\bullet$] Roughly speaking, the algebraic variety $M_{V',n}$ over $B_{V', n}$ parameterises principally polarised 1-motives of type $[V/V'^{\perp}\rightarrow G_{V'}]$ together with a ``principal $p^n$-level structure''. In particular, over $M_{V',n}$, there is a universal 1-motive
$$\widetilde{M}_{V'}=[V/V'^{\perp}\rightarrow G_{V'}]$$ 
together with a universal decomposition
$$\widetilde{M}_{V'}[p^n]\simeq (V'/p^nV'\otimes\mu_{p^n})\oplus A_{V'}[p^n]\oplus (V/V'^{\perp}\otimes \Z/p^n\Z).$$
It turns out $M_{V', n}$ is a torus over $B_{V', n}$ with the torus $$\Hom\left(\frac{1}{Np^n}\Sym^2(V/V'^{\perp}), \bbG_m\right).$$ 

\item[$\bullet$] The morphism $M_{V', n}\rightarrow M_{V', n, \sigma}$ is the affine toroidal embedding attached to the cone $\sigma\in C(V/V'^{\perp})$. Let $Z_{V', n, \sigma}:=M_{V', n, \sigma}\smallsetminus M_{V',n}$ denote the closed stratum of $M_{V', n, \sigma}$. Since $\sigma$ uniquely determines $V'$, we might simply write $Z_{n, \sigma}$.
    \item[$\bullet$] The morphism $M_{V', n}\rightarrow M_{V', n, \frakS}$ is the toroidal embedding attached to the polyhedral decomposition $\frakS$. Let $Z_{V', n, \frakS}:=M_{V', n, \frakS}\smallsetminus M_{V', n}$ denote the closed stratum of $M_{V', n, \frakS}$.
\end{enumerate} 

\begin{Theorem}[$\text{\cite[Th\'{e}or\`{e}me 4.1]{Pilloni-Stroh-CoherentCohomologyandGaloisRepresentations}}$]
We have 
\begin{enumerate}
    \item[(i)] The toroidal compactification $\overline{X}_{\Gamma(p^n)}$ admits a stratification indexed by the finite set $\frakS/\widetilde{\Gamma}(p^n)$. For any $\sigma\in \frakS$, the corresponding stratum in $\overline{X}_{\Gamma(p^n)}$ is isomorphic to $Z_{V', n, \sigma}$.
    \item[(ii)] The boundary $\overline{X}_{\Gamma(p^n)}\smallsetminus X_{\Gamma(p^n)}$ is given by a normal crossing divisor. The codimension-one strata $Z_{V', n, \sigma}$ are in bijection with the irreducible components of the normal crossing divisor. Such $V'$ necessarily has rank 1.
    \item[(iii)] The toroidal compactification is compatible with change of levels. In particular, there are natural finite morphisms $\overline{X}_{\Gamma(p^n)}\rightarrow \overline{X}_{\Gamma(p^m)}$ for $n\geq m$. 
    \item[(iv)] There is a natural action of $\GSp_{2g}(\Z_p)/\Gamma(p^n)$ on $\overline{X}_{\Gamma(p^n)}$. It permutes the boundary strata accordingly.
   \end{enumerate}
\end{Theorem}

On the other hand, the case for $\Gamma = \Iw$ is carefully studied in \cite{Stroh-TorComp}. However, instead of following \emph{loc. cit.}, we propose an alternative way to obtain $\overline{X}_{\Gamma}$ with the desired properties (Tor1), (Tor2) and (Tor3). To this end, we recall a theorem of K. Fujiwara and K. Kato (\cite[Theorem 7.6]{Illusie}):

\begin{Theorem}[Fujiwara--Kato]\label{Theorem: Fujiwara--Kato}
Let $Y$ be a regular scheme, $D$ an effective divisor of $Y$ with normal crossing and $U := Y \smallsetminus D$. Equip $Y$ with the divisorial log structure defined by $D$. Then, the restriction functor \[
    \left[\begin{array}{c}
        \text{finite Kummer \'etale}  \\
        \text{cover over $Y$} 
    \end{array}\right] \rightarrow \left[\begin{array}{c}
        \text{finite \'etale}  \\
         \text{cover over $U$}
    \end{array}\right], \quad T \mapsto T\times_Y U
\] if fully faithful. The essential image of this functor consists of those finite \'etale covers over $U$ which are tamely ramified along $D$.
\end{Theorem}

In particular, when $Y$ is further a variety over a field of characteristic $0$, every finite \'etale cover over $U$ is tamely ramified along $D$. That is, one obtains an isomorphism between the finite Kummer \'etale site $Y_{\fket}$ and the finite \'etale site $U_{\fet}$.

\begin{Proposition}\label{Proposition: toridal compactification of algebraic Siegel varieties}
Let $\Gamma$ denote either $\Gamma(p^n)$ (for some $n>0$), $\Iw$, or $\Iw^+$. There exists a unique fs log scheme $\overline{X}_{\Gamma}$ over $\overline{X}$ satisfying (Tor1),  (Tor2), and (Tor3).
\end{Proposition}
\begin{proof}
Recall that $\overline{X}$ is equipped with the divisorial log structure given by the boundary divisor 
$Z = \overline{X}\smallsetminus X$ of normal crossing (by \cite[Chapter IV, Theorem 6.7 (1)]{Faltings-Chai}). Theorem \ref{Theorem: Fujiwara--Kato} yields a unique log scheme $\overline{X}_{\Gamma}$, which is finite Kummer \'etale over $\overline{X}$, extending the finite \'etale morphism $X_{\Gamma} \rightarrow X$. This shows that $\overline{X}_{\Gamma}$ satisfies (Tor1) and (Tor2). Finally, by applying a scheme-theoretic version of Lemma \ref{Kummer etale Galois cover}, we conclude that $\overline{X}_{\Gamma}$ also satisfies (Tor3). 
\end{proof}

\begin{Remark}\label{Remark: comparison of constructions of toroidal compactification}
\normalfont When $\Gamma \in \{\Gamma(p^n), \Iw\}$, one should ask whether our construction of $\overline{X}_{\Gamma}$ coincides with the ones constructed in \cite{Pilloni-Stroh-CoherentCohomologyandGaloisRepresentations} and \cite{Stroh-TorComp}. The answer to this question is affirmative. When $\Gamma = \Gamma(p^n)$, \cite[Chapter IV, Theorem 6.7(6)]{Faltings-Chai} implies that $\overline{X}_{\Gamma(p^n)}$ is finite Kummer \'etale over $\overline{X}$ with Galois group $\GSp_{2g}(\Z/p^n\Z)$. The uniqueness of $\overline{X}_{\Gamma}$ then yields the identification. For $\Gamma = \Iw$, it follows similarly by applying \cite[Théorème 3.2.7.1]{Stroh-TorComp}. 
\end{Remark}

To wrap up the subsection, we pass to the realm of adic spaces. Let $\calX_{\Gamma}$ (resp., $\overline{\calX}_{\Gamma}$) denote the adic space over $\Spa(\C_p, \calO_{\C_p})$ associated with $X_{\Gamma}$ (resp., $\overline{X}_{\Gamma}$). In particular, we refer $\overline{\calX}_{\Gamma}$ as the \emph{toroidal compactification} of $\calX_{\Gamma}$ determined by the fixed polyheral decomposition $\frakS$. It satisfies the following analogues of (Tor1), (Tor2), and (Tor3):

\begin{enumerate}
    \item[(Tor1')] The log adic space $\overline{\calX}_{\Gamma}$, equipped with the divisorial log structure given by the boundary divisor $\calZ_{\Gamma} = \overline{\calX}_{\Gamma}\smallsetminus \calX_{\Gamma}$, is finite Kummer \'etale over $\overline{\calX}$; 
    \item[(Tor2')] There is a cartesian diagram \[
        \begin{tikzcd}
            \calX_{\Gamma} \arrow[r, hook]\arrow[d] & \overline{\calX}_{\Gamma}\arrow[d]\\
            \calX\arrow[r, hook] & \overline{\calX}
        \end{tikzcd}
    \]
    \item[(Tor3')] \begin{enumerate}
        \item[(i)] If $\Gamma = \Gamma(p^n)$, then \[
            \overline{\calX}_{\Gamma} \rightarrow \overline{\calX}
        \] is Galois with Galois group $\GSp_{2g}(\Z/p^n\Z)$. 
        \item[(ii)] If $\Gamma = \Iw$, then \[
            \overline{\calX}_{\Gamma(p)} \rightarrow \overline{\calX}_{\Iw}
        \] is Galois with Galois group $B_{\GSp_{2g}}(\Z/p\Z)$. 
        \item[(iii)] If $\Gamma = \Iw^+$, then \[
            \overline{\calX}_{\Gamma(p)} \rightarrow \overline{\calX}_{\Iw^+}
        \] is Galois with Galois group $T_{\GSp_{2g}}(\Z/p\Z)$.
    \end{enumerate}
\end{enumerate}

\subsection{The perfectoid Siegel modular variety}\label{subsection: perfectoid Siegel modular variety}
Let $\frakX$ (resp., $\overline{\frakX}$) be the formal completion of $X_0$ (resp., $\overline{X}_0$) along its special fibre. Let $\frakX_{\Gamma(p^n)}$ (resp., $\overline{\frakX}_{\Gamma(p^n)}$) be the normalisation of $\frakX$ (resp., $\overline{\frakX}$) inside the rigid analytic space associated with $X_{\Gamma(p^n)}$ (resp., $\overline{X}_{\Gamma(p^n)}$). 

In order to work with the toroidal compactification at the infinite level, the authors of \cite{Pilloni-Stroh-CoherentCohomologyandGaloisRepresentations} consider modified versions $\overline{\frakX}_{\Gamma(p^n)}^{\text{mod}}$ of the formal schemes $\overline{\frakX}_{\Gamma(p^n)}$, which we briefly recall. 

Let $n\in \Z_{\geq 0}$ and let $\frakG$ be the tautological semiabelian scheme over $\overline{\frakX}_{\Gamma(p^n)}$. Let $$\pi:\frakG\rightarrow \overline{\frakX}_{\Gamma(p^n)}$$
be the natural projection with identity section $e$ and let
$$\underline{\Omega}_{\Gamma(p^n)}:=e^*\Omega^1_{\frakG/\overline{\frakX}_{\Gamma(p^n)}}. \footnote{The sheaf $\underline{\Omega}_{\Gamma(p^n)}$ is denoted by $\omega_A$ in \cite{Pilloni-Stroh-CoherentCohomologyandGaloisRepresentations}.}$$
Over $\frakX_{\Gamma(p^n)}$, composing the dual of the universal trivialisation
$$\psi_{p^n}: V\otimes_{\Z}(\Z/p^n\Z)\simeq \frakG[p^n]$$
(which becomes an isomorphism on the rigid generic fibre) and the Hodge--Tate map
$$\frakG[p^n]^{\vee}\rightarrow \underline{\Omega}_{\Gamma(p^n)}/p^n\underline{\Omega}_{\Gamma(p^n)}$$
we obtain
$$\HT_{\Gamma(p^n)}: V^{\vee}\otimes_{\Z}(\Z/p^n\Z)\rightarrow \underline{\Omega}_{\Gamma(p^n)}/p^n\underline{\Omega}_{\Gamma(p^n)}$$ 
which induces
$$\HT_{\Gamma(p^n)}\otimes \id: \big(V^{\vee}\otimes_{\Z}(\Z/p^n\Z)\big)\otimes_{\Z} \scrO_{\frakX_{\Gamma(p^n)}}\rightarrow \underline{\Omega}_{\Gamma(p^n)}/p^n\underline{\Omega}_{\Gamma(p^n)}.$$ 
According to \cite[Proposition 1.2]{Pilloni-Stroh-CoherentCohomologyandGaloisRepresentations}, this map extends to the toroidal compactification:
\begin{equation}\label{eq: extended HT map}
    \HT_{\Gamma(p^n)}\otimes \id: \big(V^{\vee}\otimes_{\Z}(\Z/p^n\Z)\big)\otimes_{\Z} \scrO_{\overline{\frakX}_{\Gamma(p^n)}}\rightarrow \underline{\Omega}_{\Gamma(p^n)}/p^n\underline{\Omega}_{\Gamma(p^n)}.
\end{equation}
More precisely, in terms of the explicit description in \S \ref{subsection: boundary strata}, \'etale locally at the boundary stratum, there is a universal semiabelian scheme $G_{V'}$ with constant toric rank sitting in an exact sequence \[
    0 \rightarrow T_{V'} \rightarrow G_{V'} \rightarrow A_{V'} \rightarrow 0
\] as well as a principally polarised 1-motive $\widetilde{M}_{V'} = [V'^{\perp}/V' \rightarrow G_{V'}]$. We consider the composition \[
    \widetilde{M}_{V'}[p^n]^{\vee} \twoheadrightarrow G_{V'}[p^n]^{\vee} \xrightarrow{\HT_{G_{V'}[p^n]^{\vee}}} \underline{\omega}_{G_{V'}}/p^n.
\] Composing this with the dual of the universal trivialisation of $\widetilde{M}_{V'}[p^n]$ and tensoring with the structure sheaf, we arrive at the desired morphism (\ref{eq: extended HT map}).

Consider the image of $\HT_{\Gamma(p^n)}\otimes \id$ and then consider its preimage inside $\underline{\Omega}_{\Gamma(p^n)}$. This yields a subsheaf $\underline{\Omega}_{\Gamma(p^n)}^{\text{mod}}\subset \underline{\Omega}_{\Gamma(p^n)}$. In fact, $\underline{\Omega}_{\Gamma(p^n)}^{\text{mod}}$ does not depend on $n$; \emph{i.e.,} if $n\geq m$ and $\overline{\frakX}_{\Gamma(p^n)} \rightarrow \overline{\frakX}_{\Gamma(p^m)}$ is the natural projection, then the pullback of $\underline{\Omega}_{\Gamma(p^m)}^{\text{mod}}$ coincides with $\underline{\Omega}_{\Gamma(p^n)}^{\text{mod}}$.

Now, let $n$ be any positive integer greater than $\frac{g}{p-1}$. Consider ideals $\scrI_1, \ldots, \scrI_g\subset \scrO_{\overline{\frakX}_{\Gamma(p^n)}}$ generated by the lifts of the determinants of the minors of rank $g,\ldots, 1$ of the map
$$\HT_{\Gamma(p^n)}\otimes \id: \big(V^{\vee}\otimes_{\Z}(\Z/p^n\Z)\big)\otimes_{\Z} \scrO_{\overline{\frakX}_{\Gamma(p^n)}}\rightarrow \underline{\Omega}_{\Gamma(p^n)}/p^n\underline{\Omega}_{\Gamma(p^n)}.$$ 
Notice that these ideals are invertible on the rigid generic fibre. Let $\widetilde{\frakX}_{\Gamma(p^n)}$ be the formal scheme obtained by consecutive formal blowups of $\overline{\frakX}_{\Gamma(p^n)}$ along these ideals. In particular, $\underline{\Omega}_{\Gamma(p^n)}^{\text{mod}}$ becomes locally free over $\widetilde{\frakX}_{\Gamma(p^n)}$.  

Let $\overline{\frakX}_{\Gamma(p^n)}^{\text{mod}}$ be the normalisation of $\widetilde{\frakX}_{\Gamma(p^n)}$ inside its adic generic fibre. We remark that the adic generic fibre of $\overline{\frakX}_{\Gamma(p^n)}^{\text{mod}}$ coincides with the one of $\overline{\frakX}_{\Gamma(p^n)}$. For any $m\geq n>\frac{g}{p-1}$, there is a natural finite morphism
$$\overline{\frakX}_{\Gamma(p^{m})}^{\text{mod}}\rightarrow \overline{\frakX}_{\Gamma(p^n)}^{\text{mod}}.$$

Notice that the adic generic fibre of $\overline{\frakX}_{\Gamma(p^n)}^{\text{mod}}$ coincides with $\overline{\calX}_{\Gamma(p^n)}$. The locally free sheaf $\underline{\Omega}_{\Gamma(p^n)}^{\text{mod}}$ gives rise to a locally free $\scrO^+_{\overline{\calX}_{\Gamma(p^n)}}$-module $\underline{\omega}^{\text{mod},+}_{\Gamma(p^n)}$ on $\overline{\calX}_{\Gamma(p^n)}$. Inverting $p$, we obtain the locally free $\scrO_{\overline{\calX}_{\Gamma(p^n)}}$-module $\underline{\omega}_{\Gamma(p^n)}$. Notice that $\underline{\omega}_{\Gamma(p^n)}$ is just the usual sheaf of invariant differentials defined using the universal semiabelian varieties.

Consider the projective limit
$$\overline{\frakX}_{\Gamma(p^{\infty})}^{\text{mod}}:=\varprojlim \overline{\frakX}_{\Gamma(p^n)}^{\text{mod}}$$
in the category of $p$-adic formal schemes. Let $\overline{\calX}_{\Gamma(p^{\infty})}$ be its adic generic fibre in the sense of \cite{Scholze-Weinstein}.

\begin{Proposition}[$\text{\cite[Proposition 4.9 \& Corollaire 4.14]{Pilloni-Stroh-CoherentCohomologyandGaloisRepresentations}}$]\label{Proposition: perfectoid toroidal compactification}
 We have \begin{enumerate}
   \item[(i)] The adic generic fibre $\overline{\calX}_{\Gamma(p^{\infty})}$ is a perfectoid space such that  $$\overline{\calX}_{\Gamma(p^{\infty})}\sim\varprojlim_{n}\overline{\calX}_{\Gamma(p^n)}$$ in the sense of \cite[Definition 2.4.1]{Scholze-Weinstein}.
   \item[(ii)] For every $n\in \Z_{\geq 0}$, the natural morphism
   $$\overline{\calX}_{\Gamma(p^{\infty})}\rightarrow \overline{\calX}_{\Gamma(p^n)}$$ is a pro-Kummer \'{e}tale Galois cover with Galois group $\Gamma(p^n)$. (Here we have abused the notation and identify $\overline{\calX}_{\Gamma(p^{\infty})}$ with the object $\varprojlim_{n}\overline{\calX}_{\Gamma(p^n)}$ in the pro-Kummer \'etale site.) Simiarly, the natural morphism
   $$\overline{\calX}_{\Gamma(p^{\infty})}\rightarrow \overline{\calX}_{\Iw} \quad (\text{resp., }\overline{\calX}_{\Gamma(p^{\infty})}\rightarrow \overline{\calX}_{\Iw^+})$$ is a pro-Kummer \'{e}tale Galois cover with Galois group $\Iw_{\GSp_{2g}}$ (resp., $\Iw_{\GSp_{2g}}^+$).
    \end{enumerate}
\end{Proposition}

\begin{Remark}
\normalfont Induced from the stratification on the finite levels, the perfectoid Siegel modular variety $\overline{\calX}_{\Gamma(p^{\infty})}$ admits a stratification by the profinite set
$$\hat{\frakS}:=\varprojlim_n \frakS/\widetilde{\Gamma}(p^n).$$
For each $\hat{\sigma}=(\sigma_n)_{n\geq 0}\in \hat{\frakS}$, the $\hat{\sigma}$-stratum is canonically isomorphic to 
$$\calZ_{\infty, \hat{\sigma}}:=\varprojlim_n\calZ_{n, \sigma_n}$$
where $\calZ_{n,\sigma_n}$ is the adic spaces given by the analytification of $Z_{n,\sigma_n}$.
\end{Remark}

Finally, we recall the construction of the Hodge--Tate period map in the case of toroidal compactification. By definition of $\underline{\omega}^{\text{mod},+}_{\Gamma(p^n)}$, the Hodge--Tate map $\HT_{\Gamma(p^n)}$ induces a map (which we abuse the notation and still denote by $\HT_{\Gamma(p^n)}$)
$$\HT_{\Gamma(p^n)}: V^{\vee}\otimes_{\Z}(\Z/p^n\Z)\rightarrow \underline{\omega}^{\text{mod},+}_{\Gamma(p^n)}/p^n\underline{\omega}^{\text{mod},+}_{\Gamma(p^n)}.$$ 
Let $\underline{\omega}^{\text{mod},+}_{\Gamma(p^{\infty})}$ and $\underline{\omega}_{\Gamma(p^{\infty})}$ denote the pullbacks of $\underline{\omega}^{\text{mod},+}_{\Gamma(p^n)}$ and $\underline{\omega}_{\Gamma(p^n)}$, respectively, to $\overline{\calX}_{\Gamma(p^{\infty})}$. Pulling back $\HT_{\Gamma(p^n)}$ to the infinite level and taking inverse limit, we obtain 
$$\HT_{\Gamma(p^{\infty})}: V^{\vee}_p\rightarrow \underline{\omega}^{\text{mod},+}_{\Gamma(p^{\infty})}$$
which induces a surjection
$$\HT_{\Gamma(p^{\infty})}\otimes \id: V^{\vee}_p\otimes_{\Z_p} \scrO^+_{\overline{\calX}_{\Gamma(p^{\infty})}}\rightarrow \underline{\omega}^{\text{mod},+}_{\Gamma(p^{\infty})}.$$
Inverting $p$, the surjection
$$\HT_{\Gamma(p^{\infty})}\otimes \id: V^{\vee}_p\otimes_{\Z_p} \scrO_{\overline{\calX}_{\Gamma(p^{\infty})}}\rightarrow \underline{\omega}_{\Gamma(p^{\infty})}$$
induces the \textbf{\textit{Hodge--Tate period map}}
$$\pi_{\HT}: \overline{\calX}_{\Gamma(p^{\infty})}\rightarrow \adicFL$$
where $\adicFL$ is the (adic) flag variety parameterising the maximal lagrangians of $V_p$.
\end{appendix}

%\printbibliography
\bibliographystyle{amsalpha}
\bibliography{Reference.bib}

@incollection {RosnerWissauer,
    AUTHOR = {R\"osner, Mirko and Weissauer, Rainer},
     TITLE = {Multiplicity one for certain paramodular forms of genus two},
 BOOKTITLE = {$L$-functions and automorphic forms},
    SERIES = {Contrib. Math. Comput. Sci.},
    VOLUME = {10},
     PAGES = {251--264},
 PUBLISHER = {Springer, Cham},
      YEAR = {2017}
}

@incollection {Weissauer-4dimGalRep,
    AUTHOR = {Weissauer, Rainer},
     TITLE = {Four dimensional {G}alois representations},
      BOOKTITLE = {Formes automorphes. II. Le cas du groupe $\rm GSp(4)$},
    PUBLISHER = {Soci\'et\'e math\'ematique de France},
   JOURNAL = {Ast\'erisque},
    NUMBER = {302},
      YEAR = {2005},
     PAGES = {67--150}
}

@book {RS-NewformGSp4,
    AUTHOR = {Roberts, Brooks and Schmidt, Ralf},
     TITLE = {Local newforms for $\mathrm{GSp}$(4)},
    SERIES = {Lecture Notes in Mathematics},
    VOLUME = {1918},
 PUBLISHER = {Springer, Berlin},
      YEAR = {2007}
}

@article {Soudry,
    AUTHOR = {Soudry, David},
     TITLE = {A uniqueness theorem for representations of {${\rm GSO}(6)$}
              and the strong multiplicity one theorem for generic
              representations of {${\rm GSp}(4)$}},
   JOURNAL = {Israel J. Math.},
    VOLUME = {58},
      YEAR = {1987},
    NUMBER = {3},
     PAGES = {257--287}
}

@article {LoefflerZerbesColeman,
    AUTHOR = {Loeffler, David and Zerbes, Sarah Livia},
     TITLE = {Rankin-{E}isenstein classes in {C}oleman families},
   JOURNAL = {Res. Math. Sci.},
  FJOURNAL = {Research in the Mathematical Sciences},
    VOLUME = {3},
      YEAR = {2016},
     PAGES = {Paper No. 29, 53},
      ISSN = {2522-0144},
   MRCLASS = {11F85 (11F67 11G40 14F30)},
  MRNUMBER = {3552987},
MRREVIEWER = {Ivan Mati\'{c}},
       DOI = {10.1186/s40687-016-0077-6},
       URL = {https://doi-org.lib-ezproxy.concordia.ca/10.1186/s40687-016-0077-6},
}

@article{BDJ,
    AUTHOR = {Barrera, Daniel and Dimitrov, Mladen and Jorza, Andrei},
     TITLE = {{$p$}-adic {$L$}-functions of {H}ilbert cusp forms and the
              trivial zero conjecture},
   JOURNAL = {J. Eur. Math. Soc. (JEMS)},
  FJOURNAL = {Journal of the European Mathematical Society (JEMS)},
    VOLUME = {24},
      YEAR = {2022},
    NUMBER = {10},
     PAGES = {3439--3503},
      ISSN = {1435-9855,1435-9863},
   MRCLASS = {11F67 (11F41 11G40)},
  MRNUMBER = {4432904},
}

@article{KretShin,
  title={Galois representations for general symplectic groups},
  author={Kret, Arno and Shin, Sug Woo},
  journal={Journal of the European Mathematical Society},
  year={2022}
}

@incollection {Diao,
    AUTHOR = {Diao, Hansheng and Lan, Kai-Wen and Liu, Ruochuan and Zhu,
              Xinwen},
     TITLE = {Logarithmic adic spaces: some foundational results},
 BOOKTITLE = {{$p$}-adic {H}odge theory, singular varieties, and non-abelian
              aspects},
    SERIES = {Simons Symp.},
     PAGES = {65--182},
 PUBLISHER = {Springer, Cham},
      YEAR = {2023},
      ISBN = {978-3-031-21549-0},
   MRCLASS = {14A21},
  MRNUMBER = {4592580},
}

@article{Diao-Lan-Liu-Zhu, 
    AUTHOR = {Diao, Hansheng and Lan, Kai-Wen and Liu, Ruochuan and Zhu,
              Xinwen},
     TITLE = {Logarithmic {R}iemann-{H}ilbert correspondences for rigid
              varieties},
   JOURNAL = {J. Amer. Math. Soc.},
  FJOURNAL = {Journal of the American Mathematical Society},
    VOLUME = {36},
      YEAR = {2023},
    NUMBER = {2},
     PAGES = {483--562},
      ISSN = {0894-0347,1088-6834},
   MRCLASS = {14F40 (14D07 14F30 14G22 14G35)},
  MRNUMBER = {4536903},
}

@misc{DY-2023,
    author = {Hansheng Diao and Zijian Yao}, 
    title = {The {H}alo Conjecture for $\mathrm{GL}_2$}, 
    year = {2023}, 
    howpublished = {Preprint. Avaialble at: \url{https://arxiv.org/abs/2302.07987}}
}

@incollection{Illusie,
     author = {Illusie, Luc},
     title = {An overview of the work of {K}. {F}ujiwara, {K}. {K}ato, and {C}. {N}akayama on logarithmic \'etale cohomology},
     booktitle = {Cohomologies $p$-adiques et applications arithm\'etiques (II)},
     editor = {Berthelot, Pierre and Fontaine, Jean-Marc and Illusie, Luc and Kato, Kazuya and Rapoport, Michael},
     series = {Ast\'erisque},
     publisher = {Soci\'et\'e math\'ematique de France},
     number = {279},
     year = {2002},
     pages = {271-322},
     zbl = {1052.14005},
     mrnumber = {1922832},
     language = {en},
}

@misc{stacks-project,
  author       = {The {Stacks project authors}},
  shorthand    = {Stacks},
  title        = {The Stacks project},
  howpublished = {\url{https://stacks.math.columbia.edu}},
  year         = {2021},
}

@article{AI-2020,
    title={Overconvergent de {R}ham {E}ichler--{S}himura morphisms}, 
    journal={Journal of the Institute of Mathematics of Jussieu}, 
    publisher={Cambridge University Press}, 
    author={Andreatta, Fabrizio and Iovita, Adrian}, 
    year={2022}, 
    pages={1–57},
}

@article{AIP-2015,
    author = {Fabrizio Andreatta and Adrian Iovita and Vincent Pilloni},
    journal = {Annals of Mathematics},
    number = {2},
    pages = {623--697},
    publisher = {Annals of Mathematics},
    title = {$p$-adic families of {S}iegel modular cuspforms},
    volume = {181},
    year = {2015}
}

@article{AIS-2015, 
    title={Overconvergent {E}ichler–{S}himura isomorphisms}, 
    volume={14}, 
    number={2}, 
    journal={Journal of the Institute of Mathematics of Jussieu}, 
    publisher={Cambridge University Press}, 
    author={Andreatta, Fabrizio and Iovita, Adrian and Stevens, Glenn}, 
    year={2015}, 
    pages={221–274}}

@misc{BHW-2019,
     author = {Christopher Birkbeck and Ben Heuer and Chris Williams},
     title = {Overconvergent {H}ilbert modular forms via perfectoid modular varieties},
     year = {2019},
     howpublished = {Preprint. 	arXiv:1902.03985}
     }

@article{CHJ-2017,
    author = {Przemyslaw Chojecki and David Hansen and Christian Johansson},
    language = {en},
    title = {Overconvergent Modular Forms and Perfectoid Shimura Curves},
    journal = {Documenta Mathematica 2017},
    volume = {vol. 22},
    pages = {1431-0643},
    publisher = {Fak. für Mathematik, Universität Bielefeld, D-33501 Bielefeld},
    year = {2017}
}

@book{Huber-2013, 
    author = {Roland Huber}, 
    year = {2013}, 
    title = {Étale Cohomology of Rigid Analytic Varieties and Adic Spaces}, 
    publisher = {Springer-Verlag},
    series = {Aspects of Mathematics},
    volume = {30}
}

@article {LanKS,
    AUTHOR = {Lan, Kai-Wen},
     TITLE = {Toroidal compactifications of {PEL}-type {K}uga families},
   JOURNAL = {Algebra Number Theory},
  FJOURNAL = {Algebra \& Number Theory},
    VOLUME = {6},
      YEAR = {2012},
    NUMBER = {5},
     PAGES = {885--966},
      ISSN = {1937-0652},
   MRCLASS = {11G18 (11G15 14D06)},
  MRNUMBER = {2968629},
MRREVIEWER = {Mahesh Kakde},
}

@inbook{Scholze-perfectoid-survey,
    author = {Peter Scholze}, 
    title = {Perfectoid spaces: A survey},
    booktitle = {Current Developments in Mathematics, 2012},
    year = {2012},
    editor = {David Jerison and Mark Kisin and Tomasz Mrowka and Richard Stanley and Horng-Tzer Yau and Shing-Tung Yau},
    publisher = {International Press of Boston, Inc.}
}

@article{Scholze_2013, 
    title={$p$-adic {H}odge theory for rigid--analytic varieties}, 
    volume={1}, 
    DOI={10.1017/fmp.2013.1}, 
    journal={Forum of Mathematics, Pi}, 
    publisher={Cambridge University Press}, 
    author={Peter Scholze}, 
    year={2013}, 
}

@article{Scholze-2015,
    author = {Peter Scholze},
    journal = {Annals of Mathematics},
    number = {3},
    pages = {945--1066},
    publisher = {Annals of Mathematics},
    title = {On torsion in the cohomology of locally symmetric varieties},
    volume = {182},
    year = {2015}
}

@article{Scholze-Weinstein,
    author = {Peter Scholze and Jared Weinstein},
    journal = {Cambridge Journal of Mathematics}, 
    title = {Moduli of $p$-divisible groups},
    volume = {1},
    number = {2}, 
    year = {2013},
    page = {145--237},
    doi = {https://dx.doi.org/10.4310/CJM.2013.v1.n2.a1}
}

@book{Scholze-Weinstein-Berkeley,
    title = {Berkeley lectures on $p$-adic geometry}, 
    author = {Peter Scholze and Jared Weinstein}, 
    year = {2020}, 
    series = {Annals of Mathematics Studies}, 
    publisher = {Princeton University Press},
}

@article{Hansen-PhD, 
    author = {David Hansen},
    title = {Universal eigenvarieties, trianguline Galois representations, and $p$-adic {L}anglands functoriality},
    journal = {Journal für die reine und angewandte Mathematik}, 
    year = {2017},
    issue = {730},
    pages = {1--64}
}

@article{Johansson-Newton,
    title = "Extended eigenvarieties for overconvergent cohomology",
    author = "Christian Johansson and James Newton",
    year = "2019",
    month = "2",
    day = "13",
    language = "English",
    volume = "13",
    pages = "93--158",
    journal = "Algebra and Number Theory",
    issn = "1937-0652",
    publisher = "Mathematical Sciences Publishers",
    number = "1",
}

@inbook{Buzzard_2007, 
    place={Cambridge}, 
    series={London Mathematical Society Lecture Note Series}, 
    title={Eigenvarieties}, 
    booktitle={$L$-Functions and {G}alois {R}epresentations}, 
    publisher={Cambridge University Press}, 
    author={Buzzard, Kevin}, 
    editor={Burns, David and Buzzard, Kevin and Nekovář, Jan}, 
    year={2007}, 
    pages={59–120}, 
    collection={London Mathematical Society Lecture Note Series}}

@book{Fulton-Harris,
    author = {William Fulton and Joe Harris}, 
    title = {Representation Theory: a first course}, 
    publisher = {Springer, New York}, 
    series = {Graduate Texts in Mathematics},
    volume = {129}, 
    year = {1991},
    doi = {https://doi.org/10.1007/978-1-4612-0979-9}
}

@article{Stroh-TorComp,
    author = {Stroh, Benoît},
    journal = {Bulletin de la Société Mathématique de France},
    number = {2},
    pages = {259-315},
    publisher = {Société mathématique de France},
    title = {Compactification de variétés de {S}iegel aux places de mauvaise réduction},
    url = {http://eudml.org/doc/272517},
    volume = {138},
    year = {2010},
}

@article{Pilloni-Stroh-CoherentCohomologyandGaloisRepresentations,
    author="Pilloni, Vincent and Stroh, Beno{\^i}t",
    title="Cohomologie coh{\'e}rente et repr{\'e}sentations {G}aloisiennes",
    journal="Annales math{\'e}matiques du Qu{\'e}bec",
    year="2016",
    month="Jun",
    day="01",
    volume="40",
    number="1",
    pages="167--202",
    issn="2195-4763",
    doi="10.1007/s40316-015-0056-0",
    url="https://doi.org/10.1007/s40316-015-0056-0"
}

@book{Faltings-Chai,
    author = {Faltings, Gred and Chai, Ching-Li}, 
    title = {Degeneration of Abelian Varieties},
    series = {Ergebnisse der Mathematik und ihrer Grenzgebiete},
    publisher = {Springer-Verlag Berlin Heidelberg}, 
    year = {1990}
}

@book{NSW-cohomology,
    author = {Jürgen Neukirch and Alexander Schmidt and Kay Wingberg}, 
    title = {Cohomology of Number Fields},
    publisher = {Springer, Berlin, Heidelberg},
    year = {2008},
    series = {Grundlehren der mathematischen Wissenschaften}, 
    volume = {323},
}

@misc{Ash-Stevens,
    author = {Avner Ash and Glenn Stevens},
    title = {$p$-adic deformations of arithmetic cohomology},
    howpublished = {Preprint. Available at \url{http://math.bu.edu/people/ghs/preprints/Ash-Stevens-02-08.pdf}},
    year = {2008}
}

@article{Kisin-2003,
    author = {Mark Kisin}, 
    title = {Overconvergent modular forms and the {F}ontaine--{M}azur conjecture}, 
    journal = {Inventiones mathematicae},
    volume = {153}, 
    year = {2003}, 
    page = {373–-454},
}

@article{Hida_2002, 
    title={Control theorems of coherent sheaves on {S}himura varieties of {PEL} type}, 
    volume={1}, 
    number={1}, 
    journal={Journal of the Institute of Mathematics of Jussieu}, 
    publisher={Cambridge University Press}, 
    author={Hida, Haruzo}, 
    year={2002}, 
    pages={1–76}
}

@article{Sen-analytic,
    author = {Shankar Sen},
    journal = {Annals of Mathematics},
    number = {3},
    pages = {647--661},
    publisher = {Annals of Mathematics},
    title = {The Analytic Variation of $p$-Adic {H}odge Structure},
    volume = {127},
    year = {1988}
}

@article {FaltingsHT,
    AUTHOR = {Faltings, Gerd},
     TITLE = {Hodge-{T}ate structures and modular forms},
   JOURNAL = {Math. Ann.},
    VOLUME = {278},
      YEAR = {1987},
    NUMBER = {1-4},
     PAGES = {133--149},
}

@misc{Hansen-Iwasawa, 
    title = {Iwasawa theory of overconvergent modular forms, {I}: {C}ritical-slope $p$-adic {$L$}-functions}, 
    author = {David Hansen}, 
    year = {2015}, 
    howpublished = {Preprint. Available at: \url{https://arxiv.org/abs/1508.03982}}
}

@article{Hansen-2020,
	Author = {Hansen, David},
	Journal = {Compositio Mathematica},
	Number = {2},
	Pages = {299-324},
	Title = {Vanishing and comparison theorems in rigid analytic geometry},
	Volume = {156},
	Year = {2020}}

@article{Fargues-canonical,
    author = {Fargues, Laurent},
    title = {La filtration canonique des points de torsion des groupes $p$-divisibles},
    journal = {Annales scientifiques de l'\'Ecole Normale Sup\'erieure},
    pages = {905--961},
    publisher = {Soci\'et\'e math\'ematique de France},
    volume = {4e s{\'e}rie, 44},
    number = {6},
    year = {2011},
    zbl = {1331.14044},
    mrnumber = {2919687},
    language = {fr},
    url = {www.numdam.org/item/ASENS_2011_4_44_6_905_0/}
}

@book{Fargues-Genestier-Lafforgue,
    title = {L’isomorphisme entre les tours de {L}ubin-{T}ate et de {D}rinfeld}, 
    author = {Laurent Fargues and Alain Genestier and Vincent Lafforgue}, 
    year = {2008}, 
    publisher = {Birkhäuser Basel}, 
    series = {Progress in Mathematics}, 
    volume = {262},
}

@article{conrad-conn,
    author = {Conrad, Brian},
    journal = {Annales de l'institut Fourier},
    language = {eng},
    number = {2},
    pages = {473-541},
    publisher = {Association des Annales de l'Institut Fourier},
    title = {Irreducible components of rigid spaces},
    volume = {49},
    year = {1999},
}

@incollection{Deligne-Shimura,
    author = {Deligne, Pierre},
    title = {Travaux de {S}himura},
    booktitle = {S\'eminaire Bourbaki : vol. 1970/71, expos\'es 382-399},
    author = {Collectif},
    series = {S\'eminaire Bourbaki},
    note = {talk:389},
    publisher = {Springer-Verlag},
    number = {13},
    year = {1971},
    zbl = {0225.14007},
    mrnumber = {498581},
    language = {fr},
    url = {http://www.numdam.org/item/SB_1970-1971__13__123_0/}
}

@misc{Berthelot-rigid_cohomology,
    title = {Cohomologie rigide et cohomologie rigide à supports propres}, 
    author = {Pierre Berthelot}, 
    year = {1996},
    howpublished = {Preprint. Available at: \url{https://perso.univ-rennes1.fr/pierre.berthelot/publis/Cohomologie_Rigide_I.pdf}},
}

@misc{BP-highercolemannots,
    title = {Notes on higher {C}oleman theory},
    author = {George Boxer and Vincent Pilloni}, 
    year = {2020},
    howpublished = {Preprint available at \url{https://www.ma.imperial.ac.uk/~gboxer/montrealnotes.pdf}}
}

@misc{Boxer--Pilloni--higherColeman,
    title = {Higher {C}oleman theory},
    author = {George Boxer and Vincent Pilloni}, 
    year = {2020},
    howpublished = {Preprint available at \url{https://perso.ens-lyon.fr/vincent.pilloni/HigherColeman.pdf}},
}

@article{Lutkebohmert,
    title = {Der Satz von {R}emmert-{S}tein in der nichtarchimedischen Funktionen}, 
    author = {Werner Lütkebohmert}, 
    journal = {Mathematische Zeitschrift}, 
    volume = {139}, 
    pages = {69 -- 84}, 
    year = {1974}, 
    doi = {https://doi.org/10.1007/BF01194146},
}

@article{Amice,
     author = {Amice, Yvette},
     title = {Interpolation $p$-adique},
     journal = {Bulletin de la Soci\'et\'e Math\'ematique de France},
     pages = {117 -- 180},
     publisher = {Soci\'et\'e math\'ematique de France},
     volume = {92},
     year = {1964},
}

@article{Larzard,
     author = {Lazard, Michel},
     title = {Groupes analytiques $p$-adiques},
     journal = {Publications Math\'ematiques de l'IH\'ES},
     pages = {5--219},
     publisher = {Institut des Hautes \'Etudes Scientifiques},
     volume = {26},
     year = {1965},
}

@article{Wu-2020,
    author = {Ju-Feng Wu}, 
    title = {A pairing on the cuspidal eigenvariety for $\mathrm{GSp}_{2g}$ and the ramification locus},
    year = {2021}, 
    journal = {Documenta Mathematica}, 
    volume = {26}, 
    pages = {675 -- 711}, 
}

@misc{Rodriguez-BGG,
    title = {Locally analytic completed cohomology}, 
    author = {Juan Esteban Rodríguez Camargo}, 
    year = {2022}, 
    howpublished = {Preprint. Available at: \url{https://arxiv.org/abs/2209.01057}}, 
    shorthand = {RC22}
}

%\newpage

%\textbf{Data Availability}

%Data sharing not applicable to this article as no datasets were generated or analyzed during the current study.

%\vspace{3mm}

%\textbf{Conflict of interest}

%The authors declare that there is no conflict of interest in this work.

%\vspace{3mm}
\vspace{15mm}

\begin{tabular}{l}
    H.D.\\
    Tsinghua University,   \\
    Yau Mathematical Sciences Center\\
    Beijing, China\\
    \textit{E-mail address: }\texttt{hdiao@mail.tsinghua.edu.cn }\\
    \\
    G.R. \\
    Concordia University   \\
    Department of Mathematics and Statistics\\
    Montr\'{e}al, Qu\'{e}bec, Canada\\
    \textit{E-mail address: }\texttt{giovanni.rosso@concordia.ca }\\
    \\
    J.-F.W.\\
    School of Mathematics and Statistics\\
    University College Dublin\\
    Belfield, Dublin 4, Ireland\\
    \textit{E-mail address: }\texttt{ju-feng.wu@ucd.ie}
\end{tabular}

\end{document}